\documentclass{memo-l}

\usepackage[all]{xy}
\usepackage{eucal}
\usepackage{amssymb}

\newtheorem{theorem}{Theorem}[chapter]
\newtheorem{lemma}[theorem]{Lemma}
\newtheorem{proposition}[theorem]{Proposition}
\newtheorem{corollary}[theorem]{Corollary}
\newtheorem*{proposition*}{Proposition}

\theoremstyle{definition}
\newtheorem{definition}[theorem]{Definition}
\newtheorem{example}[theorem]{Example}

\newtheorem{construction}[theorem]{Construction}

\theoremstyle{remark}
\newtheorem{remark}[theorem]{Remark}

\numberwithin{section}{chapter}
\numberwithin{equation}{chapter}

\newcommand{\N}{\mathrm{N}}
\newcommand{\Fun}{\mathrm{Fun}}

\newcommand{\NFin}{\N\mathcal{F}\mathrm{in}_*}

\makeindex

\begin{document}

\frontmatter

\title{Goodwillie Approximations to Higher Categories}

\author{Gijs Heuts}
\address{Mathematical Institute, Utrecht University, Budapestlaan 6, 3584CD Utrecht, the Netherlands}
\thanks{During part of the writing of this paper the author was supported by the European Research Council (ERC) under the European Union's Horizon 2020 research and innovation programme (grant agreement No. 682922).}

\date{}



\begin{abstract}
We construct a Goodwillie tower of categories which interpolates between the category of pointed spaces and the category of spectra. This tower of categories refines the Goodwillie tower of the identity functor in a precise sense. More generally, we construct such a tower for a large class of $\infty$-categories $\mathcal{C}$ and classify such Goodwillie towers in terms of the derivatives of the identity functor of $\mathcal{C}$. As a particular application we show how this provides a model for the homotopy theory of simply-connected spaces in terms of coalgebras in spectra with Tate diagonals. Our classification of Goodwillie towers simplifies considerably in settings where the Tate cohomology of the symmetric groups vanishes. As an example we apply our methods to rational homotopy theory. Another application identifies the homotopy theory of $p$-local spaces with homotopy groups in a certain finite range with the homotopy theory of certain algebras over Ching's spectral version of the Lie operad. This is a close analogue of Quillen's results on rational homotopy. 
\end{abstract}

\maketitle

\tableofcontents

\chapter*{Introduction}

Write $\mathcal{S}_\ast$ for the category of pointed spaces. The \emph{Goodwillie tower of the identity functor} \cite{goodwillie3} on $\mathcal{S}_\ast$ gives, for each pointed space $X$, a tower of spaces
\[
\xymatrix{
&& \vdots \ar[d] & \\
&\vdots& P_3 X \ar[d] & \\
&& P_2 X \ar[d] & \\
X \ar[rr]\ar[urr]\ar[uurr] && P_1 X & \Omega^\infty \Sigma^\infty X, \ar@{=}[l]
}
\]
which interpolates between the stable and unstable homotopy type of $X$. The homotopy fiber of the map $P_n X \rightarrow P_{n-1} X$ is usually denoted $D_n X$ and may be expressed as follows:
\begin{equation*}
D_n X = \Omega^\infty \bigl((\partial_n \mathrm{id} \wedge X^{\wedge n})_{h\Sigma_n} \bigr).
\end{equation*}
Here $\partial_n \mathrm{id}$ is a spectrum carrying an action of the symmetric group $\Sigma_n$ and is called the \emph{$n$th derivative} of the identity functor. Ching \cite{ching} showed that the symmetric sequence of derivatives $\partial_\ast \mathrm{id}$ has a natural operad structure; furthermore, this operad is the cobar construction of the commutative cooperad and as such could be considered as the (desuspension of) the Lie operad in the category of spectra. In particular, taking integral homology reproduces (a degree shift of) the ordinary Lie operad in the category of abelian groups. 

We will refine this picture in the following way: we will produce a Goodwillie tower of \emph{categories}
\[
\xymatrix{
&& \vdots \ar[d] & \\
&\vdots& \mathcal{P}_3 \mathcal{S}_\ast \ar[d] & \\
&& \mathcal{P}_2 \mathcal{S}_\ast \ar[d] & \\
\mathcal{S}_\ast \ar[rr]\ar[urr]\ar[uurr] && \mathcal{P}_1 \mathcal{S}_\ast & \mathrm{Sp}, \ar@{=}[l].
}
\]
which interpolates between the category of pointed spaces and the category $\mathrm{Sp}$ of spectra. All functors in this diagram are left adjoints. The unit of the adjunction
\[
\xymatrix{
\mathcal{S}_\ast \ar@<.5ex>[r] & \mathcal{P}_n\mathcal{S}_\ast \ar@<.5ex>[l]
}
\]
is the natural transformation $\mathrm{id}_{\mathcal{S}_\ast} \rightarrow P_n$ described above. We will in fact construct such Goodwillie towers for a large class of $\infty$-categories $\mathcal{C}$. The functor $\mathcal{C} \rightarrow \mathcal{P}_n\mathcal{C}$ satisfies a universal property with respect to $n$-excisive functors out of $\mathcal{C}$, see Theorem \ref{thm:TheoremA}.

The main result of this paper, Theorem \ref{thm:classification}, is a classification of such Goodwillie towers of $\infty$-categories in terms of certain \emph{Tate cohomology} \cite{greenleesmay, klein, ambidexterity} associated to the symmetric groups acting on the derivatives of the identity functor. An informal description of Tate cohomology is as follows. Say $X$ is an object of a stable homotopy theory (for example a spectrum or a chain complex) and $X$ carries an action of a finite group $G$. We can then form the homotopy coinvariants $X_{hG}$ and the homotopy invariants $X^{hG}$ of this action. Furthermore, there is a natural \emph{norm map}
\begin{equation*}
\mathrm{Nm}: X_{hG} \longrightarrow X^{hG}.
\end{equation*}
Informally speaking, this map is induced by summing over the group $G$. The associated Tate cohomology $X^{tG}$ is defined to be the cofiber of this map. Under special circumstances such Tate cohomology objects will be contractible. For example, when working in chain complexes over a field in which the order of $G$ is invertible, dividing by this order then induces an inverse to the norm map. In particular, in the homotopy theory of chain complexes over $\mathbb{Q}$ all Tate cohomology vanishes. A more striking example is the homotopy theory of $K(n)$-local spectra, where $K(n)$ is the $n$th Morava $K$-theory at a prime $p$ \cite{greenleessadofsky}.

An informal summary of the main consequence of Theorem \ref{thm:classification} is below; a precise statement is Corollary \ref{cor:notatecohomology2}, see also Remark \ref{rmk:notatecohomology}.

\begin{theorem}
\label{thm:informalmain}
If $\mathcal{C}$ is an $\infty$-category such that the Tate cohomology of the symmetric groups $\Sigma_k$ vanishes in the stabilization $\mathrm{Sp}(\mathcal{C})$ of $\mathcal{C}$ for $k \leq n$, then there is a canonical equivalence of $\infty$-categories
\begin{equation*}
\mathcal{P}_n\mathcal{C} \simeq \mathcal{P}_n\mathrm{Alg}(\partial_* \mathrm{id}_{\mathcal{C}}).
\end{equation*}
Here $\mathrm{Alg}(\partial_*\mathrm{id}_{\mathcal{C}})$ denotes the $\infty$-category of algebras for the operad formed by the derivatives of the identity functor of $\mathcal{C}$.
\end{theorem}

\begin{remark}
In this introduction we suppress the technical assumptions needed to `do Goodwillie calculus in an $\infty$-category $\mathcal{C}$'. Also, for simplicity of language, we are assuming a natural operad structure on the derivatives of the identity functor of $\mathcal{C}$; such a structure has recently been described in this generality by Ching \cite{chingderivatives}. However, in the main body of this paper we will circumvent the direct use of such operad structures by working in a Koszul dual  setting.
\end{remark}

A consequence of Theorem \ref{thm:informalmain} is that the $\infty$-category of rational pointed spaces and the $\infty$-category of rational differential graded Lie algebras have equivalent Goodwillie towers. Using this observation, one can reprove some of Quillen's results on rational homotopy theory (see Theorem \ref{thm:rationalhomotopy}). Our proof of Theorem \ref{thm:classification} also gives an explicit (although slightly more complicated) description of the Goodwillie tower when Tate cohomology does not vanish. As an example we highlight the case $\mathcal{C} = \mathcal{S}_*$ of pointed spaces, although the general case is much the same. The Goodwillie tower of the $\infty$-category $\mathcal{S}_*$ can be described in terms of the $\infty$-category of \emph{Tate coalgebras} in spectra. A Tate coalgebra is a spectrum $E$ which is first of all a (nonunital) commutative coalgebra, meaning it is equipped with comultiplication maps
\begin{equation*}
\delta_n: E \rightarrow (E^{\wedge n})^{h\Sigma_n}
\end{equation*}
for $n \geq 2$, which are in a suitable sense compatible for different values of $n$. Furthermore, the compositions 
\begin{equation*}
E \xrightarrow{\delta_n} (E^{\wedge n})^{h\Sigma_n} \rightarrow (E^{\wedge n})^{t\Sigma_n}
\end{equation*}
should be equipped with homotopies $h_n$ to certain maps
\begin{equation*}
\tau_n: E \rightarrow (E^{\wedge n})^{t\Sigma_n}
\end{equation*}
which we refer to as the \emph{Tate diagonals}. We will construct these maps by induction on $n$. The first Tate diagonal $\tau_2$ has been considered previously in various contexts \cite{kleinmoduli,LNR} and can be characterized as follows. The expression $(E \wedge E)^{t\Sigma_2}$ is an \emph{exact} functor of $E$ (i.e., it preserves cofiber sequences). Therefore $\tau_2$ is essentially uniquely determined by its values on suspension spectra $\Sigma^\infty X$, where it is simply constructed from the diagonal map of $X$:
\begin{equation*}
\Sigma^\infty X \xrightarrow{\Delta} (\Sigma^\infty X \wedge \Sigma^\infty X)^{h\Sigma_2} \rightarrow (\Sigma^\infty X \wedge \Sigma^\infty X)^{t\Sigma_2}.
\end{equation*}
For $n > 2$ the Tate diagonal $\tau_n$ is constructed inductively. It need not exist for a general spectrum $E$, but if $E$ is equipped with comultiplication maps $\delta_k$ for $k < n$ and homotopies $h_k$ corresponding to $\tau_k$ for $k<n$, we will show that $E$ can naturally be equipped with a further Tate diagonal $\tau_n$. The homotopy $h_n$ to be defined must then also be equipped with certain coherence data relating it to $\delta_k$ for $k < n$. We will begin making this discussion precise in Section \ref{sec:informalconstr} by describing the relevant coherence more explicitly. 

If $E$ is a suspension spectrum $\Sigma^\infty X$ with its evident coalgebra structure, the Tate diagonals $\tau_n$ coincide with the obvious maps 
\begin{equation*}
\Sigma^\infty X \rightarrow (\Sigma^\infty X^{\wedge n})^{t\Sigma_n}
\end{equation*}
defined as above. This provides a functor
\begin{equation*}
\mathcal{S}_* \rightarrow \mathrm{coAlg}_{\mathrm{Tate}}(\mathrm{Sp}^\otimes): X \mapsto \Sigma^\infty X,
\end{equation*}
with $\mathrm{coAlg}_{\mathrm{Tate}}(\mathrm{Sp}^\otimes)$ denoting the $\infty$-category of Tate coalgebras in spectra. We prove the following in Section \ref{subsec:modulisuspension}:

\begin{theorem}
\label{thm:Tatecoalgebras}
The functor above gives an equivalence of $\infty$-categories 
\begin{equation*}
\mathcal{S}_*^{\geq 2} \rightarrow \mathrm{coAlg}_{\mathrm{Tate}}(\mathrm{Sp}^\otimes)^{\geq 2}
\end{equation*}
from the $\infty$-category of simply connected pointed spaces to the $\infty$-category of simply-connected Tate coalgebras.
\end{theorem}

One can think of this result as a refinement of the coalgebra model of rational homotopy theory. At the same time there is a close connection between Theorem \ref{thm:Tatecoalgebras} and Mandell's work on $p$-adic homotopy theory, which we will explore in joint work with Nikolaus. From another perspective, the theorem above gives a (somewhat) concrete description of what it means for a simply-connected spectrum to be a coalgebra for the comonad $\Sigma^\infty\Omega^\infty$. It has recently been pointed out to us that Theorem \ref{thm:Tatecoalgebras} provides answers to some questions raised by Klein in his work on moduli of suspension spectra \cite{kleinmoduli} (see also the appendix of \cite{kleinpeter}, which suggests a picture close to what we described above). 

Write $\mathbf{L}$ for Ching's Lie operad in spectra, whose terms are the derivatives of the identity functor $\partial_*\mathrm{id}$. Another illustration of the use of Theorem \ref{thm:informalmain} is the following:

\begin{theorem}
\label{thm:truncatedspaces}
Let $p$ be a prime and $n \geq 2$ an integer. Write $\mathcal{S}_*^{[n,p(n-1)]}$ for the $\infty$-category of pointed spaces $X$ such that $\pi_i X = 0$ if $i$ is not contained in the interval $[n,p(n-1)]$. Similarly, write $\mathrm{Alg}(\mathbf{L})^{[n,p(n-1)]}$ for the $\infty$-category of algebras for Ching's operad $\mathbf{L}$ in spectra which have nontrivial homotopy groups concentrated in the same range. After inverting $(p-1)!$ (or localizing at $p$) there is an equivalence of $\infty$-categories
\begin{equation*}
\mathcal{S}_*^{[n,p(n-1)]} \longrightarrow \mathrm{Alg}(\mathbf{L})^{[n,p(n-1)]}.
\end{equation*}
Under this equivalence, the homotopy groups of a pointed space $X \in \mathcal{S}_*^{[n,p(n-1)]}$ are isomorphic to the homotopy groups of the spectrum underlying the corresponding $\mathbf{L}$-algebra.
\end{theorem}

The $\mathbf{L}$-algebra corresponding to a space $X$ in Theorem \ref{thm:truncatedspaces} can be realized by a variant of the construction of the topological Andr\'{e}-Quillen cohomology of the commutative ring spectrum $\mathbf{D}X$, the Spanier-Whitehead dual of $X$, although the precise formula seems too involved to be of much practical use. It is tempting to think that the $\mathbf{L}$-algebra structure of Theorem \ref{thm:truncatedspaces} in particular produces the Whitehead products on $\pi_* X$, but we do not have a proof of this. 

The theorem above is a close analogue of rational homotopy theory; however, the kinds of Lie algebras featuring in this result are Lie algebras of spectra, rather than of chain complexes. Other interesting settings where Tate cohomology vanishes and our theorem can be fruitfully applied arise in chromatic homotopy theory. Indeed, as alluded to above, a result of Greenlees and Sadofsky \cite{greenleessadofsky} states that all norm maps are equivalences in the homotopy theory $\mathrm{Sp}_{K(n)}$ of $K(n)$-local spectra, where $n > 0$ and $K(n)$ is the $n$th Morava $K$-theory at a prime $p$. Furthermore, Kuhn \cite{kuhntate} proved that the same is true in $\mathrm{Sp}_{T(n)}$, the homotopy theory of $T(n)$-local spectra, where now $T(n)$ is the telescope of a $v_n$-self map on a finite $p$-local type $n$ spectrum. In \cite{unstableperiodicity} we set up an analogue of Quillen's results in these settings, providing a Lie algebra model for $v_n$-periodic unstable homotopy theory at every height $n$. This description also yields a new perspective on the Bousfield-Kuhn functors \cite{kuhntelescopic}.

\subsection*{Acknowledgments}

I am very grateful to my PhD advisor Jacob Lurie for countless suggestions and insights. Furthermore, I would like to thank Omar Antol\'{i}n Camarena, Mark Behrens, and Mike Hopkins for conversations providing inspiration, either direct or indirect, that went into writing this paper. I thank John Klein for pointing out the relation between this work and his earlier work on the moduli of suspension spectra. Finally, I thank the anonymous referee for many useful suggestions. During part of the writing of this paper the author was supported by the European Research Council (ERC) under the European Union's Horizon 2020 research and innovation programme (grant agreement No. 682922).

\aufm{Gijs Heuts}

\mainmatter

\chapter{Main results}
\label{sec:mainresults}

Throughout this paper we will use \emph{$\infty$-categories}, or \emph{quasicategories}, as our preferred formalism for higher category theory. Our reason for using these instead of for example model categories is that there are good and tractable notions of homotopy limits and colimits of $\infty$-categories, which we will make extensive use of. The works of Joyal \cite{joyalpaper, joyal} and Lurie \cite{htt} are the basic references.

Our results will apply to the class of $\infty$-categories that are \emph{pointed}, meaning they have an object that is both initial and final, and \emph{compactly generated}. Recall from Section 5.5.7 of \cite{htt} that an $\infty$-category is compactly generated if it is presentable and $\omega$-accessible. Alternatively, an $\infty$-category is compactly generated if and only if it is equivalent to an $\infty$-category of the form $\mathrm{Ind}(\mathcal{D})$, where $\mathcal{D}$ is a small $\infty$-category which admits finite colimits. Here $\mathrm{Ind}(\mathcal{D})$ denotes the category of $\mathrm{Ind}$-objects in $\mathcal{D}$. It is the free cocompletion of $\mathcal{D}$ with respect to filtered colimits. 

The typical examples of pointed compactly generated $\infty$-categories we have in mind are $\mathcal{S}_\ast$ (the $\infty$-category of pointed spaces) and $\mathrm{Alg}_\mathbf{O}(\mathrm{Sp})$, the $\infty$-category of $\mathbf{O}$-algebras in spectra, for $\mathbf{O}$ a nonunital operad in spectra. We will also consider variants of this latter category, for example replacing spectra with chain complexes over a field.

\emph{Goodwillie calculus} was originally introduced in the context of spaces and spectra \cite{goodwillie3}, but later generalized to apply to more general homotopy theories \cite{kuhngoodwillie, higheralgebra, pereira, biedermannroendigs}. We will use the version of the theory developed in Chapter 6 of \cite{higheralgebra}. In particular, it shows that all the standard methods of Goodwillie calculus apply to pointed compactly generated $\infty$-categories. For the reader unfamiliar with Goodwillie calculus, the original paper \cite{goodwillie3} is still the most accessible introduction to the topic. Throughout this paper, we will freely make use of such terminology as \emph{$n$-excisive functors} and \emph{derivatives of functors}, although we recall several of the standard constructions throughout the text and in the appendix. We will usually consider functors which preserve zero objects and \emph{filtered} colimits. The first is just a matter of convenience, the second is essential for some of the fundamental results of calculus, e.g. the characterization of homogeneous functors. We will make our assumptions on functors clear whenever they come up.

We will associate to a pointed compactly generated $\infty$-category $\mathcal{C}$ an \emph{$n$-excisive approximation} $\mathcal{P}_n\mathcal{C}$, which is universal in a certain sense. Before defining these approximations, let us make some obvious but necessary observations:

\begin{lemma}
\label{lem:observations}
Consider an adjunction between $\infty$-categories (left adjoint on the left)
\[
\xymatrix@C=25pt{
F: \mathcal{C} \ar@<.5ex>[r] & \mathcal{D}: G \ar@<.5ex>[l] 
}
\]
and suppose the identity functor $\mathrm{id}_\mathcal{D}$ of $\mathcal{D}$ is $n$-excisive. Then both $F$ and $G$ are $n$-excisive functors and so is the composition $GF$. By the universal property of the $n$-excisive approximation $P_n$, the unit and counit of this adjunction induce maps $P_n\mathrm{id}_\mathcal{C} \rightarrow GF$ and $P_n(FG) \rightarrow \mathrm{id}_\mathcal{D}$.
\end{lemma}

\begin{definition}
\label{def:nexcapprox}
Let $\mathcal{C}$ be a pointed, compactly generated $\infty$-category. An adjunction (left adjoint on the left)
\[
\xymatrix@C=25pt{
F: \mathcal{C} \ar@<.5ex>[r] & \mathcal{D}: G \ar@<.5ex>[l] 
}
\]
is a \emph{weak $n$-excisive approximation} to $\mathcal{C}$ if the following two conditions are satisfied:
\begin{itemize}
\item[(a)] The $\infty$-category $\mathcal{D}$ is pointed and compactly generated. Moreover, its identity functor $\mathrm{id}_\mathcal{D}$ is $n$-excisive.
\item[(b)] The map $P_n\mathrm{id}_\mathcal{C} \rightarrow GF$ induced by the unit (see Lemma \ref{lem:observations}) is an equivalence. Also, the natural transformation $P_n(FG) \rightarrow \mathrm{id}_\mathcal{D}$ induced by the counit is an equivalence.
\end{itemize}
\end{definition}

If $\mathcal{C}$ is a pointed compactly generated $\infty$-category whose identity functor is itself $n$-excisive, then for any weak $n$-excisive approximation $F: \mathcal{C} \rightleftarrows \mathcal{D}: G$ the functor $F$ is fully faithful. Indeed, in this case property (b) above tells us that the unit $\mathrm{id}_\mathcal{C} \rightarrow GF$ is an equivalence. In other words, the functor $F$ exhibits $\mathcal{C}$ as a colocalization of $\mathcal{D}$. We will be especially interested in those $n$-excisive approximations which are `maximal' in the following sense:

\begin{definition}
\label{def:strongapprox}
A pointed compactly generated $\infty$-category $\mathcal{D}$ is \emph{$n$-excisive} if every weak $n$-excisive approximation $F': \mathcal{D} \rightleftarrows \mathcal{E}: G'$ is an equivalence. A weak $n$-excisive approximation $F: \mathcal{C} \rightleftarrows \mathcal{D}: G$ as in Definition \ref{def:nexcapprox} is said to be a \emph{strong $n$-excisive approximation} if the $\infty$-category $\mathcal{D}$ is $n$-excisive.
\end{definition} 

An $\infty$-category is 1-excisive if and only if it is stable (Corollary \ref{cor:1excstable}); one can think of the notion of $n$-excisiveness as a generalization of stability to the cases $n \geq 1$. In Corollary \ref{cor:bnexcisive} we will give an alternative and more explicit characterization of the maximality property of Definition \ref{def:strongapprox} which makes this clear.

\begin{remark}
We will often abuse language and refer to the $\infty$-category $\mathcal{D}$ as a strong $n$-excisive approximation to $\mathcal{C}$, leaving the adjunction between $\mathcal{C}$ and $\mathcal{D}$ understood. If $\mathcal{C}$ is the $\infty$-category of pointed spaces, then the $\infty$-category $\mathrm{Sp}^{\geq n}$ of $n$-connected spectra, for any $n < 0$, is a weak 1-excisive approximation to $\mathcal{C}$. The $\infty$-category $\mathrm{Sp}$ of spectra is a strong 1-excisive approximation. In fact, for any pointed compactly generated $\infty$-category $\mathcal{C}$, the stabilization $\mathrm{Sp}(\mathcal{C})$ is a strong 1-excisive approximation.
\end{remark}

As far as $n$-excisive functors are concerned, any weak approximation of $\mathcal{C}$ is `as good as' $\mathcal{C}$ itself. The proof of the following lemma is an elementary illustration of our definitions. Here $\mathrm{Fun}^{\leq n}$ denotes the $\infty$-category of $n$-excisive functors that commute with filtered colimits.

\begin{lemma}
\label{lem:nexcfunctors}
Let $F_{\mathcal{C}}: \mathcal{C} \rightleftarrows \mathcal{C}_n: G_{\mathcal{C}}$ and $F_{\mathcal{D}}: \mathcal{D} \rightleftarrows \mathcal{D}_n: G_{\mathcal{D}}$ be weak $n$-excisive approximations to $\mathcal{C}$ and $\mathcal{D}$ respectively. Then composing with $G_{\mathcal{D}}$ and $F_{\mathcal{C}}$ yields an equivalence
\begin{equation*}
G_{\mathcal{D}} \circ - \circ F_{\mathcal{C}}: \mathrm{Fun}^{\leq n}(\mathcal{C}_n, \mathcal{D}_n) \longrightarrow \mathrm{Fun}^{\leq n}(\mathcal{C}, \mathcal{D}).
\end{equation*}
\end{lemma}
\begin{proof}
Note that since $G_{\mathcal{D}}$ preserves limits and $F_{\mathcal{C}}$ preserves colimits, a functor of the form $G_{\mathcal{D}} \circ H \circ F_{\mathcal{C}}$ is $n$-excisive whenever $H$ is $n$-excisive, so that the statement of the lemma makes sense. An inverse (up to homotopy) to the functor in the statement of the lemma is the functor described by
\begin{equation*}
\mathrm{Fun}^{\leq n}(\mathcal{C}, \mathcal{D}) \rightarrow \mathrm{Fun}^{\leq n}(\mathcal{C}_n, \mathcal{D}_n): H \mapsto P_n(F_{\mathcal{D}} \circ H \circ G_{\mathcal{C}}).
\end{equation*}
Indeed, that this is an inverse follows from the natural equivalences
\begin{equation*}
P_n(F_{\mathcal{D}}G_{\mathcal{D}} H F_{\mathcal{C}}G_{\mathcal{C}}) \simeq P_n H \simeq H
\end{equation*}
for $H \in \mathrm{Fun}^{\leq n}(\mathcal{C}_n, \mathcal{D}_n)$ and
\begin{equation*}
G_{\mathcal{D}} P_n(F_{\mathcal{D}} H G_{\mathcal{C}}) F_{\mathcal{C}} \simeq P_n(G_{\mathcal{D}} F_{\mathcal{D}} H G_{\mathcal{C}} F_{\mathcal{C}}) \simeq P_n H \simeq H
\end{equation*}
for $H \in \mathrm{Fun}^{\leq n}(\mathcal{C}, \mathcal{D})$. Both these equivalences follow from our assumptions combined with the elementary fact that 
\begin{equation*}
P_n(H_1 \circ H_2) \simeq P_n(P_n(H_1) \circ H_2) \simeq P_n(H_1 \circ P_n(H_2))
\end{equation*}
for any two functors $H_1$ and $H_2$, at least if one assumes that $H_1$ preserves filtered colimits.
\end{proof}

Our first main result concerns the existence, naturality, and uniqueness of $n$-excisive approximations. To phrase the kind of functoriality that $n$-excisive approximations satisfy we introduce some notation: let us write $\mathbf{Cat}^\omega_*$ for the (large) $\infty$-category of pointed compactly generated $\infty$-categories and functors which preserve small colimits and compact objects. Also, let us write $\mathbf{Cat}^\omega_{*, \leq n}$ for the full subcategory of $\mathbf{Cat}^\omega_*$ on the $n$-excisive $\infty$-categories of Definition \ref{def:strongapprox}. Note that colimit-preserving functors between $n$-excisive categories are automatically $n$-excisive functors.

\begin{remark}
By the adjoint functor theorem (Corollary 5.5.2.9 of \cite{htt}), a colimit-preserving functor $F$ between compactly generated categories admits a right adjoint $G$. The condition that $F$ preserves compact objects is equivalent to the condition that $G$ preserves filtered colimits.
\end{remark}

\begin{theorem}
\label{thm:TheoremA}
The inclusion $i_n: \mathbf{Cat}^\omega_{*,\leq n} \rightarrow \mathbf{Cat}^\omega_*$ admits a left adjoint $\mathcal{P}_n$, which satisfies the following properties:
\begin{itemize}
\item[(a)] For any $\mathcal{C} \in \mathbf{Cat}^\omega_*$, the unit
\begin{equation*}
\mathcal{C} \longrightarrow i_n \mathcal{P}_n\mathcal{C}
\end{equation*}
is the left adjoint of a strong $n$-excisive approximation to $\mathcal{C}$.
\item[(b)] The counit is an equivalence between $\mathcal{P}_n i_n$ and the identity functor of $\mathbf{Cat}^\omega_{*,\leq n}$. In other words, $\mathcal{P}_n$ exhibits $\mathbf{Cat}^\omega_{*,\leq n}$ as a localization of $\mathbf{Cat}^\omega_*$.
\item[(c)] The functor $\mathcal{P}_n$ preserves finite limits.
\end{itemize}
In particular, any pointed compactly generated $\infty$-category $\mathcal{C}$ admits a strong $n$-excisive approximation. Moreover, such an approximation is unique up to canonical equivalence.
\end{theorem}

We will prove Theorem \ref{thm:TheoremA} in Chapter \ref{sec:constructingPnC} by providing a direct construction of $\mathcal{P}_n\mathcal{C}$ and proving the necessary properties. We will usually omit the inclusion $i_n$ from the notation; it should be clear from context in which category we are working. Also, for $\mathcal{C} \in \mathbf{Cat}^\omega_*$, we will use the notation
\[
\xymatrix@C=25pt{
\mathcal{C} \ar@<.5ex>[r]^{\Sigma_n^\infty} & \mathcal{P}_n\mathcal{C} \ar@<.5ex>[l]^{\Omega_n^\infty}
}
\]
for the $n$-excisive approximation to $\mathcal{C}$ provided by Theorem \ref{thm:TheoremA}. With this notation, the unit
\begin{equation*}
\eta: \mathrm{id}_{\mathcal{C}} \rightarrow \Omega^\infty_n\Sigma^\infty_n
\end{equation*}
is equivalent to the $n$-excisive approximation $\mathrm{id}_{\mathcal{C}} \rightarrow P_n\mathrm{id}_{\mathcal{C}}$. The notation we use is derived from the case $n=1$. Indeed, recall that a pointed compactly generated $\infty$-category is 1-excisive if and only if it is stable (Corollary \ref{cor:1excstable}). The functor $\mathcal{P}_1$ of Theorem \ref{thm:TheoremA} is then the stabilization functor $\mathrm{Sp}(-)$ and the resulting adjoint pair 
\[
\xymatrix@C=25pt{
\mathcal{C} \ar@<.5ex>[r]^{\Sigma_1^\infty} & \mathcal{P}_1\mathcal{C} \ar@<.5ex>[l]^{\Omega_1^\infty}
}
\]
can be identified with the usual stabilization adjunction
\[
\xymatrix@C=25pt{
\mathcal{C} \ar@<.5ex>[r]^{\Sigma^\infty} & \mathrm{Sp}(\mathcal{C}) \ar@<.5ex>[l]^{\Omega^\infty}.
}
\]
It is a straightforward formal exercise to see that for $m \leq n$, one has the transitivity property $\mathcal{P}_m(\mathcal{P}_n\mathcal{C}) \simeq \mathcal{P}_m\mathcal{C}$ and so in particular a functor
\begin{equation*}
\mathcal{P}_n\mathcal{C} \longrightarrow \mathcal{P}_m\mathcal{C}.
\end{equation*}
With this observation one can assemble the $n$-excisive approximations for various $n$ into a \emph{Goodwillie tower} for $\mathcal{C}$:

\[
\xymatrix{
& & \vdots \ar[d] \\
& \vdots & \mathcal{P}_3\mathcal{C}\ar[d] \\
& & \mathcal{P}_2\mathcal{C}\ar[d] \\
\mathcal{C} \ar[uurr]|{\Sigma_3^\infty}\ar[urr]|{\Sigma_2^\infty}\ar[rr]|{\Sigma_1^\infty}& & \mathcal{P}_1\mathcal{C}.
}
\]

The construction of this Goodwillie tower is natural with respect to functors preserving colimits and compact objects. For $m < n$, we will denote the functor $\mathcal{P}_n\mathcal{C} \rightarrow \mathcal{P}_m\mathcal{C}$ by $\Sigma_{n,m}^\infty$ and its right adjoint by $\Omega_{n,m}^\infty$. The reader should in particular observe that the unit of the adjunction
\[
\xymatrix{
\mathcal{C} \ar@<.5ex>[r] & \varprojlim \mathcal{P}_n\mathcal{C} \ar@<.5ex>[l]
}
\] 
is the natural transformation
\begin{equation*}
\mathrm{id}_{\mathcal{C}} \longrightarrow \varprojlim P_n\mathrm{id}_{\mathcal{C}}.
\end{equation*}
Write $\mathcal{C}^{\mathrm{conv}}$ for the full subcategory of $\mathcal{C}$ on objects for which the Goodwillie tower of the identity converges, i.e. those $X$ at which the evaluation of the natural transformation above is an equivalence. Then we immediately obtain the following:

\begin{lemma}
\label{lem:convergence}
The functor $\mathcal{C}^{\mathrm{conv}} \rightarrow  \varprojlim \mathcal{P}_n\mathcal{C}$ is fully faithful.
\end{lemma}

\begin{remark}
We have tacitly used the existence of limits in $\mathbf{Cat}^\omega_*$, which is guaranteed by Lemma \ref{lem:catomegabicomplete}.
\end{remark}

\begin{remark}
The reader might note the absence of a definition for the notion of an \emph{$n$-homogenous $\infty$-category}, which would parallel the notion of $n$-homogeneous functor. Simply `taking the fiber' of the functor $\mathcal{P}_n\mathcal{C} \rightarrow \mathcal{P}_{n-1}\mathcal{C}$ yields the trivial $\infty$-category and is therefore not of interest. We will encounter several examples of $\infty$-categories which are perhaps deserving of the adjective $n$-homogeneous, e.g. in Corollary \ref{cor:fibercoalgs}. They are the $\infty$-categories of `coalgebras' for an $n$-homogeneous functor $F$ from a stable $\infty$-category to itself. We will not belabour the issue, as this notion will not play a prominent role for us.
\end{remark}



Our goal is to provide a classification of the Goodwillie towers of $\infty$-categories as described above. First we discuss the extra structure present on the derivatives of the identity functor of a pointed compactly generated $\infty$-category $\mathcal{C}$. In fact, as already hinted at in the introduction, it will be more convenient for us to focus on the `Koszul dual' structure (see Remark \ref{rmk:derivativesidentity} for a more elaborate statement). To do this we need the language of \emph{stable $\infty$-operads}. For background on $\infty$-operads see Chapter 2 of \cite{higheralgebra}. For more on stable $\infty$-operads the reader can consult Chapter 6 of the same reference. Throughout this text we will almost exclusively work with \emph{nonunital} $\infty$-operads, i.e. $\infty$-operads whose structure map to $\NFin$ factors through $\N\mathrm{Surj}$, with $\mathrm{Surj}$ denoting the subcategory of $\mathcal{F}\mathrm{in}_*$ containing only the surjective maps of finite pointed sets. For any $\infty$-operad $\mathcal{O}^\otimes$ one may construct a nonunital variant $\mathcal{O}^\otimes_{\mathrm{nu}}$ of it by pulling it back along the map $\N\mathrm{Surj} \rightarrow \NFin$. For an $\infty$-operad $\mathcal{O}^\otimes \rightarrow \NFin$ we will always write $\mathcal{O}$ for its underlying $\infty$-category, i.e. its fiber over $\langle 1 \rangle$.

\begin{definition}
A nonunital $\infty$-operad $p: \mathcal{O}^\otimes \rightarrow \N\mathrm{Surj}$ is \emph{stable} if it satisfies the following conditions:
\begin{itemize}
\item[(1)] It is \emph{corepresentable}, meaning the map $p$ is a locally coCartesian fibration. Equivalently, this means that for every non-empty collection of objects $\{X_1, \ldots, X_n\}$ of $\mathcal{O}$, the functor
\begin{equation*}
\mathcal{O}^\otimes(X_1, \ldots, X_n; -): \mathcal{O} \longrightarrow \mathcal{S}
\end{equation*}
parametrizing operations in $\mathcal{O}^\otimes$ with these inputs is corepresentable by an object we denote $X_1 \otimes^n \cdots \otimes^n X_n$. This determines, for every non-empty finite set $I$, a functor
\begin{equation*}
\mathcal{O}^I \longrightarrow \mathcal{O}: \{X_i\}_{i \in I} \longmapsto \otimes^I\{X_i\}_{i \in I}.
\end{equation*}
\item[(2)] Its underlying $\infty$-category $\mathcal{O}$ is stable and compactly generated.
\item[(3)] For every non-empty finite set $I$, the tensor product functor $\otimes^I: \mathcal{O}^I \rightarrow \mathcal{O}$ preserves colimits in each variable separately. 
\end{itemize}
\end{definition}

\begin{remark}
Observe that our definition of a stable $\infty$-operad is slightly more restrictive than Lurie's: he does not require the underlying $\infty$-category to be compactly generated and only requires the functors $\otimes_I$ to preserve \emph{finite} colimits in each variable separately. Also, he does not restrict to the nonunital case.
\end{remark}

\begin{remark}
\label{rmk:dictionary}
As an example, suppose $p: \mathcal{O}^\otimes \rightarrow \N\mathrm{Surj}$ is stable and that the underlying $\infty$-category $\mathcal{O}$ is the $\infty$-category $\mathrm{Sp}$ of spectra. Then condition (3) forces the tensor products $\otimes_I$ to be of the form
\begin{equation*}
\otimes^I(\{X_i\}_{i \in I}) \simeq C_I \wedge \bigl(\bigwedge_{i \in I} X_i \bigr)
\end{equation*}
for some fixed spectrum $C_I$, where $\wedge$ denotes the smash product of spectra. Using the composition maps of $\mathcal{O}^\otimes$ and its corepresentability, one finds natural maps
\[
\xymatrix{
& X_1 \otimes^3 X_2 \otimes^3 X_3 \ar[dl]\ar[dr] & \\
(X_1 \otimes^2 X_2) \otimes^2 X_3 & & X_1 \otimes^2 (X_2 \otimes^2 X_3)
}
\]
and hence two different maps
\[
\xymatrix{
& C_3 \ar[dl]\ar[dr]& \\
C_2 \wedge C_2 && C_2 \wedge C_2.
}
\]
Such maps exist more generally for any finite set $I$ and a partition of it. They give the collection of the spectra $C_I$ the structure of a \emph{cooperad}. See Section 6.3 of \cite{higheralgebra} for more on this dictionary between stable $\infty$-operads and cooperads. The reader might wish to keep it in mind to relate our results and techniques to those found in the literature, for example in the papers of Arone and Ching \cite{aroneching}.
\end{remark}

\begin{remark}
In the main body of this text we use Lurie's formalism for $\infty$-operads because it allows us to cite results from \cite{higheralgebra}. However, in the appendix we switch to the formalism of \emph{dendroidal sets} to prove some of the more technical results. The equivalence between these two formalisms is proved in \cite{hhm} (in the setting of nonunital $\infty$-operads) and in \cite{chh}.
\end{remark}

Let $\mathcal{C}$ be a pointed compactly generated $\infty$-category. The Cartesian product on $\mathcal{C}$ gives it the structure of a symmetric monoidal $\infty$-category, which is encoded as an $\infty$-operad $\mathcal{C}^\times \rightarrow \NFin$ whose structure map is a coCartesian fibration. In Section 6.2 of \cite{higheralgebra}, Lurie defines the \emph{stabilization} $\mathrm{Sp}(\mathcal{C})^\otimes \rightarrow \N\mathrm{Surj}$ of the $\infty$-operad $\mathcal{C}_{\mathrm{nu}}^\times$ and proves its existence. It has the following properties:
\begin{itemize}
\item[(1)] It is stable.
\item[(2)] There is a map of $\infty$-operads $\mathrm{Sp}(\mathcal{C})^\otimes \rightarrow \mathcal{C}_{\mathrm{nu}}^\times$ whose fiber over $\langle 1 \rangle$ is the functor $\Omega^\infty: \mathrm{Sp}(\mathcal{C}) \rightarrow \mathcal{C}$.
\item[(3)] For every finite set $I$, the induced natural transformation
\begin{equation*}
\times^I \circ \bigl(\Omega^\infty \bigr)^I \longrightarrow \Omega^\infty \circ \otimes^I 
\end{equation*}
exhibits $\otimes^I$ as a derivative (or \emph{multilinearization}) of $\times^I$.
\end{itemize}

\begin{remark}
Lurie in fact works with the stabilization of $\mathcal{C}^\times$ rather than the nonunital $\mathcal{C}_{\mathrm{nu}}^\times$. There is no essential loss of information in passing from the first to the second, but the second will be more convenient to us.
\end{remark}

\begin{remark}
\label{rmk:derivativesidentity}
Unraveling the definitions (see Lemma \ref{lem:codersigmaomega}), one sees that $D_n(\Sigma^\infty\Omega^\infty)(X)$ is equivalent to $(X \otimes^n \cdots \otimes^n X)_{h\Sigma_n}$. Thus, informally speaking, the data of the stable $\infty$-operad $\mathrm{Sp}(\mathcal{C})^\otimes$ is equivalent to the data of the stabilization $\mathrm{Sp}(\mathcal{C})$ and the derivatives of the functor $\Sigma^\infty\Omega^\infty$ together with their cooperadic structure. The derivatives of the identity functor on $\mathcal{C}$ can be extracted from these by a cobar construction.  This was done for the case of spaces and spectra in the work of Arone and Ching \cite{aroneching}. In this paper we will mostly work with the stable $\infty$-operad $\mathrm{Sp}(\mathcal{C})^\otimes$, but the correspondence just described can be useful to keep in mind as a guiding principle. See also the discussion in Section 6.3 of \cite{higheralgebra}.
\end{remark}

\begin{example}
In case $\mathcal{C} = \mathcal{S}_*$ one has $\mathrm{Sp}(\mathcal{C})^\otimes \simeq \mathrm{Sp}^\otimes$, i.e. the stabilization of $\mathcal{S}_*^{\times}$ is the symmetric monoidal $\infty$-category of spectra with their smash product (or, more precisely, its nonunital variant). This corresponds to the result of Arone and Ching that the derivatives of the functor $\Sigma^\infty\Omega^\infty$ form the commutative cooperad in spectra \cite{aroneching}.
\end{example}

\begin{remark}
\label{rmk:prodpresfunctor}
A product-preserving functor $F: \mathcal{C} \rightarrow \mathcal{D}$ induces a map $\mathcal{C}^\times \rightarrow \mathcal{D}^\times$ of $\infty$-operads.  Moreover, if $F$ preserves all finite limits, it induces a map of stable $\infty$-operads $\mathrm{Sp}(\mathcal{C})^\otimes \rightarrow \mathrm{Sp}(\mathcal{D})^\otimes$. 
\end{remark}

We now phrase our classification problem as follows:

\textbf{Question:} Given a nonunital stable $\infty$-operad $p: \mathcal{O}^\otimes \rightarrow \N\mathrm{Surj}$, can we classify the possible Goodwillie towers of pointed compactly generated $\infty$-categories $\mathcal{C}$ for which $\mathcal{O}^\otimes$ is the stabilization of $\mathcal{C}_{\mathrm{nu}}^\times$? 

We will take an inductive approach to answering this question. Write $\mathrm{Surj}_{\leq n}$ for the full subcategory of the category $\mathrm{Surj}$ on the objects $\langle i \rangle$ for $0 \leq i \leq n$. For a nonunital $\infty$-operad $\mathcal{O}^\otimes \longrightarrow \N\mathrm{Surj}$ we use the notation
\begin{equation*}
\mathcal{O}^\otimes_{\leq n} := \mathcal{O}^\otimes \times_{\N\mathrm{Surj}} \N\mathrm{Surj}_{\leq n}.
\end{equation*}
We will refer to $\mathcal{O}^\otimes_{\leq n}$ as the \emph{$n$-truncation} of $\mathcal{O}$. Informally speaking, $\mathcal{O}^\otimes_{\leq n}$ knows about the operations of $\mathcal{O}^\otimes$ with at most $n$ inputs. 

\begin{definition}
\label{def:nstage}
Let $p: \mathcal{O}^\otimes \rightarrow \N\mathrm{Surj}$ be a nonunital stable $\infty$-operad. An \emph{$n$-stage} for $\mathcal{O}^\otimes$ is an $n$-excisive $\infty$-category $\mathcal{C}$ together with an equivalence of $n$-truncations $\mathrm{Sp}(\mathcal{C})^\otimes_{\leq n} \rightarrow \mathcal{O}^\otimes_{\leq n}$.
\end{definition}


The following is part of Proposition \ref{prop:SpPnCntruncated}, which will be proved in Section \ref{subsec:cobar}:

\begin{lemma}
For a pointed compactly generated $\infty$-category $\mathcal{C}$, the functor $\Omega^\infty_n: \mathcal{P}_n\mathcal{C} \longrightarrow \mathcal{C}$ induces an  equivalence of $n$-truncations
\begin{equation*}
\mathrm{Sp}(\mathcal{P}_n\mathcal{C})^\otimes_{\leq n} \longrightarrow \mathrm{Sp}(\mathcal{C})^\otimes_{\leq n}.
\end{equation*}
In particular, $\Omega^\infty_n$ exhibits $\mathcal{P}_n\mathcal{C}$ as an $n$-stage for $\mathrm{Sp}(\mathcal{C})^\otimes$.
\end{lemma}

\begin{remark}
\label{rmk:universalntruncation}
One can in fact show that $\mathrm{Sp}(\mathcal{P}_n\mathcal{C})^\otimes$ is the \emph{initial} stable $\infty$-operad which maps to $\mathrm{Sp}(\mathcal{C})^\otimes$ and induces an equivalence on $n$-truncations. We will show in Section \ref{subsec:truncations} that for any nonunital stable $\infty$-operad $\mathcal{O}^\otimes$ there is a stable $\infty$-operad $\tau_n\mathcal{O}^\otimes$ with this universal property. One can think of it as agreeing with $\mathcal{O}^\otimes$ up to operations of arity $n$ and being free above that. With this notation, the first sentence of this remark can be summarized as $\mathrm{Sp}(\mathcal{P}_n\mathcal{C})^\otimes \simeq \tau_n\mathrm{Sp}(\mathcal{C})^\otimes$. This claim also follows from Proposition \ref{prop:SpPnCntruncated}.
\end{remark}

\begin{definition}
Let $\mathcal{O}^\otimes$ be a nonunital stable $\infty$-operad. Denote by $\mathbf{Cat}_n$ the maximal Kan complex contained in the $\infty$-category $\mathbf{Cat}^\omega_{*,\leq n}$ of $n$-excisive $\infty$-categories. Then denote by $\mathcal{G}_n(\mathcal{O}^\otimes)$ the homotopy fiber defined by the following diagram:
\[
\xymatrix{
\mathcal{G}_n(\mathcal{O}^\otimes) \ar[d]\ar[r] & \mathbf{Cat}_n \ar[d] \\
\ast \ar[r]_-{\mathcal{O}^\otimes_{\leq n}} & \mathbf{Cat}_\infty / \N\mathrm{Surj}_{\leq n}.
}
\]
The vertical map on the right takes the $n$-truncation of the stabilization of the $\infty$-operad $\mathcal{C}^\times_{\mathrm{nu}}$, for $\mathcal{C}$ an $n$-excisive $\infty$-category. We will refer to $\mathcal{G}_n(\mathcal{O}^\otimes)$ as the \emph{space of $n$-stages for $\mathcal{O}^\otimes$}. 
\end{definition}

For the sake of concreteness, whenever required we will adopt the following explicit model for the simplicial set $\mathcal{G}_n(\mathcal{O}^\otimes)$. Write $\mathbf{Op}^{\mathrm{St}}_{\leq n}$ for the $\infty$-category of stable $n$-truncated $\infty$-operads (see Definition \ref{def:truncatedOpSt}) and $K_n$ for the maximal Kan complex contained in it. The forgetful map $(K_n)_{\mathcal{O}^\otimes_{\leq n}/} \rightarrow K_n$ is a Kan fibration and its domain is a contractible Kan complex. Take $\mathcal{G}_n(\mathcal{O}^\otimes)$ to be the pullback of simplicial sets
\[
\xymatrix{
\mathcal{G}_n(\mathcal{O}^\otimes) \ar[d]\ar[r] & \mathbf{Cat}_n \ar[d] \\
(K_n)_{\mathcal{O}^\otimes_{\leq n}/} \ar[r] & K_n.
}
\]
It is easily checked that this gives a Kan complex equivalent to the definition above.

The formation of $n$-excisive approximations yields a sequence of maps
\begin{equation*}
\cdots \longrightarrow \mathcal{G}_3(\mathcal{O}^\otimes) \longrightarrow \mathcal{G}_2(\mathcal{O}^\otimes) \longrightarrow \mathcal{G}_1(\mathcal{O}^\otimes).
\end{equation*}
Note that $\mathcal{G}_1(\mathcal{O}^\otimes)$ is contractible. We will now give an inductive description of the homotopy type of $\mathcal{G}_n(\mathcal{O}^\otimes)$. First we recall the Tate construction. Let $X$ be an object of a compactly generated stable $\infty$-category $\mathcal{O}$ and assume $X$ is equipped with a $\Sigma_n$-action. Then the homotopy invariants $X^{h\Sigma_n}$ and coinvariants $X_{h\Sigma_n}$ both exist in $\mathcal{O}$ and there is a natural \emph{norm map}
\begin{equation*}
\mathrm{Nm}: X_{h\Sigma_n} \longrightarrow X^{h\Sigma_n}
\end{equation*}
between the two (see \cite{greenleesmay, klein, ambidexterity}). Among natural transformations $F(X) \rightarrow X^{h\Sigma_n}$ between functors from $\mathrm{Fun}(B\Sigma_n,\mathcal{O})$ to $\mathcal{O}$, it is characterized up to canonical equivalence by the requirements that it be an equivalence on compact objects of $\mathrm{Fun}(B\Sigma_n,\mathcal{O})$ (the $\infty$-category of objects of $\mathcal{O}$ with a $\Sigma_n$-action) and that $F$ preserves colimits. The \emph{Tate construction} $X^{t\Sigma_n}$ is defined to be the cofiber of this map.

Now suppose $\mathcal{C}$ is an $(n-1)$-stage for $\mathcal{O}^\otimes$, so that in particular we have an equivalence $\mathrm{Sp}(\mathcal{C})^\otimes \simeq \tau_{n-1}\mathcal{O}^\otimes$ (see Remark \ref{rmk:universalntruncation}). Write 
\begin{equation*}
\odot^n:  \mathcal{O}^{n} \longrightarrow \mathcal{O}
\end{equation*}
for the $n$-fold tensor product on $\mathcal{O}$ determined by this stable $\infty$-operad. The map $\tau_{n-1}\mathcal{O}^\otimes \rightarrow \mathcal{O}^\otimes$ of $\infty$-operads induces a natural transformation $\otimes^{n} \rightarrow \odot^n$, with $\otimes^{n}$ denoting the $n$-fold tensor product determined by the stable $\infty$-operad $\mathcal{O}^\otimes$. Consider an object $X \in \mathcal{C}$ and its $n$-fold diagonal map $\Delta_n(X): X \rightarrow X^{\times n}$. We write 
\begin{equation*}
\Sigma^\infty_{\mathcal{C}}: \mathcal{C} \longrightarrow \mathrm{Sp}(\mathcal{C}) \simeq \mathcal{O}
\end{equation*}
for the stabilization of $\mathcal{C}$. Linearizing the natural transformation $\Delta_n$ then gives a map
\begin{equation*}
\delta_n: \Sigma^\infty_{\mathcal{C}} X \longrightarrow \Sigma^\infty_{\mathcal{C}} X \odot^n \cdots \odot^n \Sigma^\infty_{\mathcal{C}} X
\end{equation*}
which is part of a natural transformation between functors from $\mathcal{C}$ to $\mathcal{O}$. The symmetric group $\Sigma_n$ acts on the codomain of this natural transformation. Denote the Tate construction of this action by $\Psi_{\mathcal{C}}$. Then $\delta_{n}$ induces a natural transformation
\begin{equation*}
\psi_{n}: \Sigma^\infty_{\mathcal{C}} \longrightarrow \Psi_{\mathcal{C}}.
\end{equation*}
In Section \ref{subsec:tatediagonal} we will construct a space $\widehat{\mathcal{T}}_n$ fibered over $\mathcal{G}_{n-1}(\mathcal{O}^\otimes)$, whose fiber over $\mathcal{C}$ is the space of natural transformations $\mathrm{Nat}(\Sigma^\infty_{\mathcal{C}}, \Psi_{\mathcal{C}})$. The construction of $\psi_{n}$ can be made functorial to yield a section of this fibration. 

Similarly, write $\Theta_{\mathcal{C}}$ for the Tate construction of the natural $\Sigma_n$-action on the functor $\otimes^n \circ (\Sigma^\infty)^n$. We will construct another space $\mathcal{T}_n$ fibered over $\mathcal{G}_{n-1}(\mathcal{O}^\otimes)$ whose fiber over $\mathcal{C}$ is the space of natural transformations $\mathrm{Nat}(\Sigma^\infty_{\mathcal{C}},\Theta_{\mathcal{C}})$. The second main result of this paper is the following:

\begin{theorem}
\label{thm:classification}
There exists a pullback square in the $\infty$-category $\mathcal{S}$ of spaces as follows:
\[
\xymatrix{
\mathcal{G}_n(\mathcal{O}^\otimes) \ar[r]\ar[d] &  \mathcal{T}_n \ar[d] \\
\mathcal{G}_{n-1}(\mathcal{O}^\otimes) \ar[r] & \widehat{\mathcal{T}}_n.
}
\]
\end{theorem}

\begin{remark}
In Section \ref{subsec:goodwilliespaces} we will consider the fiber of the map $\mathcal{T}_n \rightarrow \widehat{\mathcal{T}}_n$. It can be described in terms of a cobar construction formed from the stable $\infty$-operad $\mathrm{Sp}(\mathcal{C})^\otimes$.
\end{remark}

From Theorem \ref{thm:classification} we immediately deduce the following:

\begin{corollary}
\label{cor:notatecohomology}
Let $\mathcal{O}^\otimes$ be a nonunital stable $\infty$-operad and assume that the Tate cohomology of the symmetric groups $\Sigma_k$ vanishes in $\mathcal{O}$ for $k \leq n$, i.e. for every object $X$ of $\mathcal{O}$ with $\Sigma_k$-action the Tate construction $X^{t\Sigma_k}$ is contractible. Then the spaces $\mathcal{G}_k(\mathcal{O}^\otimes)$ are contractible for $k \leq n$.
\end{corollary}

In Section \ref{subsec:vanishingTate} we will explicitly describe the $n$-stages for $\mathcal{O}^\otimes$ in this special case. The relevant statement is Proposition \ref{prop:nstagenotate}, which together with Corollary \ref{cor:notatecohomology} gives the following:

\begin{corollary}
\label{cor:notatecohomology2}
Let $\mathcal{O}^\otimes$ be as in Corollary \ref{cor:notatecohomology}. If $\mathcal{C}$ is an $n$-stage for $\mathcal{O}^\otimes$, then there is a canonical equivalence of $\infty$-categories
\begin{equation*}
\mathcal{C} \longrightarrow \mathrm{coAlg}^{\mathrm{ind}}(\tau_n\mathcal{O}^\otimes).
\end{equation*}
Here $\tau_n\mathcal{O}^\otimes$ denotes the stable $\infty$-operad mentioned in Remark \ref{rmk:universalntruncation} and $\mathrm{coAlg}^{\mathrm{ind}}$ indicates the $\infty$-category of ind-coalgebras for that $\infty$-operad, see Definitions \ref{def:coalgebra} and \ref{def:indcoalgebras}.
\end{corollary}

\begin{remark}
\label{rmk:notatecohomology}
In concrete examples (such as when $\mathcal{C}$ is the $\infty$-category of pointed spaces or of algebras over an operad $\mathbf{O}$ in the category of spectra), the derivatives of the identity functor of $\mathcal{C}$ are known to form an operad. To relate the $\infty$-category $\mathrm{coAlg}^{\mathrm{ind}}(\tau_n\mathcal{O}^\otimes)$ to the $\infty$-category of algebras over this operad and retrieve the statement of Theorem \ref{thm:informalmain} of the introduction, one applies a form of Koszul duality, specifically Proposition \ref{prop:PnAlgO}.
\end{remark}

In particular, the results above imply that if the Tate cohomology of all the symmetric groups vanishes in $\mathcal{O}$ there is (up to equivalence) only one possible Goodwillie tower of $\infty$-categories associated to the $\infty$-operad $\mathcal{O}^\otimes$, namely that of the $\infty$-category of (ind-)coalgebras in $\mathcal{O}^\otimes$. As an example of how our methods apply in the setting of vanishing Tate cohomology we will reprove some well-known results from the rational homotopy theory of Quillen \cite{rationalhomotopy} (see Section \ref{sec:examples}):

\begin{theorem}
\label{thm:rationalhomotopy}
Let $\mathcal{S}_\mathbb{Q}^{\geq 2}$ denote the $\infty$-category of pointed simply connected rational spaces, $\mathrm{coAlg}_\mathbb{Q}^{\geq 2}$ the $\infty$-category of simply connected differential graded commutative coalgebras over $\mathbb{Q}$ and $\mathrm{Lie}_{\mathbb{Q}}^{\geq 1}$ the $\infty$-category of connected differential graded Lie algebras over $\mathbb{Q}$. Then there exists a diagram
\[
\xymatrix{
& \mathcal{S}_\mathbb{Q}^{\geq 2} \ar[dl]\ar[dr] & \\
\mathrm{Lie}_{\mathbb{Q}}^{\geq 1} \ar[rr] & & \mathrm{coAlg}_\mathbb{Q}^{\geq 2} 
}
\]
in which each of the three functors is an equivalence of $\infty$-categories. 
\end{theorem}

A more novel application of our results is Theorem \ref{thm:truncatedspaces} mentioned in the introduction. We will prove this result in Section \ref{subsec:truncatedspaces}. Finally, we will deduce some consequences of our results in the case where Tate spectra do not vanish. Specifically, we analyze the Goodwillie tower of the $\infty$-category $\mathcal{S}_*$ of pointed spaces in Sections \ref{subsec:modulisuspension} and \ref{subsec:goodwilliespaces}. In particular, we prove Theorem \ref{thm:Tatecoalgebras} of the introduction in Section \ref{subsec:modulisuspension}.

\chapter{Constructing $n$-excisive approximations}
\label{sec:constructingPnC}

The goal of this chapter is to give an explicit construction of the functor $\mathcal{P}_n$ of Theorem \ref{thm:TheoremA} and establish the necessary properties. This construction is closely related to Goodwillie's construction of $n$-excisive approximations to functors \cite{goodwillie3}, which we will also briefly review below. First we introduce some notation and terminology. \par 
For every integer $n \geq 1$, write $\mathbf{P}(n)$ for the power set of the set $\{1, \ldots, n\}$, regarded as a partially ordered set under inclusion. Also, write $\mathbf{P}_0(n)$ for the partially ordered set $\mathbf{P}(n) - \{\varnothing\}$. If $\mathcal{C}$ is an $\infty$-category, an \emph{$n$-cube} (resp. a \emph{punctured $n$-cube}) in $\mathcal{C}$ is a functor
\begin{equation*}
\mathcal{X}: \N\mathbf{P}(n) \longrightarrow \mathcal{C}
\end{equation*}
(resp. $\mathcal{X}: \N\mathbf{P}_0(n) \longrightarrow \mathcal{C}$). An $n$-cube is \emph{strongly coCartesian} if every face of it is a pushout square or, more precisely, if for every $I, J \subseteq \{1, \ldots, n\}$ the square
\[
\xymatrix{
\mathcal{X}(I \cap J) \ar[r]\ar[d] & \mathcal{X}(I) \ar[d] \\
\mathcal{X}(J) \ar[r] & \mathcal{X}(I \cup J)
}
\]
is a pushout square. Similarly, we say a punctured $n$-cube is strongly coCartesian if every face of it is a pushout square. This amounts to the same condition as above, with the added requirement that $I$, $J$ and $I \cap J$ all be non-empty. If the $\infty$-category $\mathcal{C}$ has a terminal object $\ast$, we will say that a punctured $n$-cube $\mathcal{X}$ in $\mathcal{C}$ is \emph{special} if it is strongly coCartesian and moreover satisfies $\mathcal{X}(\{i\}) \simeq \ast$ for all $1 \leq i \leq n$.

\begin{example}
A 2-cube is simply a square and a punctured 2-cube a diagram of the form
\[
\xymatrix{
\mathcal{X}_1 \ar[r] & \mathcal{X}_{12} & \mathcal{X}_2, \ar[l]
}
\]
in what should be obvious notation. The requirement that a punctured 2-cube be strongly coCartesian is vacuous. A punctured 3-cube is a diagram of the following form:
\[
\xymatrix@!0{
 & \mathcal{X}_1 \ar[rr]\ar[dd] & & \mathcal{X}_{12} \ar[dd] \\
& & \mathcal{X}_2 \ar[ur]\ar[dd] & \\
& \mathcal{X}_{13}\ar'[r][rr] & & \mathcal{X}_{123} \\
\mathcal{X}_3 \ar[ur]\ar[rr] & & \mathcal{X}_{23}\ar[ur] 
}
\]
Such a punctured 3-cube is strongly coCartesian if the three squares in this diagram are pushouts. It is special if moreover $\mathcal{X}_1 \simeq \mathcal{X}_2 \simeq \mathcal{X}_3 \simeq \ast$.
\end{example}

\begin{definition}
Let $\mathcal{C}$ be an $\infty$-category which has a terminal object. Then define $\mathcal{T}_n\mathcal{C}$ to be the full subcategory of the $\infty$-category $\Fun(\mathbf{P}_0(n+1),\mathcal{C})$ spanned by the special punctured $(n+1)$-cubes. 
\end{definition}

Observe that the construction $\mathcal{T}_n$ is functorial with respect to functors preserving terminal objects and finite colimits. Moreover, if $\mathcal{C}$ admits finite colimits, we can construct a functor
\begin{equation*}
L_n: \mathcal{C} \longrightarrow \mathcal{T}_n\mathcal{C}
\end{equation*} 
as follows. Consider the full subcategory $\overline{\mathcal{T}}_n\mathcal{C}$ of $\Fun(\mathbf{P}(n+1), \mathcal{C})$ spanned by the $(n+1)$-cubes $\mathcal{X}$ satisfying the following two conditions:
\begin{itemize}
\item[(i)] For each $1 \leq i \leq n + 1$, we have $\mathcal{X}(\{i\}) \simeq \ast$.
\item[(ii)] The cube $\mathcal{X}$ is strongly coCartesian.
\end{itemize}
There are obvious functors
\[
\xymatrix{
& \overline{\mathcal{T}}_n\mathcal{C} \ar[dl]\ar[dr] & \\
\mathcal{C} & & \mathcal{T}_n\mathcal{C}.
}
\]
The left functor evaluates at $\varnothing$, the right functor forgets the initial vertex of the cube. Moreover, the left map is a trivial Kan fibration (this follows from Proposition 4.3.2.15 of \cite{htt}). We may therefore pick a section, which we will denote $C_n$, and compose with the right arrow to obtain a functor $L_n: \mathcal{C} \longrightarrow \mathcal{T}_n\mathcal{C}$ as claimed above. If $\mathcal{C}$ has finite limits, this functor admits a right adjoint $R_n$, which can be described on objects by the formula
\begin{equation*}
R_n(\mathcal{X}) = \varprojlim \mathcal{X},
\end{equation*}
where $\mathcal{X}$ is a special punctured $(n+1)$-cube. The discussion above can be refined to make the assignment which sends $\mathcal{C}$ to the adjunction
\[
\xymatrix@C=25pt{
\mathcal{C} \ar@<.5ex>[r]^{L_n} & \mathcal{T}_n\mathcal{C} \ar@<.5ex>[l]^{R_n}
}
\]
natural in $\mathcal{C}$, at least with respect to functors preserving terminal objects and finite colimits. To make this refinement one considers the span above involving $\mathcal{C}$, $\overline{\mathcal{T}}_n\mathcal{C}$, and $\mathcal{T}_n\mathcal{C}$ not just for individual $\mathcal{C}$, but for a family of such $\infty$-categories in which the functors are as specified. The functors $L_n$ and $R_n$ were considered independently and in a different context by Eldred \cite{eldred}.

\begin{remark}
Observe that our discussion in particular applies to pointed compactly generated $\infty$-categories $\mathcal{C}$. In this case, the functor $L_n$ preserves compact objects. This can be seen from the fact that $R_n$ preserves filtered colimits, which is clear since it is the functor taking a limit over the punctured cube, which is a finite diagram.
\end{remark}

\begin{remark}
In the construction above we defined for every $X \in \mathcal{C}$ a strongly coCartesian $(n+1)$-cube $C_n(X)$ in $\mathcal{C}$. This strongly coCartesian cube is essentially uniquely determined by the fact that its initial vertex is $X$ and the vertices corresponding to one-element subsets of $\{1, \ldots, n+1\}$ are terminal objects. In the case where $\mathcal{C} = \mathcal{S}_*$ one can describe the cube $C_n(X)$ very explicitly and indeed this is Goodwillie's original construction: for a pointed space $X$ one can take 
\begin{equation*}
C_n(X): \mathbf{P}(n+1) \rightarrow \mathcal{C}: S \mapsto X \star S,
\end{equation*}
with $\star$ denoting the join. Indeed, the resulting cube is strongly coCartesian and the join of $X$ with a one-point set is contractible (being a cone on $X$).
\end{remark}

\begin{remark}
\label{rmk:Cn}
The reader should observe that for every non-empty subset $S \subseteq \{1, \ldots, n+1\}$ there is an equivalence
\begin{equation*} 
C_n(X)(S) \simeq \bigvee_{|S|-1} \Sigma X,
\end{equation*}
i.e. the $(|S|-1)$-fold coproduct of the suspension of $X$ with itself. In particular, if $f: X \rightarrow Y$ is a map such that $\Sigma f$ is an equivalence, then $C_n(f)$ is an equivalence at every vertex corresponding to a non-empty subset $S$.
\end{remark}

\begin{definition}
For a pointed compactly generated $\infty$-category $\mathcal{C}$, define 
\begin{equation*}
\mathcal{P}_n\mathcal{C} := \varinjlim(\mathcal{C} \longrightarrow \mathcal{T}_n\mathcal{C} \longrightarrow \mathcal{T}_n(\mathcal{T}_n\mathcal{C}) \longrightarrow \cdots ),
\end{equation*}
where the colimit is taken inside the $\infty$-category $\mathbf{Cat}_*^\omega$ and the arrows are the functors $L_n$. Denote the resulting adjunction by
\[
\xymatrix@C=25pt{
\mathcal{C} \ar@<.5ex>[r]^{\Sigma_n^\infty} & \mathcal{P}_n\mathcal{C}. \ar@<.5ex>[l]^{\Omega_n^\infty}
}
\]
\end{definition}

\begin{remark}
The assignment $\mathcal{P}_n$ is a functor $\mathbf{Cat}_*^\omega \rightarrow \mathbf{Cat}_*^\omega$. Instead of forming the above colimit in $\mathbf{Cat}_*^\omega$ one could instead form the following colimit in $\mathbf{Cat}_*$, the $\infty$-category of small pointed $\infty$-categories:
\begin{equation*}
\mathcal{P}_n\mathcal{C}^c := \varinjlim(\mathcal{C}^c \longrightarrow \mathcal{T}_n\mathcal{C}^c \longrightarrow \mathcal{T}_n(\mathcal{T}_n\mathcal{C})^c \longrightarrow \cdots ).
\end{equation*}
Here a superscript $c$ denotes taking the full subcategory spanned by compact objects. Indeed, Lemma \ref{lem:filteredcolimcompacts} shows that
\begin{equation*}
\mathcal{P}_n\mathcal{C} := \mathrm{Ind}(\mathcal{P}_n\mathcal{C}^c).
\end{equation*}
\end{remark}

\begin{remark}
\label{rmk:P1stabilization}
The $\infty$-category $\mathcal{P}_1\mathcal{C}$ is the stabilization $\mathrm{Sp}(\mathcal{C})$ of $\mathcal{C}$. Indeed, note that a special punctured 2-cube in $\mathcal{C}$ is just a span
\begin{equation*}
* \rightarrow X \leftarrow *,
\end{equation*}
so that there is an evident equivalence $\mathcal{T}_1\mathcal{C} \simeq \mathcal{C}$. The composite functor
\begin{equation*}
\mathcal{C} \xrightarrow{L_1} \mathcal{T}_1 \simeq \mathcal{C}
\end{equation*}
is then precisely the suspension functor.
\end{remark}

\begin{lemma}
\label{lem:Pnfinitelimits}
The functor $\mathcal{P}_n: \mathbf{Cat}_*^\omega \rightarrow \mathbf{Cat}_*^\omega$ preserves finite limits.
\end{lemma}
\begin{proof}
This is immediate from Lemma \ref{lem:filteredcolimleftexact}.
\end{proof}

Let us briefly review Goodwillie's construction of $n$-excisive approximations to functors. The reader can consult \cite{goodwillie3} for the original treatment (valid for the categories of spaces and spectra) or Chapter 6 of \cite{higheralgebra} for an exposition that applies to the current setting. Given a functor $F: \mathcal{C} \longrightarrow \mathcal{D}$ between pointed compactly generated $\infty$-categories, one defines a new functor $T_n F: \mathcal{C} \longrightarrow \mathcal{D}$ as the composite
\[
\xymatrix@C=25pt{
\mathcal{C} \ar[r]^-{L_n} & \mathcal{T}_n\mathcal{C} \ar[r]^-{F \circ -} & \mathrm{Fun}(\mathbf{P}_0(n+1), \mathcal{D}) \ar[r]^-{\varprojlim} & \mathcal{D}.  
}
\]
There is an evident natural transformation $t_n F: F \rightarrow T_n F$ which we use to define
\[
\xymatrix{
P_n F := \varinjlim( F \ar[r]^-{t_n F} & T_n F \ar[r]^-{t_n(t_n F)} & T_n^2 F \ar[r] & \cdots ).
}
\]

The following observations can now easily be deduced from our constructions:

\begin{lemma}
\label{lem:Pnobservations}
Let $F: \mathcal{C} \longrightarrow \mathcal{D}$ be a functor between pointed compactly generated $\infty$-categories. Then:
\begin{itemize}
\item[(a)] The $n$-excisive approximation $P_n F$ canonically factors as follows:
\[
\xymatrix{
\mathcal{C} \ar[d]_{\Sigma_n^\infty}\ar[r]^{P_n F} & \mathcal{D}. \\
\mathcal{P}_n\mathcal{C} \ar[ur] &
}
\]
\item[(b)] If $F$ preserves colimits, compact objects, and terminal objects, then $P_n F$ canonically factors as follows:
\[
\xymatrix{
\mathcal{C} \ar[d]_{\Sigma_n^\infty}\ar[r]^{P_n F} & \mathcal{D} \\
\mathcal{P}_n\mathcal{C} \ar[r]_{\mathcal{P}_n F} & \mathcal{P}_n\mathcal{D}. \ar[u]_{\Omega^\infty_n}
}
\]
\item[(c)] The unit $\mathrm{id}_{\mathcal{C}} \rightarrow \Omega^\infty_n \Sigma^\infty_n$ of the adjunction $\mathcal{C} \rightleftarrows \mathcal{P}_n\mathcal{C}$ coincides with the natural transformation
\begin{equation*}
\mathrm{id}_\mathcal{C} \longrightarrow P_n\mathrm{id}_\mathcal{C}.
\end{equation*}  
\end{itemize} 
\end{lemma}
\begin{proof}
(a). Since $T_n F$ factors over $\mathcal{T}_n\mathcal{C}$ by definition, the colimit $P_n F = \varinjlim_k T_n^k F$ factors over the colimit $\mathcal{P}_n\mathcal{C}$.

(b). Under the stated assumptions $F$ induces a functor $\mathcal{T}_n\mathcal{C} \rightarrow \mathcal{T}_n\mathcal{D}$ and  $T_n F$ may be factored as follows:
\[
\xymatrix@C=25pt{
\mathcal{C} \ar[r]^-{L_n} & \mathcal{T}_n\mathcal{C} \ar[r]^-{F \circ -} & \mathcal{T}_n\mathcal{D} \ar[r]^-{\varprojlim} & \mathcal{D}.  
}
\]
Iterating $T_n$ and taking the colimit gives the result.

(c). Observe that
\begin{equation*}
P_n \mathrm{id}_{\mathcal{C}} = \varinjlim_k T_n^k\mathrm{id}_{\mathcal{C}} = \varinjlim_k R_n^k L_n^k = \Omega^\infty_n\Sigma^\infty_n.
\end{equation*}
\end{proof}

The following lemma and its corollaries will be needed later in this section.

\begin{lemma}
\label{lem:Pnstab}
The functor $\Sigma_n^\infty: \mathcal{C} \rightarrow \mathcal{P}_n\mathcal{C}$ induces an equivalence on stabilizations. More precisely, there is a commutative diagram of functors
\[
\xymatrix{
\mathcal{C} \ar[r]^{\Sigma_n^\infty} \ar[d]_{\Sigma^\infty} & \mathcal{P}_n\mathcal{C} \ar[d]^{\Sigma^\infty} \\
\mathrm{Sp}(\mathcal{C}) \ar[r]_{\partial \Sigma_n^\infty} & \mathrm{Sp}(\mathcal{P}_n\mathcal{C}) 
}
\]
in which the bottom functor is an equivalence.
\end{lemma}
\begin{proof}
We will prove a slightly stronger claim, namely that the functor $L_n: \mathcal{C} \rightarrow \mathcal{T}_n\mathcal{C}$ induces an equivalence on stabilizations. Indeed, chasing through the definitions, the induced functor can be identified with $L_n: \mathrm{Sp}(\mathcal{C}) \rightarrow \mathcal{T}_n\mathrm{Sp}(\mathcal{C})$, where $L_n$ now denotes the same construction as before, but applied to the $\infty$-category $\mathrm{Sp}(\mathcal{C})$. We claim that this functor is an equivalence. Indeed, the composite
\[
\xymatrix{
\mathrm{Sp}(\mathcal{C}) \ar[r]^-{L_n} & \mathcal{T}_n\mathrm{Sp}(\mathcal{C}) \ar[r]^-{R_n} & \mathrm{Sp}(\mathcal{C})
}
\]
is equivalent to the identity, since in a stable $\infty$-category a coCartesian cube is also Cartesian. To see that $L_n R_n$ is equivalent to the identity, suppose we are given $\mathcal{X}_0 \in \mathcal{T}_n\mathrm{Sp}(\mathcal{C})$, i.e. a special punctured $(n+1)$-cube in $\mathrm{Sp}(\mathcal{C})$. We can complete this to a Cartesian $(n+1)$-cube $\mathcal{X}$ in $\mathrm{Sp}(\mathcal{C})$ with
\begin{equation*}
\mathcal{X}(\varnothing) = \varprojlim \mathcal{X}_0.
\end{equation*}
Since we are working in a stable $\infty$-category, $\mathcal{X}$ is also coCartesian. We claim it is in fact strongly coCartesian, which follows from Lemma \ref{lem:cocartstrong} below. Now, since $\mathcal{X}$ is a strongly coCartesian cube, it follows that the map $L_n \mathcal{X}(\varnothing) \rightarrow \mathcal{X}_0$ is an equivalence, which concludes the proof.
\end{proof}

In the previous proof we used the following general fact in the particular case where $\mathcal{D} = \mathrm{Sp}(\mathcal{C})$.

\begin{lemma}
\label{lem:cocartstrong}
Suppose $\mathcal{D}$ is a stable $\infty$-category, $k \geq 1$, and $\mathcal{X}: \mathbf{P}(k) \rightarrow \mathcal{D}$ is a $k$-cube such that the restriction $\mathcal{X}_0$ is a strongly coCartesian punctured cube. If moreover $\mathcal{X}$ is coCartesian, then it is in fact strongly coCartesian.
\end{lemma}
\begin{proof}
The proof is by induction on $k$. For $k=1$ and $k=2$ there is nothing to prove. Suppose the claim has been established for $k - 1$ and we wish to prove it for $k$. We need to show that the squares
\[
\xymatrix{
\mathcal{X}(I \cap J) \ar[r]\ar[d] & \mathcal{X}(I) \ar[d] \\
\mathcal{X}(J) \ar[r] & \mathcal{X}(I \cup J)
}
\]
are pushouts in the cases where $I \cap J$ is empty. By the pasting lemma for pushouts, it suffices to treat the cases where $I$ and $J$ are singletons, say $I = \{i\}$ and $J = \{j\}$. Pick an $l \in \{1, \ldots, k\}$ which is unequal to both $i$ and $j$ (note that this is possible, since we are in the case $k \geq 3$). Denote by $\mathbf{P}^{l \in}(k)$ the poset of subsets of $\{1, \ldots, k\}$ containing $l$ and by $\mathbf{P}^{l \notin}(k)$ the poset of subsets not containing $l$. Then consider the diagrams
\begin{eqnarray*}
\mathcal{Y}_0 & = & \mathcal{X}|_{\N\mathbf{P}^{l \notin}(k)}, \\
\mathcal{Y}_1 & = & \mathcal{X}|_{\N\mathbf{P}^{l \in}(k)},
\end{eqnarray*}
which are both $(k-1)$-cubes. The cube $\mathcal{Y}_1$ is strongly coCartesian by assumption, because the subsets it is indexed on are all non-empty. Since $\mathcal{X}$ is coCartesian, its total cofiber vanishes, so that the canonical map 
\begin{equation*}
\mathrm{tcof}(\mathcal{Y}_0) \longrightarrow \mathrm{tcof}(\mathcal{Y}_1)
\end{equation*}
is an equivalence. The latter vanishes again, since $\mathcal{Y}_1$ is coCartesian, so that $\mathcal{Y}_0$ must be coCartesian as well. Note that we use the assumption that $\mathcal{D}$ is stable to conclude that a cube with vanishing total cofiber is coCartesian. By the inductive hypothesis on $k-1$, we conclude that $\mathcal{Y}_0$ is in fact strongly coCartesian, which finishes the proof.
\end{proof}

\begin{corollary}[Corollary of Lemma \ref{lem:Pnstab}]
\label{cor:suspcomp}
If $X \in \mathcal{P}_n\mathcal{C}$ is a compact object then the $k$th suspension $\Sigma^k X$ is in the essential image of $\Sigma_n^\infty: \mathcal{C} \rightarrow \mathcal{P}_n\mathcal{C}$ for some $k \geq 0$.
\end{corollary}
\begin{proof}
Suppose $X \in \mathcal{P}_n\mathcal{C}$ is compact. Consider the diagram
\[
\xymatrix{
\mathcal{C}^c \ar[r]^\Sigma \ar[d]_{\Sigma_n^\infty} & \mathcal{C}^c \ar[r]^\Sigma \ar[d]_{\Sigma_n^\infty}  & \mathcal{C}^c \ar[r]^\Sigma \ar[d]_{\Sigma_n^\infty} & \cdots \ar[r] & \mathrm{Sp}(\mathcal{C})^c \ar[d]^{\partial \Sigma_n^\infty} \\
\mathcal{P}_n\mathcal{C}^c \ar[r]^\Sigma & \mathcal{P}_n\mathcal{C}^c \ar[r]^\Sigma & \mathcal{P}_n\mathcal{C}^c \ar[r]^\Sigma & \cdots \ar[r] & \mathrm{Sp}(\mathcal{P}_n\mathcal{C})^c.
}
\]
Both rows are colimit diagrams in $\mathbf{Cat}_\ast$ and the rightmost vertical arrow is an equivalence by Lemma \ref{lem:Pnstab}. Thus, there exists an object $Y \in \mathrm{Sp}(\mathcal{C})^c$ whose image in $\mathrm{Sp}(\mathcal{P}_n\mathcal{C})^c$ is equivalent to $\Sigma^\infty X$. Since the top row is a filtered colimit there exists a $k \geq 0$ and an object $Y'$ in the $k$th entry of that row whose image in $\mathrm{Sp}(\mathcal{C})$ is equivalent to $Y$. It is then not necessarily true that $\Sigma_n^\infty Y'$ is equivalent to $\Sigma^k X$, but this will be true after several suspensions; indeed, since $\Sigma_n^\infty Y'$ and $\Sigma^k X$ have equivalent images in the colimit of the bottom row, there exists an $l \geq 0$ such that $\Sigma_n^\infty \Sigma^l Y'$ and $\Sigma^{k+l} X$ are equivalent.
\end{proof}


Our goal for the rest of this section is to prove Theorem \ref{thm:TheoremA}. We start with the following:

\begin{proposition}
\label{prop:Pnweakapprox}
The adjunction $\mathcal{C} \rightleftarrows \mathcal{P}_n\mathcal{C}$ is a weak $n$-excisive approximation to $\mathcal{C}$.
\end{proposition}
\begin{proof}
Part (c) of Lemma \ref{lem:Pnobservations} says that $P_n\mathrm{id}_\mathcal{C} \rightarrow \Omega^\infty_n\Sigma^\infty_n$ is an equivalence. Let us now show that the identity functor of $\mathcal{P}_n\mathcal{C}$ is $n$-excisive, which is also more or less immediate from our constructions. Indeed, $T_n(\mathrm{id}_{\mathcal{P}_n\mathcal{C}})$ is given by the composite
\[
\xymatrix{
\mathcal{P}_n\mathcal{C} \ar[r]^-{L_n} & \mathcal{T}_n\mathcal{P}_n\mathcal{C} \ar[r]^-{R_n} & \mathcal{P}_n\mathcal{C}. 
}
\]
The functor $L_n$ is an equivalence by construction, so that this composite is equivalent to the identity. Therefore
\begin{equation*}
P_n(\mathrm{id}_{\mathcal{P}_n\mathcal{C}}) = \varinjlim_k T_n^k (\mathrm{id}_{\mathcal{P}_n\mathcal{C}}) \simeq \mathrm{id}_{\mathcal{P}_n\mathcal{C}}.
\end{equation*}
It remains to show that the natural transformation
\begin{equation*}
P_n(\Sigma_n^\infty\Omega_n^\infty) \longrightarrow \mathrm{id}_{\mathcal{P}_n\mathcal{C}}
\end{equation*}
is an equivalence. Since all functors involved commute with filtered colimits it suffices to show that this natural transformation is an equivalence after evaluating on every compact object. First, consider an object $X \in \mathcal{P}_n\mathcal{C}$ that is equivalent to $\Sigma_n^\infty Y$ for some $Y \in \mathcal{C}$. The triangle identities for the adjunction $(\Sigma_n^\infty, \Omega_n^\infty)$ yield a diagram
\[
\xymatrix{
\Sigma_n^\infty\ar@{=}[dr]\ar[d] & \\
\Sigma_n^\infty\Omega_n^\infty\Sigma_n^\infty \ar[r] & \Sigma_n^\infty.
}
\]
Observe that $\Sigma_n^\infty$ is $n$-excisive and furthermore
\begin{equation*}
P_n(\Sigma_n^\infty\Omega_n^\infty\Sigma_n^\infty) \simeq P_n(\Sigma_n^\infty\Omega_n^\infty)\Sigma_n^\infty.
\end{equation*}
This latter observation follows from the fact that $\Sigma_n^\infty$ preserves colimits. Therefore, applying $P_n$ to the previous diagram and evaluating at $Y$ yields the diagram
\[
\xymatrix{
\Sigma_n^\infty Y\ar@{=}[dr]\ar[d] & \\
P_n(\Sigma_n^\infty\Omega_n^\infty)\Sigma_n^\infty Y \ar[r] & \Sigma_n^\infty Y.
}
\]
We claim that the vertical map is an equivalence. Indeed, since the unit $\mathrm{id}_\mathcal{C} \rightarrow \Omega_n^\infty\Sigma_n^\infty$ is a $P_n$-equivalence (i.e. an equivalence after applying $P_n$), so is the natural transformation
\begin{equation*}
\Sigma_n^\infty \longrightarrow \Sigma_n^\infty\Omega_n^\infty\Sigma_n^\infty
\end{equation*}
obtained by whiskering the unit with $\Sigma_n^\infty$. We conclude that the map
\begin{equation*}
P_n(\Sigma_n^\infty\Omega_n^\infty)\Sigma_n^\infty Y \longrightarrow \Sigma_n^\infty Y
\end{equation*}
is an equivalence as well. Let us now show how to reduce the case of a general compact object $X$ to this one. By Corollary \ref{cor:suspcomp} there exists a $k \geq 0$ such that the $k$th suspension $\Sigma^k X$ is in the essential image of $\Sigma_n^\infty$, so that
\begin{equation*}
P_n(\Sigma_n^\infty\Omega_n^\infty)\Sigma^k X \longrightarrow \Sigma^k X
\end{equation*}
is an equivalence by the argument above. Let us show that 
\begin{equation*}
P_n(\Sigma_n^\infty\Omega_n^\infty)\Sigma^{k-1} X \longrightarrow \Sigma^{k-1} X
\end{equation*}
is also an equivalence. Iterating our argument $k$ times will then finish the proof. By Remark \ref{rmk:Cn} the map of $(n+1)$-cubes
\begin{equation*}
P_n(\Sigma_n^\infty\Omega_n^\infty) C_n(\Sigma^{k-1} X) \longrightarrow C_n(\Sigma^{k-1} X)
\end{equation*}
is an equivalence at every vertex corresponding to a non-empty $S \subseteq \{1, \ldots, n+1\}$. The cube $C_n(\Sigma^{k-1} X)$ is strongly coCartesian. Since $P_n(\Sigma_n^\infty\Omega_n^\infty)$ and $\mathrm{id}_{\mathcal{P}_n\mathcal{C}}$ are $n$-excisive, both the domain and codomain of the map above are then Cartesian cubes. Therefore the map
\begin{equation*}
P_n(\Sigma_n^\infty\Omega_n^\infty)\Sigma^{k-1} X \longrightarrow \Sigma^{k-1} X
\end{equation*}
obtained by evaluating at the initial vertex must be an equivalence as well, establishing the inductive step.
\end{proof}

We are after the following strengthening of the previous proposition:

\begin{proposition}
\label{prop:Pnapprox}
The adjunction $\mathcal{C} \rightleftarrows \mathcal{P}_n\mathcal{C}$ is a strong $n$-excisive approximation to $\mathcal{C}$.
\end{proposition}

This proposition follows directly from Proposition \ref{prop:Pnweakapprox} and the following characterization of $n$-excisive categories:

\begin{proposition}
\label{prop:nexcchar}
Let $\mathcal{C}$ be a pointed, compactly generated $\infty$-category. Then $\mathcal{C}$ is $n$-excisive if and only if the functor $L_n: \mathcal{C} \longrightarrow \mathcal{T}_n\mathcal{C}$ is an equivalence.
\end{proposition}
\begin{proof}
First suppose $\mathcal{C}$ is $n$-excisive. Consider the diagram
\[
\xymatrix{
\mathcal{C} \ar[r]^{L_n}\ar[d]_{\Sigma_n^\infty} & \mathcal{T}_n\mathcal{C} \ar[d]^{\mathcal{T}_n\Sigma_n^\infty} \\
\mathcal{P}_n\mathcal{C} \ar[r]_{\mathcal{P}_n L_n} & \mathcal{T}_n\mathcal{P}_n\mathcal{C}.
}
\]
The bottom horizontal map is an equivalence by construction; the left vertical map is an equivalence since it is a weak $n$-excisive approximation and we assumed $\mathcal{C}$ to be $n$-excisive; the right vertical arrow is obtained by applying $\mathcal{T}_n$ to the left one and is therefore an equivalence as well. We conclude that $L_n$ must be an equivalence. Conversely, assume $L_n: \mathcal{C} \longrightarrow \mathcal{T}_n\mathcal{C}$ is an equivalence. Then first of all
\begin{equation*}
\mathrm{id}_{\mathcal{C}} \simeq R_n L_n \simeq T_n\mathrm{id}_\mathcal{C}.
\end{equation*}
Iterating $T_n$ gives $\mathrm{id}_{\mathcal{C}} \simeq P_n \mathrm{id}_{\mathcal{C}}$, so the identity functor of $\mathcal{C}$ is $n$-excisive. Now let $F: \mathcal{C} \rightleftarrows \mathcal{D} : G$ be a weak $n$-excisive approximation to $\mathcal{C}$. We wish to show that $F$ is an equivalence. First recall that $F$ must be a fully faithful functor as a consequence of the fact that $\mathrm{id}_{\mathcal{C}}$ is $n$-excisive. We should argue that it is essentially surjective. Since $F$ preserves compact objects and filtered colimits it suffices to show that every compact object $X$ of $\mathcal{D}$ is in the essential image of $F$. Lemma \ref{lem:stabapprox} below tells us that $F$ induces an equivalence on stabilizations, so that by the same reasoning as in Corollary \ref{cor:suspcomp} there exists a $k \geq 0$ such that $\Sigma^k X$ is in the essential image of $F$. We will show that $X$ itself is also in this essential image. Suppose $k >0$ (the case $k=0$ being trivial) and consider the strongly coCartesian $(n+1)$-cube $C_n(\Sigma^{k-1}X)$. Denote by $\mathcal{X}_0$ its restriction to $\mathbf{P}_0(n)$. The special punctured cube $\mathcal{X}_0$ features only $\Sigma^k X$ and coproducts of this object with itself and is therefore contained in the essential image of $F$, so that we may lift it (in an essentially unique way) to a punctured cube $\mathcal{X}'_0$ in $\mathcal{C}$. This punctured cube is again special; recall that $\mathcal{C}$ is a colocalization of $\mathcal{D}$, so that $F$ creates colimits in $\mathcal{C}$. We may therefore think of $\mathcal{X}'_0$ as an object of $\mathcal{T}_n\mathcal{C}$. By taking the limit, we may extend $\mathcal{X}'_0$ to an $n$-cube $\mathcal{X}'$ in $\mathcal{C}$ satisfying $\mathcal{X}'(\varnothing) = R_n\mathcal{X}'_0$. We claim that $\mathcal{X}'$ is a strongly coCartesian cube. Consider the natural map of $(n+1)$-cubes
\begin{equation*}
C_n(\mathcal{X}'(\varnothing)) \longrightarrow \mathcal{X}'.
\end{equation*}
Both cubes have the same initial vertex and the restriction of the left-hand side to $\mathbf{P}_0(n+1)$ is by definition $L_n R_n \mathcal{X}'_0$, which is equivalent to $\mathcal{X}'_0$ by assumption (b). This shows that the map above is an equivalence, so that $\mathcal{X}'$ is indeed strongly coCartesian. The cube $F\mathcal{X}'$ is then strongly coCartesian as well. Since $\mathcal{D}$ has $n$-excisive identity functor, this cube is Cartesian, so we may conclude that
\begin{equation*}
F\mathcal{X}'(\varnothing) \simeq \varprojlim F\mathcal{X}'_0 \simeq \varprojlim \mathcal{X}_0 \simeq \Sigma^{k-1} X.
\end{equation*}
This proves $\Sigma^{k-1}X$ is in the essential image of $F$ as well. A descending induction on $k$ now shows that $X$ itself is in the essential image of $F$, finishing the proof.
\end{proof}

\begin{corollary}
\label{cor:1excstable}
A pointed compactly generated $\infty$-category $\mathcal{C}$ is 1-excisive if and only if it is stable.
\end{corollary}
\begin{proof}
As in Remark \ref{rmk:P1stabilization}, the functor $L_1$ can be identified with the suspension functor $\Sigma: \mathcal{C} \rightarrow \mathcal{C}$. But this is an equivalence if and only if $\mathcal{C}$ is stable (cf. Corollary 1.4.2.27 of \cite{higheralgebra}).
\end{proof}

The following is a more explicit reformulation of the characterization of Proposition \ref{prop:nexcchar}:

\begin{corollary}
\label{cor:bnexcisive}
Let $\mathcal{C}$ be as in the previous proposition. Then $\mathcal{C}$ is $n$-excisive if and only if $\mathcal{C}$ satisfies the following:
\begin{itemize}
\item[(a)] The identity functor $\mathrm{id}_\mathcal{C}$ is $n$-excisive.
\item[(b)] If $\mathcal{X}: \mathbf{P}(n+1) \rightarrow \mathcal{C}$ is a Cartesian $(n+1)$-cube such that its restriction $\mathcal{X}|_{\mathbf{P}_0(n+1)}$ is a special punctured $(n+1)$-cube, then $\mathcal{X}$ is in fact strongly coCartesian.
\end{itemize}
\end{corollary}
\begin{proof}
Suppose $\mathcal{C}$ is $n$-excisive. Then (a) is immediate and (b) is just an explicit formulation of what it means for the counit $L_n R_n \rightarrow \mathrm{id}_{\mathcal{T}_n\mathcal{C}}$ to be an equivalence (which is the case for $n$-excisive $\mathcal{C}$, by Proposition \ref{prop:nexcchar}). Conversely, suppose $\mathcal{C}$ satisfies (a) and (b). We already observed that if $\mathcal{C}$ satisfies (a), then the unit map $\mathrm{id}_{\mathcal{C}} \rightarrow R_n L_n$ is an equivalence. Now (b) guarantees the same for the counit, so that the pair $(L_n, R_n)$ is an adjoint equivalence of $\infty$-categories.
\end{proof}

In the proof of Proposition \ref{prop:nexcchar} we used the following lemma:

\begin{lemma}
\label{lem:stabapprox}
Let $F: \mathcal{C} \rightleftarrows \mathcal{D}: G$ be a weak $n$-excisive approximation. Then $F$ induces an equivalence on stabilizations.
\end{lemma}
\begin{proof}
Consider the linear functors $D_1 F$, $D_1 G$, $D_1(FG)$ and $D_1(GF)$ obtained by taking first Goodwillie derivatives. By the general theory of calculus, the functor $D_1 F$ canonically factors as
\[
\xymatrix{
\mathcal{C} \ar[r]^-{\Sigma^\infty_\mathcal{C}} & \mathrm{Sp}(\mathcal{C}) \ar[r]^-{\partial F} & \mathrm{Sp}(\mathcal{D}) \ar[r]^-{\Omega^\infty_\mathcal{D}} & \mathcal{D}
} 
\]
and of course a similar factorization exists for the other functors. By the Klein-Rognes chain rule (Corollary 6.2.1.24 of \cite{higheralgebra}), there are equivalences
\begin{eqnarray*}
\partial(FG) & \simeq & \partial F \circ \partial G, \\
\partial(GF) & \simeq & \partial G \circ \partial F.
\end{eqnarray*}
By the assumption that $F$ and $G$ form a weak $n$-excisive approximation, we know that $FG$ and $GF$ are $P_n$-equivalent to the identity functors of $\mathcal{C}$ and $\mathcal{D}$ respectively. In particular,
\begin{eqnarray*}
D_1(FG) & \simeq & D_1 \mathrm{id}_\mathcal{C}, \\ 
D_1(GF) & \simeq & D_1 \mathrm{id}_\mathcal{D}. 
\end{eqnarray*}
It follows that $\partial F \circ \partial G \simeq \mathrm{id}_{\mathrm{Sp}(\mathcal{C})}$ and $\partial G \circ \partial F \simeq \mathrm{id}_{\mathrm{Sp}(\mathcal{D})}$, so that in particular $\partial F$ is an equivalence of $\infty$-categories.
\end{proof}

We can now prove our first main result.

\begin{proof}[Proof of Theorem \ref{thm:TheoremA}]
If $\mathcal{C}$ is an $n$-excisive $\infty$-category, then Proposition \ref{prop:nexcchar} lets us conclude that the functor $\mathcal{C} \rightarrow \mathcal{P}_n\mathcal{C}$ is an equivalence. To prove that $\mathcal{P}_n$ is indeed a localization functor, it suffices to prove the following statement: if $\mathcal{C}$ and $\mathcal{D}$ are pointed, compactly generated $\infty$-categories, where furthermore $\mathcal{D}$ is $n$-excisive, then precomposition with the functor $\Sigma_n^\infty: \mathcal{C} \rightarrow \mathcal{P}_n\mathcal{C}$ induces an equivalence
\begin{equation*}
(\Sigma_n^\infty)^\ast: \mathrm{Fun}^{\omega, L}(\mathcal{P}_n\mathcal{C}, \mathcal{D}) \longrightarrow \mathrm{Fun}^{\omega, L}(\mathcal{C}, \mathcal{D}). 
\end{equation*}
Here $\mathrm{Fun}^{\omega,L}$ denotes the $\infty$-category of colimit-preserving functors which in addition preserve compact objects. An explicit inverse to $(\Sigma_n^\infty)^\ast$ can be described as follows: given a functor $F \in \mathrm{Fun}^{\omega, L}(\mathcal{C}, \mathcal{D})$, form the composition of $\mathcal{P}_n F: \mathcal{P}_n\mathcal{C} \rightarrow \mathcal{P}_n\mathcal{D}$ with the equivalence $\Omega_n^\infty: \mathcal{P}_n\mathcal{D} \rightarrow \mathcal{D}$. To see that this gives an inverse, first observe that for such an $F$, the composite
\[
\xymatrix{
\mathcal{C} \ar[r]^-{\Sigma_n^\infty} & \mathcal{P}_n\mathcal{C} \ar[r]^-{\mathcal{P}_n F} & \mathcal{P}_n\mathcal{D} \ar[r]^-{\Omega_n^\infty} & \mathcal{D}
}
\] 
is $P_n F$ by part (b) of Lemma \ref{lem:Pnobservations}. Since $F$ is $n$-excisive, this is naturally equivalent to $F$ itself. For the other direction, suppose $G \in \mathrm{Fun}^{\omega, L}(\mathcal{P}_n\mathcal{C}, \mathcal{D})$ and consider the composite
\[
\xymatrix@C=35pt{
\mathcal{P}_n\mathcal{C} \ar[r]^-{\mathcal{P}_n(G\Sigma_n^\infty)} & \mathcal{P}_n\mathcal{D} \ar[r]^-{\Omega_n^\infty} & \mathcal{D}.
}
\] 
Again, a simple unravelling of the definitions shows that this composite is $P_n G$, which is naturally equivalent to $G$, since $G$ itself is $n$-excisive. Part (c) of the theorem is precisely Lemma \ref{lem:Pnfinitelimits}. 

The final claim of the theorem is that strong $n$-excisive approximations are unique up to canonical equivalence. To see this, suppose $F: \mathcal{C} \rightleftarrows \mathcal{D}: G$ is a strong $n$-excisive approximation to $\mathcal{C}$. By what we have proved, there is a canonical factorization of $F$ into colimit-preserving functors
\[
\xymatrix{
\mathcal{C} \ar[r]^{\Sigma_n^\infty} & \mathcal{P}_n\mathcal{C} \ar[r]^{F'} & \mathcal{D}.
}
\]
Denote the right adjoint to $F'$ by $G'$. We claim that the pair $(F', G')$ is a weak $n$-excisive approximation to $\mathcal{P}_n\mathcal{C}$. Since the latter $\infty$-category is $n$-excisive, it then follows that $F'$ must be an equivalence. To prove our claim, first observe that repeatedly using the equivalence between $P_n(\Sigma_n^\infty\Omega_n^\infty)$ and the identity functor of $\mathcal{P}_n\mathcal{C}$ yields the following chain of equivalences:
\begin{eqnarray*}
P_n(G'F') & \simeq & P_n(\Sigma_n^\infty \Omega_n^\infty G'F' \Sigma_n^\infty \Omega_n^\infty) \\
& \simeq & P_n(\Sigma_n^\infty GF \Omega_n^\infty) \\
& \simeq & P_n(\Sigma_n^\infty \Omega_n^\infty) \\
& \simeq & \mathrm{id}_{\mathcal{P}_n\mathcal{C}}.
\end{eqnarray*}
Furthermore, this composite of equivalences can easily be seen to be inverse to the image under $P_n$ of the unit of the adjoint pair $(F',G')$. Similarly, one has the chain of equivalences
\begin{eqnarray*}
P_n(F'G') & \simeq & P_n(F' \Sigma_n^\infty \Omega_n^\infty G') \\
& \simeq & P_n(FG) \\
& \simeq & \mathrm{id}_{\mathcal{D}},
\end{eqnarray*}
which together with our previous observation proves that $(F',G')$ is indeed a weak $n$-excisive approximation.
\end{proof}

\chapter{Another construction of polynomial approximations}
\label{sec:informalconstr}

Our goal is to provide a classification of the $n$-stages of a nonunital stable $\infty$-operad $\mathcal{O}^\otimes$, in the sense of Definition \ref{def:nstage}. To achieve this we will establish a different construction of the polynomial approximations $\mathcal{P}_n\mathcal{C}$ of the previous chapter. It is essentially given by the \emph{Tate coalgebras} already loosely described in the introduction. This construction can be modified to obtain $n$-stages for a stable $\infty$-operad $\mathcal{O}^\otimes$ and we will exploit this to prove Theorem \ref{thm:classification}. In this section we give an informal outline of our strategy in order to orient the reader. A more detailed treatment and the proof of Theorem \ref{thm:classification} will be given in the following chapters.

Almost all constructions and proofs in the next few chapters proceed by induction on $n$. They are therefore less direct than what was done before. However, they yield a much more explicit understanding of the $\infty$-categories $\mathcal{P}_n\mathcal{C}$ and the way they relate for different values of $n$. Throughout this chapter $\mathcal{C}$ denotes a pointed compactly generated $\infty$-category. 

Out of the $(n-1)$-excisive approximation $\mathcal{P}_{n-1}\mathcal{C}$ we will construct an $n$-excisive $\infty$-category $\mathcal{Q}_n\mathcal{C}$, which will turn out to be naturally equivalent to $\mathcal{P}_n\mathcal{C}$. Any object $X \in \mathcal{C}$ is canonically a \emph{coalgebra} for the Cartesian product, meaning there are diagonal maps $X \rightarrow (X^{\times k})^{h\Sigma_k}$ for every $k \geq 1$ satisfying appropriate coherence relations. This gives $\Sigma^\infty X$ the structure of a coalgebra in $\mathrm{Sp}(\mathcal{C})^\otimes$; more precisely, we obtain `diagonal maps' $\delta_k: \Sigma^\infty X \rightarrow (\Sigma^\infty X\otimes^k \cdots \otimes^k \Sigma^\infty X)^{h\Sigma_k}$ satisfying similar coherence relations, where $\otimes^k$ denotes the $k$-fold tensor product arising from the stable $\infty$-operad $\mathrm{Sp}(\mathcal{C})^\otimes$. We describe such coalgebra structures in detail in the next chapter. Similarly, for $Y \in \mathcal{P}_{n-1}\mathcal{C}$, the object $\Sigma^\infty_{n-1,1}Y$ acquires a coalgebra structure in $\mathrm{Sp}(\mathcal{P}_{n-1}\mathcal{C})^\otimes$. Write $\odot^k$ for the $k$-fold tensor product determined by this stable $\infty$-operad. For $1 \leq k \leq n-1$ there are natural equivalences $\odot^k \simeq \otimes^k$. For $k \geq n$ the tensor product $\odot^k$ is the limit of a diagram recording the ways in which the $\otimes^j$ for $j < n$ can be combined to build products of $k$ variables (see Proposition \ref{prop:truncatedtensor} for a precise statement). The $n$-fold comultiplication of $\Sigma^\infty_{n-1,1}Y$ with respect to $\odot^n$ yields a map for which we write
\begin{equation*}
\delta_{<n}: \Sigma^\infty_{n-1,1}Y \longrightarrow (\Sigma^\infty_{n-1,1}Y \odot^n \cdots \odot^n \Sigma^\infty_{n-1,1}Y)^{h\Sigma_n}.
\end{equation*}
Here the notation $\delta_{<n}$ is meant to indicate the fact that it is built from the comultiplication maps $\delta_j$ for $j<n$, in the same way that $\odot^n$ is built from the tensor products $\otimes^j$ with $j < n$. If $Y$ is of the form $\Sigma_{n-1}^\infty X$ for some $X \in \mathcal{C}$, the coalgebra structure of $\Sigma^\infty X$ described previously also yields a map
\begin{equation*}
\tau_n: \Sigma^\infty_{n-1,1}Y \longrightarrow (\Sigma^\infty_{n-1,1}Y \otimes^n \cdots \otimes^n \Sigma^\infty_{n-1,1}Y)^{t\Sigma_n}
\end{equation*} 
by composing $\delta_n$ with the natural map from fixed points to Tate construction. The codomain of this map is an $(n-1)$-excisive functor of $Y$ (see Lemma \ref{lem:normseq}). It is not hard to see that using this fact in combination with Lemma \ref{lem:nexcfunctors} one obtains an extension of the above natural transformation to \emph{all} objects $Y$ of $\mathcal{P}_{n-1}\mathcal{C}$, rather than just those in the image of $\Sigma^\infty_{n-1}$. We will refer to this natural transformation as the \emph{Tate diagonal}. The map of $\infty$-operads $\mathrm{Sp}(\mathcal{P}_{n-1}\mathcal{C})^\otimes \rightarrow \mathrm{Sp}(\mathcal{C})^\otimes$ induces a natural transformation $\otimes^n \rightarrow \odot^n$ under which the map $\tau_n$ is compatible with the composition of $\delta_{<n}$ with the map from fixed points to Tate construction.

We can now informally describe $\mathcal{Q}_n\mathcal{C}$ as follows. A compact object of this category is a compact object $Y$ of $\mathcal{P}_{n-1}\mathcal{C}$, equipped with an $n$-fold comultiplication
\begin{equation*}
\delta_n: \Sigma^\infty_{n-1,1}Y \longrightarrow (\Sigma^\infty_{n-1,1}Y \otimes^n \cdots \otimes^n \Sigma^\infty_{n-1,1}Y)^{h\Sigma_n}.
\end{equation*}
This map should be compatible both with the Tate diagonal $\tau_n$ described above (under the natural map from fixed points to Tate construction) and with the coalgebra structure of $\Sigma^\infty_{n-1,1}Y$ in $\mathrm{Sp}(\mathcal{P}_{n-1}\mathcal{C})^\otimes$ (under the natural transformation $\otimes^n \rightarrow \odot^n$), as expressed by the map $\delta_{<n}$. More precisely, the object $\overline{Y} = \Sigma^\infty_{n-1,1}Y$ should be equipped with a lift as indicated by the dashed arrow in the following diagram:
\[
\xymatrix{
\overline{Y} \ar@/_1pc/[ddr]_{\delta_{<n}}\ar@/^1pc/[drr]^{\tau_n} \ar@{-->}[dr]^{\delta_n}& & \\
& \bigl(\overline{Y}^{\otimes n}\bigr)^{h\Sigma_n} \ar[r]\ar[d] & \bigl(\overline{Y}^{\otimes n}\bigr)^{t\Sigma_n} \ar[d] \\
& \bigl(\overline{Y}^{\odot n}\bigr)^{h\Sigma_n} \ar[r] &  \bigl(\overline{Y}^{\odot n}\bigr)^{t\Sigma_n}.
}
\]
Maps in $\mathcal{Q}_n\mathcal{C}$ between such objects are maps in $\mathcal{P}_{n-1}\mathcal{C}$ that respect this extra structure. 

The reader should observe that the only piece of input from $\mathcal{C}$ needed to construct $\mathcal{Q}_n\mathcal{C}$ out of $\mathcal{P}_{n-1}\mathcal{C}$ is the Tate diagonal. For any other choice of natural transformation
\begin{equation*}
\overline{Y} \longrightarrow (\overline{Y} \otimes^n \cdots \otimes^n \overline{Y})^{t\Sigma_n}
\end{equation*}
compatible with the map 
\begin{equation*}
\overline{Y} \longrightarrow (\overline{Y} \odot^n \cdots \odot^n \overline{Y})^{t\Sigma_n}
\end{equation*}
we could have carried out the same construction. The resulting $\infty$-category is not necessarily equivalent to $\mathcal{P}_n\mathcal{C}$, but is still an $n$-stage for $\mathrm{Sp}(\mathcal{C})^\otimes$. We will show that choosing the Tate diagonal is essentially the \emph{only} freedom we have when building $n$-stages out of $(n-1)$-stages. Verifying this claim will yield a proof of Theorem \ref{thm:classification}.

\begin{example}
\label{ex:P2S}
Let $\mathcal{C}$ be the $\infty$-category $\mathcal{S}_*$ of pointed spaces, for which $\mathrm{Sp}(\mathcal{C})^\otimes$ is the usual symmetric monoidal $\infty$-category of spectra with its smash product, so $\otimes^n = \wedge^n$. Consider a finite pointed space $X$. The diagonal $X \rightarrow (X \times X)^{h\Sigma_2}$ gives rise to the Tate diagonal $\tau_2: \Sigma^\infty X \rightarrow (\Sigma^\infty X \wedge \Sigma^\infty X)^{t\Sigma_2}$. The codomain of this map is an \emph{exact} functor of $\Sigma^\infty X$. Therefore, the Tate diagonal extends to a natural map $\tau_2: E \rightarrow (E \wedge E)^{t\Sigma_2}$ for any finite spectrum $E$. The $\infty$-category of compact objects in $\mathcal{P}_2{\mathcal{S}_*}$ is then the $\infty$-category of finite spectra $Y$ equipped with a lift as in the following diagram:
\[
\xymatrix{
& (E \wedge E)^{h\Sigma_2} \ar[d] \\
E \ar@{-->}[ur]^{\delta_2}\ar[r]_-{\tau_2} & (E \wedge E)^{t\Sigma_2}.  
}
\]
\end{example}

\begin{example}
\label{ex:P3S}
Let us also give a description of $\mathcal{P}_3\mathcal{S}_*$. This example already contains all the essential features of the general case. Write $\odot^3$ for the 3-fold tensor product determined by the stable $\infty$-operad $\mathrm{Sp}(\mathcal{P}_2\mathcal{C})^\otimes$. We give a formula for such tensor products in Proposition \ref{prop:truncatedtensor} but in this special case it is very simple, as we will now explain. The poset $\mathbf{Part}_2(3)$ of nontrivial and nondiscrete partitions of the set $\{1,2,3\}$ conists of three elements (namely $(12)3$ and its cyclic permutations) and no nontrivial relations between them. The symmetric group $\Sigma_3$ acts in the evident way by permuting the three letters, so every element of $\mathbf{Part}_2(3)$ has a stabilizer of order two. Then for a spectrum $E$ we have
\begin{equation*}
E \odot^3 E \odot^3 E \simeq N\mathbf{Part}_2(3) \wedge E^{\wedge 3}.
\end{equation*}
In other words it is just a sum of three copies of the ordinary threefold smash product, but with a twisted action of $\Sigma_3$. The natural transformation $\wedge^3 = \otimes^3 \rightarrow \odot^3$ is the diagonal of this threefold sum. Now suppose $Y$ is a compact object of $\mathcal{P}_2\mathcal{S}_*$, i.e. a finite spectrum equipped with a twofold diagonal map $\delta_2$ compatible with the Tate diagonal $\tau_2$ as in Example \ref{ex:P2S} above. This structure in particular equips the underlying spectrum of $Y$ (which we denote by $\overline{Y}$) with a further Tate diagonal
\begin{equation*}
\tau_3: \overline{Y} \rightarrow (\overline{Y}^{\wedge 3})^{t\Sigma_3}.
\end{equation*}
Indeed, for the moment writing $F$ for the 2-excisive functor $E \mapsto (E^{\wedge 3})^{t\Sigma_3}$, Lemma \ref{lem:nexcfunctors} guarantees that the evident map
\begin{equation*}
\mathrm{Nat}(\Sigma^\infty_{2,1}, F \circ \Sigma^\infty_{2,1}) \rightarrow \mathrm{Nat}(\Sigma^\infty, F \circ \Sigma^\infty)
\end{equation*}
is an equivalence. Here the domain (resp. codomain) concerns natural transformations between 2-excisive functors $\mathcal{P}_2\mathcal{S}_* \rightarrow \mathrm{Sp}$ (resp. 2-excisive functors $\mathcal{S}_* \rightarrow \mathrm{Sp}$). Then $\tau_3$ is the preimage (which is well-defined up to contractible ambiguity) of the natural transformation
\begin{equation*}
\Sigma^\infty X \rightarrow (\Sigma^\infty X^{\wedge 3})^{h\Sigma_3} \rightarrow (\Sigma^\infty X^{\wedge 3})^{t\Sigma_3}. 
\end{equation*}

The twofold comultiplication map $\delta_2: \overline{Y} \rightarrow \overline{Y} \wedge \overline{Y}$ (now thought of as a $\Sigma_2$-equivariant map) can be used to construct a $\Sigma_3$-equivariant map
\begin{equation*}
\delta_{<3}: \overline{Y} \rightarrow \overline{Y} \odot^3 \overline{Y} \odot^3 \overline{Y}.
\end{equation*}
Indeed, on the component of the tensor product corresponding to the partition $(12)3$ one uses
\begin{equation*}
\overline{Y} \xrightarrow{\delta_2} \overline{Y} \wedge \overline{Y} \xrightarrow{\delta_2 \wedge \mathrm{id}}  (\overline{Y} \wedge \overline{Y}) \wedge \overline{Y}
\end{equation*}
and similarly for the others. This is compatible with the Tate diagonal as in the following diagram of solid arrows:
\[
\xymatrix{
\overline{Y} \ar@/_1pc/[ddr]_{\delta_{<3}}\ar@/^1pc/[drr]^{\tau_3} \ar@{-->}[dr]^{\delta_3}& & \\
& \bigl(\overline{Y}^{\wedge 3}\bigr)^{h\Sigma_3} \ar[r]\ar[d] & \bigl(\overline{Y}^{\wedge 3}\bigr)^{t\Sigma_3} \ar[d] \\
& \bigl(\overline{Y}^{\odot 3}\bigr)^{h\Sigma_3} \ar[r] &  \bigl(\overline{Y}^{\odot 3}\bigr)^{t\Sigma_3}.
}
\]
A compact object of $\mathcal{P}_3\mathcal{S}_*$ is then precisely a compact object $Y$ of $\mathcal{P}_2\mathcal{S}_*$ together with a lift $\delta_3$ as indicated by the dashed arrow in the diagram (which should be thought of as a `threefold comultiplication' compatible with all previously defined structure).
\end{example}

\begin{remark}
Observe that in Example \ref{ex:P3S} the Tate diagonal $\tau_3$ arises from general Goodwillie calculus results, rather than from direct construction. Although the map $\tau_3: \overline{Y} \rightarrow (\overline{Y}^{\wedge 3})^{t\Sigma_3}$ only depends on the comultiplication $\delta_2: \overline{Y} \rightarrow (\overline{Y}^{\wedge 2})^{h\Sigma_2}$ and its compatibility with $\tau_2$, it is not clear how to give a direct formula for $\tau_3$ in terms of this data. 
\end{remark}

\chapter{Coalgebras in stable $\infty$-operads}
\label{sec:coalgebras}

To make the ideas of the previous chapter precise we need to investigate the homotopy theory of \emph{coalgebras} in stable $\infty$-operads and their relation to the \emph{truncations} of a stable $\infty$-operad. In Section \ref{subsec:truncations} we define such truncations and state their universal property. In Section \ref{subsec:coalgebras} we define and study coalgebras in corepresentable $\infty$-operads. Then in Section \ref{subsec:truncatedcoalgebras} we specialize to the setting of coalgebras in truncated stable $\infty$-operads. This material is of crucial importance to our proofs but rather technical in nature. The reader might therefore wish to skip this chapter on first reading and refer back to it as needed. The homotopy theory of coalgebras we use here is closely related to the one used by Arone and Ching in \cite{aronechingtowers}.

\section{Truncations of $\infty$-operads}
\label{subsec:truncations}

Suppose $\mathbf{O}$ is a \emph{nonunital} operad in the (ordinary) symmetric monoidal category of spectra (in any convenient formalism for such; it will not matter for the purposes of this discussion). From such an $\mathbf{O}$ we can construct its \emph{$n$-truncation} $\tau_n\mathbf{O}$, whose underlying symmetric sequence is defined by
\begin{equation*}
\tau_n\mathbf{O}(k) = \begin{cases} \mathbf{O}(k) & \mbox{if } 1 \leq k \leq n, \\ \ast & \mbox{if } k > n. \end{cases}
\end{equation*}
The operad structure on $\tau_n\mathbf{O}$ is inherited in the obvious way from $\mathbf{O}$, simply setting operations of arity greater than $n$ to zero. Observe that there is an evident map of operads $\mathbf{O} \rightarrow \tau_n\mathbf{O}$ (although the evident map of symmetric sequences going in the opposite direction is generally \emph{not} a map of operads). This map induces an isomorphism $\mathbf{O}(k) \rightarrow \tau_n\mathbf{O}(k)$ for $k \leq n$ and is the terminal map out of $\mathbf{O}$ with this property. In fact, writing $\mathbf{Op}$ for the category of nonunital operads in spectra and $\mathbf{Op}_n$ for its full subcategory spanned by operads whose terms $\mathbf{O}(k)$ are isomorphic to $\ast$ for $k >n$, the process of $n$-truncation described above provides a left adjoint $\mathbf{Op} \rightarrow \mathbf{Op}_n$ to the inclusion $\mathbf{Op}_n \rightarrow \mathbf{Op}$, exhibiting the category of $n$-truncated nonunital operads as a localization of $\mathbf{Op}$.

\begin{remark}
For the discussion above it is essential to work with nonunital operads. Indeed, suppose $\mathbf{P}$ is an operad with constant term $\mathbf{P}(0)$ and $\tau_n\mathbf{P}$ is defined by setting all operations of arity greater than $n$ to zero. Then the map of symmetric sequences $\mathbf{P} \rightarrow \tau_n\mathbf{P}$ need not be a map of operads; indeed it might not respect, for example, composition maps of the form $\mathbf{P}(n+1) \otimes \mathbf{P}(0) \rightarrow \mathbf{P}(n)$.
\end{remark}

Relevant to us will be a Koszul dual version of this discussion. To be precise, for a map of operads $\mathbf{O} \rightarrow \tau_n\mathbf{O}$ as above one can perform a simplicial bar construction \cite{ching} to obtain a map of \emph{cooperads} $\mathbf{B}(\mathbf{O}) \rightarrow \mathbf{B}(\tau_n\mathbf{O})$. This map gives an isomorphism of spectra
\begin{equation*}
\mathbf{B}(\mathbf{O})(k) \simeq \mathbf{B}(\tau_n\mathbf{O})(k)
\end{equation*}
for $1 \leq k \leq n$ (and is the terminal map of cooperads out of $\mathbf{B}(\mathbf{O})$ with this property). Indeed, the $k$th term of the bar construction $\mathbf{B}(\mathbf{O})$ depends only on $\mathbf{O}(m)$ for $m \leq k$, so that this observation follows from the corresponding statement for $\mathbf{O}$ itself. However, the terms $\mathbf{B}(\tau_n\mathbf{O})(k)$ need \emph{not} be contractible for $k > n$. Rather, one should think of the cooperad $\mathbf{B}(\tau_n\mathbf{O})$ as being `cofreely generated' by the terms of arity up to $n$.

For the purposes of informal discussion, let us think of the derivatives of the identity functor on a pointed compactly generated $\infty$-category $\mathcal{C}$ as an operad $\mathbf{O}$ (which is, of course, accurate in many settings of interest). The derivatives of the identity functor on the $n$-excisive approximation $\mathcal{P}_n\mathcal{C}$ are then the $n$-truncation of this operad. We wish to adapt this discussion to the actual setting in which we work here, namely the stable $\infty$-operads $\mathrm{Sp}(\mathcal{C})^\otimes$ and $\mathrm{Sp}(\mathcal{P}_n\mathcal{C})^\otimes$. As explained in Remark \ref{rmk:dictionary} there is a dictionary between stable $\infty$-operads and cooperads, under which $\mathrm{Sp}(\mathcal{C})^\otimes$ corresponds to the cooperad $\partial_*(\Sigma^\infty\Omega^\infty)$ and similarly for $\mathrm{Sp}(\mathcal{P}_n)^\otimes$. The $n$-truncation map $\partial_*\mathrm{id}_{\mathcal{C}} \rightarrow \partial_*\mathrm{id}_{\mathcal{P}_n\mathcal{C}}$ corresponds (under the bar construction) to a map between these cooperads and then, by this dictionary, to the map of stable $\infty$-operads 
\begin{equation*}
\mathrm{Sp}(\mathcal{P}_n\mathcal{C})^\otimes \longrightarrow \mathrm{Sp}(\mathcal{C})^\otimes.
\end{equation*}
Note the switch of direction here. The goal of this section is to transfer our discussion to this setting, in particular describing and characterizing $n$-truncations of stable $\infty$-operads, of which the above map is an example. We will state our results but defer all proofs in this section to Appendix \ref{sec:apptruncations}.

\begin{definition}
Write $\mathbf{Op}^\mathrm{St}$ for the $\infty$-category of nonunital stable $\infty$-operads and maps $\mathcal{N}^\otimes \rightarrow \mathcal{O}^\otimes$ whose underlying functor $\mathcal{N} \rightarrow \mathcal{O}$ preserves limits and filtered colimits.
\end{definition}

Using Propositions 6.2.4.14 and 6.2.4.15 of \cite{higheralgebra} one can show that the stabilization procedure which assigns to a pointed compactly generated $\infty$-category $\mathcal{C}$ the $\infty$-operad $\mathrm{Sp}(\mathcal{C})^\otimes$ can be made into a functor
\begin{equation*}
\mathrm{Sp}(-)^\otimes: \bigl(\mathbf{Cat}^\omega_*\bigr)^{\mathrm{op}} \longrightarrow \mathbf{Op}^{\mathrm{St}}.
\end{equation*}

Recall that for a nonunital $\infty$-operad we introduced the simplicial set $\mathcal{O}^\otimes_{\leq n}$ in Chapter \ref{sec:mainresults}. To state the necessary results, we will need the $\infty$-category of \emph{$n$-truncated} $\infty$-operads. Informally speaking a nonunital $n$-truncated $\infty$-operad is a categorical fibration over $\N\mathrm{Surj}_{\leq n}$ satisfying the evident analogues of the axioms required of an $\infty$-operad. One way to make this precise is as follows: consider the subcategory of $\mathbf{sSets}/\N\mathrm{Surj}_{\leq n}$ spanned by objects of the form $\mathcal{O}^\otimes_{\leq n}$, with $\mathcal{O}^\otimes$ ranging through all nonunital $\infty$-operads, and arrows those maps over $\N\mathrm{Surj}_{\leq n}$ preserving inert morphisms. This subcategory is a simplicial category in an evident way and one can define $\mathbf{Op}_{\leq n}$ to be its homotopy-coherent nerve.

\begin{remark}
\label{rmk:adhoctruncation}
A less ad hoc way of defining $\mathbf{Op}_{\leq n}$ is by constructing a model structure on the category of marked simplicial sets over $(\N\mathrm{Surj}_{\leq n})^\natural$ in analogy with Lurie's construction of the model category of $\infty$-preoperads in Section 2.1.4 of \cite{higheralgebra}. The details are straightforward, simply copying Lurie's work and replacing $\NFin$ by $\N\mathrm{Surj}_{\leq n}$ throughout. An alternative way of describing this homotopy theory using dendroidal sets is given in the appendix.
\end{remark}

We call an $n$-truncated $\infty$-operad $\mathcal{O}_{\leq n}^\otimes$ \emph{stable} if $\mathcal{O}^\otimes$ itself is stable.

\begin{definition}
\label{def:truncatedOpSt}
Write $\mathbf{Op}_{\leq n}^{\mathrm{St}}$ for the $\infty$-category of nonunital stable $n$-truncated $\infty$-operads and maps $\mathcal{N}_{\leq n}^\otimes \rightarrow \mathcal{O}_{\leq n}^\otimes$ whose underlying functor $\mathcal{N} \rightarrow \mathcal{O}$ preserves limits and filtered colimits.
\end{definition}

Pulling back from $\N\mathrm{Surj}$ to $\N\mathrm{Surj}_{\leq n}$ defines a functor
\begin{equation*}
(-)_{\leq n}: \mathbf{Op}^{\mathrm{St}} \longrightarrow \mathbf{Op}_{\leq n}^{\mathrm{St}}.
\end{equation*}

The following result is a shadow of Theorem \ref{thm:TheoremA} in the world of $\infty$-operads, with $(-)_{\leq n}$ being the analogue of the $n$-excisive approximation $\mathcal{P}_n$.

\begin{theorem}
\label{thm:ntruncation}
The functor $(-)_{\leq n}$ admits a left adjoint $i_n:  \mathbf{Op}_{\leq n}^{\mathrm{St}} \rightarrow \mathbf{Op}^{\mathrm{St}}$. Furthermore, for an $n$-truncated $\infty$-operad $\mathcal{N}^\otimes$ the unit $\mathcal{N}^\otimes \rightarrow (i_n\mathcal{N}^\otimes)_{\leq n}$ is an equivalence.
\end{theorem}

This theorem can be rephrased as saying that the $n$-truncated nonunital stable $\infty$-operads form a coreflective subcategory of $\mathbf{Op}^{\mathrm{St}}$. We write $\tau_n$ for the composite $i_n \circ (-)_{\leq n}$. The motivating example of $n$-truncations (which makes the comment above the previous theorem precise) is the following:

\begin{proposition}
\label{prop:SpPnCntruncated}
Let $\mathcal{C}$ be a pointed compactly generated $\infty$-category. Then the map of $\infty$-operads $\tau_n\mathrm{Sp}(\mathcal{P}_n\mathcal{C})^\otimes \rightarrow \mathrm{Sp}(\mathcal{P}_n\mathcal{C})^\otimes$ is an equivalence. Furthermore, the map $\tau_n\mathrm{Sp}(\mathcal{P}_n\mathcal{C})^\otimes \rightarrow \tau_n\mathrm{Sp}(\mathcal{C})^\otimes$ is also an equivalence.
\end{proposition}

Another way of phrasing this proposition is to say that $\mathrm{Sp}(\mathcal{P}_n\mathcal{C})^\otimes$ is the coreflection of $\mathrm{Sp}(\mathcal{C})^\otimes$ into the $\infty$-category of $n$-truncated stable $\infty$-operads. 

Let $\mathcal{O}^\otimes$ be a nonunital stable $\infty$-operad and write $\otimes^k$ for the $k$-fold tensor product on the $\infty$-category $\mathcal{O}$ it defines. Similarly, write $\odot^k$ for the $k$-fold tensor product on $\mathcal{O}$ determined by the $\infty$-operad $\tau_n\mathcal{O}^\otimes$. It is immediate from Theorem \ref{thm:ntruncation} that the functors $\odot^k$ are canonically equivalent to $\otimes^k$ for $k \leq n$. It will be useful to have an explicit description of $\odot^k$ in terms of $\otimes^k$ for general $k$. We first need to introduce some notation.

Let $\mathbf{Equiv}(k)$ denote the set of equivalence relations on the set $\{1, \ldots, k\}$. It is a partially ordered set under refinement of equivalence relations, with minimal element the discrete equivalence relation and maximal element the trivial equivalence relation with one equivalence class. Consider the functor
\begin{equation*}
q: \mathbf{Equiv}(k) \longrightarrow \mathcal{F}\mathrm{in}_*: E \longmapsto \bigl(\{1, \ldots, k\}/E\bigr)_*,
\end{equation*}
where the subscript $*$ denotes the addition of a disjoint basepoint. Write $\mathbf{\Delta}/\N\mathbf{Equiv}(k)$ for the category of simplices of the nerve of $\mathbf{Equiv}(k)$ and consider the full subcategory $\mathbf{Part}(k) \subseteq \mathbf{\Delta}/\N\mathbf{Equiv}(k)$ consisting of simplices of dimension at least 1 with initial (resp. final) vertex the discrete (resp. trivial) equivalence relation. In other words, $\mathbf{Part}(k)$ is the partially ordered set whose elements are the non-empty linearly ordered subsets of $\mathbf{Equiv}(k)$ whose minimal (resp. maximal) element maps to $\langle k \rangle$ (resp. $\langle 1 \rangle$) under the functor $q$. We will need a further subset $\mathbf{Part}_n(k) \subseteq \mathbf{Part}(k)$. It consists of those elements $E_0 < \cdots < E_j$ of $\mathbf{Part}(k)$ such that for each $1 \leq i \leq j$ the fibers of the map $q(E_{i-1}) \rightarrow q(E_i)$ have cardinality at most $n$. Observe that if $n \geq k$ then the inclusion $\mathbf{Part}_n(k) \subseteq \mathbf{Part}(k)$ is an equality, but this is not the case if $n < k$. Also, $q$ induces a functor
\begin{equation*}
Q: \mathbf{Part}(k) \longrightarrow (\mathbf{\Delta}/\N\mathcal{F}\mathrm{in}_*)^{\langle k \rangle, \langle 1 \rangle},
\end{equation*}
where $(\mathbf{\Delta}/\N\mathcal{F}\mathrm{in}_*)^{\langle k \rangle, \langle 1 \rangle}$ denotes the full subcategory of $\mathbf{\Delta}/\N\mathcal{F}\mathrm{in}_*$ on simplices starting at $\langle k \rangle$ and ending at $\langle 1 \rangle$.

\begin{example}
\label{ex:nest}
The poset $\mathbf{Part}(3)$ has 4 elements. Its image under $Q$ can be schematically drawn as follows:
\[
\xymatrix{
& (\langle 3 \rangle \rightarrow \langle 1 \rangle) \ar[dr]\ar[d]\ar[dl] & \\
(\langle 3 \rangle \rightarrow \langle 2 \rangle \rightarrow \langle 1 \rangle) & (\langle 3 \rangle \rightarrow \langle 2 \rangle \rightarrow \langle 1 \rangle) & (\langle 3 \rangle \rightarrow \langle 2 \rangle \rightarrow \langle 1 \rangle). 
}
\]
The three elements on the bottom row correspond to the three different equivalence relations on the set $\{1,2,3\}$ that are neither trivial nor discrete. The subset $\mathbf{Part}_2(3) \subset \mathbf{Part}(3)$ consists of precisely those three elements.
\end{example} 

We will now define a functor 
\begin{equation*}
\psi^k: \N\mathbf{Part}(k) \longrightarrow \mathrm{Fun}(\mathcal{O}^\otimes_{\langle k \rangle}, \mathcal{O}^\otimes_{\langle 1 \rangle}).
\end{equation*}
By assumption, the map $p: \mathcal{O}^\otimes \longrightarrow \N\mathcal{F}\mathrm{in}_*$ is a locally coCartesian fibration. In particular, any edge $f: x \rightarrow y$ in $\N\mathcal{F}\mathrm{in}_*$ determines a functor $f_!: \mathcal{O}^\otimes_x \rightarrow \mathcal{O}^\otimes_y$, which is canonical up to homotopy. More generally, any simplex $\sigma: \Delta^j \longrightarrow \N\mathcal{F}\mathrm{in}_*$, given by a sequence of edges
\[
\xymatrix{
\sigma(0) \ar[r]^{f(1)} & \sigma(1) \ar[r]^{f(2)} & \cdots \ar[r]^{f(j)} & \sigma(j),
}
\] 
determines a functor 
\begin{equation*}
f(j)_! \circ f(j-1)_! \circ \cdots \circ f(1)_!: \mathcal{O}^\otimes_{\sigma(0)} \longrightarrow \mathcal{O}^\otimes_{\sigma(j)}.
\end{equation*}

\begin{remark}
The composite above is generally \emph{not} equivalent to the functor $\mathcal{O}^\otimes_{\sigma(0)} \longrightarrow \mathcal{O}^\otimes_{\sigma(j)}$ determined by the edge $\sigma(0) \rightarrow \sigma(j)$. However, if $p$ is a coCartesian fibration (rather than just a locally coCartesian fibration), these are indeed equivalent.
\end{remark}

The construction above can be made natural in the simplex $\sigma$, so as to yield a functor
\begin{equation*}
\theta: \N(\mathbf{\Delta}/\N\mathcal{F}\mathrm{in}_*)^{\langle k \rangle, \langle 1 \rangle} \longrightarrow \mathrm{Fun}(\mathcal{O}^\otimes_{\langle k \rangle}, \mathcal{O}^\otimes_{\langle 1 \rangle}).
\end{equation*}
This is done in Definition 7.2.3.8 of \cite{higheralgebra} and called the \emph{spray} associated to $p$. Given $\theta$, we can now define $\psi^k$ to be the composite $\theta \circ Q$. Also, we write $\psi_n^k$ for the restriction of $\psi^k$ to $\N\mathbf{Part}_n(k)$. 

\begin{example}
The diagram $\psi^3$ is obtained by applying $\theta$ to the diagram of Example \ref{ex:nest}. The top vertex of the resulting diagram is the functor assigning to a tuple $(X,Y,Z)$ the tensor product $X \otimes^3 Y \otimes^3 Z$. The bottom three vertices correspond to the expression $(X \otimes^2 Y) \otimes^2 Z$ and permutations of it. The arrows in the diagram correspond to decomposition maps of the form $X \otimes^3 Y \otimes^3 Z \rightarrow (X \otimes^2 Y) \otimes^2 Z$ as discussed in Remark \ref{rmk:dictionary}. The diagram $\psi^4$ is already rather large. A small part of it can be pictured as follows, when evaluated on a tuple $(X,Y,Z,W)$:
\[
\xymatrix{
X \otimes^4 Y \otimes^4 Z \otimes^4 W \ar[r] \ar[d] & (X \otimes^3 Y \otimes^3 Z) \otimes^2 W \ar[d] \\
X \otimes^3 (Y \otimes^2 Z) \otimes^3 W \ar[r] & (X \otimes^2 (Y \otimes^2 Z)) \otimes^2 W.
}
\]
\end{example}

Recall that we write $\odot^k$ for the $k$-fold tensor product determined by the $\infty$-operad $\tau_n\mathcal{O}^\otimes$. The following result gives a concrete description of this functor in terms of the tensor products determined by $\mathcal{O}^\otimes$ itself:

\begin{proposition}
\label{prop:truncatedtensor}
There is a canonical equivalence 
\begin{equation*}
\odot^k \longrightarrow \varprojlim_{\N\mathbf{Part}_n(k)} \psi^k_n.
\end{equation*}
\end{proposition}

Note in particular that when $k \leq n$ the diagram $\psi^k_n$ equals $\psi^k$ itself, which has an initial vertex $\otimes^k$, so that indeed $\odot^k \simeq \otimes^k$ in this case. Finally, Proposition \ref{prop:mapstruncations} below will be a useful tool in inductively producing maps between various $\infty$-operads. To state it, suppose $\mathcal{O}^\otimes$ and $\mathcal{N}^\otimes$  are nonunital stable $\infty$-operads. For simplicity, we will assume that the underlying $\infty$-category $\mathcal{N}$ of $\mathcal{N}^\otimes$ is equal to $\mathcal{O}$, although this is only for notational convenience. Write $\otimes^k_{\mathcal{O}}$ and $\otimes^k_{\mathcal{N}}$ for the $k$-fold tensor products on $\mathcal{O}$ determined by these $\infty$-operads. Write $\mathrm{Map}_{\mathcal{O}}(\mathcal{O}^\otimes, \mathcal{N}^\otimes)$ for the space of maps of $\infty$-operads which restrict to the identity on the level of underlying $\infty$-categories (i.e. after taking fibers over $\langle 1 \rangle$). Any such map $\mathcal{O}^\otimes \rightarrow \mathcal{N}^\otimes$ determines natural transformations $\otimes_{\mathcal{N}}^k \rightarrow \otimes_{\mathcal{O}}^k$. For notational convenience let us write $D^{\otimes_{\mathcal{O}}}_k$ and $D^{\otimes_{\mathcal{N}}}_k$ for the functors described by 
\begin{equation*}
X \mapsto (X \otimes_\mathcal{O}^k \cdots \otimes_{\mathcal{O}}^k X)_{h\Sigma_k} \quad\quad \text{and} \quad\quad X \mapsto (X \otimes_\mathcal{N}^k \cdots \otimes_{\mathcal{N}}^k X)_{h\Sigma_k}
\end{equation*}
respectively. There is then a canonical map
\begin{equation*}
\mathrm{Map}_{\mathcal{O}}(\mathcal{O}^\otimes, \mathcal{N}^\otimes) \longrightarrow \mathrm{Nat}(D^{\otimes_{\mathcal{N}}}_k, D^{\otimes_{\mathcal{O}}}_k),
\end{equation*} 
where $\mathrm{Nat}$ on the right refers to the space of natural transformation between functors from $\mathcal{O}$ to itself. Similarly, writing $\odot^k_{\mathcal{O}}$ for the tensor products on $\mathcal{O}$ determined by $\tau_{n-1}\mathcal{O}^\otimes$, there is the analogous map
\begin{equation*}
\mathrm{Map}_{\mathcal{O}}(\tau_{n-1}\mathcal{O}^\otimes, \mathcal{N}^\otimes) \longrightarrow  \mathrm{Nat}(D^{\otimes_{\mathcal{N}}}_k, D^{\odot_{\mathcal{O}}}_k).
\end{equation*}

\begin{remark}
Exploiting the correspondence between $k$-homogeneous functors and symmetric multilinear functors of $k$ variables, one sees that the space $\mathrm{Nat}(D^{\otimes_{\mathcal{N}}}_k, D^{\otimes_{\mathcal{O}}}_k)$ is equivalent to the space of $\Sigma_k$-equivariant natural transformations between $\otimes_{\mathcal{N}}^k$ and $\otimes_{\mathcal{O}}^k$. 
\end{remark}

\begin{proposition}
\label{prop:mapstruncations}
Suppose $\mathcal{O}^\otimes$ and $\mathcal{N}^\otimes$ are as above. Then the diagram
\[
\xymatrix{
\mathrm{Map}_{\mathcal{O}}(\tau_n\mathcal{O}^\otimes, \mathcal{N}^\otimes)  \ar[r]\ar[d] & \mathrm{Nat}(D^{\otimes_{\mathcal{N}}}_n, D^{\otimes_{\mathcal{O}}}_n) \ar[d] \\
\mathrm{Map}_{\mathcal{O}}(\tau_{n-1}\mathcal{O}^\otimes, \mathcal{N}^\otimes) \ar[r] & \mathrm{Nat}(D^{\otimes_{\mathcal{N}}}_n, D^{\odot_{\mathcal{O}}}_n)
}
\]
is a pullback square in the $\infty$-category of spaces.
\end{proposition}

In words, to extend a map $\tau_{n-1}\mathcal{O}^\otimes \rightarrow \mathcal{N}^\otimes$ to a map $\tau_n\mathcal{O}^\otimes \rightarrow \mathcal{N}^\otimes$, one needs to (equivariantly) lift the natural transformation $\otimes^n_{\mathcal{N}} \rightarrow \odot^n_{\mathcal{O}}$ to a natural transformation $\otimes^n_{\mathcal{N}}  \rightarrow \otimes^n_{\mathcal{O}}$.


\section{Coalgebras in a corepresentable $\infty$-operad}
\label{subsec:coalgebras}

Let $\mathcal{O}^\otimes$ be a corepresentable $\infty$-operad, so that its underlying $\infty$-category $\mathcal{O}$ comes equipped with tensor product functors $\otimes^k$. In this section we discuss what it means to equip an object $X \in \mathcal{O}$ with the structure of a \emph{coalgebra} in $\mathcal{O}^\otimes$, meaning a sequence of `diagonal maps'
\begin{equation*}
\delta_k: X \longrightarrow X \otimes^k X \otimes^k \cdots \otimes^k X
\end{equation*}
which are compatible with the operad structure of $\mathcal{O}^\otimes$ in a suitable way.

\begin{remark}
Under the dictionary of Remark \ref{rmk:dictionary}, coalgebras in $\mathcal{O}^\otimes$ match with coalgebras over the corresponding cooperad.
\end{remark}
 
Roughly speaking, supplying a coalgebra structure on $X$ is equivalent to upgrading the slice category $\mathcal{O}_{X/}$ to a corepresentable $\infty$-operad in a way that is compatible with the operad structure of $\mathcal{O}^\otimes$. More precisely, let $p^\otimes: \mathcal{X}^\otimes \rightarrow \mathcal{O}^\otimes$ be a fibration of corepresentable $\infty$-operads and suppose that the underlying functor $p: \mathcal{X} \rightarrow \mathcal{O}$ is equivalent (as a fibration over $\mathcal{O}$) to a projection of the form $\mathcal{O}_{X/} \rightarrow \mathcal{O}$ for some object $X \in \mathcal{O}$. Fix such an equivalence (which need not be unique, but our definition will not depend on the choice). Then, for any tuple of maps $f_1: X \rightarrow Y_1, \ldots, f_n: X \rightarrow Y_n$, the map $p^\otimes$ induces a map
\begin{equation*}
\varphi_{f_1, \ldots, f_n}: Y_1 \otimes \cdots \otimes Y_n = p(f_1) \otimes \cdots \otimes p(f_n) \longrightarrow p(f_1 \otimes \cdots \otimes f_n).
\end{equation*} 

\begin{definition}
\label{def:coalgebra}
We say \emph{$p^\otimes$ exhibits $X$ as a coalgebra in $\mathcal{O}^\otimes$} if the maps $\varphi_{f_1, \ldots, f_n}$ are equivalences, for any $n \geq 0$ and choice of maps $f_1, \ldots, f_n$. Write $\mathrm{coAlg}(\mathcal{O}^\otimes)$ for the $\infty$-category of coalgebras in $\mathcal{O}^\otimes$, which is the opposite of the full subcategory of the $\infty$-category of $\infty$-operads over $\mathcal{O}^\otimes$ spanned by the coalgebras.
\end{definition}

One can think of this definition as follows. Consider the identity map $X = X$ as an object $\mathrm{id}_X$ of $\mathcal{O}_{X/}$. The corepresentable $\infty$-operad $\mathcal{X}^\otimes$ then determines for every $n$ a tensor product $\mathrm{id}_X \otimes^n \cdots \otimes^n \mathrm{id}_X$, which is another object of $\mathcal{O}_{X/}$. By the requirement that $\varphi_{\mathrm{id}_X, \ldots, \mathrm{id}_X}$ is an equivalence, this object determines a map
\begin{equation*}
\delta_n: X \longrightarrow X \otimes^n \cdots \otimes^n X.
\end{equation*}
Furthermore, these maps (for various $n$) satisfy certain coherence relations. For example, the stable $\infty$-operad $\mathcal{O}^\otimes$ gives a natural transformation
\begin{equation*}
X \otimes^3 X \otimes^3 X \longrightarrow (X \otimes^2 X) \otimes^2 X
\end{equation*}
and a coalgebra as above provides a 2-simplex
\[
\xymatrix@C=35pt{
& X \otimes^3 X \otimes^3 X \ar[d] \\
X \ar[ur]^{\delta_3} \ar[r]_-{(\delta_2 \otimes 1) \circ \delta_2} & (X \otimes^2 X) \otimes^2 X
}
\]
in $\mathcal{O}$. Furthermore, if $\mathcal{X}^\otimes \rightarrow \mathcal{O}^\otimes$ and $\mathcal{Y}^\otimes \rightarrow \mathcal{O}^\otimes$ exhibit objects $X$ and $Y$ respectively as coalgebras in $\mathcal{O}^\otimes$, then a triangle of maps of $\infty$-operads
\[
\xymatrix{
\mathcal{Y}^\otimes \ar[rr]^F\ar[dr] && \mathcal{X}^\otimes \ar[dl] \\
& \mathcal{O}^\otimes &
}
\]
in particular (after passing to fibers over $\langle 1 \rangle$) yields a functor $\mathcal{O}_{Y/} \rightarrow \mathcal{O}_{X/}$ over $\mathcal{O}$, which corresponds to a morphism $f: X \rightarrow Y$. The fact that $F$ is a map of $\infty$-operads over $\mathcal{O}^\otimes$ makes $f$ compatible with the coalgebra structures on $X$ and $Y$.

\begin{remark}
We noted above that the equivalence between $\mathcal{X}$ and $\mathcal{O}_{X/}$ is not necessarily unique; however, the object $X$ itself is unique up to equivalence. We can define a forgetful functor $U: \mathrm{coAlg}(\mathcal{O}^\otimes) \rightarrow \mathcal{O}$, taking the underlying object of a coalgebra, as follows. Consider the functor 
\begin{equation*}
j: \mathcal{O} \longrightarrow (\mathbf{Cat}/\mathcal{O})^{\mathrm{op}}: X \longmapsto (\mathcal{O}_{X/} \rightarrow \mathcal{O}),
\end{equation*}
which is an embedding by the Yoneda lemma, and write $\mathrm{Repr}(\mathcal{O})$ for its essential image. (To be precise, $j$ is the composition of the Yoneda embedding $\mathcal{O} \rightarrow \mathrm{Fun}(\mathcal{O}^{\mathrm{op}},\mathcal{S})$ with the unstraightening construction of Section 3.2 of \cite{htt}.) Then there is an essentially unique inverse $k: \mathrm{Repr}(\mathcal{O}) \rightarrow \mathcal{O}$ to $j$. By definition the underlying functor $p: \mathcal{X} \rightarrow \mathcal{O}$ of a coalgebra $p^\otimes$ is in $\mathrm{Repr}(\mathcal{O})$ and we set $U(p^\otimes) = k(p)$. 
\end{remark}

We need to investigate the behaviour of coalgebras under maps of $\infty$-operads. Suppose $g^\otimes: \mathcal{N}^\otimes \rightarrow \mathcal{O}^\otimes$ is a map of $\infty$-operads such that the induced functor of underlying $\infty$-categories $g: \mathcal{N} \rightarrow \mathcal{O}$ admits a left adjoint $f$. We claim that for any coalgebra $X$ in $\mathcal{O}^\otimes$, the object $f(X)$ can be given the structure of a coalgebra in $\mathcal{N}^\otimes$ in a canonical way. Let us first explain this heuristically; a rigorous construction is \ref{constr:pullbackcoalg} below.

For any $Y \in \mathcal{N}$, the map $g$ yields a natural map $g(Y) \otimes_{\mathcal{O}} g(Y) \rightarrow g(Y \otimes_{\mathcal{N}} Y)$. Using this, and the unit and counit of the adjunction between $f$ and $g$, we get for any $X \in \mathcal{O}$ a sequence of natural maps
\begin{eqnarray*}
f(X \otimes_{\mathcal{O}} X) & \longrightarrow & f(gf(X) \otimes_{\mathcal{O}} gf(X)) \\
& \longrightarrow & fg(f(X) \otimes_{\mathcal{N}} f(X)) \\
& \longrightarrow & f(X) \otimes_{\mathcal{N}} f(X). 
\end{eqnarray*}
If $X$ is a coalgebra, so that it comes with a natural map $X \rightarrow X \otimes_{\mathcal{O}} X$, we may form the composite of the maps
\begin{equation*}
f(X) \rightarrow f(X \otimes_{\mathcal{O}} X) \rightarrow f(X) \otimes_{\mathcal{N}} f(X)
\end{equation*}
to find a diagonal for $f(X)$. The higher diagonals can be treated analogously. To be more precise, we can construct a functor $f_*: \mathrm{coAlg}(\mathcal{O}^\otimes) \rightarrow \mathrm{coAlg}(\mathcal{N}^\otimes)$ using the following result:

\begin{lemma}
Let $p^\otimes: \mathcal{X}^\otimes \rightarrow \mathcal{O}^\otimes$ be a map of $\infty$-operads exhibiting an object $X$ as a coalgebra in $\mathcal{O}$ and let $g^\otimes: \mathcal{N}^\otimes \rightarrow \mathcal{O}^\otimes$ be as above. Form a pullback square
\[
\xymatrix{
f_*\mathcal{X}^\otimes \ar[r]\ar[d]_{q^\otimes} & \mathcal{X}^\otimes \ar[d]^{p^\otimes} \\
\mathcal{N}^\otimes \ar[r]_{g^\otimes} & \mathcal{O}^\otimes.
}
\]
Then $q$ exhibits the object $f(X)$ as a coalgebra in $\mathcal{N}^\otimes$.
\end{lemma}
\begin{proof}
Observe that the underlying $\infty$-category of $\mathcal{X}$ fits into a pullback square
\[
\xymatrix{
f_*\mathcal{X} \ar[r]\ar[d] & \mathcal{O}_{X/} \ar[d] \\
\mathcal{N} \ar[r] & \mathcal{O}
}
\]
and is therefore equivalent to $\mathcal{N}_{f(X)/}$. Furthermore, for a collection of maps $\varphi_1: f(X) \rightarrow Y_1, \ldots, \varphi_n: f(X) \rightarrow Y_n$ it is straightforward to check that the functor
\begin{equation*}
f_*\mathcal{X}(\varphi_1, \ldots, \varphi_n; -): \mathcal{N}_{f(X)/} \longrightarrow \mathcal{S}
\end{equation*}
is corepresented by the map $f(X) \rightarrow Y_1 \otimes_{\mathcal{N}} \cdots \otimes_{\mathcal{N}} Y_n$ adjoint to the composition of the maps
\begin{equation*}
X \longrightarrow g(Y_1) \otimes_{\mathcal{O}} \cdots \otimes_{\mathcal{O}} g(Y_n) \longrightarrow g(Y_1 \otimes_{\mathcal{N}} \cdots \otimes_{\mathcal{N}} Y_n).
\end{equation*}
The first of these maps is determined by the tensor product induced by $\mathcal{X}^\otimes$ on $\mathcal{O}_{X/}$, the second is determined by the map of $\infty$-operads $g^\otimes$. Therefore $f_*\mathcal{X}$ is corepresentable. Also, this description of $f_*\mathcal{X}(\varphi_1, \ldots, \varphi_n; -)$ makes it clear that the map $q^\otimes$ satisfies the requirements for exhibiting $f(X)$ as a coalgebra in $\mathcal{N}^\otimes$.
\end{proof}

\begin{construction}
\label{constr:pullbackcoalg}
The pullback square of the previous lemma defines a functor $f_*: \mathrm{coAlg}(\mathcal{O}^\otimes) \rightarrow \mathrm{coAlg}(\mathcal{N}^\otimes)$.
\end{construction}

Our source of coalgebras will be a combination of Construction \ref{constr:pullbackcoalg} and the construction of `diagonal' coalgebras in Cartesian symmetric monoidal $\infty$-categories, which we provide now. Consider an $\infty$-category $\mathcal{C}$ which admits finite products. In Construction 2.4.1.4 of \cite{higheralgebra} Lurie defines a symmetric monoidal $\infty$-category $\mathcal{C}^\times$ whose monoidal structure is given by the Cartesian product. We write $\mathcal{C}_{\mathrm{nu}}^\times$ for its nonunital variant, which is the pullback
\begin{equation*}
\N\mathrm{Surj} \times_{\NFin} \mathcal{C}^\times.
\end{equation*}
For $X$ an object of $\mathcal{C}$ the category $\mathcal{C}_{X/}$ admits finite products as well: for maps $X \rightarrow Y$ and $X \rightarrow Z$ their product in this $\infty$-category is the composition $X \rightarrow X \times X \rightarrow Y \times Z$, where the first map is the diagonal. Therefore we may construct another $\infty$-operad $\mathcal{C}_{X/}^{\times}$ which admits an evident forgetful functor to $\mathcal{C}^\times$. The proof of the following is completely straightforward and left to the reader:

\begin{lemma}
\label{lem:cartcoalg}
The functor $(\mathcal{C}_{X/}^{\times})_{\mathrm{nu}} \rightarrow \mathcal{C}^\times_{\mathrm{nu}}$ exhibits $X$ as a coalgebra in $\mathcal{C}^\times_{\mathrm{nu}}$.
\end{lemma}

\begin{definition}
\label{def:diag}
Let $\mathcal{C}$ be an $\infty$-category with finite products. Then by the previous lemma we can define a functor which may be described as follows:
\begin{equation*}
\mathrm{diag}: \mathcal{C} \longrightarrow \mathrm{coAlg}(\mathcal{C}^{\times}_{\mathrm{nu}}): X \longmapsto \bigl((\mathcal{C}_{X/}^{\times})_{\mathrm{nu}} \rightarrow \mathcal{C}^\times_{\mathrm{nu}}\bigr).
\end{equation*}
More precisely, one can construct a functor $\mathcal{C} \rightarrow (\mathbf{Cat}_\infty)^{\mathrm{op}}$ which assigns to $X$ an $\infty$-category equivalent to the slice category $\mathcal{C}_{X/}$ by straightening the Cartesian fibration $\mathcal{C}^{\Delta^1} \xrightarrow{\mathrm{ev}_0} \mathcal{C}$. Then one applies Construction 2.4.1.4 of \cite{higheralgebra} pointwise. The notation $\mathrm{diag}$ refers to the fact that this construction yields the usual `diagonal' coalgebra structure on an object $X$ of an $\infty$-category with finite products.
\end{definition}

\begin{construction}
\label{constr:coalgebras}
Suppose $\mathcal{C}$ is a pointed compactly generated $\infty$-category. Then $\Omega_n^\infty: \mathcal{P}_n\mathcal{C} \rightarrow \mathcal{C}$ induces a commutative diagram of $\infty$-operads
\[
\xymatrix{
\mathcal{C}_{\mathrm{nu}}^\times & (\mathcal{P}_n\mathcal{C})_{\mathrm{nu}}^\times \ar[l] \\
\mathrm{Sp}(\mathcal{C})^\otimes \ar[u] & \mathrm{Sp}(\mathcal{P}_n\mathcal{C})^\otimes . \ar[u]\ar[l]
}
\]
Applying Construction \ref{constr:pullbackcoalg} then yields a commutative diagram
\[
\xymatrix{
\mathrm{coAlg}(\mathcal{C}_{\mathrm{nu}}^\times) \ar[d]\ar[r] & \mathrm{coAlg}((\mathcal{P}_n\mathcal{C})_{\mathrm{nu}}^\times) \ar[d] \\
\mathrm{coAlg}(\mathrm{Sp}(\mathcal{C})^\otimes) \ar[r] & \mathrm{coAlg}(\mathrm{Sp}(\mathcal{P}_n\mathcal{C})^\otimes).
}
\]
On underlying objects the top horizontal functor may be identified with $\Sigma_n^\infty$, the bottom horizontal functor simply with $\mathrm{id}_{\mathrm{Sp}(\mathcal{C})}$. Finally, Definition \ref{def:diag} provides a functor $\mathrm{diag}: \mathcal{C} \rightarrow \mathrm{coAlg}(\mathcal{C}^\times_{\mathrm{nu}})$ and hence we obtain a functor from $\mathcal{C}$ to each of the four categories of coalgebras in the previous square. In what follows we will make frequent use of the sequence of functors
\[
\xymatrix{
\mathcal{C} \ar[r] & \mathrm{coAlg}(\mathrm{Sp}(\mathcal{C})^\otimes) \ar[r] & \mathrm{coAlg}(\mathrm{Sp}(\mathcal{P}_n\mathcal{C})^\otimes)
}
\]
thus obtained. Note that this sequence is also natural in $\mathcal{C}$ with respect to functors preserving colimits and compact objects. On underlying objects the composite of these functors may be identified with $\Sigma^\infty_{\mathcal{C}}$, with the understanding that the underlying $\infty$-category of $\mathrm{Sp}(\mathcal{P}_n\mathcal{C})^\otimes$ is identified with $\mathrm{Sp}(\mathcal{C})$.
\end{construction}

\section{Coalgebras in an $n$-truncated stable $\infty$-operad}
\label{subsec:truncatedcoalgebras}

Fix a nonunital stable $\infty$-operad $\mathcal{O}^\otimes$. We will need a collection of technical results on the $\infty$-categories of coalgebras in the truncated $\infty$-operads $\tau_n\mathcal{O}^\otimes$. The $\infty$-category $\mathrm{coAlg}(\mathcal{O}^\otimes)$ need not be compactly generated, so that it does not necessarily admit a good theory of calculus. To circumvent this defect we consider the following:

\begin{definition}
\label{def:indcoalgebras}
Write $\mathrm{coAlg}(\mathcal{O}^\otimes)^c$ for the full subcategory of $\mathrm{coAlg}(\mathcal{O}^\otimes)$ on coalgebras whose underlying object of $\mathcal{O}$ is compact. Then define the $\infty$-category of \emph{ind-coalgebras in $\mathcal{O}^\otimes$} to be
\begin{equation*}
\mathrm{coAlg}^{\mathrm{ind}}(\mathcal{O}^\otimes) := \mathrm{Ind}(\mathrm{coAlg}\bigl(\mathcal{O}^\otimes)^c\bigr).
\end{equation*}
\end{definition}

\begin{remark}
In Construction \ref{constr:coalgebras} we defined functors $\mathcal{C} \rightarrow \mathrm{coAlg}(\mathrm{Sp}(\mathcal{C})^\otimes)$. This construction allows a variant for ind-coalgebras: indeed, applying the previous construction to compact objects of $\mathcal{C}$ and formally extending by filtered colimits yields a functor
\begin{equation*}
\mathcal{C} \longrightarrow \mathrm{coAlg}^{\mathrm{ind}}(\mathrm{Sp}(\mathcal{C})^\otimes).
\end{equation*}
\end{remark}

The $\infty$-category of ind-coalgebras in $\tau_n\mathcal{O}^\otimes$ can be understood more explicitly by inductively constructing it out of the $\infty$-category of coalgebras in $\tau_{n-1}\mathcal{O}^\otimes$. For convenience of stating the necessary results we introduce some notation. Suppose $F: \mathcal{O} \rightarrow \mathcal{O}$ is a functor. Then we write
\begin{equation*}
\bigl\{X \rightarrow F(X)\bigr\}^c_{\mathcal{O}}
\end{equation*}
for the $\infty$-category of compact objects $X \in \mathcal{O}$ equipped with a map $X \rightarrow F(X)$. More precisely, this $\infty$-category can be defined as the pullback of the span
\[
\xymatrix@C=30pt{
\mathcal{O}^c \ar[r]^-{(\mathrm{id},F)} & \mathcal{O} \times \mathcal{O} & \mathcal{O}^{\Delta^1} \ar[l]_-{(\mathrm{ev}_0,\mathrm{ev}_1)}.
}
\]
As before, write $\odot^k$ (resp. $\otimes^k$) for the tensor products determined by $\tau_{n-1}\mathcal{O}^\otimes$ (resp. $\tau_n\mathcal{O}^\otimes$). For any $k \geq 1$ there is an evident functor
\begin{equation*}
\mathrm{coAlg}^c(\tau_{n-1}\mathcal{O}^\otimes) \longrightarrow \bigl\{X \rightarrow (X \odot^{k} \cdots \odot^{k} X)^{h\Sigma_{k}} \bigr\}^c_{\mathcal{O}},
\end{equation*}
and similarly for $\mathrm{coAlg}^c(\tau_n\mathcal{O}^\otimes)$, taking the $k$-fold diagonal map of a coalgebra structure. Given the description of $\odot^k$ of Proposition \ref{prop:truncatedtensor} the following (which we prove in Section \ref{subsec:appendixcoalgebras}) should not be surprising:

\begin{lemma}
\label{lem:algtaun+1}
The following is a pullback square of compactly generated $\infty$-categories:
\[
\xymatrix{
\mathrm{coAlg}^{\mathrm{ind}}(\tau_n\mathcal{O}^\otimes) \ar[r]\ar[d] & \mathrm{Ind}\bigl\{X \rightarrow (X \otimes^n \cdots \otimes^n X)^{h\Sigma_n} \bigr\}^c_{\mathcal{O}} \ar[d] \\
\mathrm{coAlg}^{\mathrm{ind}}(\tau_{n-1}\mathcal{O}^\otimes) \ar[r] & \mathrm{Ind}\bigl\{X \rightarrow (X \odot^n \cdots \odot^n X)^{h\Sigma_n} \bigr\}^c_{\mathcal{O}}.
}
\]
\end{lemma}

\begin{remark}
This lemma expresses the idea that to lift a compact coalgebra $X$ in $\tau_{n-1}\mathcal{O}^\otimes$ to a coalgebra in $\tau_n\mathcal{O}^\otimes$, it suffices to lift the $n$-fold diagonal $X \rightarrow X^{\odot n}$ to a map $X \rightarrow X^{\otimes n}$.
\end{remark}

From the previous lemma we may conclude the following key fact about the $\infty$-category of ind-coalgebras in $\tau_n\mathcal{O}^\otimes$:

\begin{proposition}
\label{prop:indcoalgnexcisive}
The $\infty$-category $\mathrm{coAlg}^{\mathrm{ind}}(\tau_n\mathcal{O}^\otimes)$ is $n$-excisive.
\end{proposition}
\begin{proof}
The proof is by induction on $n$. The case $n=1$ is clear, since $\mathrm{coAlg}^{\mathrm{ind}}(\tau_1\mathcal{O}^\otimes)$ is just $\mathcal{O}$ itself and $\mathcal{O}$ is stable by assumption. For the inductive step from $n-1$ to $n$, consider the square of Lemma \ref{lem:algtaun+1}. By the inductive hypothesis, the $\infty$-category $\mathrm{coAlg}^{\mathrm{ind}}(\tau_{n-1}\mathcal{O}^\otimes)$ is $(n-1)$-excisive. By the first part of Lemma \ref{lem:Fcatnexc} below, the two $\infty$-categories on the right are $n$-excisive: indeed, an expression of the form $X^{\otimes n}$ or $X^{\odot n}$ is $n$-excisive as a functor of $X$ because it is the diagonal of a multilinear functor of $n$ variables, and the class of $n$-excisive functors is closed under limits. The proposition follows from this, since the class of $n$-excisive $\infty$-categories is closed under taking limits.
\end{proof}

\begin{remark}
Note that we do \emph{not} claim that $\mathrm{coAlg}^{\mathrm{ind}}(\tau_{n-1}\mathcal{O}^\otimes)$ is equivalent to $\mathcal{P}_{n-1}\mathrm{coAlg}^{\mathrm{ind}}(\mathcal{O}^\otimes)$. In fact this is usually not the case. See, however, Corollary \ref{cor:Pnindcoalg} below.
\end{remark}

\begin{lemma}
\label{lem:Fcatnexc}
If $F: \mathcal{O} \rightarrow \mathcal{O}$ is an $n$-excisive functor, then the $\infty$-category $\mathrm{Ind}\bigl\{X \rightarrow F(X)\bigr\}^c_{\mathcal{O}}$ is $n$-excisive. Furthermore, the obvious functor
\begin{equation*}
\mathrm{Ind}\bigl\{X \rightarrow F(X)\bigr\}^c_{\mathcal{O}} \longrightarrow \mathrm{Ind}\bigl\{X \rightarrow P_{n-1}F(X)\bigr\}^c_{\mathcal{O}}
\end{equation*}
is (the left adjoint of) a strong $(n-1)$-excisive approximation.
\end{lemma}
\begin{proof}
Write $\mathcal{D}$ for $\mathrm{Ind}\bigl\{X \rightarrow F(X)\bigr\}^c_{\mathcal{O}}$ and $u: \mathcal{D} \rightarrow \mathcal{O}$ for the forgetful functor, which preserves colimits and is conservative (i.e. detects equivalences). For objects $X, Y \in \mathcal{D}$ there is an equalizer diagram as follows:
\[
\xymatrix{
\mathrm{Map}_{\mathcal{D}}(X, Y)  \ar[r] &  \mathrm{Map}_{\mathcal{O}}(uX, uY) \ar@<.5ex>[r] \ar@<-.5ex>[r] & \mathrm{Map}_{\mathcal{O}}(uX, F(uY)).
}
\]
Writing $r$ for the right adjoint of $u$, it follows that there is an equalizer diagram of functors
\[
\xymatrix{
\mathrm{id}_{\mathcal{D}} \ar[r] & ru \ar@<.5ex>[r] \ar@<-.5ex>[r] & rFu.
}
\]
The composition $ru$ is $1$-excisive, whereas $rFu$ is $n$-excisive. Therefore $\mathrm{id}_{\mathcal{D}}$ is $n$-excisive, being a limit of $n$-excisive functors. 

To prove that $\mathcal{D}$ is an $n$-excisive $\infty$-category we verify condition (b) of Corollary \ref{cor:bnexcisive}. Let $\mathcal{X}: \mathbf{P}(n+1) \rightarrow \mathcal{D}$ be a Cartesian $(n+1)$-cube such that the restriction $\mathcal{X}_0 := \mathcal{X}|_{\mathbf{P}_0(n+1)}$ is a special punctured $(n+1)$-cube. We need to show that $\mathcal{X}$ is strongly coCartesian. Note that it suffices to treat the case where the vertices of $\mathcal{X}$ are compact objects. We will define a strongly coCartesian cube $\mathcal{X}'$ in $\mathcal{D}$ whose restriction to $\mathbf{P}_0(n+1)$ coincides with $\mathcal{X}_0$. We will then prove that $\mathcal{X}'$ is Cartesian, so that it is equivalent to $\mathcal{X}$, proving the lemma. First, consider the special punctured $(n+1)$-cube $u\mathcal{X}_0$ in $\mathcal{O}$ and complete it to a Cartesian cube (which we suggestively denote $u\mathcal{X}'$) by setting
\begin{equation*}
u\mathcal{X}'(\varnothing) := \varprojlim u\mathcal{X}_0.
\end{equation*}
Since $\mathcal{O}$ is 1-excisive (and hence a fortiori $n$-excisive) the cube $u\mathcal{X}'$ is strongly coCartesian by condition (b) of Corollary \ref{cor:bnexcisive} (or directly from Lemma \ref{lem:cocartstrong}). To define an $(n+1)$-cube $\mathcal{X}'$ in $\mathcal{D}$ we should specify a natural transformation $\nu: u\mathcal{X}' \rightarrow F(u\mathcal{X}')$. We already have a natural transformation $u\mathcal{X}_0 \rightarrow F(u\mathcal{X}_0)$. One completes the definition of $\nu$ by considering the induced map
\begin{equation*}
u\mathcal{X}'(\varnothing) \longrightarrow \varprojlim F(u\mathcal{X}_0) \simeq F(u\mathcal{X}'(\varnothing)).
\end{equation*}
The equivalence above follows from the fact that $F$ is $n$-excisive. The $(n+1)$-cube $\mathcal{X}'$ we have defined is strongly coCartesian simply because $u$ creates colimits. To see that it is Cartesian, consider (for any $Y \in \mathcal{D}$) the equalizer diagram
\[
\xymatrix{
\mathrm{Map}_{\mathcal{D}}(Y, \mathcal{X}'(\varnothing))  \ar[r] &  \mathrm{Map}_{\mathcal{O}}(uY, u\mathcal{X}'(\varnothing)) \ar@<.5ex>[r] \ar@<-.5ex>[r] & \mathrm{Map}_{\mathcal{O}}(uY, F(u\mathcal{X}'(\varnothing)))
}
\]
and observe that the second and third terms are canonically equivalent to
\begin{equation*}
\varprojlim \mathrm{Map}_{\mathcal{O}}(uY, u\mathcal{X}_0)\quad\quad \text{and} \quad\quad \varprojlim\mathrm{Map}_{\mathcal{O}}(uY, F(u\mathcal{X}_0))
\end{equation*}
respectively. It follows that the map
\begin{equation*}
\mathrm{Map}_{\mathcal{D}}(Y, \mathcal{X}'(\varnothing)) \longrightarrow \varprojlim \mathrm{Map}_{\mathcal{D}}(Y, \mathcal{X}_0) 
\end{equation*}
is an equivalence.

Finally we need to prove the last claim of the lemma concerning $\mathcal{P}_{n-1}\mathcal{D}$. Recall that $\mathcal{T}_{n-1}\mathcal{D}$ is the full subcategory of $\mathrm{Fun}(\mathbf{P}_0(n), \mathcal{D})$ spanned by the special punctured $n$-cubes. Note that objects of $\mathcal{T}_{n-1}\mathcal{D}$ may be identified with special punctured $n$-cubes $\mathcal{X}_0$ in $\mathcal{O}$ together with a natural transformation $\nu: \mathcal{X}_0 \rightarrow F(\mathcal{X}_0)$. It is then straightforward to see that the following functor is an equivalence of $\infty$-categories:
\begin{equation*}
\varprojlim: \mathcal{T}_{n-1}\mathcal{D}^c \longrightarrow \bigl\{X \rightarrow T_{n-1}F(X)\bigr\}^c_{\mathcal{O}}: (\nu: \mathcal{X}_0 \rightarrow F(\mathcal{X}_0)) \longmapsto (\varprojlim \nu: \varprojlim\mathcal{X}_0 \rightarrow \varprojlim F(\mathcal{X}_0)).
\end{equation*}
From here it is a simple formal exercise to find an equivalence
\begin{equation*}
\mathcal{P}_{n-1}\mathcal{D} \longrightarrow \mathrm{Ind}\bigl\{X \rightarrow \varinjlim_k T_{n-1}^kF(X)\bigr\}^c_{\mathcal{O}} \simeq \mathrm{Ind}\bigl\{X \rightarrow P_{n-1}F(X)\bigr\}^c_{\mathcal{O}}
\end{equation*}
which completes the proof.
\end{proof}
\begin{remark}
Note that the statement and proof of the previous lemma only depend on the restriction of $F$ to compact objects of $\mathcal{O}$, so that it is not necessary to assume that $F$ preserves filtered colimits.
\end{remark}

The following observation (see \cite{mccarthy}) will be of crucial importance:

\begin{lemma}
\label{lem:normseq}
The norm sequence
\begin{equation*}
(X^{\otimes n})_{h\Sigma_n} \longrightarrow (X^{\otimes n})^{h\Sigma_n} \longrightarrow (X^{\otimes n})^{t\Sigma_n}
\end{equation*}
exhibits the first term (resp. the last term) as the $n$th homogeneous layer (resp. the $(n-1)$-excisive approximation) of the functor in the middle. 
\end{lemma}
\begin{proof}
A simple calculation shows that the norm map induces an equivalence on $n$th cross effects, so that the first map induces an equivalence on $n$th derivatives.
\end{proof}

Recall from Theorem \ref{thm:TheoremA} that the functor $\mathcal{P}_{n-1}$ preserves pullbacks. Applying $\mathcal{P}_{n-1}$ to the square of Lemma \ref{lem:algtaun+1} and using Lemma \ref{lem:Fcatnexc} gives the following:

\begin{corollary}
\label{cor:Pnindcoalg}
The following is a pullback square of compactly generated $\infty$-categories:
\[
\xymatrix{
\mathcal{P}_{n-1}\mathrm{coAlg}^{\mathrm{ind}}(\tau_n\mathcal{O}^\otimes) \ar[r]\ar[d] & \mathrm{Ind}\bigl\{X \rightarrow (X \otimes^n \cdots \otimes^n X)^{t\Sigma_n} \bigr\}^c_{\mathcal{O}} \ar[d] \\
\mathrm{coAlg}^{\mathrm{ind}}(\tau_{n-1}\mathcal{O}^\otimes) \ar[r] & \mathrm{Ind}\bigl\{X \rightarrow (X \odot^n \cdots \odot^n X)^{t\Sigma_n} \bigr\}^c_{\mathcal{O}}.
}
\]
\end{corollary}

It is also straightforward to describe the relation between $\mathrm{coAlg}^{\mathrm{ind}}(\tau_n\mathcal{O}^\otimes)$ and its $(n-1)$-excisive approximation:

\begin{lemma}
\label{lem:coalgvsPn}
The following is a pullback square of compactly generated $\infty$-categories:
\[
\xymatrix{
\mathrm{coAlg}^{\mathrm{ind}}(\tau_n\mathcal{O}^\otimes) \ar[r]\ar[d] & \mathrm{Ind}\bigl\{X \rightarrow (X^{\otimes^n})^{h\Sigma_n} \bigr\}^c_{\mathcal{O}} \ar[d] \\
\mathcal{P}_{n-1}\mathrm{coAlg}^{\mathrm{ind}}(\tau_n\mathcal{O}^\otimes) \ar[r] & \mathrm{Ind}\bigl\{X \rightarrow (X^{\odot^n})^{h\Sigma_n} \times_{(X^{\odot^n})^{t\Sigma_n}} (X^{\otimes^n})^{t\Sigma_n} \bigr\}^c_{\mathcal{O}}.
}
\]
\end{lemma}
\begin{proof}
Consider the following diagram of compactly generated $\infty$-categories:
\[
\xymatrix{
\mathrm{coAlg}^{\mathrm{ind}}(\tau_n\mathcal{O}^\otimes) \ar[r]\ar[d] & \mathrm{Ind}\bigl\{X \rightarrow (X^{\otimes^n})^{h\Sigma_n} \bigr\}^c_{\mathcal{O}} \ar[d] & & \\
\mathcal{P}_{n-1}\mathrm{coAlg}^{\mathrm{ind}}(\tau_n\mathcal{O}^\otimes) \ar[r]\ar[d] & \mathrm{Ind}\bigl\{X \rightarrow (X^{\odot^n})^{h\Sigma_n} \times_{(X^{\odot^n})^{t\Sigma_n}} (X^{\otimes^n})^{t\Sigma_n} \bigr\}^c_{\mathcal{O}} \ar[d]\ar[r] & \mathrm{Ind}\bigl\{X \rightarrow (X^{\otimes^n})^{t\Sigma_n} \bigr\}^c_{\mathcal{O}} \ar[d] \\
\mathrm{coAlg}^{\mathrm{ind}}(\tau_{n-1}\mathcal{O}^\otimes) \ar[r] & \mathrm{Ind}\bigl\{X \rightarrow (X^{\odot^n})^{h\Sigma_n} \bigr\}^c_{\mathcal{O}} \ar[r] & \mathrm{Ind}\bigl\{X \rightarrow (X^{\odot^n})^{t\Sigma_n} \bigr\}^c_{\mathcal{O}}.
}
\]
The lower right square is clearly a pullback. Also, the rectangle formed by the lower two squares is a pullback by Corollary \ref{cor:Pnindcoalg}, so that the bottom left square must be a pullback by the usual pasting lemma for pullbacks. Similarly, the vertical rectangle formed by the left two squares is a pullback by Lemma \ref{lem:algtaun+1}. It follows that the top left square is a pullback.
\end{proof}

There is an evident functor
\begin{equation*}
\mathrm{triv}: \mathcal{O} \longrightarrow \mathrm{coAlg}^{\mathrm{ind}}(\tau_n\mathcal{O}^\otimes)
\end{equation*}
assigning to each object $X$ of $\mathcal{O}$ the trivial coalgebra structure on $X$, i.e. a coalgebra equipped with (a coherent system of) nullhomotopies for each of the maps $X \rightarrow (X^{\otimes k})^{h\Sigma_k}$ with $1 \leq k \leq n$. To be precise, one can construct this functor inductively using Lemma \ref{lem:algtaun+1} and the evident functors
\begin{equation*}
\mathcal{O} \longrightarrow \mathrm{Ind}\bigl\{X \longrightarrow (X \otimes^k \cdots \otimes^k X)^{h\Sigma_k}\bigr\}^c_{\mathcal{O}}
\end{equation*}
assigning to each $X$ the zero map into $(X^{\otimes k})^{h\Sigma_k}$. The following is then a straightforward consequence of what we have done so far:

\begin{corollary}
\label{cor:fibercoalgs}
The following is a pullback square of compactly generated $\infty$-categories:
\[
\xymatrix{
\mathrm{Ind}\bigl\{X \rightarrow \mathrm{fib}(X^{\otimes n} \rightarrow X^{\odot n})_{h\Sigma_n} \bigr\}^c_{\mathcal{O}} \ar[r] \ar[d] & \mathrm{coAlg}^{\mathrm{ind}}(\tau_n\mathcal{O}^\otimes) \ar[d] \\
\mathcal{O} \ar[r]_-{\mathrm{triv}} & \mathcal{P}_{n-1}\mathrm{coAlg}^{\mathrm{ind}}(\tau_n\mathcal{O}^\otimes).
}
\]
\end{corollary}
\begin{proof}
This is immediate from Lemma \ref{lem:coalgvsPn} and the observation that the fiber of
\begin{equation*}
(X^{\otimes n})^{h\Sigma_n} \longrightarrow (X^{\odot n})^{h\Sigma_n} \times_{(X^{\odot n})^{t\Sigma_n}} (X^{\otimes n})^{t\Sigma_n}
\end{equation*}
is canonically equivalent to the fiber of the map
\begin{equation*}
(X^{\otimes n})_{h\Sigma_n} \longrightarrow (X^{\odot n})_{h\Sigma_n}.
\end{equation*}
Indeed, the former is the total fiber in the following square:
\[
\xymatrix{
(X^{\otimes n})^{h\Sigma_n} \ar[r] \ar[d] & (X^{\otimes n})^{t\Sigma_n} \ar[d] \\
(X^{\odot n})^{h\Sigma_n} \ar[r] & (X^{\odot n})^{t\Sigma_n}.   
}
\]
This total fiber may be computed by first taking the fibers of the rows, yielding the functors 
\begin{equation*}
(X^{\otimes n})_{h\Sigma_n} \quad\quad \text{and} \quad\quad (X^{\odot n})_{h\Sigma_n}, 
\end{equation*}
and then taking the fiber of the evident map between those.
\end{proof}


\chapter{The space of Goodwillie towers}
\label{sec:classification}

In this chapter we prove Theorem \ref{thm:classification}. We construct the square of that theorem in Section \ref{subsec:tatediagonal}. Then we make precise the construction of $n$-stages outlined in Chapter \ref{sec:informalconstr}. Finally, in Section \ref{subsec:classification}, we prove that the square of the theorem is in fact a homotopy pullback of spaces.

\section{The Tate diagonal}
\label{subsec:tatediagonal}

In this section we construct the square
\[
\xymatrix{
\mathcal{G}_{n}(\mathcal{O}^\otimes) \ar[r]^-{T_n}\ar[d]_{p_n} &  \mathcal{T}_n \ar[d] \\
\mathcal{G}_{n-1}(\mathcal{O}^\otimes) \ar[r] & \widehat{\mathcal{T}}_n.
}
\]
of Theorem \ref{thm:classification}. We will describe the spaces $\mathcal{T}_n$ and $\widehat{\mathcal{T}}_n$ as the total spaces of fibrations $t_n$ and $\widehat{t}_n$ over $\mathcal{G}_{n-1}(\mathcal{O}^\otimes)$, with the maps in the square above arising from certain sections of these. The fibers of $\mathcal{T}_n$ and $\widehat{\mathcal{T}}_n$ over a fixed $(n-1)$-stage $\mathcal{C}$ are the spaces of natural transformations $\mathrm{Nat}(\Sigma^\infty_{\mathcal{C}}, \Theta_{\mathcal{C}})$ and $\mathrm{Nat}(\Sigma^\infty_{\mathcal{C}}, \Psi_{\mathcal{C}})$ respectively. Recall that the functors $\Theta_{\mathcal{C}}$ and $\Psi_{\mathcal{C}}$ are the following:
\begin{eqnarray*}
\Theta_{\mathcal{C}}: \mathcal{C} \longrightarrow \mathcal{O}: && X \longmapsto \bigl(\Sigma^\infty_{\mathcal{C}} X \otimes^n \cdots \otimes^n \Sigma^\infty_{\mathcal{C}} X \bigr)^{t\Sigma_n}, \\
\Psi_{\mathcal{C}}: \mathcal{C} \longrightarrow \mathcal{O}: && X \longmapsto \bigl(\Sigma^\infty_{\mathcal{C}} X \odot^n \cdots \odot^n \Sigma^\infty_{\mathcal{C}} X \bigr)^{t\Sigma_n}.
\end{eqnarray*}
Here we have suppressed the identification of $\mathrm{Sp}(\mathcal{C})$ with $\mathcal{O}$ in our notation, which we will continue to do in order to avoid cluttering. The vertical map $p_n$ in the square is given by the formation of $(n-1)$-excisive approximations. The map $T_n: \mathcal{G}_n(\mathcal{O}^\otimes) \rightarrow  \mathcal{T}_n$ (to be constructed below) assigns to an $n$-excisive category the Tate diagonal described informally in Chapter \ref{sec:informalconstr}. 

To begin with, we note that there is a `tautological' coCartesian fibration
\begin{equation*}
\gamma: \Gamma \rightarrow \mathcal{G}_{n-1}(\mathcal{O}^\otimes)
\end{equation*}
which is classified (in the sense of Definition 3.3.2.2 of \cite{htt}) by the evident functor $\mathcal{G}_{n-1}(\mathcal{O}^\otimes) \rightarrow \mathbf{Cat}_\infty$. In particular, the fiber of $\gamma$ over a vertex representing an $n-1$-stage $\mathcal{C}$ is canonically equivalent to the $\infty$-category $\mathcal{C}$. In fact $\gamma$ is also a Cartesian fibration; this is a general feature of coCartesian fibrations over a Kan complex (cf. Proposition 3.3.1.8 of \cite{htt}), but explicitly $\gamma$ (as a Cartesian fibration) is classified by the functor
\begin{equation*}
\mathcal{G}_{n-1}(\mathcal{O}^\otimes)^{\mathrm{op}} \rightarrow (\mathcal{P}r^L)^{\mathrm{op}} \rightarrow \mathcal{P}r^R \rightarrow \mathbf{Cat}_\infty.
\end{equation*}
The first arrow denotes the evident functor sending an $n-1$-stage to the corresponding $n-1$-excisive $\infty$-category, with $\mathcal{P}r^L$ the $\infty$-category of presentable $\infty$-categories and left adjoint functors, whereas the second arrow is the equivalence which takes right adjoints (and is the identity on objects), cf. Corollary 5.5.3.4 of \cite{htt}. The final arrow is the inclusion.

As in Construction 6.2.2.2 of \cite{higheralgebra}, there is a \emph{relative stabilization}
\[
\xymatrix{
\mathrm{St}(\gamma) \ar[rr]^-{\Omega^\infty_\gamma}\ar[dr] &&  \Gamma \ar[dl]^{\gamma} \\
& \mathcal{G}_{n-1}(\mathcal{O}^\otimes) &
}
\]
of the fibration $\gamma$. The map $\mathrm{St}(\gamma) \rightarrow \mathcal{G}_{n-1}(\mathcal{O}^\otimes)$ is a coCartesian fibration by Proposition 6.2.2.8 of \cite{higheralgebra}. The fiber of the map $\Omega^\infty_\gamma$ over a vertex $\mathcal{C}$ is the (absolute) stabilization
\begin{equation*}
\Omega^\infty_{\mathcal{C}}: \mathrm{Sp}(\mathcal{C}) \rightarrow \mathcal{C}.
\end{equation*}
One can think of the existence of this relative stabilization as expressing the functoriality of the assignment
\begin{equation*}
\mathcal{G}_{n-1}(\mathcal{O}^\otimes) \rightarrow \mathbf{Cat}_\infty: \mathcal{C} \mapsto \mathrm{Sp}(\mathcal{C}).
\end{equation*}
The functor $\Omega^\infty_\gamma$ admits a \emph{relative left adjoint}
\[
\xymatrix{
\mathrm{St}(\gamma) \ar[dr] &&  \Gamma \ar[dl]^{\gamma}\ar[ll]_{\Sigma^\infty_\gamma} \\
& \mathcal{G}_{n-1}(\mathcal{O}^\otimes) &
}
\]
by Proposition 7.3.2.6 of \cite{higheralgebra}. Its fiber over $\mathcal{C}$ is of course the left adjoint functor $\Sigma^\infty_{\mathcal{C}}: \mathcal{C} \rightarrow \mathrm{Sp}(\mathcal{C})$.

\begin{remark}
In fact our definition of $\mathcal{G}_{n-1}(\mathcal{O}^\otimes)$, which includes for every $\mathcal{C}$ an equivalence $\mathrm{Sp}(\mathcal{C}) \simeq \mathcal{O}$, implies that the fibration $\mathrm{St}(\gamma) \rightarrow \mathcal{G}_{n-1}(\mathcal{O}^\otimes)$ is equivalent to the `constant' fibration
\begin{equation*}
\mathcal{O} \times \mathcal{G}_{n-1}(\mathcal{O}^\otimes) \rightarrow \mathcal{G}_{n-1}(\mathcal{O}^\otimes).
\end{equation*}
The only reason for introducing $\mathrm{St}(\gamma)$ as above is to make explicit the map $\Sigma^\infty_\gamma$, whose domain is generally \emph{not} a constant fibration.
\end{remark}

We define another map of simplicial sets
\begin{equation*}
\mathrm{Fun}_{\mathcal{G}_{n-1}(\mathcal{O}^\otimes)}(\Gamma, \mathrm{St}(\gamma)) \xrightarrow{f_n} \mathcal{G}_{n-1}(\mathcal{O}^\otimes).
\end{equation*}
which is characterized by the formula
\begin{equation*}
\mathrm{Hom}_{\mathcal{G}_{n-1}(\mathcal{O}^\otimes)}\bigl(\Delta^k, \mathrm{Fun}_{\mathcal{G}_{n-1}(\mathcal{O}^\otimes)}(\Gamma, \mathrm{St}(\gamma))\bigr) = \mathrm{Hom}_{\mathcal{G}_{n-1}(\mathcal{O}^\otimes)}\bigl(\Gamma \times_{{\mathcal{G}_{n-1}(\mathcal{O}^\otimes)}} \Delta^k, \mathrm{St}(\gamma)\bigr).
\end{equation*}
In particular, the fiber of $f_n$ over $\mathcal{C}$ is the $\infty$-category $\mathrm{Fun}(\mathcal{C}, \mathrm{Sp}(\mathcal{C}))$. Corollary 3.2.2.13 of \cite{htt} guarantees that $f_n$ is again a coCartesian fibration (or see 3.10 of \cite{barwickshah} for a discussion of this construction). Let us record the following fairly evident property:

\begin{lemma}
\label{lem:sectionfn}
Assign to a map of simplicial sets $\alpha: \Gamma \rightarrow \mathrm{St}(\gamma)$ over $\mathcal{G}_{n-1}(\mathcal{O}^\otimes)$ the section of $f_n$ defined by the formula
\begin{equation*}
\bigl(\Delta^k \rightarrow \mathcal{G}_{n-1}(\mathcal{O}^\otimes)\bigr) \mapsto \bigl(\Gamma \times_{{\mathcal{G}_{n-1}(\mathcal{O}^\otimes)}} \Delta^k \rightarrow \Gamma \xrightarrow{\alpha} \mathrm{St}(\gamma)\bigr).
\end{equation*}
Then this assignment determines a bijection between such maps $\alpha$ and sections of $f_n$.
\end{lemma}
\begin{proof}
It is easy to check that the inverse construction is given as follows: say $s$ is a section of $f_n$ and consider a simplex 
\begin{equation*}
\xi: \Delta^k \rightarrow \Gamma.
\end{equation*}
Then $\xi$ can also be viewed as a $k$-simplex of $\Gamma \times_{\mathcal{G}_{n-1}(\mathcal{O}^\otimes)} \Delta^k$ and one defines
\begin{equation*}
\widehat{s}(\xi) := \bigl(\Delta^k \xrightarrow{\xi} \Gamma \times_{\mathcal{G}_{n-1}(\mathcal{O}^\otimes)} \Delta^k \xrightarrow{s(\gamma\xi)} \mathrm{St}(\gamma)\bigr).
\end{equation*}
One verifies that $\widehat{s}$ is a map of simplicial sets and the assignment $s \mapsto \widehat{s}$ is the desired inverse.\end{proof}

In particular, the left adjoint $\Sigma^\infty_\gamma: \Gamma \rightarrow \mathrm{St}(\gamma)$ gives rise to a section
\begin{equation*}
\sigma: \mathcal{G}_{n-1}(\mathcal{O}^\otimes) \rightarrow \mathrm{Fun}_{\mathcal{G}_{n-1}(\mathcal{O}^\otimes)}(\Gamma, \mathrm{St}(\gamma))
\end{equation*}
of $f_n$. Similarly, the assignments $\mathcal{C} \mapsto \Psi_\mathcal{C}$ and $\mathcal{C} \mapsto \Theta_\mathcal{C}$ give rise to sections $\psi$ and $\theta$ of $f_n$ respectively. Indeed, these are easily obtained from $\sigma$ by identifying $\mathrm{St}(\gamma)$ with $\mathcal{O} \times \mathcal{G}_{n-1}(\mathcal{O}^\otimes)$ and postcomposing with the functors
\begin{equation*}
X \mapsto (X^{\odot^n})^{t\Sigma_n} \quad\quad \text{resp.} \quad\quad X \mapsto (X^{\otimes^n})^{t\Sigma_n}
\end{equation*}
on the first factor. In fact, we will also need the following variant of this construction. Write $N_\mathcal{C}$ for the evident natural transformation $\Theta_{\mathcal{C}} \rightarrow \Psi_{\mathcal{C}}$ induced by the natural map
\begin{equation*}
(X^{\otimes^n})^{t\Sigma_n} \rightarrow (X^{\odot^n})^{t\Sigma_n}.
\end{equation*}
Then there is a corresponding map
\begin{equation*}
\nu: \mathcal{G}_{n-1}(\mathcal{O}^\otimes) \rightarrow \mathrm{Fun}_{\mathcal{G}_{n-1}(\mathcal{O}^\otimes)}(\Gamma, \mathrm{St}(\gamma))^{\Delta^1}
\end{equation*}
whose value at a vertex $\mathcal{C}$ is described by $\nu(\mathcal{C}) = N_{\mathcal{C}}$.

Now define a map $\widehat{t}_n: \widehat{\mathcal{T}}_n \rightarrow \mathcal{G}_{n-1}(\mathcal{O}^\otimes)$ by the pullback square
\[
\xymatrix{
\widehat{\mathcal{T}}_n \ar[d]_{\widehat{t}_n} \ar[r] & \mathrm{Fun}_{\mathcal{G}_{n-1}(\mathcal{O}^\otimes)}(\Gamma, \mathrm{St}(\gamma))^{\Delta^1} \ar[d]^{(\mathrm{ev}_0, \, \mathrm{ev}_1)} \\
\mathcal{G}_{n-1}(\mathcal{O}^\otimes) \ar[r]_-{(\sigma, \psi)} & \mathrm{Fun}_{\mathcal{G}_{n-1}(\mathcal{O}^\otimes)}(\Gamma, \mathrm{St}(\gamma)) \times \mathrm{Fun}_{\mathcal{G}_{n-1}(\mathcal{O}^\otimes)}(\Gamma, \mathrm{St}(\gamma))
}
\]
and similarly define a map $t_n$ by a pullback square
\[
\xymatrix{
\mathcal{T}_n \ar[d]_{t_n} \ar[r] & \mathrm{Fun}_{\mathcal{G}_{n-1}(\mathcal{O}^\otimes)}(\Gamma, \mathrm{St}(\gamma))^{\Delta^2} \ar[d]^{(\mathrm{ev}_0, \, \mathrm{ev}_{\{1,2\}})} \\
\mathcal{G}_{n-1}(\mathcal{O}^\otimes) \ar[r]_-{(\sigma, \nu)} & \mathrm{Fun}_{\mathcal{G}_{n-1}(\mathcal{O}^\otimes)}(\Gamma, \mathrm{St}(\gamma)) \times \mathrm{Fun}_{\mathcal{G}_{n-1}(\mathcal{O}^\otimes)}(\Gamma, \mathrm{St}(\gamma))^{\Delta^1}.
}
\]
Observe that the inclusion
\begin{equation*}
\Delta^1 \simeq \Delta^{\{0,1\}} \subseteq \Delta^2
\end{equation*}
defines a map $\mathcal{T}_n \rightarrow \widehat{\mathcal{T}}_n$ compatible with the maps to $\mathcal{G}_{n-1}(\mathcal{O}^\otimes)$. Also, note that the fiber of $\widehat{t}_n$ over a vertex $\mathcal{C}$ is given by (a model for) the space of natural transformations 
\begin{equation*}
\mathrm{Nat}(\Sigma^\infty_{\mathcal{C}}, \Psi_\mathcal{C})
\end{equation*}
between functors from $\mathcal{C}$ to $\mathrm{Sp}(\mathcal{C}) \simeq \mathcal{O}$. Similarly, the fiber of $t_n$ over $\mathcal{C}$ is a model for the space of natural transformations
\begin{equation*}
\mathrm{Nat}(\Sigma^\infty_{\mathcal{C}}, \Theta_\mathcal{C}).
\end{equation*}

In analogy with the construction of $\widehat{\mathcal{T}}_n$ it might seem more obvious to construct $\mathcal{T}_n$ by exactly the same procedure with $\Theta_\mathcal{C}$ in place of $\Psi_{\mathcal{C}}$. This would yield a fibration equivalent to the map $t_n$ constructed above, simply because the forgetful functor 
\begin{equation*}
\mathrm{Fun}(\mathcal{C}, \mathrm{Sp}(\mathcal{C}))_{/N_{\mathcal{C}}} \longrightarrow \mathrm{Fun}(\mathcal{C}, \mathrm{Sp}(\mathcal{C}))_{/\Theta_{\mathcal{C}}}
\end{equation*}
is a trivial fibration. However, the construction of $\mathcal{T}_n$ we use admits an evident map to $\widehat{\mathcal{T}}_n$, which is moreover a Kan fibration:

\begin{lemma}
\label{lem:tnKanfibration}
The maps $t_n$ and $\widehat{t}_n$ are Kan fibrations. The map $\mathcal{T}_n \rightarrow \widehat{\mathcal{T}}_n$ described above is a Kan fibration as well.
\end{lemma}
\begin{proof}
It is convenient to phrase the proof in terms of marked simplicial sets (as in Chapter 3 of \cite{htt}); recall that those are pairs $(X, \mathcal{E})$ with $X$ a simplicial set and $\mathcal{E}$ a subset of the set of 1-simplices of $X$ containing all degenerate edges. For a simplicial set $X$, one writes $X^\flat$ for $X$ with only degenerate edges marked, $X^\sharp$ for $X$ with all edges marked, and if $X$ happens to be an $\infty$-category then $X^\natural$ denotes $X$ with the equivalences marked. To prove that $\widehat{t}_n$ is a Kan fibration we should argue the existence of solutions to lifting problems of the following form, with $0 \leq k \leq m$:
\[
\xymatrix{
\Lambda^m_k \ar[d]\ar[r] & \widehat{\mathcal{T}}_n \ar[d]^{\widehat{t}_n} \\
\Delta^m \ar[r]\ar@{-->}[ur] & \mathcal{G}_{n-1}(\mathcal{O}^\otimes).
}
\]
Equivalently, we should solve the following lifting problem of marked simplicial sets:
\[
\xymatrix{
(\Lambda^m_k)^\sharp \ar[d]\ar[r] & \widehat{\mathcal{T}}_n^\sharp \ar[d]^{\widehat{t}_n} \\
(\Delta^m)^\sharp \ar[r]\ar@{-->}[ur] & \mathcal{G}_{n-1}(\mathcal{O}^\otimes)^\sharp.
}
\]
By the pullback square defining $\widehat{t}_n$ this is equivalent to
\[
\xymatrix{
(\Lambda^m_k)^\sharp \ar[d]\ar[r] & \bigl(\mathrm{Fun}_{\mathcal{G}_{n-1}(\mathcal{O}^\otimes)}(\Gamma, \mathrm{St}(\gamma))^{\Delta^1}\bigr)^\natural \ar[d]^{(\mathrm{ev}_0, \, \mathrm{ev}_1)} \\
(\Delta^m)^\sharp \ar[r]\ar@{-->}[ur] & \mathrm{Fun}_{\mathcal{G}_{n-1}(\mathcal{O}^\otimes)}(\Gamma, \mathrm{St}(\gamma))^\natural \times \mathrm{Fun}_{\mathcal{G}_{n-1}(\mathcal{O}^\otimes)}(\Gamma, \mathrm{St}(\gamma))^\natural,
}
\]
which by adjunction is equivalent to
\[
\xymatrix{
(\Lambda^m_k)^\sharp \times (\Delta^1)^\flat \amalg_{(\Lambda^m_k)^\sharp \times (\partial\Delta^1)^\flat} (\Delta^m)^\sharp \times (\partial\Delta^1)^\flat \ar[d]\ar[r] & \mathrm{Fun}_{\mathcal{G}_{n-1}(\mathcal{O}^\otimes)}(\Gamma, \mathrm{St}(\gamma))^\natural \\
(\Delta^m)^\sharp \times (\Delta^1)^\flat. \ar@{-->}[ur] &
}
\]
For any trivial cofibration $f: K \rightarrow L$ and monomorphism $i: A \rightarrow B$ of simplicial sets, the pushout-product of $f^{\sharp}$ and $i^\flat$ is a trivial cofibration of marked simplicial sets in the model structure of Theorem 3.1.3.7 of \cite{htt} (using Corollary 3.1.4.4 for the pushout-product property). Thus the left vertical map is a trivial cofibration of marked simplicial sets. Because the fibrant objects in this model structure are marked simplicial sets of the form $X^\natural$ (Proposition 3.1.4.1 of \cite{htt}), with $X$ an $\infty$-category, a solution as indicated by the dashed arrow exists. The proof that $t_n$ is Kan fibration proceeds similarly, now using the monomorphism
\begin{equation*}
\Delta^{\{0\}} \amalg \Delta^{\{1,2\}} \subseteq \Delta^2
\end{equation*}
in place of $\partial\Delta^1 \rightarrow \Delta^1$. Finally, to show that $\mathcal{T}_n \rightarrow \widehat{\mathcal{T}}_n$ is a Kan fibration one applies the same argument, now using the monomorphism of simpicial sets
\begin{equation*}
\Lambda_2^2 \subseteq \Delta^2.
\end{equation*}
\end{proof}

It remains to construct the maps in the square at the start of this section, i.e. a section
\begin{equation*}
\mathcal{G}_{n-1}(\mathcal{O}^\otimes) \rightarrow \widehat{\mathcal{T}}_n
\end{equation*}
of $\widehat{t}_n$, as well as the map
\begin{equation*}
T_n: \mathcal{G}_n(\mathcal{O}^\otimes) \rightarrow \mathcal{T}_n
\end{equation*}
and establish a homotopy between the two composites in the square. Note that for an $n$-stage $\mathcal{C}$ for $\mathcal{O}^\otimes$, the constructions of Section \ref{subsec:coalgebras} give a square
\[
\xymatrix{
\mathcal{C} \ar[d]\ar[r] & \mathrm{coAlg}(\tau_n\mathcal{O}^\otimes) \ar[d] \\
\mathcal{P}_{n-1}\mathcal{C} \ar[r] & \mathrm{coAlg}(\tau_{n-1}\mathcal{O}^\otimes).
}
\]
Applying the functor $\mathcal{P}_{n-1}$ and using the square of Corollary \ref{cor:Pnindcoalg} then gives a further diagram
\[
\xymatrix{
\mathcal{P}_{n-1}\mathcal{C} \ar@{=}[d]\ar[r] & \mathrm{Ind}\bigl\{X \rightarrow \Theta_{\mathcal{P}_{n-1}\mathcal{C}}(X)\bigr\}^c_{\mathcal{O}} \ar[d] \\
\mathcal{P}_{n-1}\mathcal{C} \ar[r] & \mathrm{Ind}\bigl\{X \rightarrow \Psi_{\mathcal{P}_{n-1}\mathcal{C}}(X)\bigr\}^c_{\mathcal{O}}.
}
\]
The bottom horizontal arrow exists for any $(n-1)$-stage, not necessarily of the form $\mathcal{P}_{n-1}\mathcal{C}$, and defines in a straightforward way a map $\mathcal{G}_{n-1}(\mathcal{O}^\otimes) \rightarrow \widehat{\mathcal{T}}_n$ which is a section of $\widehat{t}_n$. Similarly, the square above also defines the map $T_n:\mathcal{G}_n(\mathcal{O}^\otimes) \rightarrow \mathcal{T}_n$. Moreover, the commutativity of this square also provides the desired commutativity of the square at the start of this section.

\begin{remark}
Lemma \ref{lem:nexcfunctors} provides another way to think about the construction of the map $T_n$. Indeed, consider an $n$-stage $\mathcal{C}_n$ and write $\mathcal{C}_{n-1}$ for the $(n-1)$-excisive approximation $\mathcal{P}_{n-1}\mathcal{C}$. Also, write $F: \mathcal{C}_n \rightarrow \mathcal{C}_{n-1}$ for the induced functor. For $X \in \mathcal{C}_n$, with image $\overline{X} := \Sigma_{n,1}^\infty X$ in $\mathcal{O}$, there is a natural map
\begin{equation*}
\overline{X} \longrightarrow (\overline{X}^{\otimes n})^{t\Sigma_n}
\end{equation*}
induced by the coalgebra structure of $\overline{X}$ in $\tau_n\mathcal{O}^\otimes$. One can now use the fact that the codomain of this map is an $(n-1)$-excisive functor of $X$ and apply Lemma \ref{lem:nexcfunctors} to show that precomposition with $F$ induces a weak equivalence
\begin{equation*}
\mathrm{Nat}\bigl(\Sigma_{n-1,1}^\infty, ((\Sigma_{n-1,1}^\infty)^{\otimes n})^{t\Sigma_n}\bigr) \longrightarrow \mathrm{Nat}\bigl(\Sigma_{n,1}^\infty, ((\Sigma_{n,1}^\infty)^{\otimes n})^{t\Sigma_n}\bigr).
\end{equation*}
Here the left-hand side refers to natural transformations between functors $\mathcal{C}_{n-1} \rightarrow \mathrm{Sp}(\mathcal{C}_{n-1})$, the right-hand side to natural transformations between functors $\mathcal{C}_n \rightarrow \mathrm{Sp}(\mathcal{C}_n)$. In particular, one finds for $Y \in \mathcal{C}_{n-1}$ and $\overline{Y} := \Sigma_{n-1,1}^\infty Y$ a natural map
\begin{equation*}
\overline{Y} \longrightarrow  (\overline{Y}^{\otimes n})^{t\Sigma_n}
\end{equation*}
which extends the natural map defined for $X \in \mathcal{C}_n$ above. This map is the Tate diagonal.
\end{remark}

\section{Constructing $n$-stages}
\label{subsec:nstages}

In this section we study the pullback $\mathcal{A}$ defined by the square
\[
\xymatrix{
\mathcal{A} \ar[r]\ar[d] &  \mathcal{T}_n \ar[d] \\
\mathcal{G}_{n-1}(\mathcal{O}^\otimes) \ar[r] & \widehat{\mathcal{T}}_n.
}
\]
Since the simplicial sets in this square are Kan complexes and the right vertical map is a Kan fibration (by Lemma \ref{lem:tnKanfibration}), this pullback is also a homotopy pullback. As a consequence of the constructions of the previous section there is (up to contractible ambiguity) a canonical map $\alpha: \mathcal{G}_n(\mathcal{O}^\otimes) \rightarrow \mathcal{A}$. We will construct a map $\beta: \mathcal{A} \rightarrow \mathcal{G}_n(\mathcal{O}^\otimes)$. In the next section we demonstrate that $\alpha$ and $\beta$ are homotopy inverse to each other, proving Theorem \ref{thm:classification}.

A vertex $\mathcal{Z}$ of $\mathcal{A}$ may be identified with an $(n-1)$-stage $\mathcal{C}$ for $\mathcal{O}^\otimes$ together with a 2-simplex
\[
\xymatrix{
& \Theta_{\mathcal{C}} \ar[d] \\
\Sigma^\infty_{\mathcal{C}} \ar[r]\ar[ur] &  \Psi_{\mathcal{C}}
}
\]
in the $\infty$-category $\mathrm{Fun}(\mathcal{C}, \mathrm{Sp}(\mathcal{C}))$. In other words, adopting the usual short-hand $\overline{X} = \Sigma^\infty_{\mathcal{C}} X$ for objects $X$ of $\mathcal{C}$, we have a natural diagram
\[
\xymatrix{
& (\overline{X}^{\otimes n})^{t\Sigma_n} \ar[d] \\
\overline{X} \ar[r]\ar[ur] &  (\overline{X}^{\odot n})^{t\Sigma_n}
}
\]
where the horizontal arrow arises from the coalgebra structure on $\overline{X}$ supplied by the functor $\mathcal{C} \rightarrow \mathrm{coAlg}^{\mathrm{ind}}(\tau_{n-1}\mathcal{O}^\otimes)$. Invoking Corollary \ref{cor:Pnindcoalg} we see that this data is equivalent to a lift of that functor as follows:
\[
\xymatrix{
& \mathcal{P}_{n-1}\mathrm{coAlg}^{\mathrm{ind}}(\tau_n\mathcal{O}^\otimes) \ar[d] \\
\mathcal{C} \ar[r]\ar[ur] & \mathrm{coAlg}^{\mathrm{ind}}(\tau_{n-1}\mathcal{O}^\otimes).
}
\]

\begin{definition}
Define $\beta_0(\mathcal{Z})$ to be the pullback, formed in $\mathbf{Cat}_*^{\omega}$, in the following square:
\[
\xymatrix{
\beta_0(\mathcal{Z}) \ar[r]\ar[d] & \mathrm{coAlg}^{\mathrm{ind}}(\tau_n\mathcal{O}^\otimes) \ar[d] \\
\mathcal{C} \ar[r] & \mathcal{P}_{n-1}\mathrm{coAlg}^{\mathrm{ind}}(\tau_n\mathcal{O}^\otimes).
}
\]
Fixing a choice of pullback functor for $\mathbf{Cat}_*^{\omega}$ (which is of course unique up to contractible ambiguity) and observing that the constructions of the paragraph above make sense for families of $(n-1)$-stages $\mathcal{C}$, one obtains a functor
\begin{equation*}
\beta_0: \mathcal{A} \rightarrow \mathbf{Cat}_*^\omega.
\end{equation*}
\end{definition}

\begin{remark}
A particular way of fixing the choice of pullback is as follows: one replaces the right vertical map (which does not depend on $\mathcal{Z}$) by a categorical fibration and takes the actual pullback of simplicial sets. 
\end{remark}

The crucial properties of this construction are the following:

\begin{proposition}
\label{prop:betanexc}
The $\infty$-category $\beta_0(\mathcal{Z})$ is $n$-excisive.
\end{proposition}

This proposition is immediate from the fact that the class of $n$-excisive $\infty$-categories is closed under taking limits.

\begin{proposition}
\label{prop:betastab}
There is a canonical equivalence of $\infty$-operads
\begin{equation*}
\varphi_{\mathcal{Z}}: \tau_n\mathcal{O}^\otimes \longrightarrow \mathrm{Sp}(\beta_0(\mathcal{Z}))^\otimes. 
\end{equation*}
This equivalence is functorial, in the sense that $\varphi$ can be made into a map
\begin{equation*}
\mathcal{A} \rightarrow \mathbf{Op}^{\mathrm{St}}_{\tau_n\mathcal{O}^\otimes/}.
\end{equation*}
\end{proposition}

The previous two propositions provide a map
\begin{equation*}
\beta: \mathcal{A} \longrightarrow \mathcal{G}_n(\mathcal{O}^\otimes)
\end{equation*}
which on vertices can be described by the formula
\begin{equation*}
\beta(\mathcal{Z}) = (\beta_0(\mathcal{Z}), \varphi_{\mathcal{Z}}).
\end{equation*}
In the next section we will prove that $\beta$ is a homotopy equivalence. 

We give a proof of Proposition \ref{prop:betastab} below. Roughly speaking we should show that the first $n$ tensor products in the stable $\infty$-operad $\mathrm{Sp}(\beta_0(\mathcal{Z}))^\otimes$ agree with the first $n$ tensor products defined by the stable $\infty$-operad $\mathcal{O}^\otimes$. This will be rather obvious for the first $n-1$. For the $n$th one, the square defining $\beta_0(\mathcal{Z})$ will yield a pullback square, natural for $X \in \mathcal{O}$, as follows:
\[
\xymatrix{
(X^{\otimes_{\mathcal{Z}} n})_{h\Sigma_n} \ar[r]\ar[d] & (X^{\otimes n})^{h\Sigma_n}  \ar[d] \\
(X^{\odot n})_{h\Sigma_n} \ar[r] & (X^{\odot n})^{h\Sigma_n} \times_{(X^{\odot n})^{t\Sigma_n}} (X^{\otimes n})^{t\Sigma_n}.
}
\]
Here the symbol $\otimes_{\mathcal{Z}}$ denotes the tensor products defined by $\mathrm{Sp}(\beta_0(\mathcal{Z}))^\otimes$. Using that the fiber of the right-hand vertical map is $\mathrm{fib}(X^{\otimes n} \rightarrow X^{\odot n})_{h\Sigma_n}$ this will give $\otimes^n_{\mathcal{Z}} \simeq \otimes^n$. We now set out to make this precise:

\begin{proof}[Proof of Proposition \ref{prop:betastab}]
To avoid overburdening the exposition further than necessary we give the proof for a vertex $\mathcal{Z}$ of $\mathcal{A}$, but it is straightforward to verify that everything we do makes sense in families, i.e., works equally well for a general simplex $\zeta$ of $\mathcal{A}$. Applying $\mathcal{P}_{n-1}$ to the square defining $\beta_0(\mathcal{Z})$ and using that this functor preserves pullbacks, we conclude that the induced functor
\begin{equation*}
\mathcal{P}_{n-1}\beta_0(\mathcal{Z}) \longrightarrow \mathcal{C}
\end{equation*}
is an equivalence. By Proposition \ref{prop:SpPnCntruncated} we then find an equivalence of $\infty$-operads
\begin{equation*}
\mathrm{Sp}(\mathcal{C})^\otimes \longrightarrow \tau_{n-1}\mathrm{Sp}(\beta_0\mathcal{Z})^\otimes.
\end{equation*}
Since $\mathcal{C}$ is an $(n-1)$-stage for $\mathcal{O}^\otimes$ we have an equivalence $\tau_{n-1}\mathcal{O}^\otimes \rightarrow \mathrm{Sp}(\mathcal{C})^\otimes$, thus providing an equivalence
\begin{equation*}
\varphi'_{\mathcal{Z}}: \tau_{n-1}\mathcal{O}^\otimes \longrightarrow \tau_{n-1}\mathrm{Sp}(\beta_0\mathcal{Z})^\otimes.
\end{equation*}
We now wish to apply Proposition \ref{prop:mapstruncations} to obtain a map
\begin{equation*}
\varphi_\mathcal{Z}: \tau_n\mathcal{O}^\otimes \longrightarrow \mathrm{Sp}(\beta_0(\mathcal{Z}))^\otimes.
\end{equation*}
For simplicity of notation, let us identify $\mathrm{Sp}(\beta_0(\mathcal{Z}))$ with $\mathcal{O}$ and write $\otimes^k_{\mathcal{Z}}$ for the $k$-fold tensor product on $\mathcal{O}$ induced by the stable $\infty$-operad $ \mathrm{Sp}(\beta_0(\mathcal{Z}))^\otimes$. To extend $\varphi'_\mathcal{Z}$ to a map $\varphi_{\mathcal{Z}}$ as above we should supply a lift, natural in $X \in \mathcal{O}$, as follows:
\[
\xymatrix{
&  (X \otimes^n \cdots \otimes^n X)_{h\Sigma_n} \ar[d] \\
(X \otimes_{\mathcal{Z}}^n \cdots \otimes_{\mathcal{Z}}^n X)_{h\Sigma_n} \ar[r] \ar@{-->}[ur] & (X \odot^n \cdots \odot^n X)_{h\Sigma_n}.
}
\]
Here the bottom horizontal arrow is induced by the map $\varphi'_{\mathcal{Z}}$. Moreover, if we prove that the diagonal map is an equivalence it will follow that $\varphi_{\mathcal{Z}}$ is an equivalence of $\infty$-operads. We will do both these things at once.

Observe that the four $\infty$-categories in the square defining $\beta_0(\mathcal{Z})$ admit left adjoint functors to $\mathcal{O}$; for $\beta_0(\mathcal{Z})$ and $\mathcal{C}$ these are the stabilizations $\Sigma^\infty_{\mathcal{Z}}$ and $\Sigma^\infty_{\mathcal{C}}$ respectively, while for $\mathrm{coAlg}^{\mathrm{ind}}(\tau_n\mathcal{O}^\otimes)$ and $\mathcal{P}_{n-1}\mathrm{coAlg}^{\mathrm{ind}}(\tau_n\mathcal{O}^\otimes)$ these are the obvious forgetful functors. Associated to each of these four adjunctions is a comonad on $\mathcal{O}$. Writing $\mathrm{cofree}_n$ and $\mathrm{cofree}_n^t$ for the last two, we obtain a diagram of functors on $\mathcal{O}$ as follows:
\[
\xymatrix{
D_n(\Sigma^\infty_{\mathcal{Z}}\Omega^\infty_{\mathcal{Z}}) \ar[d]\ar[r] & D_n(\mathrm{cofree}_n) \ar[d] \\
D_n(\Sigma^\infty_{\mathcal{C}}\Omega^\infty_{\mathcal{C}}) \ar[r] & D_n(\mathrm{cofree}_n^t). 
}
\]
By standard formal reasoning this diagram is a pullback square of functors. 
Recall that $(X \otimes_{\mathcal{Z}}^n \cdots \otimes_{\mathcal{Z}}^n X)_{h\Sigma_n}$ is precisely $D_n(\Sigma^\infty_{\mathcal{Z}}\Omega^\infty_{\mathcal{Z}})(X)$.  As the notation suggests, the comonad $\mathrm{cofree}_n$ should be thought of as the cofree ind-coalgebra functor associated to the stable $\infty$-operad $\tau_n\mathcal{O}^\otimes$. We wish to apply Corollary \ref{cor:fibercoalgs} to describe the fiber of the right vertical map. The $\infty$-category 
\begin{equation*}
\mathrm{Ind}\bigl\{X \rightarrow \mathrm{fib}(X^{\otimes n} \rightarrow X^{\odot n})_{h\Sigma_n} \bigr\}^c_{\mathcal{O}}
\end{equation*}
appearing in that result induces yet another comonad on $\mathcal{O}$, which we denote $\mathrm{cofree}_{=n}^{\mathrm{dp}}$. (The superscript refers to divided powers, which we will investigate in Section \ref{subsec:dividedpowers}.) There is then a canonical equivalence
\begin{equation*}
D_n(\mathrm{cofree}_{=n}^{\mathrm{dp}})(X) \longrightarrow \mathrm{fib}(X^{\otimes n} \rightarrow X^{\odot n})_{h\Sigma_n}.
\end{equation*}
Indeed this can be checked directly, or one uses Proposition \ref{prop:derivativesid} and the fact that the $k$th derivative of the identity functor on the $\infty$-category above is trivial for $1 < k < n$. Corollary \ref{cor:fibercoalgs} gives a square
\[
\xymatrix{
\mathrm{cofree}_{=n}^{\mathrm{dp}} \ar[r]\ar[d] & \mathrm{cofree}_n \ar[d] \\
\mathrm{id}_{\mathcal{O}} \ar[r] & \mathrm{cofree}_n^t.
}
\]
Applying $D_n$ then yields a pullback
\[
\xymatrix{
\mathrm{fib}(X^{\otimes n} \rightarrow X^{\odot n})_{h\Sigma_n} \ar[r] \ar[d] & D_n(\mathrm{cofree}_n) \ar[d] \\
\ast \ar[r] & D_n(\mathrm{cofree}_n^t).
}
\]
We deduce that the fiber of the map $D_n(\Sigma^\infty_{\mathcal{Z}}\Omega^\infty_{\mathcal{Z}})(X) \rightarrow D_n(\Sigma^\infty_{\mathcal{C}}\Omega^\infty_{\mathcal{C}})(X)$ is canonically equivalent to the fiber of $(X^{\otimes n})_{h\Sigma_n} \rightarrow (X^{\odot n})_{h\Sigma_n}$. Applying $D_n$ to the composition of maps
\begin{equation*}
\Sigma^\infty_{\mathcal{Z}}\Omega^\infty_{\mathcal{Z}} \longrightarrow \mathrm{cofree}_n(X) \longrightarrow (X^{\otimes n})^{h\Sigma_n} 
\end{equation*}
yields a map
\begin{equation*}
D_n(\Sigma^\infty_{\mathcal{Z}}\Omega^\infty_{\mathcal{Z}})(X) \longrightarrow (X^{\otimes n})_{h\Sigma_n}.
\end{equation*}
This map fits into a diagram
\[
\xymatrix{
\mathrm{fib}(X^{\otimes n} \rightarrow X^{\odot n})_{h\Sigma_n} \ar[r]\ar@{=}[d] & D_n(\Sigma^\infty_{\mathcal{Z}}\Omega^\infty_{\mathcal{Z}})(X) \ar[r]\ar[d] & D_n(\Sigma^\infty_{\mathcal{C}}\Omega^\infty_{\mathcal{C}})(X) \ar[d] \\
\mathrm{fib}(X^{\otimes n} \rightarrow X^{\odot n})_{h\Sigma_n}  \ar[r] & (X^{\otimes n})_{h\Sigma_n} \ar[r] & (X^{\odot n})_{h\Sigma_n}.
}
\]
Both rows are fiber sequences by our discussion above. Also, the rightmost vertical map is an equivalence by assumption. We conclude that the middle vertical map is an equivalence, which concludes our proof.
\end{proof}

We conclude this section with some observations about our construction of $n$-stages which are not necessary for the proofs of our main results, but which will be useful to sharpen some of our results in the next section, as well as being useful in \cite{unstableperiodicity}. One could consider a variant of our construction of $\beta_0(\mathcal{Z})$ `without the ind', where one defines an $\infty$-category $\beta'_0(\mathcal{Z})$ by the following pullback square in $\mathbf{Cat}_\infty$:
\[
\xymatrix{
\beta'_0(\mathcal{Z}) \ar[r]\ar[d] & \mathrm{coAlg}(\tau_n\mathcal{O}^\otimes) \ar[d] \\
\mathcal{C} \ar[r] & \mathcal{P}_{n-1}\mathrm{coAlg}(\tau_n\mathcal{O}^\otimes).
}
\]
Here the lower right-hand corner requires some care: given that $\mathrm{coAlg}(\tau_n\mathcal{O}^\otimes)$ is not necessarily compactly generated, it does not make sense to apply $\mathcal{P}_{n-1}$ to it. Rather we take our cue from Corollary \ref{cor:Pnindcoalg} and simply define it by a pullback square as follows:
\[
\xymatrix{
\mathcal{P}_{n-1}\mathrm{coAlg}(\tau_n\mathcal{O}^\otimes) \ar[r]\ar[d] &  \bigl\{X \rightarrow (X^{\otimes n})^{t\Sigma_n}\bigr\}_{\mathcal{O}} \ar[d] \\
\mathrm{coAlg}(\tau_{n-1}\mathcal{O}^\otimes)\ar[r] & \bigl\{X \rightarrow (X^{\odot n})^{t\Sigma_n}\bigr\}_{\mathcal{O}}.
}
\]
Note that by construction the $\infty$-category of compact objects $\beta_0(\mathcal{Z})^c$ inside $\beta_0(\mathcal{Z})$ is a still a full subcategory of $\beta'_0(\mathcal{Z})$. In fact, it is easy to show that $\beta'_0(\mathcal{Z})$ is a presentable $\infty$-category, so in particular it has all colimits. Left Kan extension of the inclusion $\beta_0(\mathcal{Z})^c \rightarrow \beta'_0(\mathcal{Z})$ defines a functor
\begin{equation*}
\mathrm{Ind}(\beta_0(\mathcal{Z})^c) = \beta_0(\mathcal{Z}) \rightarrow \beta'_0(\mathcal{Z}).
\end{equation*}
It is not at all clear whether the $\infty$-category $\beta'_0(\mathcal{Z})$ itself is compactly generated. However, we will show that at least it has a good supply of compact objects:

\begin{lemma}
\label{lem:compactTatecoalgebras}
If $X \in \beta'_0(\mathcal{Z})$ is an object whose image under
\begin{equation*}
\beta'_0(\mathcal{Z}) \rightarrow \mathrm{coAlg}(\tau_n\mathcal{O}^\otimes)  \xrightarrow{\mathrm{forget}} \mathcal{O}
\end{equation*}
is a compact object of $\mathcal{O}$, then $X$ itself is a compact object of $\beta'_0(\mathcal{Z})$.
\end{lemma}
\begin{proof}
For the duration of this proof, write $u: \beta_0'(\mathcal{Z}) \rightarrow \mathcal{C}$ for the functor featuring in the defining square of $\beta_0'(\mathcal{Z})$ and $U: \beta_0'(\mathcal{Z}) \rightarrow \mathcal{O}$ for the composite described in the lemma. A chase through the definitions of these $\infty$-categories combined with Lemma \ref{lem:coalgvsPn} (or its variant for ordinary coalgebras rather than ind-coalgebras) shows that there is a fiber sequence as follows (cf. the first part of the proof of Lemma \ref{lem:Fcatnexc}):
\begin{equation*}
\mathrm{Map}_\mathcal{O}(UX, \Omega\mathrm{fib}((UY)^{\otimes n} \rightarrow (UY)^{\odot n})_{h\Sigma_n}) \rightarrow \mathrm{Map}_{\beta'_0(\mathcal{Z})}(X, Y) \rightarrow \mathrm{Map}_{\mathcal{\mathcal{C}}}(uX, uY).
\end{equation*}
This identification of the fiber works for any choice of basepoint in $\mathrm{Map}_{\mathcal{\mathcal{C}}}(uX, uY)$ (or one can argue directly that the sequence is in fact a principal fiber sequence). Also, note that if $X$ is assumed to be such that $UX$ (and hence also $uX$) are compact, then the base and fiber commute with filtered colimits when interpreted as functors of $Y$. Hence the same is true of the middle term and we conclude that $X$ itself is compact.
\end{proof}

\begin{corollary}
\label{cor:compactTatecoalgebras}
The functor $\beta_0(\mathcal{Z}) \rightarrow \beta'_0(\mathcal{Z})$ is fully faithful.
\end{corollary}

\begin{remark}
\label{rmk:compactTatecoalgebras}
If $\mathcal{O}^\otimes = \mathrm{Sp}(\mathcal{C})^\otimes$, one can use Proposition \ref{prop:derivativesid} to identify the fiber in the proof above as
\begin{equation*}
\mathrm{Map}_\mathcal{O}(UX, \partial_n \mathrm{id}_{\mathcal{C}}(UY, \ldots, UY)_{h\Sigma_n}).
\end{equation*}
Here $\partial_n\mathrm{id}_{\mathcal{C}}: \mathrm{Sp}(\mathcal{C})^{\times n} \rightarrow \mathrm{Sp}(\mathcal{C})$ denotes the $n$-multilinear functor corresponding to the $n$th derivative of the identity on $\mathcal{C}$.
\end{remark}

\section{A classification of $n$-stages}
\label{subsec:classification}

We have constructed maps $\alpha: \mathcal{G}_n(\mathcal{O}^\otimes) \rightarrow \mathcal{A}$ and $\beta: \mathcal{A} \rightarrow \mathcal{G}_n(\mathcal{O}^\otimes)$, which we claim to be homotopy equivalences. This follows from Propositions $\ref{prop:CnisbaCn}$ and $\ref{prop:Tatebeta}$ below, which will complete the proof of Theorem \ref{thm:classification}.

Suppose $\mathcal{C}$ is an $n$-stage for $\mathcal{O}^\otimes$, so that in particular we have an equivalence $\varphi_{\mathcal{C}}: \tau_n\mathcal{O}^\otimes \rightarrow \mathrm{Sp}(\mathcal{C})^\otimes$. Then we find a square
\[
\xymatrix{
\mathcal{C} \ar[r]\ar[d] & \mathrm{coAlg}^{\mathrm{ind}}(\tau_n\mathcal{O}^\otimes) \ar[d] \\
\mathcal{P}_{n-1}\mathcal{C} \ar[r] & \mathcal{P}_{n-1} \mathrm{coAlg}^{\mathrm{ind}}(\tau_n\mathcal{O}^\otimes).
}
\]
The top horizontal arrow is the formation of coalgebras discussed previously, the bottom arrow obtained from the top by applying $\mathcal{P}_{n-1}$. Write $\mathcal{D} := \beta_0(\alpha(\mathcal{C}))$. By the definition of $\beta_0$, the square above induces a functor $F: \mathcal{C} \rightarrow \mathcal{D}$. Since $F$ preserves colimits it admits a right adjoint $G$. Considering our definitions it should be clear that $G$ provides a 2-simplex of equivalences as follows:
\[
\xymatrix{
& \tau_n\mathcal{O}^\otimes \ar[dl]_{\varphi_{\mathcal{D}}} \ar[dr]^{\varphi_{\mathcal{C}}} & \\
\mathrm{Sp}(\mathcal{D})^\otimes \ar[rr] & & \mathrm{Sp}(\mathcal{C})^\otimes.
} 
\]
This construction provides a homotopy between the identity map of $\mathcal{G}_n(\mathcal{O}^\otimes)$ and the composition $\beta \circ \alpha$ by virtue of the following:

\begin{proposition}
\label{prop:CnisbaCn}
The functor $F: \mathcal{C} \rightarrow \mathcal{D}$ above is an equivalence of $\infty$-categories.
\end{proposition}
\begin{proof}
Since $\mathcal{C}$ is an $n$-excisive $\infty$-category, it suffices to show that the adjunction $(F,G)$ is a weak $n$-excisive approximation. In other words, we should show that the natural transformations $\mathrm{id}_{\mathcal{C}} \rightarrow GF$ and $P_n(FG) \rightarrow \mathrm{id}_{\mathcal{Z}}$ are equivalences. For simplicity of notation we identify the stabilizations of $\mathcal{C}$ and $\mathcal{D}$ with $\mathcal{O}$ (although to be precise one would have to carry along the equivalences $\varphi_{\mathcal{C}}$ and $\varphi_{\mathcal{D}}$). These $\infty$-categories induce comonads $\Sigma^\infty_{\mathcal{C}}\Omega^\infty_{\mathcal{C}}$ and $\Sigma^\infty_{\mathcal{D}}\Omega^\infty_{\mathcal{D}}$ on $\mathcal{O}$ respectively. Furthermore, by construction we have equivalences
\begin{equation*}
\Sigma^\infty_{\mathcal{D}} F \simeq \Sigma^\infty_{\mathcal{C}} \quad\quad \text{and} \quad\quad G \Omega^\infty_{\mathcal{D}} \simeq \Omega^\infty_{\mathcal{C}}.
\end{equation*}
Observe also that the natural transformation
\begin{equation*}
P_k(\Sigma^\infty_{\mathcal{D}}FG\Omega^\infty_{\mathcal{D}}) \rightarrow P_k(\Sigma^\infty_{\mathcal{D}}\Omega^\infty_{\mathcal{D}})
\end{equation*}
is an equivalence for every $k$. Indeed, this is a direct consequence of the fact that $D_k(\Sigma^\infty_{\mathcal{C}}\Omega^\infty_{\mathcal{C}}) \rightarrow D_k(\Sigma^\infty_{\mathcal{D}}\Omega^\infty_{\mathcal{D}})$ is an equivalence for every $k$, which is a reformulation of the equivalence of $\infty$-operads $\mathrm{Sp}(\mathcal{D})^\otimes \rightarrow \mathrm{Sp}(\mathcal{C})^\otimes$. Using unit and counit of the stabilization adjunctions of $\mathcal{C}$ and $\mathcal{D}$ we may form the cosimplicial objects featured in the following diagram:
\[
\xymatrix{
\mathrm{id}_{\mathcal{C}} \ar[d]\ar[r] & \mathrm{Tot}\Bigl(P_n\bigl(\Omega^\infty_{\mathcal{C}} (\Sigma^\infty_{\mathcal{C}}\Omega^\infty_{\mathcal{C}})^{\bullet} \Sigma^\infty_{\mathcal{C}} \bigr)\Bigr) \ar[d] \\
GF \ar[r] &  \mathrm{Tot}\Bigl(P_n\bigl(G \Omega^\infty_{\mathcal{D}} (\Sigma^\infty_{\mathcal{D}}\Omega^\infty_{\mathcal{D}})^{\bullet} \Sigma^\infty_{\mathcal{D}} F\bigr)\Bigr). 
}
\]
The vertical arrow on the right is an equivalence by our previous remark. Moreover, the horizontal arrows are equivalences by a result of Arone and Ching (see Proposition \ref{prop:cobarderivatives} in the appendix). It follows that the left vertical arrow is an equivalence. We treat the map $P_n(FG) \rightarrow \mathrm{id}_{\mathcal{D}}$ similarly: indeed, by the same proposition we have equivalences
\begin{eqnarray*}
F & \longrightarrow & \mathrm{Tot}\Bigl(P_n\bigl(\Omega^\infty_{\mathcal{D}} (\Sigma^\infty_{\mathcal{D}}\Omega^\infty_{\mathcal{D}})^\bullet \Sigma^\infty_{\mathcal{D}} F \bigr)\Bigr), \\
G & \longrightarrow & \mathrm{Tot}\Bigl(P_n\bigl(G \Omega^\infty_{\mathcal{D}} (\Sigma^\infty_{\mathcal{D}}\Omega^\infty_{\mathcal{D}})^\bullet \Sigma^\infty_{\mathcal{D}} \bigr)\Bigr).
\end{eqnarray*}
Together these yield a diagram
\[
\xymatrix{
P_n(FG) \ar[d]\ar[r] &  \mathrm{Tot}\Bigl(P_n\bigl(\Omega^\infty_{\mathcal{D}} (\Sigma^\infty_{\mathcal{D}}\Omega^\infty_{\mathcal{D}})^\bullet \Sigma^\infty_{\mathcal{D}} F  G \Omega^\infty_{\mathcal{D}} (\Sigma^\infty_{\mathcal{D}}\Omega^\infty_{\mathcal{D}})^\bullet \Sigma^\infty_{\mathcal{D}}\bigr)\Bigr) \ar[d] \\
\mathrm{id}_{\mathcal{D}} \ar[r] & \mathrm{Tot}\Bigl(P_n\bigl(\Omega^\infty_{\mathcal{D}} (\Sigma^\infty_{\mathcal{D}}\Omega^\infty_{\mathcal{D}})^\bullet \Sigma^\infty_{\mathcal{D}} \Omega^\infty_{\mathcal{D}} (\Sigma^\infty_{\mathcal{D}}\Omega^\infty_{\mathcal{D}})^\bullet \Sigma^\infty_{\mathcal{D}}\bigr)\Bigr),
}
\]
where the cosimplicial object on the top right is the diagonal of the evident bicosimplicial object formed from the resolutions of $F$ and $G$ described above (and similarly for the bottom right, applying the standard resolution of the identity twice). One deduces that the horizontal maps are equivalences by applying Proposition \ref{prop:cobarderivatives} twice. The vertical map on the right is an equivalence again because the natural transformation $P_n(\Sigma^\infty_{\mathcal{D}}FG\Omega^\infty_{\mathcal{D}}) \rightarrow P_n(\Sigma^\infty_{\mathcal{D}}\Omega^\infty_{\mathcal{D}})$ is an equivalence. Finally, we conclude that the left vertical map $P_n(FG) \rightarrow \mathrm{id}_{\mathcal{D}}$ in the square above is an equivalence.
\end{proof}

It remains to deal with the composition $\alpha \circ \beta$. Consider a vertex $\mathcal{Z} \in \mathcal{A}$, which in particular determines an $(n-1)$-excisive $\infty$-category $\mathcal{C}$ and a natural transformation $N: \Sigma^\infty_{\mathcal{C}} \rightarrow \Theta_{\mathcal{C}}$. The map $T_n: \mathcal{G}_n(\mathcal{O}^\otimes) \rightarrow \mathcal{T}_n$ assigns to $\beta(\mathcal{Z})$ another such natural transformation $T_n(\beta(\mathcal{Z})): \Sigma^\infty_{\mathcal{C}} \rightarrow \Theta_{\mathcal{C}}$. A homotopy between $\alpha \circ \beta$ and the identity map of $\mathcal{A}$ is then provided by the following proposition, which is almost tautologous from what we have done so far:

\begin{proposition}
\label{prop:Tatebeta}
The natural transformations $N$ and $T_n(\beta(\mathcal{Z}))$ are canonically equivalent relative to $\Psi_{\mathcal{C}}$.
\end{proposition}
\begin{proof}
For the purposes of this proof let us write $\mathcal{D} := \beta_0(\mathcal{Z})$. Recall that this $\infty$-category is defined by a pullback square
\[
\xymatrix{
\mathcal{D} \ar[r]^-{A_{\mathcal{D}}} \ar[d]& \mathrm{coAlg}^{\mathrm{ind}}(\tau_n\mathcal{O}^\otimes) \ar[d] \\
\mathcal{C} \ar[r] & \mathcal{P}_{n-1}\mathrm{coAlg}^{\mathrm{ind}}(\tau_n\mathcal{O}^\otimes).
}
\]
The vertical arrows become equivalences after applying $\mathcal{P}_{n-1}$, which yields an equivalence between the bottom horizontal arrow and $\mathcal{P}_{n-1}A_{\mathcal{D}}$. Recall that our definition of $T_n(\beta(\mathcal{Z}))$ arose from applying $\mathcal{P}_{n-1}$ to the construction of coalgebras of Section \ref{subsec:coalgebras}, for which we write
\begin{equation*}
B_{\mathcal{D}}: \mathcal{D} \longrightarrow \mathrm{coAlg}^{\mathrm{ind}}(\tau_n\mathcal{O}^\otimes). 
\end{equation*}
Both $A_{\mathcal{D}}$ and $B_{\mathcal{D}}$ supply for each $X \in \mathcal{D}$ an $n$-fold `diagonal'
\begin{equation*}
\Sigma^\infty_{\mathcal{D}} X \longrightarrow (\Sigma^\infty_{\mathcal{D}} X \otimes^n \cdots \otimes^n \Sigma^\infty_{\mathcal{D}} X)^{h\Sigma_n}.
\end{equation*}
Write $\overline{X}$ for $\Sigma^\infty_{\mathcal{D}} X$. To prove the proposition it suffices to show that these maps are canonically equivalent relative to $(\overline{X}^{\odot n})^{h\Sigma_n}$. This is mostly an unraveling of definitions. One way to describe the $n$-fold diagonal induced by $B_{\mathcal{D}}$ is by using the coalgebra structure $\overline{X} \rightarrow \Sigma^\infty_{\mathcal{D}}\Omega^\infty_{\mathcal{D}}\overline{X}$ given by the unit of the adjunction $(\Sigma^\infty_{\mathcal{D}},\Omega^\infty_{\mathcal{D}})$ and composing with the map $\Sigma^\infty_{\mathcal{D}}\Omega^\infty_{\mathcal{D}}\overline{X} \rightarrow (\overline{X}^{\otimes n})^{h\Sigma_n}$ coming from the formation of coderivatives (see Lemma \ref{lem:codersigmaomega}). In case of $A_{\mathcal{D}}$, this $n$-fold diagonal arises by using the same structure $\overline{X} \rightarrow \Sigma^\infty_{\mathcal{D}}\Omega^\infty_{\mathcal{D}}\overline{X}$, but then considering the natural transformation
\begin{equation*}
\Sigma^\infty_{\mathcal{D}}\Omega^\infty_{\mathcal{D}} \longrightarrow \mathrm{cofree}_n
\end{equation*}
induced by $A_{\mathcal{D}}$ and composing with the projection
\begin{equation*}
\mathrm{cofree}_n(\overline{X}) \longrightarrow (\overline{X} \otimes^n \cdots \otimes^n \overline{X})^{h\Sigma_n}.
\end{equation*}
In the proof of Proposition \ref{prop:betastab} we showed that this composite induces an equivalences on $n$th derivatives and hence also on $n$th coderivatives (see Remark \ref{rmk:derivcoderiv}), which implies what we need. The compatibility with maps down to $(\overline{X}^{\odot n})^{h\Sigma_n}$ is immediate from the commutativity of the square defining $\mathcal{D}$.
\end{proof}

\section{The case of vanishing Tate constructions}
\label{subsec:vanishingTate}

Let $\mathcal{O}^\otimes$ be a nonunital stable $\infty$-operad. As in Corollary \ref{cor:notatecohomology}, assume that for $k \leq n$ and any object $X \in \mathcal{O}$ equipped with an action of $\Sigma_k$ the object $X^{t\Sigma_k}$ is trivial, i.e. is a zero object of $\mathcal{O}$. Then Theorem \ref{thm:classification} implies that the space $\mathcal{G}_k(\mathcal{O}^\otimes)$ is contractible for every $k \leq n$. In particular, all $n$-stages for $\mathcal{O}^\otimes$ are canonically equivalent, up to contractible ambiguity. The following result gives an explicit description of these $n$-stages:

\begin{proposition}
\label{prop:nstagenotate}
Assume $\mathcal{O}^\otimes$ as above. Write $\mathcal{C}_n$ for the $n$-excisive $\infty$-category $\mathrm{coAlg}^{\mathrm{ind}}(\tau_n\mathcal{O}^\otimes)$. Then there is an equivalence of $\infty$-operads
\begin{equation*}
\tau_n\mathcal{O}^\otimes \longrightarrow \mathrm{Sp}(\mathcal{C}_n)^\otimes
\end{equation*}
exhibiting $\mathcal{C}_n$ as an $n$-stage for $\mathcal{O}^\otimes$.
\end{proposition}
\begin{proof}
The proof is more or less direct from our construction of $n$-stages as in Section \ref{subsec:nstages}. The case $n=1$ is trivial, since $\mathrm{coAlg}^{\mathrm{ind}}(\tau_1\mathcal{O}^\otimes)$ is just the $\infty$-category $\mathcal{O}$ itself. Suppose we have proved the claim for $n-1$ and we wish to establish it for $n$. Let $\mathcal{C}_{n-1}$ be an $(n-1)$-stage for $\mathcal{O}^\otimes$, so that by our inductive hypothesis there is an equivalence
\begin{equation*}
\mathcal{C}_{n-1} \longrightarrow \mathrm{coAlg}^{\mathrm{ind}}(\tau_{n-1}\mathcal{O}^\otimes).
\end{equation*}
Note that by Corollary \ref{cor:Pnindcoalg} and our assumption on vanishing Tate constructions, there is an equivalence
\begin{equation*}
\mathrm{coAlg}^{\mathrm{ind}}(\tau_{n-1}\mathcal{O}^\otimes) \longrightarrow \mathcal{P}_{n-1}\mathrm{coAlg}^{\mathrm{ind}}(\tau_n\mathcal{O}^\otimes).
\end{equation*}
The constructions and results of Section \ref{subsec:nstages} then show that we may form an $n$-stage $\mathcal{C}$ for $\mathcal{O}^\otimes$ by forming the following pullback square:
\[
\xymatrix{
\mathcal{C} \ar[d]\ar[r] & \mathrm{coAlg}^{\mathrm{ind}}(\tau_n\mathcal{O}^\otimes) \ar[d] \\
\mathcal{C}_{n-1} \ar[r] & \mathcal{P}_{n-1}\mathrm{coAlg}^{\mathrm{ind}}(\tau_n\mathcal{O}^\otimes).
}
\]
Since the bottom horizontal arrow is an equivalence, so is the top arrow. In particular, $\mathrm{coAlg}^{\mathrm{ind}}(\tau_n\mathcal{O}^\otimes)$ can be made into an $n$-stage for $\mathcal{O}^\otimes$.
\end{proof}

In the case of vanishing Tate constructions we may also use Lemma \ref{lem:compactTatecoalgebras} and Corollary \ref{cor:compactTatecoalgebras} to conclude:

\begin{corollary}
Let $\mathcal{O}^\otimes$ be as above, i.e., all Tate constructions for $\Sigma_k$ vanish for $k \leq n$. Then $X \in \mathrm{coAlg}(\tau_n\mathcal{O}^\otimes)$ is a compact object whenever its underlying object of $\mathcal{O}$ is compact. As a consequence, the evident functor
\begin{equation*}
\mathrm{coAlg}^{\mathrm{ind}}(\tau_n\mathcal{O}^\otimes) \rightarrow \mathrm{coAlg}(\tau_n\mathcal{O}^\otimes)
\end{equation*}
is fully faithful.
\end{corollary}

\chapter{Examples}
\label{sec:examples}

In this chapter we apply our techniques to analyze the Goodwillie towers of several more or less familiar homotopy theories. In Section \ref{subsec:dividedpowers} we start with the notion of \emph{divided power coalgebras}, which in a sense give the simplest (or `untwisted') example of $n$-stages associated to a nonunital stable $\infty$-operad. They arise from Theorem \ref{thm:classification} by choosing the Tate diagonals to be null for every $n$. In case $\mathbf{O}$ is an operad in the category of spectra (or some related stable homotopy theory), we consider the Goodwillie tower of the homotopy theory $\mathrm{Alg}(\mathbf{O})$ of algebras over $\mathbf{O}$ and show that it coincides with the Goodwillie tower of divided power coalgebras over the cooperad $\mathbf{B}(\mathbf{O})$, the bar construction of $\mathbf{O}$. This is an instance of Koszul duality. 

In Section \ref{subsec:rationalhomotopy} we consider rational homotopy theory. In the homotopy theory of rational spectra all Tate cohomology vanishes, so that the classification of Theorem \ref{thm:classification} degenerates completely. We use this fact to identify the Goodwillie tower of the homotopy theory of pointed rational spaces with that of differential graded coalgebras over the rational numbers and with the Goodwillie tower of rational differential graded Lie algebras. We reproduce some of Quillen's results from this. 

In Section \ref{subsec:truncatedspaces} we localize the $\infty$-category of spectra so that $(p-1)!$ becomes invertible. In this localized homotopy theory the Tate cohomology of $\Sigma_k$ vanishes for $k < p$. We will use this observation to prove Theorem \ref{thm:truncatedspaces}, giving a Lie algebra model for the $\infty$-category $\mathcal{S}_\ast^{[n,p(n-1)]}$ of pointed spaces whose nontrivial homotopy groups are in the range of dimensions $[n,p(n-1)]$.

In Section \ref{subsec:modulisuspension} we investigate the Goodwillie tower of the homotopy theory of pointed spaces and prove Theorem \ref{thm:Tatecoalgebras}, which describes the $\infty$-category of simply connected pointed spaces in terms of Tate coalgebras.

Finally, in Section \ref{subsec:goodwilliespaces}, we make some further observations on the Goodwillie tower of the homotopy theory of pointed spaces. The stable $\infty$-operad of interest in this case is $\mathrm{Sp}^{\otimes}$, the symmetric monoidal $\infty$-category of spectra with the smash product. We give an explicit description of the fibers of the maps $\mathcal{G}_n(\mathrm{Sp}^\otimes) \rightarrow \mathcal{G}_{n-1}(\mathrm{Sp}^\otimes)$ in terms of the Tate spectra of the derivatives of the identity on $\mathcal{S}_{\ast}$, i.e. the Spanier-Whitehead duals of the partition complexes. 

The aim of the current chapter is to illustrate the use of results and techniques developed in this paper; its style is slightly more informal and expository than the rest of this paper. A more sophisticated application of our techniques is to $v_n$-periodic unstable homotopy theory, where one obtains results similar in nature to those of rational homotopy theory. These are discussed in \cite{unstableperiodicity}.

\section{Divided power coalgebras and Koszul duality}
\label{subsec:dividedpowers}

Let $\mathcal{O}^\otimes$ be a nonunital stable $\infty$-operad. Informally speaking, a coalgebra $X$ has \emph{divided powers} if each of its diagonal maps $\delta_n$ is equipped with a factorization through the norm map as follows:
\[
\xymatrix{
& (X \otimes^n \cdots \otimes^n X)_{h\Sigma_n} \ar[d]^{\mathrm{Nm}} \ar[d] \\
X \ar[r]_-{\delta_n}\ar[ur] & (X \otimes^n \cdots \otimes^n X)^{h\Sigma_n}.
}
\]
Furthermore, these factorizations should be compatible for various $n$ with respect to the structure of $\mathcal{O}^\otimes$. More precisely, let us inductively define $\infty$-categories $\mathrm{coAlg}^{\mathrm{ind}}_{\mathrm{dp}}(\tau_n\mathcal{O}^\otimes)$ of divided power ind-coalgebras in the truncated $\infty$-operads $\tau_n\mathcal{O}^\otimes$. We start by simply setting
\begin{equation*}
\mathrm{coAlg}^{\mathrm{ind}}_{\mathrm{dp}}(\tau_1\mathcal{O}^\otimes) := \mathcal{O}.
\end{equation*}
Note that this is a 1-stage for $\mathcal{O}^\otimes$ in an evident way. Now suppose we have defined the $\infty$-category $\mathrm{coAlg}^{\mathrm{ind}}_{\mathrm{dp}}(\tau_{n-1}\mathcal{O}^\otimes)$, together with an equivalence
\begin{equation*}
\mathrm{Sp}\bigl(\mathrm{coAlg}^{\mathrm{ind}}_{\mathrm{dp}}(\tau_{n-1}\mathcal{O}^\otimes)\bigr)^\otimes \longrightarrow \tau_{n-1}\mathcal{O}^\otimes
\end{equation*}
exhibiting it as an $(n-1)$-stage for $\mathcal{O}^\otimes$. Write $\sigma_{n-1}$ for the composition of the functors
\begin{equation*}
\mathrm{coAlg}^{\mathrm{ind}}_{\mathrm{dp}}(\tau_{n-1}\mathcal{O}^\otimes) \longrightarrow \mathrm{coAlg}^{\mathrm{ind}}(\tau_{n-1}\mathcal{O}^\otimes) \longrightarrow \mathrm{Ind}\bigl\{X \rightarrow (X^{\odot n})^{t\Sigma_n}\bigr\}^c_{\mathcal{O}}.
\end{equation*}
Assume that we have a natural equivalence between $\sigma_{n-1}$ and the functor which assigns to every coalgebra $X$ simply the null map $X \rightarrow (X^{\odot n})^{t\Sigma_n}$. As usual, $\odot^n$ here denotes the $n$-fold tensor product on $\mathcal{O}$ determined by the stable $\infty$-operad $\tau_{n-1}\mathcal{O}^\otimes$. This nullhomotopy defines a functor
\begin{equation*}
\mathrm{coAlg}^{\mathrm{ind}}_{\mathrm{dp}}(\tau_{n-1}\mathcal{O}^\otimes) \longrightarrow \mathrm{Ind}\bigl\{X \rightarrow (X^{\odot n})_{h\Sigma_n}\bigr\}^c_{\mathcal{O}}
\end{equation*}
and we define the $\infty$-category of divided power coalgebras in $\tau_n\mathcal{O}^\otimes$ by the following pullback square of compactly generated $\infty$-categories:
\[
\xymatrix{
\mathrm{coAlg}^{\mathrm{ind}}_{\mathrm{dp}}(\tau_n\mathcal{O}^\otimes) \ar[r]\ar[d] & \mathrm{Ind}\bigl\{X \rightarrow (X^{\otimes n})_{h\Sigma_n}\bigr\}^c_{\mathcal{O}} \ar[d] \\
\mathrm{coAlg}^{\mathrm{ind}}_{\mathrm{dp}}(\tau_{n-1}\mathcal{O}^\otimes) \ar[r] & \mathrm{Ind}\bigl\{X \rightarrow (X^{\odot n})_{h\Sigma_n}\bigr\}^c_{\mathcal{O}}.
}
\]

That this inductive construction makes sense is guaranteed by the following:

\begin{proposition}
The $\infty$-category $\mathrm{coAlg}^{\mathrm{ind}}_{\mathrm{dp}}(\tau_n\mathcal{O}^\otimes)$ is $n$-excisive. Furthermore, the stable $\infty$-operad associated to it is canonically equivalent to $\tau_n\mathcal{O}^\otimes$, making $\mathrm{coAlg}^{\mathrm{ind}}_{\mathrm{dp}}(\tau_n\mathcal{O}^\otimes)$ an $n$-stage for $\mathcal{O}^\otimes$. The associated natural transformation $\sigma_n$ is canonically nullhomotopic in the sense described above. Finally, the following functor is an equivalence:
\begin{equation*}
\mathcal{P}_{n-1}\mathrm{coAlg}^{\mathrm{ind}}_{\mathrm{dp}}(\tau_n\mathcal{O}^\otimes) \longrightarrow \mathrm{coAlg}^{\mathrm{ind}}_{\mathrm{dp}}(\tau_{n-1}\mathcal{O}^\otimes).
\end{equation*}
\end{proposition}
\begin{proof}
The $\infty$-category $\mathrm{coAlg}^{\mathrm{ind}}_{\mathrm{dp}}(\tau_n\mathcal{O}^\otimes)$ is a pullback of $n$-excisive $\infty$-categories and therefore $n$-excisive itself. To see that it is naturally an $n$-stage for $\mathcal{O}^\otimes$, observe that we may rephrase the above construction as follows. Consider the functor
\begin{equation*}
\gamma: \mathrm{coAlg}^{\mathrm{ind}}_{\mathrm{dp}}(\tau_{n-1}\mathcal{O}^\otimes) \longrightarrow \mathrm{coAlg}^{\mathrm{ind}}(\tau_{n-1}\mathcal{O}^\otimes)
\end{equation*}
given by the formation of coalgebras. By Corollary \ref{cor:Pnindcoalg} we may use the given `nullhomotopy' of the functor $\sigma_{n-1}$ together with the functor 
\begin{equation*}
\mathrm{null}: \mathrm{coAlg}^{\mathrm{ind}}_{\mathrm{dp}}(\tau_{n-1}\mathcal{O}^\otimes) \longrightarrow \mathrm{Ind}\bigl\{X \rightarrow (X^{\otimes n})^{t\Sigma_n}\bigr\}^c_{\mathcal{O}}: X \longmapsto \bigl(X \xrightarrow{0} (X^{\otimes n})^{t\Sigma_n}\bigr)
\end{equation*}
to lift $\gamma$ to a functor
\begin{equation*}
\mathrm{coAlg}^{\mathrm{ind}}_{\mathrm{dp}}(\tau_{n-1}\mathcal{O}^\otimes) \longrightarrow \mathcal{P}_{n-1}\mathrm{coAlg}^{\mathrm{ind}}(\tau_n\mathcal{O}^\otimes).
\end{equation*}
It is then straightforward to check, using Lemma \ref{lem:coalgvsPn}, that $\mathrm{coAlg}^{\mathrm{ind}}_{\mathrm{dp}}(\tau_n\mathcal{O}^\otimes)$ as defined above fits in a pullback square
\[
\xymatrix{
\mathrm{coAlg}^{\mathrm{ind}}_{\mathrm{dp}}(\tau_n\mathcal{O}^\otimes) \ar[d]\ar[r] & \mathrm{coAlg}^{\mathrm{ind}}(\tau_n\mathcal{O}^\otimes) \ar[d] \\
\mathrm{coAlg}^{\mathrm{ind}}_{\mathrm{dp}}(\tau_{n-1}\mathcal{O}^\otimes) \ar[r] &  \mathcal{P}_{n-1}\mathrm{coAlg}^{\mathrm{ind}}(\tau_n\mathcal{O}^\otimes).
}
\]
By the results of Section \ref{subsec:nstages} this shows that $\mathrm{coAlg}^{\mathrm{ind}}_{\mathrm{dp}}(\tau_n\mathcal{O}^\otimes)$ provides an $n$-stage for $\mathcal{O}^\otimes$. Applying the functor $\mathcal{P}_{n-1}$ to this square also shows that $\mathcal{P}_{n-1}\mathrm{coAlg}^{\mathrm{ind}}_{\mathrm{dp}}(\tau_n\mathcal{O}^\otimes)$ is equivalent to $\mathrm{coAlg}^{\mathrm{ind}}_{\mathrm{dp}}(\tau_{n-1}\mathcal{O}^\otimes)$. The claim about $\sigma_n$ follows from the fact that the Goodwillie tower of the comonad on $\mathcal{O}$ associated to the adjunction
\[
\xymatrix{
\mathrm{coAlg}^{\mathrm{dp}}(\tau_n\mathcal{O}^\otimes) \ar@<.5ex>[r]^-{\Sigma^\infty} & \mathcal{O} \ar@<.5ex>[l]^-{\Omega^\infty}
}
\]
is canonically \emph{split}, meaning each stage of the tower is simply the direct sum of its homegeneous layers. Indeed, writing $\odot^{n+1}$ for the $(n+1)$-fold tensor product on $\mathcal{O}$ defined by $\tau_n\mathcal{O}^\otimes$, we have the usual pullback square
\[
\xymatrix{
P_{n+1}(\Sigma^\infty\Omega^\infty)(X) \ar[r]\ar[d] & (X^{\odot n+1})^{h\Sigma_{n+1}} \ar[d] \\
P_n(\Sigma^\infty\Omega^\infty)(X) \ar[r] & (X^{\odot n+1})^{t\Sigma_{n+1}}.
}
\]
The splitting of $\Sigma^\infty\Omega^\infty$ corresponds to a nullhomotopy of the map $P_{n}(\Sigma^\infty\Omega^\infty)(X) \rightarrow (X^{\odot n+1})^{t\Sigma_{n+1}}$, which in turn provides a nullhomotopy of $\sigma_n$.
\end{proof}

One can think of the $\infty$-categories $\mathrm{coAlg}^{\mathrm{ind}}_{\mathrm{dp}}(\tau_n\mathcal{O}^\otimes)$ as providing a compatible collection of basepoints for the spaces $\mathcal{G}_n(\mathcal{O}^\otimes)$. In particular, these spaces are non-empty for every $n$. As before there is a variant of the above constructions `without the ind', which defines $\infty$-categories of $n$-truncated divided power coalgebras fitting into pullback squares as follows:
\[
\xymatrix{
\mathrm{coAlg}_{\mathrm{dp}}(\tau_n\mathcal{O}^\otimes) \ar[r]\ar[d] & \bigl\{X \rightarrow (X^{\otimes n})_{h\Sigma_n}\bigr\}_{\mathcal{O}} \ar[d] \\
\mathrm{coAlg}_{\mathrm{dp}}(\tau_{n-1}\mathcal{O}^\otimes) \ar[r] & \bigl\{X \rightarrow (X^{\odot n})_{h\Sigma_n}\bigr\}_{\mathcal{O}}.
}
\]
Lemma \ref{lem:compactTatecoalgebras} and Corollary \ref{cor:compactTatecoalgebras} give the following:

\begin{corollary}
\label{cor:dpcompact}
An $n$-truncated divided power coalgebra $X \in \mathrm{coAlg}_{\mathrm{dp}}(\tau_n\mathcal{O}^\otimes)$ whose underlying object in $\mathcal{O}$ is compact is itself a compact object. Consequently, the evident functor
\begin{equation*}
\mathrm{coAlg}^{\mathrm{ind}}_{\mathrm{dp}}(\tau_n\mathcal{O}^\otimes) \rightarrow \mathrm{coAlg}_{\mathrm{dp}}(\tau_n\mathcal{O}^\otimes)
\end{equation*}
defined by left Kan extension from compact objects is fully faithful.
\end{corollary}

For concreteness, consider the $\infty$-category $\mathrm{Sp}$ of spectra, although the following discussion goes through for many closely related stable $\infty$-categories. Let $\mathbf{O}$ be a nonunital operad in spectra whose unary term $\mathbf{O}(1)$ is the sphere spectrum $S$. We write $\mathrm{Alg}(\mathbf{O})$ for the $\infty$-category of $\mathbf{O}$-algebras. The following is by now well-known and can be found in various places in the literature (e.g. \cite{basterramandell} or Theorem 7.3.4.7 of \cite{higheralgebra} for the first part and \cite{francisgaitsgory} for the second):

\begin{proposition}
\label{prop:stabOalgs}
There is an identification $\mathrm{Sp}(\mathrm{Alg}(\mathbf{O})) \simeq \mathrm{Sp}$. Furthermore, the functor $\Sigma^\infty_{\mathbf{O}}\Omega^\infty_{\mathbf{O}}$ on $\mathrm{Sp}$ induced by the stabilization adjunction of $\mathrm{Alg}(\mathbf{O})$ is equivalent to the functor
\begin{equation*}
X \mapsto \coprod_{n=1}^\infty \bigl(\mathbf{B}(\mathbf{O})(n) \wedge X^{\wedge n} \bigr)_{\Sigma_n},
\end{equation*}
where $\mathbf{B}(\mathbf{O})$ is the bar construction of $\mathbf{O}$.
\end{proposition}

\begin{remark}
The stabilization functor $\Sigma^\infty_{\mathbf{O}}: \mathrm{Alg}(\mathbf{O}) \rightarrow \mathrm{Sp}$ can be identified with topological Andr\'{e}-Quillen homology by a result of Basterra and Mandell \cite{basterramandell}. Loosely speaking, the result above says that for any $\mathbf{O}$-algebra $X$ the spectrum $\mathrm{TAQ}(X)$ is canonically a (conilpotent) divided power coalgebra over the cooperad $\mathbf{B}(\mathbf{O})$. 
\end{remark}

Recall from Remark \ref{rmk:dictionary} the dictionary between cooperads $\mathbf{C}$ in spectra (with $\mathbf{C}(1) = S$) and stable $\infty$-operads $\mathcal{O}^\otimes$ whose underlying $\infty$-category $\mathcal{O}$ is $\mathrm{Sp}$. Write $\mathrm{Sp}_{\mathbf{B}(\mathbf{O})}^\otimes$ for the stable $\infty$-operad corresponding to $\mathbf{B}(\mathbf{O})$ under this dictionary. The previous result can then interpreted as stating an equivalence of $\infty$-operads
\begin{equation*}
\mathrm{Sp}(\mathrm{Alg}(\mathbf{O}))^\otimes \longrightarrow \mathrm{Sp}_{\mathbf{B}(\mathbf{O})}^\otimes.
\end{equation*}
Note that it follows from Proposition \ref{prop:stabOalgs} that all Tate diagonals associated to the $\infty$-category $\mathrm{Alg}(\mathbf{O})$ are null. Thus we can immediately describe the $n$-excisive approximations of $\mathrm{Alg}(\mathbf{O})$. Let us abbreviate $\mathrm{coAlg}_{\mathrm{dp}}(\mathrm{Sp}_{\mathbf{B}(\mathbf{O})}^\otimes)$ by $\mathrm{coAlg}_{\mathrm{dp}}(\mathbf{B}(\mathbf{O}))$. As before, write $\tau_n\mathbf{O}$ for the operad obtained from $\mathbf{O}$ by killing all operations with more than $n$ inputs. 

\begin{proposition}
\label{prop:PnAlgO}
For each $n$ there is an equivalence of $\infty$-categories
\begin{equation*}
\mathcal{P}_n\mathrm{Alg}(\mathbf{O}) \longrightarrow \mathrm{coAlg}^{\mathrm{ind}}_{\mathrm{dp}}(\mathbf{B}(\tau_n\mathbf{O})).
\end{equation*}
\end{proposition}

The interesting issue here is the convergence of the Goodwillie tower
\begin{equation*}
\mathrm{Alg}(\mathbf{O}) \longrightarrow \varprojlim \mathcal{P}_n\mathrm{Alg}(\mathbf{O}).
\end{equation*}
The map of operads $\mathbf{O} \rightarrow \tau_n\mathbf{O}$ induces a transfer adjunction
\[
\xymatrix{
\mathrm{Alg}(\mathbf{O}) \ar@<.5ex>[r] & \mathrm{Alg}(\tau_n\mathbf{O}). \ar@<.5ex>[l]
}
\]
Here the right adjoint is the evident pullback functor. For an $\mathbf{O}$-algebra $X$, write $X_{\leq n}$ for the $\mathbf{O}$-algebra obtained by successively applying left and right adjoint of this transfer. The following has been proved by Pereira \cite{pereirathesis}:

\begin{proposition}
The tower
\begin{equation*}
X \longrightarrow \bigl(\cdots \rightarrow X_{\leq 3} \rightarrow X_{\leq 2} \rightarrow X_{\leq 1}\bigl)
\end{equation*}
can be identified with the Goodwillie tower of the identity functor of $\mathrm{Alg}(\mathbf{O})$.
\end{proposition}

We call an algebra $X$ \emph{complete} if the map 
\begin{equation*}
X \longrightarrow \varprojlim\bigl(\cdots \rightarrow X_{\leq 3} \rightarrow X_{\leq 2} \rightarrow X_{\leq 1}\bigl)
\end{equation*}
is an equivalence, so that $\mathrm{Alg}(\mathbf{O})^{\mathrm{conv}}$ is precisely the $\infty$-category of complete $\mathbf{O}$-algebras. From Lemma \ref{lem:convergence} we then conclude the following:

\begin{corollary}
\label{cor:completealgebras}
The functor
\begin{equation*}
\mathrm{Alg}(\mathbf{O})^{\mathrm{conv}} \longrightarrow \varprojlim_n \mathrm{coAlg}_{\mathrm{dp}}(\mathbf{B}(\tau_n\mathbf{O}))
\end{equation*}
is fully faithful.
\end{corollary}

\begin{remark}
One can show that the transfer adjunction $\mathrm{Alg}(\mathbf{O}) \leftrightarrows \mathrm{Alg}(\tau_n\mathbf{O})$ is a weak $n$-excisive approximation; indeed, this was already observed (with different terminology) in \cite{pereirathesis}. It might be tempting to conjecture that it is also a strong $n$-excisive approximation, but this is false. The $\infty$-category $\mathrm{Alg}(\tau_n\mathbf{O})$ does generally \emph{not} satisfy condition (b) of Corollary \ref{cor:bnexcisive} and is thus not $n$-excisive.
\end{remark}

Proposition \ref{prop:PnAlgO} and Corollary \ref{cor:completealgebras} are closely related to a conjecture of Francis and Gaitsgory \cite{francisgaitsgory}. This conjecture states that there should be an equivalence between the $\infty$-category of \emph{pro-nilpotent} $\mathbf{O}$-algebras and \emph{conilpotent} divided power coalgebras over the cooperad $\mathbf{B}(\mathbf{O})$. Here an $\mathbf{O}$-algebra is \emph{nilpotent} if it is in the essential image of the pullback functor
\begin{equation*}
\mathrm{Alg}(\tau_n\mathbf{O}) \longrightarrow \mathrm{Alg}(\mathbf{O})
\end{equation*}
for some $n$ and \emph{pro-nilpotent} if it is a limit of nilpotent algebras. Equivalently, an $\mathbf{O}$-algebra is pro-nilpotent if it is a limit of trivial $\mathbf{O}$-algebras, meaning algebras in the essential image of the functor
\begin{equation*}
\Omega^\infty_{\mathbf{O}}: \mathrm{Sp} \longrightarrow \mathrm{Alg}(\mathbf{O}).
\end{equation*}
Dually, a divided power $\mathbf{B}(\mathbf{O})$-coalgebra is \emph{conilpotent} if it is a colimit of trivial coalgebras. Now let us consider the special case where the operad $\mathbf{O}$ is truncated, meaning $\mathbf{O} = \tau_n\mathbf{O}$ for large enough $n$. Then every $\mathbf{O}$-algebra $X$ is nilpotent and we obtain a special case of the Francis-Gaitsgory conjecture by the following simple consequence of Proposition \ref{prop:PnAlgO}. This result was observed independently by Lee Cohn: 

\begin{proposition}
\label{prop:ntruncatedbar}
If $\mathbf{O}$ is $n$-truncated, then composing the $n$-excisive approximation $\mathrm{Alg}(\mathbf{O}) \rightarrow \mathcal{P}_n\mathrm{Alg}(\mathbf{O})$ with the equivalence of Proposition \ref{prop:PnAlgO} gives a fully faithful functor
\begin{equation*}
\mathrm{bar}_{\mathbf{O}}: \mathrm{Alg}(\mathbf{O}) \rightarrow \mathrm{coAlg}_{\mathrm{dp}}(\mathbf{B}(\mathbf{O}))
\end{equation*}
with essential image the full subcategory spanned by the conilpotent coalgebras.
\end{proposition}
\begin{proof}
The functor $\mathrm{Alg}(\mathbf{O}) \rightarrow \mathcal{P}_n\mathrm{Alg}(\mathbf{O})$ is fully faithful, since the identity functor of $\mathrm{Alg}(\mathbf{O})$ is $n$-excisive. To verify the claim about the essential image in $\mathrm{coAlg}_{\mathrm{dp}}(\mathbf{B}(\mathbf{O}))$, write
\begin{equation*}
\mathrm{free}_{\mathbf{O}}: \mathrm{Sp} \longrightarrow \mathrm{Alg}(\mathbf{O}) \quad\quad \text{and} \quad\quad \mathrm{triv}_{\mathbf{B}(\mathbf{O})}: \mathrm{Sp} \longrightarrow \mathrm{coAlg}_{\mathrm{dp}}(\mathbf{B}(\mathbf{O}))
\end{equation*}
for the functors assigning to a spectrum $X$ the free $\mathbf{O}$-algebra and the trivial divided power $\mathbf{B}(\mathbf{O})$-coalgebra on $X$ respectively. It is well-known (e.g. \cite{francisgaitsgory}) that the composition $\mathrm{bar}_{\mathbf{O}} \circ \mathrm{free}_{\mathbf{O}}$ is equivalent to the functor $\mathrm{triv}_{\mathbf{B}(\mathbf{O})}$, so that the essential image of $\mathrm{bar}_{\mathbf{O}}$ contains all trivial coalgebras. Since $\mathrm{Alg}(\mathbf{O})$ is generated under colimits by free algebras and $\mathrm{bar}_{\mathbf{O}}$ preserves colimits, its essential image is therefore generated under colimits by trivial coalgebras and thus by definition the full subcategory of conilpotent coalgebras.
\end{proof}



\section{Rational homotopy theory}
\label{subsec:rationalhomotopy}

Consider the symmetric monoidal $\infty$-category $\mathrm{Sp}_{\mathbb{Q}}^\otimes$ of rational spectra with the smash product, obtained from $\mathrm{Sp}^{\otimes}$ by localizing with respect to the Eilenberg-MacLane spectrum $H\mathbb{Q}$. Equivalently, this $\infty$-category can be thought of as the homotopy theory of (unbounded) chain complexes over $\mathbb{Q}$ with their tensor product. It is well-known (and straightforward to prove) that for any finite group $G$ and rational spectrum $X$ with $G$-action, the Tate construction $X^{tG}$ is contractible. Essentially, division by the order of the group provides an inverse to the norm map. An immediate consequence is the following:

\begin{corollary}
\label{cor:GnQcontractible}
For each $n$, the space $\mathcal{G}_n(\mathrm{Sp}_{\mathbb{Q}}^\otimes)$ is contractible.
\end{corollary}

Now consider the $\infty$-category $\mathcal{S}_{\mathbb{Q}}^{\geq 2}$ of pointed, simply-connected rational spaces. By rational here we mean spaces local with respect to rational homology or more simply spaces whose homotopy groups are vector spaces over $\mathbb{Q}$. To apply our results we need the following:

\begin{lemma}
There is an equivalence of $\infty$-operads
\begin{equation*}
\mathrm{Sp}(\mathcal{S}_{\mathbb{Q}}^{\geq 2})^\otimes \longrightarrow \mathrm{Sp}_{\mathbb{Q}}^\otimes.
\end{equation*}
\end{lemma}
\begin{proof}
The essential fact is that the rationalization functor
\begin{equation*}
L_{\mathbb{Q}}: \mathcal{S}^{\geq 2} \longrightarrow \mathcal{S}^{\geq 2}_{\mathbb{Q}}
\end{equation*}
preserves colimits (this is formal) and finite limits (since $\mathbb{Q}$-localization is exact). Because the polynomial approximations of a functor are constructed using colimits and finite limits, it follows that the Goodwillie tower of the identity functor on $\mathcal{S}^{\geq 2}_{\mathbb{Q}}$ is simply obtained by applying $L_{\mathbb{Q}}$ to the Goodwillie tower of the identity on $\mathcal{S}^{\geq 2}$. From here it is an exercise in unraveling the definitions to see that in the commutative square
\[
\xymatrix{
(\mathcal{S}_{\mathbb{Q}}^{\geq 2})^{\times} \ar[r] & (\mathcal{S}^{\geq 2})^{\times} \\
\mathrm{Sp}_{\mathbb{Q}}^\otimes \ar[u]^{\Omega^\infty_{\mathbb{Q}}} \ar[r] & \mathrm{Sp}^{\otimes} \ar[u]
}
\]
the vertical arrow exhibits $\mathrm{Sp}_{\mathbb{Q}}^\otimes$ as the stabilization of the corepresentable $\infty$-operad $(\mathcal{S}_{\mathbb{Q}}^{\geq 2})^{\times}$. (The vertical arrow on the right is the simply-connected cover of $\Omega^{\infty}$; taking such covers does not affect the stabilization of $\mathcal{S}_\ast$.)
\end{proof}

We deduce from this lemma and Corollary \ref{cor:GnQcontractible} that $\mathcal{P}_n \mathcal{S}_{\mathbb{Q}}^{\geq 2}$ is canonically equivalent to any other $n$-stage for $\mathrm{Sp}_{\mathbb{Q}}^\otimes$. By Proposition \ref{prop:nstagenotate} we obtain the following:

\begin{corollary}
\label{cor:rationalGoodwillie}
The functor $\mathcal{S}_{\mathbb{Q}}^{\geq 2} \rightarrow \mathrm{coAlg}^{\mathrm{ind}}(\mathrm{Sp}_{\mathbb{Q}})$ induced by $\Sigma^\infty_{\mathbb{Q}}: \mathcal{S}_{\mathbb{Q}}^{\geq 2} \rightarrow \mathrm{Sp}_{\mathbb{Q}}$ gives an equivalence of Goodwillie towers, so in particular a collection of equivalences
\begin{equation*}
\mathcal{P}_n \mathcal{S}_{\mathbb{Q}}^{\geq 2} \longrightarrow \mathrm{coAlg}^{\mathrm{ind}}(\tau_n\mathrm{Sp}_{\mathbb{Q}}^\otimes).
\end{equation*} 
\end{corollary}

To get a sharper result we have to deal with convergence. First of all, it is well-known \cite{aronekankaanrinta, goodwillie3} that the Goodwillie tower of the identity on $\mathcal{S}_*$ converges on simply-connected (and even nilpotent) spaces. In general $L_\mathbb{Q}$ need not preserve inverse limits, but for simply-connected $X$ it \emph{does} preserve the limit $\varprojlim P_n\mathrm{id}(X)$. Indeed, the spectrum $\partial_n\mathrm{id}$ is $(1-n)$-connective, whereas $(\Sigma^\infty X)^{\wedge n}$ is $2n$-connective, so that the homogeneous layer $D_n\mathrm{id}(X)$ is $(n+1)$-connective. Hence for fixed $n$ the map
\begin{equation*}
X \rightarrow P_n\mathrm{id}(X)
\end{equation*}
induces an isomorphism on $\pi_k$ for $k \leq n$, which is then also true on rational homotopy groups. It follows that the functor
\begin{equation*}
C_*: \mathcal{S}_{\mathbb{Q}}^{\geq 2} \longrightarrow \varprojlim_n  \mathrm{coAlg}^{\mathrm{ind}}(\tau_n\mathrm{Sp}_{\mathbb{Q}})
\end{equation*}
is fully faithful. Moreover, since $C_*$ admits a right adjoint, it embeds $\mathcal{S}_{\mathbb{Q}}^{\geq 2}$ as a coreflective subcategory. To prove Theorem \ref{thm:rationalhomotopy} we simply identify this subcategory in different ways. Note that it is generated under colimits by the image of the rational two-sphere $L_{\mathbb{Q}}S^2$.

Let us say that a coalgebra $X \in \mathrm{coAlg}(\mathrm{Sp}^\otimes_{\mathbb{Q}})$ is \emph{simply-connected} if its underlying spectrum is simply-connected and write $\mathrm{coAlg}(\mathrm{Sp}^\otimes_{\mathbb{Q}})^{\geq 2}$ for the full subcategory spanned by such. Similarly define simply-connected ind-coalgebras and truncated coalgebras. To get one half of Theorem \ref{thm:rationalhomotopy} we prove the following:

\begin{proposition}
The evident functor
\begin{equation*}
\mathcal{S}_{\mathbb{Q}}^{\geq 2} \rightarrow  \mathrm{coAlg}(\mathrm{Sp}^\otimes_{\mathbb{Q}})^{\geq 2}
\end{equation*}
is an equivalence of $\infty$-categories.
\end{proposition}
\begin{proof}
We will abuse notation and denote the functor of the proposition by $C_*$ for the length of this proof (although we will show the abuse is rather mild). Recall (as in Corollaries \ref{cor:compactTatecoalgebras} and  \ref{cor:dpcompact}) that for every $n$ the functor
\begin{equation*}
\mathrm{coAlg}^{\mathrm{ind}}(\tau_n\mathrm{Sp}_{\mathbb{Q}}) \rightarrow \mathrm{coAlg}(\tau_n\mathrm{Sp}_{\mathbb{Q}})
\end{equation*}
is fully faithful. Hence the composite
\begin{equation*}
\mathcal{S}_{\mathbb{Q}}^{\geq 2} \rightarrow \varprojlim_n \mathrm{coAlg}^{\mathrm{ind}}(\tau_n\mathrm{Sp}_{\mathbb{Q}})^{\geq 2} \rightarrow  \varprojlim_n \mathrm{coAlg}(\tau_n\mathrm{Sp}_{\mathbb{Q}})^{\geq 2}
\end{equation*}
is fully faithful as well. But by Lemma \ref{lem:coalglimit} the limit on the right is equivalent to $\mathrm{coAlg}(\tau_n\mathrm{Sp}_{\mathbb{Q}})^{\geq 2}$ itself. We conclude that the functor $C_*$ of the proposition is fully faithful. To conclude that $C_*$ is an equivalence of $\infty$-categories it remains to show that it is essentially surjective. For this it suffices to check that $C_*(L_{\mathbb{Q}}S^2)$, which is the trivial coalgebra on a class in degree 2, generates the $\infty$-category $\mathrm{coAlg}(\mathrm{Sp}_{\mathbb{Q}})^{\leq 2}$ under colimits. This is rather standard; we leave the details to the interested reader. Essentially, one builds a `cellular approximation' (much as with spaces) to any coalgebra by using $C_*(L_{\mathbb{Q}}S^2)$ and its suspensions as the cells. Alternatively, one can use a much simplified version of the connectivity arguments we use in Section \ref{subsec:modulisuspension}.
\end{proof}

We now turn to the other half of Theorem \ref{thm:rationalhomotopy}, which involves Lie algebras. Write $\mathbf{L}_{\mathbb{Q}}$ for the shifted Lie operad in $\mathrm{Sp}_{\mathbb{Q}}$, defined in terms of the usual Lie operad $\mathbf{Lie}_{\mathbb{Q}}$ in rational vector spaces by $\mathbf{L}_{\mathbb{Q}}(k) = \Sigma^{1-k}\mathbf{Lie}_{\mathbb{Q}}(k)$, or rather
\begin{equation*}
\mathbf{L}_{\mathbb{Q}}(k) = \Sigma\mathbf{Lie}_{\mathbb{Q}}(k) \otimes (S_{\mathbb{Q}}^{-1})^{\otimes k}
\end{equation*}
to make more explicit the action of the symmetric group $\Sigma_k$ permuting the (-1)-spheres. It is well-known \cite{getzlerjones, ginzburgkapranov} that the bar construction of $\mathbf{L}_{\mathbb{Q}}$ is the commutative cooperad. It can be described by $\mathbf{Com}(k) = \mathbb{Q}$ for all $k$, with its evident cooperad structure. Under the dictionary of Remark \ref{rmk:dictionary} this cooperad corresponds to the stable $\infty$-operad $\mathrm{Sp}_{\mathbb{Q}}^\otimes$. Notice that shifting up in degree provides an equivalence of $\infty$-categories $\mathrm{Alg}(\mathbf{L}_{\mathbb{Q}}) \simeq \mathrm{Alg}(\mathbf{Lie}_{\mathbb{Q}})$ and we simply write $\mathrm{Lie}_{\mathbb{Q}}$ for the latter. Furthermore, write $\mathrm{Lie}_{\mathbb{Q}}^{\geq 1}$ for the full subcategory spanned by the connected Lie algebras. Recall that the Goodwillie tower of the identity on $\mathrm{Lie}_{\mathbb{Q}}$ converges precisely on complete Lie algebras; for evident degree reasons, this in particular includes the connected Lie algebras. The remaining half of Theorem \ref{thm:rationalhomotopy} follows from:

\begin{proposition}
The functor 
\begin{equation*}
\mathrm{Lie}_{\mathbb{Q}}^{\geq 1} \longrightarrow  \varprojlim_n  \mathrm{coAlg}^{\mathrm{ind}}(\tau_n\mathrm{Sp}_{\mathbb{Q}}^\otimes)
\end{equation*}
is a fully faithful embedding of a coreflective subcategory. Furthermore, its image coincides with that of $\mathcal{S}_{\mathbb{Q}}^{\geq 2}$ embedded via $C_*$.
\end{proposition}
\begin{proof}
The first statement is a corollary of the discussion about convergence above. To characterize the image, we should show that it is generated under colimits by $C_*(L_{\mathbb{Q}}S^2)$. This coalgebra is completely characterized by the fact that its homotopy is $\mathbb{Q}$ in degree 2 and trivial in all others. The $\infty$-category $\mathrm{Lie}_{\mathbb{Q}}^{\geq 1}$ is generated by $L[x_1]$, the free differential graded Lie algebra on a single generator in degree 1. (Since this is a graded Lie algebra, it is \emph{not} trivial but has a non-zero element $[x_1, x_1]$ in degree 2.) Recall that the infinite suspension functor $\mathrm{Lie}_{\mathbb{Q}}^{\geq 1} \rightarrow \mathrm{Sp}_{\mathbb{Q}}$ can be identified with Andr\'e-Quillen homology, which in this particular case is the usual Chevalley-Eilenberg homology of (dg) Lie algebras. The Chevalley-Eilenberg complex of $L[x_1]$ is the free graded commutative coalgebra on generators $y_2$ and $y_3$ in degrees 2 and 3 respectively, with differential determined by $d(y_2)^2 = y_3$. The homology of this complex is $\mathbb{Q}$ concentrated in degree 2. In particular, $\Sigma^\infty L[x_1]$ is equivalent to the object $C_*(L_{\mathbb{Q}}S^2)$ described above.
\end{proof}

\section{Spaces with homotopy groups in a finite range}
\label{subsec:truncatedspaces}

The aim of this section is to prove Theorem \ref{thm:truncatedspaces}. For the rest of this section, fix a prime $p$ and an integer $n \geq 2$. Write $\mathcal{S}_*^{[n,p(n-1)]}$ for the $\infty$-category of pointed spaces $X$ such that $\pi_i X = 0$ if $i$ is not contained in the interval $[n,p(n-1)]$. Throughout this section we will assume that this $\infty$-category is localized so that $(p-1)!$ has been inverted, without explicitly indicating this in the notation. Our statements will remain true if one instead localizes at the prime $p$. We write $\mathbf{L}$ for Ching's operad in spectra, whose terms are given by the derivatives of the identity functor on $\mathcal{S}_*$.

For $k < p$, the order of the symmetric group $\Sigma_k$ divides $(p-1)!$. In particular, its Tate cohomology will vanish in the $\infty$-category of spectra with $(p-1)!$ inverted and Corollary \ref{cor:notatecohomology2} gives the following:

\begin{corollary}
After inverting $(p-1)!$, there is a canonical equivalence of $\infty$-categories
\begin{equation*}
\mathcal{P}_{p-1}\mathcal{S}_* \longrightarrow \mathrm{coAlg}^{\mathrm{ind}}(\tau_{p-1}\mathrm{Sp}^{\otimes}).
\end{equation*}
\end{corollary}

Since the bar construction $\mathbf{B}(\mathbf{L})$ is precisely the cocommutative cooperad in spectra \cite{chingbar}, Proposition \ref{prop:PnAlgO} implies:

\begin{corollary}
After inverting $(p-1)!$, there is a fully faithful left adjoint functor
\begin{equation*}
\mathrm{bar}_{\tau_{p-1}\mathbf{L}}: \mathrm{Alg}(\tau_{p-1}\mathbf{L}) \longrightarrow \mathrm{coAlg}^{\mathrm{ind}}(\tau_{p-1}\mathrm{Sp}^\otimes).
\end{equation*}
\end{corollary}

As before, $\tau_{p-1}\mathbf{L}$ denotes the truncation of the operad $\mathbf{L}$ which kills all operations in degrees $p$ and higher. Also, we have used the fact that the functor $\mathrm{coAlg}^{\mathrm{ind}}_{\mathrm{dp}}(\tau_{p-1}\mathrm{Sp}^\otimes) \rightarrow \mathrm{coAlg}^{\mathrm{ind}}(\tau_{p-1}\mathrm{Sp}^\otimes)$ is an equivalence, again by the vanishing of Tate cohomology of $\Sigma_k$ for $k < p$.

\begin{lemma}
\label{lem:essimageL}
Write $\mathcal{S}_*^{\geq n}$ for the full subcategory of $\mathcal{S}_*$ spanned by the $n$-connective spaces. The essential image of the composition
\begin{equation*}
\mathcal{S}_*^{\geq n} \xrightarrow{\Sigma_{p-1}^\infty} \mathcal{P}_{p-1}\mathcal{S}_* \longrightarrow \mathrm{coAlg}^{\mathrm{ind}}(\tau_{p-1}\mathrm{Sp}^{\otimes})
\end{equation*}
is contained in the essential image of the functor $\mathrm{bar}_{\tau_{p-1}\mathbf{L}}$.
\end{lemma}
\begin{proof}
All functors in the statement of the lemma preserve colimits. Since $\mathcal{S}_*^{\geq n}$ is generated under colimits by the $n$-sphere $S^n$, it suffices to check that the coalgebra $\Sigma^\infty S^n$ is in the essential image of $\mathrm{bar}_{{\tau_{p-1}\mathbf{L}}}$. We claim that $\Sigma^\infty S^n$ is a trivial coalgebra, i.e. is in the essential image of the functor (see Section \ref{subsec:truncatedcoalgebras})
\begin{equation*}
\mathrm{triv}: \mathrm{Sp} \longrightarrow \mathrm{coAlg}^{\mathrm{ind}}(\tau_{p-1}\mathrm{Sp}^{\otimes}).
\end{equation*}
Granted this claim, the lemma follows from the fact that the composition of $\mathrm{bar}_{\tau_{p-1}\mathbf{L}}$ with the free algebra functor
\begin{equation*}
\mathrm{free}_{\tau_{p-1}\mathbf{L}}: \mathrm{Sp} \longrightarrow \mathrm{Alg}(\tau_{p-1}\mathbf{L})
\end{equation*}
is equivalent to the functor $\mathrm{triv}$ above.

We verify our claim inductively. Assume that we have proved that $\Sigma^\infty S^n$ is a trivial coalgebra in $\mathrm{coAlg}^{\mathrm{ind}}(\tau_k \mathrm{Sp}^\otimes)$ for some $k < p-1$, the case of $k = 1$ being trivial. Lemma \ref{lem:algtaun+1} (and the vanishing of the relevant Tate constructions) implies that lifting $\Sigma^\infty S^n$ from a coalgebra in $\tau_k \mathrm{Sp}^\otimes$ to a coalgebra in $\tau_{k+1}\mathrm{Sp}^\otimes$ is equivalent to specifying a lift as follows:
\[
\xymatrix{
& \bigl((\Sigma^\infty S^n)^{\otimes k+1}\bigr)_{h\Sigma_{k+1}} \ar[d] \\
\Sigma^\infty S^n \ar[r]_-0 \ar@{-->}[ur] & \bigl((\Sigma^\infty S^n)^{\odot k+1}\bigr)_{h\Sigma_{k+1}} \\
}
\]
Here $\odot^{k+1}$ is now the $(k+1)$-fold tensor product defined by $\tau_k \mathrm{Sp}^\otimes$. We claim that the fiber $F$ of the vertical map has connectivity greater than $n$; in particular, any map $\Sigma^\infty S^n \rightarrow F$ is null. It follows that the lift of $\Sigma^\infty S^n$ to $\tau_{k+1}\mathrm{Sp}^\otimes$ is uniquely determined and is the trivial one. This establishes the inductive step. It remains to establish the connectivity of $F$. First, the largest possible dimension of a non-degenerate simplex in the simplicial set $\N\mathbf{Part}_{k}(k+1)$ is $k-2$. By Proposition \ref{prop:truncatedtensor}, it follows that the connectivity of $(\Sigma^\infty S^n)^{\odot k+1}$ is at least $n(k+1) - (k-2) > n + 2$. Also, the connectivity of $(\Sigma^\infty S^n)^{\otimes k+1}$ is of course $n(k+1)$, in particular it is greater than or equal to $n+2$. The claim about $F$ follows from these two estimates.
\end{proof}

Lemma \ref{lem:essimageL} guarantees that the functor $\mathcal{S}_*^{\geq n} \rightarrow \mathrm{coAlg}^{\mathrm{ind}}(\tau_{p-1}\mathrm{Sp}^{\otimes})$ factors through a functor
\begin{equation*}
\mathcal{L}: \mathcal{S}_*^{\geq n} \longrightarrow \mathrm{Alg}(\tau_{p-1}\mathbf{L})
\end{equation*}
which is uniquely determined up to equivalence. Moreover, it follows from our proof that for $m \geq n$, the algebra $\mathcal{L}(S^m)$ is the free $\tau_{p-1}\mathbf{L}$-algebra on the spectrum $\Sigma^\infty S^m$. 

By construction, $\mathcal{L}$ admits a right adjoint $\mathcal{R}$ and the unit $\mathrm{id} \rightarrow \mathcal{R}\mathcal{L}$ is equivalent to the unit of the adjoint pair $(\Sigma_{p-1}^\infty, \Omega^\infty_{p-1})$, which in turn is the natural transformation $\mathrm{id} \rightarrow P_{p-1}\mathrm{id}$. For any pointed space $X \in \mathcal{S}_*^{\geq n}$, the map
\begin{equation*}
X \longrightarrow P_{p-1}\mathrm{id}(X)
\end{equation*}
induces an isomorphism on homotopy groups in dimensions up to $p(n-1)$. Indeed, this is a consequence of the following two facts. First, the map
\begin{equation*}
X \longrightarrow \varprojlim_k P_k\mathrm{id}(X)
\end{equation*}
is an equivalence (remember that under our assumptions $X$ is simply-connected). Second, the fiber $D_k\mathrm{id}(X)$ of the map $P_k\mathrm{id}(X) \rightarrow P_{k-1}\mathrm{id}(X)$ can be described as follows:
\begin{equation*}
D_k\mathrm{id}(X) \simeq \Omega^\infty\bigl((\partial_k\mathrm{id} \otimes X^{\otimes k})_{\Sigma_k}\bigr).
\end{equation*}
The spectrum $\partial_k \mathrm{id}$ is $(1-k)$-connective (in fact, it is equivalent to a wedge of $(1-k)$-dimensional spheres), so that the homotopy groups of $D_k\mathrm{id}(X)$ vanish up to dimension $k(n-1)$.

Write $\mathrm{Alg}(\tau_{p-1}\mathbf{L})^{[n,p(n-1)]}$ for the full subcategory of $\mathrm{Alg}(\tau_{p-1}\mathbf{L})$ spanned by the algebras whose underlying spectrum has nontrivial homotopy groups only in the range $[n,p(n-1)]$ and similarly define a full subcategory $\mathcal{S}_*^{[n,p(n-1)]}$ of $\mathcal{S}_*$. The main step in proving Theorem \ref{thm:truncatedspaces} is the following:
 
\begin{proposition}
\label{prop:truncatedspaces}
The functor $\mathcal{R}$ restricts to an equivalence of $\infty$-categories
\begin{equation*}
\mathcal{R}: \mathrm{Alg}(\tau_{p-1}\mathbf{L})^{[n,p(n-1)]} \longrightarrow \mathcal{S}_*^{[n,p(n-1)]}.
\end{equation*}
\end{proposition} 
\begin{proof}
Let $X \in \mathrm{Alg}(\tau_{p-1}\mathbf{L})^{[n,p(n-1)]}$. Let us first verify that the homotopy groups of the space $\mathcal{R}X$ are indeed in the range $[n,p(n-1)]$. For $m \geq n$, the space of maps $\mathrm{Map}_{\mathcal{S}_*}(S^m, \mathcal{R}X)$ is equivalent to the space of maps between $\mathrm{free}_{\tau_{p-1}\mathbf{L}}(\Sigma^\infty S^m)$ and $X$, by adjunction and the identification of $\mathcal{L}(S^m)$ made earlier. This latter space is equivalent to $\mathrm{Map}_{\mathrm{Sp}}(\Sigma^\infty S^m, UX)$, where $UX$ denotes the underlying spectrum of the algebra $X$. This identification shows that the homotopy groups of $\mathcal{R}X$ and $UX$ coincide.

Now write
\begin{equation*}
t_{p(n-1)}: \mathcal{S}_*^{\geq n} \longrightarrow \mathcal{S}_*^{[n,p(n-1)]} \quad\quad \text{and} \quad\quad T_{p(n-1)}: \mathrm{Sp}^{\geq n} \longrightarrow \mathrm{Sp}^{[n,p(n-1)]}
\end{equation*}
for the functors taking the Postnikov sections killing all homotopy groups in dimensions higher that $p(n-1)$. These are left adjoint to the inclusions
\begin{equation*}
\mathcal{S}_*^{[n,p(n-1)]} \longrightarrow \mathcal{S}_*^{\geq n} \quad\quad \text{and} \quad\quad \mathrm{Sp}^{[n,p(n-1)]} \longrightarrow \mathrm{Sp}^{\geq n}
\end{equation*}
respectively. Note that $T_{p(n-1)}$ is a monoidal functor, so that it induces a corresponding left adjoint
\begin{equation*}
T_{p(n-1)}: \mathrm{Alg}(\tau_{p-1}\mathbf{L})^{\geq n} \longrightarrow \mathrm{Alg}(\tau_{p-1}\mathbf{L})^{[n,p(n-1)]}.
\end{equation*}
The left adjoint of the functor
\begin{equation*}
\mathcal{R}: \mathrm{Alg}(\tau_{p-1}\mathbf{L})^{[n,p(n-1)]} \longrightarrow \mathcal{S}_*^{[n,p(n-1)]}
\end{equation*}
is the composition $T_{p(n-1)} \circ \mathcal{L}$. That this adjoint pair is an equivalence now follows from the facts (to be proved below) that (1) $\mathcal{R}$ detects equivalences and (2) the unit $\mathrm{id} \rightarrow \mathcal{R} T_{p(n-1)}\mathcal{L}$ is an equivalence. Indeed, to verify that the counit $T_{p(n-1)} \mathcal{L}\mathcal{R} \rightarrow \mathrm{id}$ is an equivalence as well, it suffices (by (1)) to check this after applying $\mathcal{R}$ to both sides, when it fits into a commutative triangle
\[
\xymatrix{
& \mathcal{R} T_{p(n-1)} \mathcal{L}\mathcal{R} \ar[dr] & \\ 
\mathcal{R} \ar[ur] \ar@{=}[rr] && \mathcal{R}.
}
\]
The left slanted arrow is an equivalence by (2), so that the right slanted arrow is an equivalence by two-out-of-three.

It remains to prove (1) and (2). The first is clear from our earlier identification of the homotopy groups of $\mathcal{R}X$ with the homotopy groups of the underlying spectrum of $X$. For (2), note that the natural transformation $\mathcal{R}\mathcal{L} \rightarrow \mathcal{R} T_{p(n-1)}\mathcal{L}$ factors through a natural transformation $t_{p(n-1)} \mathcal{R}\mathcal{L} \rightarrow \mathcal{R} T_{p(n-1)}\mathcal{L}$, which is an equivalence because it induces isomorphisms on homotopy groups. The natural transformation
\begin{equation*}
\mathrm{id} \longrightarrow t_{p(n-1)} \mathcal{R}\mathcal{L}
\end{equation*}
is an equivalence as well; indeed, we already analyzed the connectivity of the map $X \rightarrow \mathcal{R}\mathcal{L}X$ before Proposition \ref{prop:truncatedspaces} and concluded it induces isomorphisms on homotopy groups in dimensions up to $p(n-1)$.
\end{proof} 
 
\begin{proof}[Proof of Theorem \ref{thm:truncatedspaces}]
The morphism of operads $\mathbf{L} \rightarrow \tau_{p-1}\mathbf{L}$ induces an adjunction on $\infty$-categories of algebras:
\[
\xymatrix{
\mathrm{Alg}(\mathbf{L})^{[n,p(n-1)]} \ar@<.5ex>[r] & \mathrm{Alg}(\tau_{p-1}\mathbf{L})^{[n,p(n-1)]}. \ar@<.5ex>[l]
}
\]
All that remains to check is that this adjunction is an equivalence. This follows easily from the observation that for $X$ an $n$-connective spectrum, the term $(\partial_k\mathrm{id} \otimes X^{\otimes k})_{h\Sigma_k}$ has vanishing homotopy groups up to dimension $k(n-1)$. In particular, for $k \geq p$, the spectrum
\begin{equation*}
T_{p(n-1)}(\partial_k\mathrm{id} \otimes X^{\otimes k})_{h\Sigma_k}
\end{equation*}
is contractible.
\end{proof}

\section{Spaces and Tate coalgebras}
\label{subsec:modulisuspension} 

The aim of this section is to prove Theorem \ref{thm:Tatecoalgebras} from the introduction, which gives a description of the homotopy theory of pointed spaces in terms of Tate coalgebras. The interested reader should compare these results to some of the questions on the moduli of suspension spectra raised by Klein in \cite{kleinmoduli} (in particular Section 9 there); also, our results allow a proof of Conjecture B of \cite{kleinpeter} by implementing the plan sketched in Section 5 of that paper.

We gave an informal description of the $n$-excisive $\infty$-categories $\mathcal{P}_n\mathcal{S}_*$ in Chapter \ref{sec:informalconstr}, using coalgebras with Tate diagonals. A precise construction was given in Chapter \ref{sec:classification}. Let us take a moment to unravel this construction and make the connection with our more informal description. Proposition \ref{prop:CnisbaCn} shows that for every $n$ there is a pullback square of compactly generated $\infty$-categories
\[
\xymatrix{
\mathcal{P}_n\mathcal{S}_* \ar[r]\ar[d] & \mathrm{coAlg}^{\mathrm{ind}}(\tau_n\mathrm{Sp}^\otimes) \ar[d] \\
\mathcal{P}_{n-1}\mathcal{S}_* \ar[r] & \mathcal{P}_{n-1} \mathrm{coAlg}^{\mathrm{ind}}(\tau_n\mathrm{Sp}^\otimes).
}
\]
Here the upper right-hand corner is the $\infty$-category of ind-coalgebras in the $\infty$-operad $\tau_n\mathrm{Sp}^\otimes$; in more elementary terms such an ind-coalgebra is an ind-object in the $\infty$-category of finite spectra $X$ equipped with comultiplication maps
\begin{equation*}
\delta_k: X \rightarrow (X^{\otimes k})^{h\Sigma_k}
\end{equation*}
for $k \leq n$, together with a coherent system of homotopies relating them. Indeed, this description is precisely the content of Lemma \ref{lem:algtaun+1}. The objects of the lower right-hand corner can be loosely described as $(n-1)$-truncated ind-coalgebras (i.e. $X$ only having $\delta_k$ for $k \leq n-1$), equipped with a further map
\begin{equation*}
X \rightarrow (X^{\otimes n})^{t\Sigma_n}
\end{equation*}
compatible with those $\delta_k$'s (or with $\delta_{<n}$, in the notation of Chapter \ref{sec:informalconstr}). To give a more precise statement, observe that Lemma \ref{lem:coalgvsPn} allows one to replace the pullback square above by the following more elementary one:
\[
\xymatrix{
\mathcal{P}_n\mathcal{S}_* \ar[r]\ar[d] & \mathrm{Ind}\bigl\{X \rightarrow (X^{\otimes^n})^{h\Sigma_n} \bigr\}^c_{\mathrm{Sp}} \ar[d] \\
\mathcal{P}_{n-1}\mathcal{S}_* \ar[r] & \mathrm{Ind}\bigl\{X \rightarrow (X^{\odot^n})^{h\Sigma_n} \times_{(X^{\odot^n})^{t\Sigma_n}} (X^{\otimes^n})^{t\Sigma_n} \bigr\}^c_{\mathrm{Sp}}.
}
\]
In words, the bottom horizontal map assigns to a compact object of $\mathcal{P}_{n-1}\mathcal{S}_*$ its suspension spectrum $X$ together with the map
\begin{equation*}
(\delta_{<n}, \tau_n): X \rightarrow (X^{\odot^n})^{h\Sigma_n} \times_{(X^{\odot^n})^{t\Sigma_n}} (X^{\otimes^n})^{t\Sigma_n}.
\end{equation*}
The first factor encodes the comultiplicative structure map $\delta_{<n}$, the second factor the Tate diagonal $\tau_n$. Recall that $\tau_n$ is uniquely determined up to equivalence by the fact that on objects of $\mathcal{P}_{n-1}\mathcal{S}_*$ of the form $\Sigma^\infty_{n-1} Y$ its value is given by the composition
\begin{equation*}
\Sigma^\infty Y \xrightarrow{\delta_n} (\Sigma^\infty Y \otimes \cdots \otimes \Sigma^\infty Y)^{h\Sigma_n} \rightarrow (\Sigma^\infty Y \otimes \cdots \otimes \Sigma^\infty Y)^{t\Sigma_n}.
\end{equation*}
The pullback square above articulates that to lift such an object of $\mathcal{P}_{n-1}\mathcal{S}_*$ to an object of $\mathcal{P}_n\mathcal{S}_*$, one should provide a comultiplication map
\begin{equation*}
\delta_n: X \rightarrow (X^{\otimes^n})^{h\Sigma_n}
\end{equation*}
which is compatible with $\delta_{<n}$ and $\tau_n$. The interested reader should compare this with the discussion of Chapter \ref{sec:informalconstr}, in particular Examples \ref{ex:P2S} and \ref{ex:P3S}. To summarize, induction on $n$ shows that a compact object of $\mathcal{P}_n\mathcal{S}_*$ is described by comultiplication maps $\delta_k$ for $k \leq n$, homotopies expressing the compatibilities between them, as well as homotopies expressing their compatibilities with the Tate diagonals $\tau_k$ for $k \leq n$. This describes all of $\mathcal{P}_n\mathcal{S}_*$, since a general object is precisely an ind-object in the subcategory of compact objects.

We will write
\begin{equation*}
\mathrm{coAlg}_{\mathrm{Tate}}^{\mathrm{ind}}(\tau_n \mathrm{Sp}^\otimes) := \mathcal{P}_n\mathcal{S}_*
\end{equation*}
for the $\infty$-category of \emph{$n$-truncated Tate ind-coalgebras}. The notation is slightly abusive; this $\infty$-category depends not only on the stable $\infty$-operad $\tau_n \mathrm{Sp}^\otimes$, but also on the Tate diagonals $\tau_n$ arising from the $\infty$-category $\mathcal{S}_*$. Note that the pullback squares above allow one to give an inductive description of the $\infty$-categories of $n$-truncated Tate ind-coalgebras (starting from $\mathrm{coAlg}_{\mathrm{Tate}}^{\mathrm{ind}}(\tau_1 \mathrm{Sp}^\otimes) = \mathrm{Sp}$) which only uses the $\infty$-operad $\mathrm{Sp}^\otimes$ and the Tate diagonals as input.

As at the end of Section \ref{subsec:nstages} there is a variant of this construction `without the ind', giving $\infty$-categories of $n$-truncated Tate coalgebras defined inductively by pullback squares
\[
\xymatrix{
\mathrm{coAlg}_{\mathrm{Tate}}(\tau_n \mathrm{Sp}^\otimes) \ar[r]\ar[d] & \mathrm{coAlg}(\tau_n\mathrm{Sp}^\otimes) \ar[d] \\
\mathrm{coAlg}_{\mathrm{Tate}}(\tau_{n-1} \mathrm{Sp}^\otimes) \ar[r] & \mathcal{P}_{n-1} \mathrm{coAlg}(\tau_n\mathrm{Sp}^\otimes),
}
\]
again starting the induction at
\begin{equation*}
\mathrm{coAlg}_{\mathrm{Tate}}(\tau_1 \mathrm{Sp}^\otimes) = \mathrm{Sp}.
\end{equation*}
Corollary \ref{cor:compactTatecoalgebras} gives fully faithful functors
\begin{equation*}
\mathrm{coAlg}_{\mathrm{Tate}}^{\mathrm{ind}}(\tau_n \mathrm{Sp}^\otimes) \rightarrow \mathrm{coAlg}_{\mathrm{Tate}}(\tau_n \mathrm{Sp}^\otimes).
\end{equation*}

\begin{definition}
The $\infty$-category of Tate coalgebras in spectra is
\begin{equation*}
\mathrm{coAlg}_{\mathrm{Tate}}(\mathrm{Sp}^\otimes) := \varprojlim_n \mathrm{coAlg}_{\mathrm{Tate}}(\tau_n \mathrm{Sp}^\otimes).
\end{equation*}
\end{definition}

Observe that the result of Lemma \ref{lem:coalglimit} supplies a functor
\begin{equation*}
\mathrm{coAlg}_{\mathrm{Tate}}(\mathrm{Sp}^\otimes)  \rightarrow \mathrm{coAlg}(\mathrm{Sp}^\otimes) 
\end{equation*}
which simply `forgets the Tate diagonals' but retains the underlying coalgebra. It is precisely the extra data supplied by the Tate diagonals that is needed to upgrade a coalgebra in spectra to a coalgebra arising as a suspension spectrum. More precisely, Theorem \ref{thm:Tatecoalgebras} states that the comparison functor on simply-connected pointed spaces (which results from taking the limit of the Goodwillie tower of $\mathcal{S}_*$)
\begin{equation*}
\Gamma: \mathcal{S}_*^{\geq 2} \rightarrow \mathrm{coAlg}_{\mathrm{Tate}}(\mathrm{Sp}^\otimes)^{\geq 2}
\end{equation*}
is an equivalence of $\infty$-categories. We will prove this result below. The main point of the proof is the following. The functor above induces a natural map of comonads on the $\infty$-category $\mathrm{Sp}^{\geq 2}$ of simply-connected spectra as follows:
\begin{equation*}
\gamma: \Sigma^\infty\Omega^\infty \rightarrow \mathrm{cofree}^{\mathrm{Tate}}.
\end{equation*}
Here $\mathrm{cofree}^{\mathrm{Tate}}$ denotes the comonad induced by the forgetful-cofree adjunction on $\mathrm{Sp}^{\geq 2}$ arising from the $\infty$-category $\mathrm{coAlg}_{\mathrm{Tate}}(\mathrm{Sp}^\otimes)^{\geq 2}$. We will prove that $\gamma$ is an equivalence. If we knew that both $\mathcal{S}_*^{\geq 2}$ and $\mathrm{coAlg}_{\mathrm{Tate}}(\mathrm{Sp}^\otimes)^{\geq 2}$ were comonadic over spectra, this would immediately finish the proof. For $\mathcal{S}_*^{\geq 2}$ this is rather well-known and corresponds essentially to the convergence of the Bousfield-Kan spectral sequence on simply-connected spaces. For the $\infty$-category of Tate coalgebras this is not immediately clear; we will provide a proof below. 

First we establish some preliminary lemmas, which refine several statements in the proof of Proposition \ref{prop:betastab} by exploiting connectivity estimates. To fix notation, write
\begin{equation*}
U_n: \mathrm{coAlg}_{\mathrm{Tate}}(\tau_n\mathrm{Sp}^\otimes)^{\geq 2} \rightarrow \mathrm{Sp}^{\geq 2}
\end{equation*}
for the forgetful functor and $\mathrm{cofree}^{\mathrm{Tate}}_n$ for the comonad on $\mathrm{Sp}^{\geq 2}$ arising by composing $U_n$ with its right adjoint. The composition of left adjoints
\begin{equation*}
\mathcal{S}_*^{\geq 2} \xrightarrow{\Sigma^\infty_n} \mathrm{coAlg}_{\mathrm{Tate}}^{\mathrm{ind}}(\tau_n\mathrm{Sp}^\otimes)^{\geq 2} \rightarrow \mathrm{coAlg}_{\mathrm{Tate}}(\tau_n\mathrm{Sp}^\otimes)^{\geq 2}
\end{equation*}
induces (via the counits of the corresponding adjoint pairs) a natural transformation 
\begin{equation*}
\gamma_n: \Sigma^\infty\Omega^\infty \rightarrow \mathrm{cofree}^{\mathrm{Tate}}_n. 
\end{equation*}

\begin{lemma}
\label{lem:PncofreeTate}
The natural transformation
\begin{equation*}
P_n(\gamma_n): P_n(\Sigma^\infty\Omega^\infty) \rightarrow P_n(\mathrm{cofree}_n^{\mathrm{Tate}})
\end{equation*}
is an equivalence.
\end{lemma}
\begin{proof}
If we were dealing with the cofree n-truncated Tate \emph{ind}-coalgebra this would follow directly from our earlier results. Indeed, that cofree coalgebra functor arises from the adjunction
\[
\xymatrix{
\mathcal{P}_n\mathcal{S}_* \ar@<.5ex>[r]^-{\Sigma^\infty_{n,1}} & \mathrm{Sp} \ar@<.5ex>[l]^-{\Omega^\infty_{n,1}}
}
\]
and the equivalence $P_n(\Sigma^\infty_n\Omega^\infty_n) \simeq \mathrm{id}_{\mathcal{P}_n\mathcal{S}_*}$ would imply the result. For the case at hand (where the only difference is the absence of the ind), a straightforward modification of the technique used in the proof of Proposition \ref{prop:betastab} gives the desired result.
\end{proof}

\begin{lemma}
\label{lem:cofreeTateconverges}
The Goodwillie tower of $\mathrm{cofree}^{\mathrm{Tate}}_n$ converges on simply-connected spectra. More precisely, for $X \in \mathrm{Sp}^{\geq 2}$ the natural map
\begin{equation*}
\mathrm{cofree}^{\mathrm{Tate}}_n(X) \rightarrow \varprojlim_m P_m(\mathrm{cofree}^{\mathrm{Tate}}_n)(X)
\end{equation*}
is an equivalence (and in fact the connectivity of the map to the $m$th stage of the limit grows linearly with $m$).
\end{lemma}
\begin{proof}
We prove this by induction on $n$, the case $n=1$ being clear because $\mathrm{cofree}^{\mathrm{Tate}}_1 = \mathrm{id}_{\mathrm{Sp}}$. For $n >1$, observe that the defining pullback square of $\mathrm{coAlg}^{\mathrm{Tate}}(\tau_n\mathrm{Sp}^\otimes)$ gives a pullback square of functors
\[
\xymatrix{
\mathrm{cofree}^{\mathrm{Tate}}_n \ar[d]\ar[r] & \mathrm{cofree}_n \ar[d] \\
\mathrm{cofree}^{\mathrm{Tate}}_{n-1} \ar[r] & \mathrm{cofree}_n^t.
}
\]
Here (like in the proof of Proposition \ref{prop:betastab}) the functors $\mathrm{cofree}_n$ and $\mathrm{cofree}_n^t$ are the comonads on $\mathrm{Sp}$ corresponding to the $\infty$-categories $\mathrm{coAlg}(\tau_n\mathrm{Sp}^\otimes)$ and $\mathcal{P}_{n-1}\mathrm{coAlg}(\tau_n\mathrm{Sp}^\otimes)$ respectively. For the length of this proof write $F_n$ for the fiber of the natural transformation $\mathrm{cofree}^{\mathrm{Tate}}_n \rightarrow \mathrm{cofree}^{\mathrm{Tate}}_{n-1}$. Since $P_m$ preserves fiber sequences, there is a corresponding fiber sequence
\begin{equation*}
\varprojlim_m P_m F_n \rightarrow \varprojlim_m P_m(\mathrm{cofree}^{\mathrm{Tate}}_n) \rightarrow \varprojlim_m P_m(\mathrm{cofree}^{\mathrm{Tate}}_{n-1}).
\end{equation*}
By the inductive hypothesis on $\mathrm{cofree}^{\mathrm{Tate}}_{n-1}$ it thus suffices to show that the Goodwillie tower of $F_n$ converges on simply-connected spectra. 

Note that $F_n$ is equivalent to the fiber of the natural transformation $\mathrm{cofree}_n \rightarrow \mathrm{cofree}_n^t$. Corollary \ref{cor:fibercoalgs} (or rather its variant without the ind) gives a pullback square
\[
\xymatrix{
\mathrm{cofree}^{\mathrm{dp}}_{=n} \ar[d]\ar[r] & \mathrm{cofree}_n \ar[d] \\
\mathrm{id}_{\mathrm{Sp}} \ar[r] & \mathrm{cofree}_n^t,
}
\]
where $\mathrm{cofree}^{\mathrm{dp}}_{=n}$ is the comonad associated to the $\infty$-category 
\begin{equation*}
\bigl\{X \rightarrow \mathrm{fib}(X^{\otimes n} \rightarrow X^{\odot n})_{h\Sigma_n}\bigr\}_{\mathrm{Sp}}
\end{equation*}
of coalgebras for the functor indicated. Thus there is a fiber sequence
\begin{equation*}
F_n \rightarrow \mathrm{cofree}^{\mathrm{dp}}_{=n} \rightarrow \mathrm{id}_{\mathrm{Sp}}
\end{equation*}
and it suffices to show that the Goodwillie tower of $\mathrm{cofree}^{\mathrm{dp}}_{=n}$ converges on simply-connected spectra.

We will use the abbreviated notations
\begin{equation*}
f_n(X) =  \mathrm{fib}(X^{\otimes n} \rightarrow X^{\odot n})_{h\Sigma_n}
\end{equation*}
for the rest of this proof. (Note that Proposition \ref{prop:derivativesid} states $f_n(X) = \Sigma D_n\mathrm{id}_{\mathcal{S}_*}(X)$, although we will not use this.) Observe that if $X$ is $k$-connective, then $X^{\otimes n}$ is $kn$-connective, whereas $X^{\odot n}$ is $(kn - n)$-connective (in fact one can add a small constant, which will not concern us). Consequently $f_n(X)$ is $n(k-1)$-connective. If we take $k \geq 2$ (i.e. $X$ simply-connected) then $\mathrm{cofree}^{\mathrm{dp}}_{=n}(X)$ can be described explicitly. Indeed, consider the functor
\begin{equation*}
\varphi_n: \mathrm{Sp} \rightarrow \mathrm{Sp}: Y \mapsto X \oplus f_n(Y).
\end{equation*}
We will use the obvious map $\pi: \varphi_n(X) \rightarrow X$ projecting onto the first summand. Form the inverse limit
\begin{equation*}
\Phi_n(X) := \varprojlim ( \cdots \xrightarrow{\varphi_n^2(\pi)} \varphi_n^2(X) \xrightarrow{\varphi_n(\pi)} \varphi_n(X) \xrightarrow{\pi} X).
\end{equation*}
Observe that the connectivity of the maps in this inverse system increases rather rapidly. To be precise, the map $\varphi_n(X) \rightarrow X$ is $n(k-1)$-connected by our estimate on the connectivity of $f_n(X)$ above. By induction, assume that the map
\begin{equation*}
\varphi_n^j(X) \xrightarrow{\varphi^{j-1}(\pi)} \varphi_n^{j-1}(X)
\end{equation*}
is $n^j(k-1)$-connected. Then
\begin{equation*}
X \oplus f_n(\varphi_n^j(X)) = \varphi_n^{j+1}(X) \xrightarrow{\varphi_n^j(\pi)} \varphi_n^j(X) = X \oplus f_n(\varphi_n^{j-1}(X))
\end{equation*}
is $n^{j+1}(k-1)$-connected, since $f_n$ essentially multiplies the connectivity of a map by $n$. These connectivity estimates also imply that $f_n$ commutes with the inverse limit defining $\Phi_n(X)$, so that the map
\begin{equation*}
\varphi_n \Phi_n(X) \rightarrow \Phi_n(\varphi_n(X))
\end{equation*}
is an equivalence. This fact implies that the map $\Phi_n(X) \rightarrow X$ onto the first term of the limit diagram above exhibits $\Phi_n(X)$ as the cofree $f_n$-coalgebra on $X$. Indeed, to make $\Phi_n(X)$ a coalgebra take the map
\begin{equation*}
\Phi_n(X) \simeq \varphi_n\Phi_n(X) = X \oplus f_n(\Phi_n(X)) \rightarrow f_n(\Phi_n(X))
\end{equation*}
where the first equivalence is the inverse of the `shift map' which applies $\varphi$ to every term of the limit. To see that it is cofree, consider a coalgebra $Y \xrightarrow{t} f_n(Y)$ and a map of coalgebras $Y \xrightarrow{\alpha} \Phi_n(X)$. The underlying map of spectra is determined by a system of compatible maps
\begin{equation*}
\alpha_j: Y \rightarrow \varphi_n^j X
\end{equation*}
and homotopies between $\pi \circ \alpha_j$ and $\alpha_{j-1}$. That $\alpha$ is a map of coalgebras is expressed by homotopies between $\alpha_j$ and $f_n(\alpha_{j-1}) \circ t$, as one sees by inspecting the square
\[
\xymatrix{
Y \ar[r]^{\alpha}\ar[d]_t & \Phi_n(X) \ar[d] \\
f_n(Y) \ar[r]_-{f_n(\alpha)} & f_n(\Phi_n(X)).
}
\]
By induction it is clear that $\alpha$ is completely determined by $\alpha_0: Y \rightarrow X$, and conversely any such $\alpha_0$ defines a map of coalgebras $\alpha$. This observation is easily made precise to show that the forgetful map
\begin{equation*}
\mathrm{Map}_{\mathrm{coAlg}_{f_n}}(Y, \Phi_n(X)) \rightarrow \mathrm{Map}_{\mathrm{Sp}}(Y,X)
\end{equation*}
is an equivalence.

Finally we are ready to conclude the proof that $\mathrm{cofree}_{=n}^{\mathrm{dp}}$ has convergent Goodwillie tower for $X \in \mathrm{Sp}^{\geq 2}$. Indeed, we observed that the map
\begin{equation*}
\mathrm{cofree}_{=n}^{\mathrm{dp}}(X) \simeq \Phi_n(X) \rightarrow \varphi_n^j (X)
\end{equation*}
is $n^{j+1}(k-1)$-connected, from which it follows straightforwardly that
\begin{equation*}
P_{n^j}(\mathrm{cofree}_{=n}^{\mathrm{dp}})(X) \rightarrow \varphi_n^j (X)
\end{equation*}
is an equivalence. Indeed, the connectivity of the error term grows so quickly with the connectivity of $X$ that its first $n^j$ derivatives must vanish and $\varphi_n^j$ itself is an $n^j$-excisive functor. Taking the limit over $j$ shows that
\begin{equation*}
\mathrm{cofree}_{=n}^{\mathrm{dp}}(X) = \varprojlim_j \varphi_n^j (X) \simeq \varprojlim_j P_{n^j}(\mathrm{cofree}_{=n}^{\mathrm{dp}})(X) 
\end{equation*}
as desired.
\end{proof}

\begin{lemma}
\label{lem:SigmaOmegacofreeTate}
Let $X$ be a simply-connected spectrum. Then the natural transformation
\begin{equation*}
\gamma: \Sigma^\infty\Omega^\infty X \rightarrow \mathrm{cofree}^{\mathrm{Tate}} X
\end{equation*}
is an equivalence.
\end{lemma}
\begin{proof}
Consider the commutative diagram
\[
\xymatrix{
\Sigma^\infty\Omega^\infty X \ar[d]\ar[r]^{\gamma} & \mathrm{cofree}^{\mathrm{Tate}} X \ar[d] \\
\varprojlim_n P_n(\Sigma^\infty\Omega^\infty) X \ar[r] & \varprojlim_n P_n(\mathrm{cofree}^{\mathrm{Tate}}_n) X.
}
\]
The left vertical map is well-known to be an equivalence for connected $X$ (see for example Corollary 1.3 of \cite{ahearnkuhn} for an explicit statement) and the lower horizontal map is an equivalence by Lemma \ref{lem:PncofreeTate}. The vertical map on the right may be factored as
\begin{eqnarray*}
\mathrm{cofree}^{\mathrm{Tate}} X & \rightarrow & \varprojlim_n \mathrm{cofree}_n^{\mathrm{Tate}} X \\
& \rightarrow & \varprojlim_n \varprojlim_m P_m(\mathrm{cofree}_n^{\mathrm{Tate}}) X \\
& \simeq & \varprojlim_n P_n(\mathrm{cofree}_n^{\mathrm{Tate}}) X.
\end{eqnarray*}
The first map is an equivalence by the definition of the $\infty$-category of Tate coalgebras as a limit of $\infty$-categories of $n$-truncated Tate coalgebras; the second is an equivalence by Lemma \ref{lem:cofreeTateconverges}. It follows that the remaining map $\gamma$ in the square is also an equivalence.
\end{proof}

\begin{proof}[Proof of Theorem \ref{thm:Tatecoalgebras}]
As mentioned before we will exploit Lurie's version of the Barr-Beck monadicity theorem (Theorem 4.7.3.5 of \cite{higheralgebra}). First of all, let us recall the well-known observation that $\mathcal{S}_*^{\geq 2}$ is comonadic over $\mathrm{Sp}^{\geq 2}$ via the adjoint pair $(\Sigma^\infty,\Omega^\infty)$. In other words, the functor $\Sigma^\infty$ induces an equivalence of $\infty$-categories between $\mathcal{S}_*^{\geq 2}$ and the $\infty$-category of coalgebras (sometimes also called comodules) for the comonad $\Sigma^\infty\Omega^\infty$. According to the Barr-Beck theorem, to prove this it suffices to check the following things:
\begin{itemize}
\item[(1)] The functor $\Sigma^\infty$ is conservative, i.e. a map $f$ of simply-connected pointed spaces is an equivalence if and only if $\Sigma^\infty f$ is an equivalence.
\item[(2)] If $X^{-1} \rightarrow X^{\bullet}$ is a coaugmented cosimplicial object which is $\Sigma^\infty$-split (i.e. whose image under $\Sigma^\infty$ admits contracting codegeneracies), then the induced map
\begin{equation*}
X^{-1} \rightarrow \mathrm{Tot}(X^\bullet)
\end{equation*}
is an equivalence.
\end{itemize}
Item (1) is completely classical; if $\Sigma^\infty f$ is an equivalence, then $f$ is a homology equivalence of simply-connected spaces and hence an actual equivalence, by the theorems of Whitehead and Hurewicz. For (2) we use that for a simply-connected pointed space $X$, the `$\Omega^\infty\Sigma^\infty$-resolution'
\[
\xymatrix{
X \ar[r] & \Omega^\infty\Sigma^\infty(X) \ar@<.5ex>[r]\ar@<-.5ex>[r] & \Omega^\infty\Sigma^\infty\Omega^\infty\Sigma^\infty(X) \ar@<1ex>[r]\ar[r]\ar@<-1ex>[r] \ar[l] & \cdots \ar@<.5ex>[l]\ar@<-.5ex>[l]
}
\]
gives an equivalence
\begin{equation*}
X \simeq \mathrm{Tot}\bigl((\Omega^\infty\Sigma^\infty)^{\bullet + 1} X\bigr).
\end{equation*}
In fact this works for any nilpotent space (see \cite{aronekankaanrinta} for a discussion of this). Now if $X^{-1} \rightarrow X^{\bullet}$ is a general $\Sigma^\infty$-split cosimplicial object of $\mathcal{S}_*^{\geq 2}$ one considers the square
\[
\xymatrix{
X^{-1} \ar[r]\ar[d] & \mathrm{Tot}\bigl(X^{\bullet}\bigr) \ar[d] \\
\mathrm{Tot}\bigl((\Omega^\infty\Sigma^\infty)^{\bullet + 1} X^{-1}\bigr) \ar[r] & \mathrm{Tot}\bigl((\Omega^\infty\Sigma^\infty)^{\bullet + 1} X^\bullet \bigr).
}
\]
The vertical maps are equivalences because the $\Omega^\infty\Sigma^\infty$-resolution converges (as just discussed), whereas the bottom horizontal arrow is an equivalence by the assumption that $X^\bullet$ is $\Sigma^\infty$-split. Therefore the top horizontal arrow is an equivalence as well, finishing the proof of comonadicity for the pair $(\Sigma^\infty,\Omega^\infty)$.

Lemma \ref{lem:SigmaOmegacofreeTate} states that the map of comonads $\Sigma^\infty\Omega^\infty \rightarrow \mathrm{cofree}^{\mathrm{Tate}}$ is an equivalence. Thus, to prove that the comparison functor $\Gamma$ is an equivalence it suffices to show that the adjoint pair
\[
\xymatrix{
\mathrm{coAlg}_{\mathrm{Tate}}(\mathrm{Sp}^\otimes)^{\geq 2} \ar@<.5ex>[r]^-{U} & \mathrm{Sp}^{\geq 2}, \ar@<.5ex>[l]^-{R}
}
\]
with $U$ the forgetful functor and $R$ its right adjoint, exhibits the $\infty$-category of simply-connected Tate coalgebras as comonadic over $\mathrm{Sp}^{\geq 2}$. Note that with this notation we have $UR = \mathrm{cofree}^{\mathrm{Tate}}$. Again we check (1) and (2) as above. The fact that $U$ is conservative is immediate from our definitions. We should therefore show that $U$-split coaugmented cosimplicial objects are limit diagrams in $\mathrm{coAlg}_{\mathrm{Tate}}(\mathrm{Sp}^\otimes)^{\geq 2}$. By the same argument as above, it suffices to show that for a simply-connected Tate coalgebra $X$ the canonical resolution
\[
\xymatrix{
X \ar[r] & RU(X) \ar@<.5ex>[r]\ar@<-.5ex>[r] & RURU(X) \ar@<1ex>[r]\ar[r]\ar@<-1ex>[r] \ar[l] & \cdots \ar@<.5ex>[l]\ar@<-.5ex>[l]
}
\]
gives an equivalence
\begin{equation*}
X \simeq \mathrm{Tot}\bigl((RU)^{\bullet + 1} X\bigr).
\end{equation*}
Write 
\[
\xymatrix{
\mathrm{coAlg}_{\mathrm{Tate}}(\mathrm{Sp}^\otimes)^{\geq 2} \ar@<.5ex>[r]^-{T_n} & \mathrm{coAlg}_{\mathrm{Tate}}(\tau_n\mathrm{Sp}^\otimes)^{\geq 2} \ar@<.5ex>[l]^-{R_n}
}
\]
for the evident adjunction, where $T_n$ is the forgetful functor and $R_n$ its right adjoint. Note that $T_1$ is the forgetful functor $U$ and $R_1 = R$. We will use the functors $T_n$ and $R_n$ to describe the `Goodwillie tower' of the identity functor on $\mathrm{coAlg}_{\mathrm{Tate}}(\mathrm{Sp}^\otimes)^{\geq 2}$; we will write $\mathrm{id}_{\mathrm{coAlg}}$ for this functor. Note that as of yet it is not clear that the latter $\infty$-category is a suitable context for Goodwillie calculus (although this will follow from the theorem), but still there is a reasonable candidate for the Goodwillie tower of the identity, namely the inverse system
\begin{equation*}
\mathrm{id}_{\mathrm{coAlg}} \rightarrow \cdots \rightarrow R_n T_n \rightarrow R_{n-1} T_{n-1} \rightarrow \cdots \rightarrow R_1 T_1.
\end{equation*} 
Indeed, since the identity functor of $\mathrm{coAlg}_{\mathrm{Tate}}(\tau_n\mathrm{Sp}^\otimes)$ is $n$-excisive, it follows formally that $R_nT_n$ is an $n$-excisive functor. More precisely, the pullback square defining $\mathrm{coAlg}_{\mathrm{Tate}}(\tau_n\mathrm{Sp}^\otimes)$ shows that there is a fiber sequence
\begin{equation*}
R\partial_n\mathrm{id}_{\mathcal{S}_*}(UX, \ldots, UX)_{h\Sigma_n} \rightarrow R_nT_n(X) \rightarrow R_{n-1}T_{n-1}(X)
\end{equation*}
(compare the proof of Lemma \ref{lem:compactTatecoalgebras} and Remark \ref{rmk:compactTatecoalgebras}) which is analogous to the fiber sequence 
\begin{equation*}
D_n \mathrm{id}_{\mathcal{S}_*} \rightarrow P_n\mathrm{id}_{\mathcal{S}_*} \rightarrow P_{n-1}\mathrm{id}_{\mathcal{S}_*}.
\end{equation*}
We will write $\mathbf{D}_n$ for the functor which assigns to a spectrum $E$ the spectrum $\partial_n\mathrm{id}_{\mathcal{S}_*}(E, \ldots, E)_{h\Sigma_n}$, so that the fiber above may be abbreviated as $R \circ \mathbf{D}_n \circ U$. Also, note that the definition of the $\infty$-category of Tate coalgebras as the limit of $\infty$-categories of $n$-truncated Tate coalgebras implies that
\begin{equation*}
\mathrm{id}_{\mathrm{coAlg}} \rightarrow \varprojlim_n R_nT_n
\end{equation*}
is an equivalence (this is a kind of unconditional convergence of the Goodwillie tower for Tate coalgebras). Moreover, since the connectivity of the fibers $R \circ \mathbf{D}_n \circ U(X)$ is $n(k-1)$ (up to a small constant), with $k$ the connectivity of $X$, one easily checks that indeed $P_n(\mathrm{id}_{\mathrm{coAlg}} ) = R_nT_n$.

The unit $\mathrm{id}_{\mathrm{coAlg}} \rightarrow RU$ factors over $R_nT_n \rightarrow RU$; in fact, the entire $RU$-resolution of the identity functor factors over
\[
\xymatrix{
R_nT_n \ar[r] & RU \ar@<.5ex>[r]\ar@<-.5ex>[r] & RURU \ar@<1ex>[r]\ar[r]\ar@<-1ex>[r] \ar[l] & \cdots .\ar@<.5ex>[l]\ar@<-.5ex>[l]
}
\]
We will prove that the resulting map
\begin{equation*}
r_n: R_nT_n \rightarrow \mathrm{Tot}\bigl(P_n(R(UR)^\bullet) U)\bigr)
\end{equation*}
is an equivalence using a modification of the argument used in the proof of Proposition \ref{prop:cobarderivatives}, which is due to Arone and Ching. We use a finite induction along the tower
\begin{equation*}
R_nT_n \rightarrow R_{n-1}T_{n-1} \rightarrow \cdots \rightarrow R_1T_1,
\end{equation*}
in which the fiber of $R_kT_k \rightarrow R_{k-1}T_{k-1}$ is $R\mathbf{D}_kU$. Observe that the cosimplicial object
\begin{equation*}
P_n\bigl(R\mathbf{D}_kUR(UR)^\bullet U\bigr)
\end{equation*}
admits contracting codegeneracies induced by the counit $UR \rightarrow \mathrm{id}_{\mathrm{Sp}^{\geq 2}}$, so that\begin{equation*}
P_n(R\mathbf{D}_k U) \simeq \mathrm{Tot}\bigl(P_n(R\mathbf{D}_kUR(UR)^\bullet U)\bigr).
\end{equation*}
Now consider the diagram
\[
\xymatrix{
P_n(R\mathbf{D}_k U) \ar[r]\ar[d] & \mathrm{Tot}\bigl(P_n(R\mathbf{D}_kUR(UR)^\bullet U)\bigr) \ar[d] \\
P_n(R_kT_k) \ar[r]\ar[d] & \mathrm{Tot}\bigl(P_n(R_kT_k R(UR)^\bullet U)\bigr) \ar[d] \\
P_n(R_{k-1}T_{k-1}) \ar[r] & \mathrm{Tot}\bigl(P_n(R_{k-1}T_{k-1}R(UR)^\bullet U)\bigr).
}
\]
Since $P_n$ commutes with fiber sequences and totalizations preserve limit diagrams, both columns are fiber sequences. By induction we may assume that the bottom horizontal arrow is an equivalence, the base of the induction being the homogeneous case $k=1$ covered above; the homogeneous case also shows that the top horizontal map is an equivalence. Therefore the middle horizontal arrow is an equivalence. Setting $k=n$ and using that $R_nT_n$ is $n$-excisive, we conclude that we have the following equivalences, where the second step uses that $P_n(\mathrm{id}_{\mathrm{coAlg}}) = R_nT_n$ and the third uses that $R$ preserves limits and $U$ preserves colimits:
\begin{eqnarray*}
R_nT_n & \simeq & \mathrm{Tot}\bigl(P_n(R_nT_n R(UR)^\bullet U)\bigr) \\
& \simeq & \mathrm{Tot}\bigl(P_n(R(UR)^\bullet U)\bigr) \\
& \simeq & \mathrm{Tot}\bigl(R P_n((UR)^\bullet) U\bigr).
\end{eqnarray*}
To finish the proof, observe that
\begin{eqnarray*}
\mathrm{id}_{\mathrm{coAlg}} & \simeq & \varprojlim_n R_nT_n \\
& \simeq & \varprojlim_n \mathrm{Tot}\bigl(R P_n((UR)^\bullet) U\bigr) \\
& \simeq &  \mathrm{Tot}\bigl(R (UR)^\bullet U\bigr).
\end{eqnarray*}
The last equivalence follows from Lemma \ref{lem:cofreeTateconverges}, which shows that the connectivity of the map $UR \rightarrow P_n(UR)$ grows linearly with $n$ (and hence the same is true with $(UR)^\bullet$ in place of $UR$).
\end{proof}

\section{Further remarks on the Goodwillie tower of the homotopy theory of spaces}
\label{subsec:goodwilliespaces}

In this section we make some further observations on the Goodwillie tower of the $\infty$-category of pointed spaces $\mathcal{S}_*$. The stable $\infty$-operad of interest here is $\mathrm{Sp}^{\otimes}$, the symmetric monoidal $\infty$-category of spectra with the smash product. The derivatives of the identity functor are given by the Spanier-Whitehead duals of the partition complexes, which we will recall below. Throughout this section we will use $\partial_n\mathrm{id}$ to simply denote the coefficient spectrum of the $n$th derivative of the identity, rather than the corresponding multilinear functor of $n$ variables on spectra. The main result of this section is the following:

\begin{proposition}
\label{prop:fiberGn}
The fiber of the map $\mathcal{G}_n(\mathrm{Sp}^\otimes) \rightarrow \mathcal{G}_{n-1}(\mathrm{Sp}^\otimes)$ over $\mathcal{P}_{n-1}\mathcal{S}_*$ is equivalent to $\Omega^{\infty-1}\bigl((\partial_n\mathrm{id})^{t\Sigma_n}\bigr)$.
\end{proposition}

\begin{example}
The spectrum $\partial_2 \mathrm{id}$ is $S^{-1}$ with trivial $\Sigma_2$-action. In this case, using that $\mathcal{G}_1(\mathcal{O}^\otimes)$ is contractible, the previous proposition identifies the space $\mathcal{G}_2(\mathrm{Sp}^{\otimes})$ with $\Omega^{\infty}S^{t\Sigma_2}$. This is in fact the zeroth space of the 2-completed sphere spectrum.
\end{example}

Let us illustrate the proof of this result for $n=2$; the higher cases (which mostly differ in notation only) are covered by Lemmas \ref{lem:nattransftheta} and \ref{lem:fibercobar} below. The space $\mathcal{G}_2(\mathrm{Sp}^\otimes)$ is equivalent to the space of natural transformations $\Sigma^\infty \rightarrow \Theta_{\mathrm{Sp}}$, where
\begin{equation*}
\Theta_{\mathrm{Sp}}(X) = (\Sigma^\infty X \wedge \Sigma^\infty X)^{t\Sigma_2}.
\end{equation*}
We claim that evaluation at $S^0$ determines an equivalence
\begin{equation*}
\mathrm{Nat}(\Sigma^\infty, \Theta_{\mathrm{Sp}}) \longrightarrow \mathrm{Map}(S, (S \wedge S)^{t\Sigma_2}) \simeq \Omega^\infty S^{t\Sigma_2}.
\end{equation*}
Write $\theta$ for the functor from spectra to spectra which assigns $(Y \wedge Y)^{t\Sigma_2}$ to $Y$, so that $\Theta_{\mathrm{Sp}} = \theta \circ \Sigma^\infty$. Then $\theta$ is an exact functor by Lemma \ref{lem:normseq}. Furthermore, by Lemma \ref{lem:nexcfunctors}, precomposition with $\Sigma^\infty$ yields an equivalence
\begin{equation*}
\mathrm{Nat}(\mathrm{id}_{\mathrm{Sp}}, \theta) \longrightarrow \mathrm{Nat}(\Sigma^\infty, \Theta_{\mathrm{Sp}}).
\end{equation*}

\begin{remark}
\label{rmk:Tatefilteredcolim}
Strictly speaking Lemma \ref{lem:nexcfunctors} does not apply directly, since $\theta$ does not preserve filtered colimits. However, one can define a functor $\theta^c$ determined by the requirements that it agrees with $\theta$ on compact objects and preserves filtered colimits. Since $\Sigma^\infty$ is compatible with filtered colimits, the spaces $\mathrm{Nat}(\Sigma^\infty, \theta^c)$ and $\mathrm{Nat}(\Sigma^\infty, \theta)$ are equivalent.
\end{remark}

Now, since the first derivative of $\Sigma^\infty \Omega^\infty$ is the identity, there is an equivalence
\begin{equation*}
\varinjlim_n \Omega^n \Sigma^\infty \Omega^\infty \Sigma^n \longrightarrow \mathrm{id}_{\mathrm{Sp}}.
\end{equation*}
This gives a sequence of equivalences
\begin{eqnarray*}
\mathrm{Nat}(\mathrm{id}, \theta) & \longrightarrow & \varprojlim_n \mathrm{Nat}(\Omega^n \Sigma^\infty \Omega^\infty \Sigma^n, \theta) \\
& \longrightarrow & \varprojlim_n \mathrm{Nat}(\Omega^\infty \Sigma^n, \Omega^\infty \Sigma^n \theta) \\
& \longrightarrow & \varprojlim_n \Omega^\infty \Sigma^n \theta(S^{-n}).
\end{eqnarray*}
The second equivalence is simply adjunction, the third equivalence is a consequence of the fact that the functor $\Omega^\infty \Sigma^n$ is corepresented by the spectrum $S^{-n}$. Since $\theta$ is exact, there is an equivalence $\Sigma^n \theta(S^{-n}) \simeq \Sigma^n \Omega^n \theta(S) \simeq \theta(S)$. We conclude that evaluation at $S$ determines an equivalence
\begin{equation*}
\mathrm{Nat}(\mathrm{id}_{\mathrm{Sp}}, \theta)  \longrightarrow \Omega^\infty\theta(S)
\end{equation*}
which is what was needed.

Recall the functors $\Sigma_n^\infty: \mathcal{S} \rightarrow \mathcal{P}_n\mathcal{S}_*$ and $\Sigma_{n,1}^\infty: \mathcal{P}_n\mathcal{S}_* \rightarrow \mathrm{Sp}$. The general case of Proposition \ref{prop:fiberGn} follows from the next two lemmas:

\begin{lemma}
\label{lem:nattransftheta}
Let $F: \mathrm{Sp} \rightarrow \mathrm{Sp}$ be an $n$-excisive functor. Then evaluation at $\Sigma_n^\infty S^0$ determines an equivalence
\begin{equation*}
\mathrm{Nat}(\Sigma^\infty_{n,1}, \, F \circ \Sigma^\infty_{n,1}) \longrightarrow \Omega^\infty F(S).
\end{equation*}
\end{lemma}

As usual, write $\odot^n$ for the $n$-fold tensor product determined by the stable $\infty$-operad $\tau_{n-1}\mathrm{Sp}^\otimes$. Also, recall that the tensor product $\otimes^n$ in $\mathrm{Sp}^\otimes$ can be identified with the smash product.

\begin{lemma}
\label{lem:fibercobar}
The fiber of $S^{\otimes n} \rightarrow S^{\odot n}$ is the spectrum $\Sigma(\partial_n\mathrm{id})$.
\end{lemma}

\begin{proof}[Proof of Lemma \ref{lem:nattransftheta}]
Recall the adjunction
\[
\xymatrix{
L_n: \mathcal{P}_n\mathcal{S}_* \ar@<.5ex>[r] & \mathcal{T}_n\mathcal{P}_n\mathcal{S}_* : R_n, \ar@<.5ex>[l]
}
\]
which is in an equivalence since $\mathcal{P}_n\mathcal{S}_*$ is $n$-excisive. We also write $L_n^k$ for the evident functor $\mathcal{P}_n\mathcal{S}_* \rightarrow \mathcal{T}_n^k\mathcal{P}_n\mathcal{S}_*$ and $R_n^k$ for its right adjoint. The colimit $\varinjlim_k R_n^k \Sigma_n^\infty\Omega_n^\infty L_n^k$ is equivalent to the identity functor of $\mathcal{P}_n\mathcal{S}_*$, simply because this colimit computes $P_n(\Sigma_n^\infty\Omega_n^\infty)$. We then have equivalences
\begin{eqnarray*}
\mathrm{Nat}(\Sigma^\infty_{n,1}, \, F \Sigma^\infty_{n,1}) & \longrightarrow & \mathrm{Nat}(\mathrm{id}_{\mathcal{P}_n\mathcal{S}_*},\, \Omega^\infty_{n,1} F \Sigma^\infty_{n,1}) \\
& \longrightarrow & \varprojlim_k \mathrm{Nat}(R_n^k \Sigma_n^\infty\Omega_n^\infty L_n^k,\, \Omega^\infty_{n,1} F \Sigma^\infty_{n,1}) \\
& \longrightarrow & \varprojlim_k \mathrm{Nat}(\Omega_n^\infty L_n^k,\, \Omega_n^\infty L_n^k \Omega^\infty_{n,1} F \Sigma^\infty_{n,1}),
\end{eqnarray*}
where the last equivalence uses the fact that $R_n^k$ is an equivalence with inverse $L_n^k$. Also, observe that there is a commutative diagram
\[
\xymatrix{
\mathcal{P}_n\mathcal{S}_* & \mathcal{T}_n^k\mathcal{P}_n\mathcal{S}_* \ar[l]_{R_n^k} \\
\mathrm{Sp} \ar[u]^{\Omega_{n,1}^\infty} & \mathcal{T}_n^k\mathrm{Sp}. \ar[l]^{R_n^k}\ar[u]_{\Omega_{n,1}^\infty}
}
\]
Inverting the two horizontal functors we obtain an equivalence $ L_n^k \Omega^\infty_{n,1} \simeq \Omega^\infty_{n,1} L_n^k$ and therefore an equivalence 
\begin{eqnarray*}
\mathrm{Nat}(\Sigma^\infty_{n,1}, \, F \Sigma^\infty_{n,1}) & \longrightarrow & \varprojlim_k \mathrm{Nat}(\Omega_n^\infty L_n^k,\, \Omega^\infty L_n^k F \Sigma^\infty_{n,1}).
\end{eqnarray*}
Precomposing with the equivalence $R_n^k$ we find a further equivalence
\begin{eqnarray*}
\mathrm{Nat}(\Sigma^\infty_{n,1}, \, F \Sigma^\infty_{n,1}) & \longrightarrow & \varprojlim_k \mathrm{Nat}(\Omega_n^\infty, \, \Omega^\infty L_n^k F \Sigma^\infty_{n,1} R_n^k).
\end{eqnarray*}
Now we use the fact that $F\Sigma^\infty_{n,1}$ is $n$-excisive to observe that
\begin{equation*}
 L_n^k F \Sigma^\infty_{n,1} R_n^k \simeq L_n^k R_n^k F \Sigma^\infty_{n,1} L_n^k R_n^k \simeq F \Sigma^\infty_{n,1}
\end{equation*}
where the right-most functor should be read as the pointwise application of $F \Sigma^\infty_{n,1}$ to yield a functor between $\mathcal{T}_n^k\mathcal{S}_*$ and $\mathcal{T}_n^k\mathrm{Sp}$. Finally then we find an equivalence
\begin{eqnarray*}
\mathrm{Nat}(\Sigma^\infty_{n,1}, \, F \Sigma^\infty_{n,1}) & \longrightarrow & \mathrm{Nat}(\Omega_n^\infty, \, \Omega^\infty F \Sigma^\infty_{n,1}).
\end{eqnarray*}
Observe that the space of natural transformations on the right is between functors $\mathcal{T}_n^k \mathcal{P}_n\mathcal{S}_* \rightarrow \mathcal{T}_n^k \mathcal{S}_*$ rather than functors $\mathcal{P}_n\mathcal{S}_* \rightarrow \mathcal{S}_*$. However, by Lemma \ref{lem:nexcfunctors}, this distinction is irrelevant, justifying our lack of notational precision. To conclude, note that the functor $\Omega_n^\infty$ is corepresented by $\Sigma^\infty_n S^0$, finally yielding the desired equivalence
\begin{eqnarray*}
\mathrm{Nat}(\Sigma^\infty_{n,1}, \, F \Sigma^\infty_{n,1}) & \longrightarrow & \Omega^\infty F \Sigma^\infty_{n,1}(\Sigma^\infty_n S^0) \simeq \Omega^\infty F(S).
\end{eqnarray*}
\end{proof}

\begin{proof}[Proof of Lemma \ref{lem:fibercobar}]
This lemma follows from the more general Proposition \ref{prop:derivativesid}, but for the convenience of the reader we offer a direct proof here. Let us first recall the standard description of the spectrum $\partial_n\mathrm{id}$. As before we write $\mathbf{Equiv}(n)$ for the partially ordered set of equivalence relations on $\{1, \ldots, n\}$. Also, write $\mathbf{Equiv}^{\pm}(n)$ for the subset obtained from $\mathbf{Equiv}(n)$ by deleting the initial and final object, i.e. the trivial and discrete equivalence relations. Then $\partial_n\mathrm{id}$ is the Spanier-Whitehead dual of the double suspension of $\N\mathbf{Equiv}^{\pm}(n)$. Note that therefore $\Sigma(\partial_n\mathrm{id})$ is the Spanier-Whitehead dual of a single suspension of $\N\mathbf{Equiv}^{\pm}(n)$. Observe that $S^{\otimes n}$ is just the sphere spectrum. To identify $S^{\odot n}$, note that the diagram $\psi_{n-1}^n$ of Proposition \ref{prop:truncatedtensor} is the constant diagram with value $S$, so that its limit over $\N\mathbf{Part}_{n-1}(n)$ is simply the Spanier-Whitehead dual of this simplicial set. Therefore the fiber of the map
\begin{equation*}
S \longrightarrow S^{\odot n}
\end{equation*}
is the Spanier-Whitehead dual of the suspension of $\N\mathbf{Part}_{n-1}(n)$. To conclude that this is equivalent to the spectrum $\Sigma(\partial_n\mathrm{id})$ it thus suffices to show that $\N\mathbf{Part}_{n-1}(n)$ is weakly equivalent to $\N\mathbf{Equiv}^{\pm}(n)$.

Recall that the poset $\mathbf{Part}_{n-1}(n)$ is defined as the poset of chains of equivalence relations $E_0 < \cdots < E_j$ so that $E_0$ is discrete, $E_j$ is trivial, and each of the maps
\begin{equation*}
\{1, \ldots, n\}/E_{i-1} \rightarrow \{1, \ldots, n\}/E_{i}
\end{equation*}
has fibers of cardinality at most $n-1$. Write $\mathbf{Part}_{n-1}(n)'$ for the subset of $\mathbf{Part}_{n-1}(n)$ consisting of simplices that are nondegenerate in $\N\mathbf{Equiv}(n)$ and similarly write $\mathbf{\Delta}/\N\mathbf{Equiv}(n)'$ for the subcategory of the category of simplices spanned by nondegenerate simplices. It is well-known (and easy to show) that the inclusion $\mathbf{Part}_{n-1}(n)' \subseteq \mathbf{Part}_{n-1}(n)$ induces a homotopy equivalence on nerves. Now observe that the map of partially ordered sets 
\begin{equation*}
\mathbf{Part}_{n-1}(n)' \longrightarrow \mathbf{\Delta}/\N\mathbf{Equiv}^{\pm}(n)'
\end{equation*}
which forgets the initial and final vertices of a chain gives an isomorphism of simplicial sets. We conclude by using the well-known fact that $\mathbf{\Delta}/\N\mathbf{Equiv}^{\pm}(n)'$ is the barycentric subdivision of $\N\mathbf{Equiv}^{\pm}(n)$, which is equivalent to $\N\mathbf{Equiv}^{\pm}(n)$ under the map taking the last vertex of a simplex. 
\end{proof}

\appendix


\chapter{Compactly generated $\infty$-categories}
\label{sec:compactlygenerated}

For the convenience of the reader we briefly recall some basic facts about compactly generated $\infty$-categories of which the proofs can be found in Section 5.5.7 of \cite{htt}. Furthermore we prove Lemma \ref{lem:filteredcolimleftexact} below, which is needed in Chapter \ref{sec:constructingPnC}. Recall that an $\infty$-category $\mathcal{C}$ is \emph{compactly generated} if it is both presentable and $\omega$-accessible. Alternatively, $\mathcal{C}$ is compactly generated if and only if it is equivalent to $\mathrm{Ind}(\mathcal{D})$ for some small $\infty$-category $\mathcal{D}$ which admits finite colimits. 

Recall that an object $X$ of an $\infty$-category $\mathcal{C}$ is \emph{compact} if the functor $\mathrm{Map}_{\mathcal{C}}(X,-): \mathcal{C} \rightarrow \mathcal{S}$ preserves filtered colimits. For a compactly generated $\infty$-category $\mathcal{C}$ we write $\mathcal{C}^{c}$ for the full subcategory of $\mathcal{C}$ spanned by its compact objects. Write $\mathbf{Cat}^{\mathrm{Rex}}$ for the $\infty$-category of essentially small $\infty$-categories which admit finite colimits, with functors preserving finite colimits. Furthermore, write $\mathbf{Cat}^{\mathrm{Rex}}_{\mathrm{idem}}$ for the full subcategory of $\mathbf{Cat}^{\mathrm{Rex}}$ spanned by those $\infty$-categories $\mathcal{C}$ that are moreover idempotent complete. The following is part of Proposition 5.5.7.8 of \cite{htt}: 

\begin{proposition}
\label{prop:compacts}
The functor which assigns to a compactly generated $\infty$-category $\mathcal{C}$ its full subcategory $\mathcal{C}^c$ of compact objects gives an equivalence of $\infty$-categories $\mathbf{Cat}^{\omega} \rightarrow \mathbf{Cat}^{\mathrm{Rex}}_{\mathrm{idem}}$. 
\end{proposition}

\begin{remark}
\label{rmk:idemcompletion}
An inverse to the construction of the previous proposition is given by assigning to $\mathcal{D} \in \mathbf{Cat}^{\mathrm{Rex}}_{\mathrm{idem}}$ the $\infty$-category $\mathrm{Ind}(\mathcal{D})$. This construction of course makes sense for $\mathcal{D} \in \mathbf{Cat}^{\mathrm{Rex}}$ not necessarily idempotent complete. The functor $\mathrm{Ind}$ factors through idempotent completion; Proposition 5.5.7.10 of \cite{htt} shows that it exhibits $\mathbf{Cat}^{\omega}$ as a localization of $\mathbf{Cat}^{\mathrm{Rex}}$.
\end{remark}

\begin{lemma}
\label{lem:catomegabicomplete}
The $\infty$-category $\mathbf{Cat}^\omega$ admits small limits and colimits.
\end{lemma}
\begin{proof}
It follows from 5.5.7.6 and 5.5.7.7 of \cite{htt} that $\mathbf{Cat}^\omega$ has small colimits. The existence of small limits is more straightforward; the $\infty$-category $\mathbf{Cat}^{\mathrm{Rex}}_{\mathrm{idem}}$ admits small limits and these can be computed in $\mathbf{Cat}$, see Lemma \ref{lem:filteredcolimcompacts} below.
\end{proof}

\begin{lemma}
\label{lem:filteredcolimcompacts}
The functor $\mathbf{Cat}^{\omega} \rightarrow \mathbf{Cat}$ which assigns to a compactly generated $\infty$-category $\mathcal{C}$ its full subcategory $\mathcal{C}^c$ preserves all small limits and filtered colimits.
\end{lemma}
\begin{proof}
By Proposition \ref{prop:compacts} it suffices to verify this claim for the inclusion $\mathbf{Cat}^{\mathrm{Rex}}_{\mathrm{idem}} \rightarrow \mathbf{Cat}$. The fact that this functor preserves small limits follows from the results of Section 5.3.6 of \cite{htt}; specifically, Corollary 5.3.6.10 shows that this functor has a left adjoint. Preservation of filtered colimits is guaranteed by Lemma 7.3.5.10 of \cite{higheralgebra}.
\end{proof}

Finally, we will need to know that filtered colimits and finite limits commute in $\mathbf{Cat}^{\omega}$:

\begin{lemma}
\label{lem:filteredcolimleftexact}
Let $I$ be a filtered $\infty$-category and write $\varinjlim_I: \mathrm{Fun}(I, \mathbf{Cat}^{\omega}) \rightarrow \mathbf{Cat}^{\omega}$ for a choice of colimit functor (which is unique up to contractible ambiguity). Then $\varinjlim_I$ preserves finite limits.
\end{lemma}
\begin{proof}
By Lemma \ref{lem:filteredcolimcompacts} it suffices to check the corresponding statement for the $\infty$-category $\mathbf{Cat}$, where it is true. Indeed, filtered colimits and finite limits commute in compactly generated $\infty$-categories, of which $\mathbf{Cat}$ is an example.
\end{proof}

\chapter{Some facts from Goodwillie calculus}
\label{sec:appcalculus}

We will not provide a comprehensive overview of the basics of Goodwillie calculus (which can for example be found in \cite{goodwillie3, higheralgebra}), but in this section we collect several additional results we need in the main text. None of this material is original.

For a functor $F: \mathcal{C} \rightarrow \mathcal{D}$ between pointed compactly generated $\infty$-categories, one can define the $n$th cross effect $\mathrm{cr}_n F$ and $n$th cocross effect $\mathrm{cr}^n F$ as follows. For objects $X_1, \ldots, X_n \in \mathcal{C}$ one considers the $n$-cube
\begin{equation*}
\mathbf{P}(n) \longrightarrow \mathcal{C}: S \longmapsto \coprod_{i \notin S} X_i
\end{equation*}
where the maps are induced by the various maps $X_i \rightarrow *$. Then the $n$th cross effect is defined as the total fiber of $F$ applied to this cube:
\begin{equation*}
\mathrm{cr}_n F(X_1, \ldots, X_n) := \mathrm{tfib}\Bigl\{F\bigl(\coprod_{i \notin S} X_i\bigr)\Bigr\}_{S \in \mathbf{P}(n)}.
\end{equation*}
Dually, one considers the cube
\begin{equation*}
\mathbf{P}(n) \longrightarrow \mathcal{C}: S \longmapsto \prod_{i \in S} X_i
\end{equation*}
and defines the cocross effect $\mathrm{cr}^n F$ as the corresponding total cofiber:
\begin{equation*}
\mathrm{cr}^n F(X_1, \ldots, X_n) := \mathrm{tcof}\Bigl\{F\bigl(\prod_{i \in S} X_i\bigr)\Bigr\}_{S \in \mathbf{P}(n)}.
\end{equation*}
There are evident maps
\begin{equation*}
\mathrm{cr}_n F(X_1, \ldots, X_n) \longrightarrow F(X_1 \amalg \cdots \amalg X_n) \quad\quad \text{and} \quad\quad F(X_1 \times \cdots \times X_n) \longrightarrow \mathrm{cr}^n F(X_1, \ldots, X_n).
\end{equation*}
Since $\mathcal{C}$ is pointed there is the obvious map
\begin{equation*}
X_1 \amalg \cdots \amalg X_n \longrightarrow X_1 \times \cdots \times X_n
\end{equation*}
which on $X_i$ is $(*, \ldots, \mathrm{id}_{X_i}, \ldots, *)$. As a consequence we obtain a natural map
\begin{equation*}
c_n: \mathrm{cr}_n F(X_1, \ldots X_n) \longrightarrow \mathrm{cr}^n F(X_1, \ldots, X_n).
\end{equation*}
The idea of considering cocross effects is due to McCarthy \cite{mccarthy}. The following is straightforward:

\begin{lemma}
\label{lem:crossequiv}
If $\mathcal{C}$ and $\mathcal{D}$ are stable then $c_n$ is an equivalence.
\end{lemma}
\begin{proof}
In the stable case finite products and coproducts coincide and we write $\oplus$ for both. Consider an inclusion $S \subseteq S'$ in $\mathbf{P}(n)$ and write $T = \{1, \ldots, n\} - S$ and $T' = \{1, \ldots, n\} - S'$. The inclusion 
\begin{equation*}
F\Bigl(\bigoplus_{i \in T'} X_i\Bigr) \rightarrow F\Bigl(\bigoplus_{i \in T} X_i\Bigr)
\end{equation*}
featuring in the definition of the cocross effect is a section of the projection
\begin{equation*}
F\Bigl(\bigoplus_{i \notin S} X_i\Bigr) \rightarrow F\Bigl(\bigoplus_{i \notin S'} X_i\Bigr)
\end{equation*}
featuring in the definition of the cross effect. In particular the cofiber of the first is equivalent to the fiber of the second. The conclusion is easily deduced from the usual calculation of the total cofiber (resp. total fiber) of a cube as an iterated cofiber (resp. iterated fiber).
\end{proof}

\begin{remark}
\label{rmk:derivcoderiv}
Recall that the $n$th derivative $\partial_n F$ (respectively $n$th coderivative $\partial^n F$) is defined by multilinearizing $\mathrm{cr}_n F$ (respectively $\mathrm{cr}^n F$). A consequence of the previous lemma is that for $\mathcal{C}$ and $\mathcal{D}$ stable the derivatives and coderivatives of $F$ are canonically equivalent. Also note that there is a natural map
\begin{equation*}
F(X) \longrightarrow F(X \times \cdots \times X) \longrightarrow \Omega^\infty_{\mathcal{D}} \partial^n F(\Sigma_{\mathcal{C}}^\infty X, \ldots, \Sigma_{\mathcal{C}}^\infty X).
\end{equation*}
\end{remark}

Recall that the $n$-fold tensor product $\otimes^n_{\mathcal{C}}$ induced by the stable $\infty$-operad $\mathrm{Sp}(\mathcal{C})^\otimes$ is by definition equipped with a natural transformation
\begin{equation*}
X_1 \times \cdots \times X_n \longrightarrow \Omega^\infty_{\mathcal{C}}(\Sigma^\infty_{\mathcal{C}}X_1 \otimes_{\mathcal{C}} \cdots \otimes_{\mathcal{C}} \Sigma^\infty X_n)
\end{equation*}
exhibiting the latter as a multilinearization of the former. In particular, for $X \in \mathcal{C}$, composing with the diagonal gives a map
\begin{equation*}
X \longrightarrow \Omega^\infty_{\mathcal{C}}(\Sigma^\infty_{\mathcal{C}} X \otimes_{\mathcal{C}} \cdots \otimes_{\mathcal{C}} \Sigma^\infty_{\mathcal{C}} X).
\end{equation*}
Its adjoint $\delta_n$ is the map featuring in the coalgebra structure of $\Sigma^\infty_{\mathcal{C}} X$ in $\mathrm{Sp}(\mathcal{C})^\otimes$. Explicitly, $\delta_n$ is the composition
\begin{equation*}
\Sigma^\infty_{\mathcal{C}} X \longrightarrow \Sigma^\infty_{\mathcal{C}}\Omega^\infty_{\mathcal{C}}(\Sigma^\infty_{\mathcal{C}} X \otimes_{\mathcal{C}} \cdots \otimes_{\mathcal{C}} \Sigma^\infty_{\mathcal{C}} X) \longrightarrow \Sigma^\infty_{\mathcal{C}} X \otimes_{\mathcal{C}} \cdots \otimes_{\mathcal{C}} \Sigma^\infty_{\mathcal{C}} X
\end{equation*}
where the second map is induced by the counit of the adjunction between $\Sigma^\infty_{\mathcal{C}}$ and $\Omega^\infty_{\mathcal{C}}$. Let us record a slightly different description of the map $\delta_n$, which we use in the proof of Proposition \ref{prop:Tatebeta}. Using the unit of the mentioned adjunction we may form the composition of maps
\begin{equation*}
\Sigma^\infty_{\mathcal{C}} X \longrightarrow \Sigma^\infty_{\mathcal{C}}\Omega^\infty_{\mathcal{C}}\Sigma^\infty_{\mathcal{C}} X \longrightarrow \partial^n\bigl(\Sigma^\infty_{\mathcal{C}}\Omega^\infty_{\mathcal{C}}\bigr)(\Sigma^\infty_{\mathcal{C}} X, \ldots, \Sigma^\infty_{\mathcal{C}} X).
\end{equation*}
Write $\delta_n'$ for this composition.

\begin{lemma}
\label{lem:codersigmaomega}
There is a natural equivalence $\varphi: \partial^n(\Sigma^\infty_{\mathcal{C}}\Omega^\infty_{\mathcal{C}}) \rightarrow \otimes^n_{\mathcal{C}}$. Furthermore, the composition $\varphi \circ \delta_n'$ is canonically homotopic to $\delta_n$.
\end{lemma}
\begin{proof}
An alternative way to factor the map $\delta_n$ is as follows:
\begin{eqnarray*}
\Sigma^\infty_{\mathcal{C}} X & \longrightarrow & \Sigma^\infty_{\mathcal{C}}\Omega^\infty_{\mathcal{C}}\Sigma^\infty_{\mathcal{C}} X \\
& \longrightarrow & \Sigma^\infty_{\mathcal{C}}(\Omega^\infty_{\mathcal{C}}\Sigma^\infty_{\mathcal{C}} X \times \cdots \times \Omega^\infty_{\mathcal{C}}\Sigma^\infty_{\mathcal{C}} X) \\
& \longrightarrow & \Sigma^\infty_{\mathcal{C}}\Omega^\infty_{\mathcal{C}}(\Sigma^\infty_{\mathcal{C}} X \otimes_{\mathcal{C}} \cdots \otimes_{\mathcal{C}} \Sigma^\infty_{\mathcal{C}}X) \\
& \longrightarrow & \Sigma^\infty_{\mathcal{C}} X \otimes_{\mathcal{C}} \cdots \otimes_{\mathcal{C}} \Sigma^\infty_{\mathcal{C}}X.
\end{eqnarray*}
The second to last map induces an equivalence after multilinearizing by definition; the last map does so as well, by the chain rule for linearizations (see Theorem 6.2.1.22 of \cite{higheralgebra}) and the fact that the linearization of $\Sigma^\infty_{\mathcal{C}}\Omega^\infty_{\mathcal{C}}$ is equivalent to the identity functor of $\mathrm{Sp}(\mathcal{C})$. Now observe that the multilinearization of the functor
\begin{equation*}
X \longmapsto \Sigma^\infty_{\mathcal{C}}(\Omega^\infty_{\mathcal{C}}\Sigma^\infty_{\mathcal{C}} X \times \cdots \times \Omega^\infty_{\mathcal{C}}\Sigma^\infty_{\mathcal{C}} X) \simeq \Sigma^\infty_{\mathcal{C}}\Omega^\infty_{\mathcal{C}}(\Sigma^\infty_{\mathcal{C}} X \times \cdots \times \Sigma^\infty_{\mathcal{C}} X) 
\end{equation*}
is precisely $\partial^n\bigl(\Sigma^\infty_{\mathcal{C}}\Omega^\infty_{\mathcal{C}}\bigr)(\Sigma^\infty_{\mathcal{C}} X, \ldots, \Sigma^\infty_{\mathcal{C}} X)$ which concludes the proof.
\end{proof}

We end this section by recalling the relation between the derivatives of the functors $\mathrm{id}_{\mathcal{C}}$ and $\Sigma^\infty_{\mathcal{C}}\Omega^\infty_{\mathcal{C}}$. Consider pointed compactly generated $\infty$-categories $\mathcal{C}$, $\mathcal{D}$ and $\mathcal{E}$ and functors $F: \mathcal{C} \rightarrow \mathcal{D}$, $G: \mathcal{D} \rightarrow \mathcal{E}$ preserving filtered colimits. Using the stabilization of $\mathcal{D}$ we may form a cosimplicial object
\[
\xymatrix{
G\Omega^\infty_{\mathcal{D}}\Sigma^\infty_{\mathcal{D}}F \ar@<.5ex>[r]\ar@<-.5ex>[r] & G\Omega^\infty_{\mathcal{D}}\Sigma^\infty_{\mathcal{D}}\Omega^\infty_{\mathcal{D}}\Sigma^\infty_{\mathcal{D}}F \ar@<1ex>[r]\ar[r]\ar@<-1ex>[r] \ar[l] & G\Omega^\infty_{\mathcal{D}}(\Sigma^\infty_{\mathcal{D}}\Omega^\infty_{\mathcal{D}})^2\Sigma^\infty_{\mathcal{D}}F \ar@<1.5ex>[r]\ar@<.5ex>[r]\ar@<-.5ex>[r]\ar@<-1.5ex>[r] \ar@<.5ex>[l]\ar@<-.5ex>[l] & \cdots \ar@<1ex>[l]\ar[l]\ar@<-1ex>[l]
}
\]
which we denote by $G\Omega^\infty_{\mathcal{D}}(\Sigma^\infty_{\mathcal{D}}\Omega^\infty_{\mathcal{D}})^{\bullet}\Sigma^\infty_{\mathcal{D}}F$. We use the notation $\mathrm{Tot}$ (i.e. totalization) for the limit of a cosimplicial diagram. The following is due to Arone and Ching \cite{aroneching}:

\begin{proposition}
\label{prop:cobarderivatives}
For each $n \geq 0$, the canonical maps
\begin{eqnarray*}
P_n(GF) & \longrightarrow & \mathrm{Tot}\bigl(P_n(G\Omega^\infty_{\mathcal{D}}(\Sigma^\infty_{\mathcal{D}}\Omega^\infty_{\mathcal{D}})^{\bullet}\Sigma^\infty_{\mathcal{D}}F)\bigr), \\
\partial_n(GF) & \longrightarrow & \mathrm{Tot}\bigl(\partial_n(G\Omega^\infty_{\mathcal{D}}(\Sigma^\infty_{\mathcal{D}}\Omega^\infty_{\mathcal{D}})^{\bullet}\Sigma^\infty_{\mathcal{D}}F)\bigr)
\end{eqnarray*}
are equivalences.
\end{proposition}
\begin{proof}
We describe the proof of the first equivalence, the second is almost identical. First suppose $G$ is homogeneous, so that it is of the form $H\Sigma^\infty_{\mathcal{D}}$ for some functor $H: \mathrm{Sp}(\mathcal{D}) \rightarrow \mathcal{E}$. Then the cosimplicial object
\begin{equation*}
P_n(H\Sigma^\infty_{\mathcal{D}}\Omega^\infty_{\mathcal{D}}(\Sigma^\infty_{\mathcal{D}}\Omega^\infty_{\mathcal{D}})^{\bullet}\Sigma^\infty_{\mathcal{D}}F)
\end{equation*}
admits extra codegeneracies (sometimes called contracting codegeneracies) induced by the counit $\Sigma^\infty_{\mathcal{D}}\Omega^\infty_{\mathcal{D}} \rightarrow \mathrm{id}_{\mathcal{D}}$, so that the claimed equivalence immediately follows from the standard lemma on contracting homotopies. For general $G$ we argue by induction on the Goodwillie tower of $G$. Consider the fiber sequence $D_k G \rightarrow P_k G \rightarrow P_{k-1} G$ and the resulting diagram
\[
\xymatrix{
P_n\bigl((D_k G) F\bigr) \ar[r]\ar[d] & \mathrm{Tot}\bigl(P_n(D_k G\Omega^\infty_{\mathcal{D}}(\Sigma^\infty_{\mathcal{D}}\Omega^\infty_{\mathcal{D}})^{\bullet}\Sigma^\infty_{\mathcal{D}}F)\bigr) \ar[d] \\
P_n\bigl((P_k G) F\bigr) \ar[r]\ar[d] & \mathrm{Tot}\bigl(P_n(P_k G\Omega^\infty_{\mathcal{D}}(\Sigma^\infty_{\mathcal{D}}\Omega^\infty_{\mathcal{D}})^{\bullet}\Sigma^\infty_{\mathcal{D}}F)\bigr) \ar[d] \\
P_n\bigl((P_{k-1} G) F\bigr) \ar[r] & \mathrm{Tot}\bigl(P_n(P_{k-1}G\Omega^\infty_{\mathcal{D}}(\Sigma^\infty_{\mathcal{D}}\Omega^\infty_{\mathcal{D}})^{\bullet}\Sigma^\infty_{\mathcal{D}}F)\bigr).
}
\]
Since $P_n$ commutes with finite limits and totalizations preserve limit diagrams, both columns are fiber sequences. By induction we may assume that the bottom horizontal map is an equivalence, the base of the induction being a consequence of the homogeneous case above; the homogeneous case also shows that the top horizontal map is an equivalence. We conclude that the map in the middle is an equivalence. The proposition follows by setting $k = n$.
\end{proof}

\begin{corollary}
\label{cor:cobarid}
For $\mathcal{C}$ a pointed compactly generated $\infty$-category, the canonical map
\begin{equation*}
\partial_n\mathrm{id}_{\mathcal{C}} \longrightarrow \mathrm{Tot}\bigl(\partial_n(\Omega^\infty_{\mathcal{C}}(\Sigma^\infty_{\mathcal{C}}\Omega^\infty_{\mathcal{C}})^{\bullet}\Sigma^\infty_{\mathcal{C}})\bigr)
\end{equation*}
is an equivalence.
\end{corollary}

\chapter{Truncations}
\label{sec:apptruncations}

In this section we provide several of the more technical proofs needed for the results of Chapter \ref{sec:coalgebras}. Specifically, we owe the reader proofs of Theorem \ref{thm:ntruncation} and of Propositions \ref{prop:SpPnCntruncated}, \ref{prop:truncatedtensor}, and \ref{prop:mapstruncations}.

In Section \ref{subsec:truncatedoperads} we start by investigating the homotopy theory of truncated $\infty$-operads using the formalism of dendroidal sets. Then, in Section \ref{subsec:truncatedstable}, we investigate the truncations of stable $\infty$-operads and prove Theorem \ref{thm:ntruncation} and Propositions \ref{prop:truncatedtensor} and \ref{prop:mapstruncations}. In Section \ref{subsec:cobar} we discuss the relation between the tensor products induced by the stable $\infty$-operad $\mathrm{Sp}(\mathcal{C})^\otimes$ and the derivatives of the identity functor of $\mathcal{C}$. We prove Proposition \ref{prop:derivativesid}, which we use several times in the body of this paper. This section also includes a proof of Proposition \ref{prop:SpPnCntruncated}. Section \ref{subsec:appendixcoalgebras} covers technical results on $\infty$-categories of coalgebras. In it we prove Lemma \ref{lem:algtaun+1}, which inductively describes $n$-truncated coalgebras in a stable $\infty$-operad, and Lemma \ref{lem:coalglimit}, which expresses the $\infty$-category of coalgebras as a limit of $\infty$-categories of truncated coalgebras.
 
We will write $\mathbf{Op}$ for the $\infty$-category of nonunital $\infty$-operads. In Lurie's formalism these are precisely the $\infty$-operads whose structure map to $\NFin$ factors through $\N\mathrm{Surj}$, where $\mathrm{Surj}$ denotes the category of finite pointed sets and surjective maps. By the results of \cite{hhm}, an equivalent way of describing the $\infty$-category $\mathbf{Op}$ is by using \emph{open} dendroidal sets. This second perspective will be more convenient when constructing truncations of operads. The basics of the theory of dendroidal sets are contained in \cite{cisinskimoerdijk1} and \cite{moerdijkweiss}. Also, \cite{hhm} contains a fairly comprehensive exposition of the background we need. It should be noted that only Section \ref{subsec:truncatedoperads} uses the formalism of dendroidal sets in a serious way. The subsequent sections consist mostly of more abstract manipulations with $\infty$-operads and could be carried out in any reasonable formalism for such as soon as the results of \ref{subsec:truncatedoperads} have been established.

\section{The homotopy theory of truncated $\infty$-operads}
\label{subsec:truncatedoperads}

Let us first fix notation and terminology concerning dendroidal sets. Write $\mathbf{\Omega}$ for the category of open rooted trees. Recall from \cite{hhm} that a tree is open if it contains no nullary vertices. This category is denoted by $\mathbf{\Omega}_o$ in \cite{hhm}, but since we will only work with open trees here we drop the subscript. Define $\mathbf{\Omega}_n$ to be the full subcategory of $\mathbf{\Omega}$ spanned by the trees with at most $n$ leaves; also, write $\mathbf{\Psi}_n$ for the full subcategory of $\mathbf{\Omega}$ spanned by those trees whose vertices have at most $n$ ingoing edges, so that we have inclusions $u: \mathbf{\Omega}_n \rightarrow \mathbf{\Psi}_n$ and $v: \mathbf{\Psi}_n \rightarrow \mathbf{\Omega}$. For the corresponding categories of presheaves we obtain adjunctions
\[
\xymatrix{
\mathbf{Sets}^{\mathbf{\Omega}_n^{\mathrm{op}}} \ar@<.5ex>[r]^{u_!} & \mathbf{Sets}^{\mathbf{\Psi}_n^{\mathrm{op}}} \ar@<.5ex>[l]^{u^*} \ar@<.5ex>[r]^{v_!} & \mathbf{Sets}^{\mathbf{\Omega}^{\mathrm{op}}}, \ar@<.5ex>[l]^{v^*}
}
\]
where $u^*$ and $v^*$ are the evident restriction functors. Note that $u_!$ and $v_!$ are fully faithful, so that we may regard the former two categories as full subcategories of the latter. Note that the inclusion of $\mathbf{\Delta} \rightarrow \mathbf{\Omega}$, obtained by considering $[n]$ as a linear tree with $n$ vertices and $n+1$ edges, factors through the two subcategories $\mathbf{\Omega}_n$ and $\mathbf{\Psi}_n$. As a consequence, the embedding of the category of simplicial sets into $\mathbf{Sets}^{\mathbf{\Omega}^{\mathrm{op}}}$ factors through $u_!$ and $v_!$.

\begin{remark}
Like $\mathbf{\Omega}$, both the categories $\mathbf{\Omega}_n$ and $\mathbf{\Psi}_n$ are generalized Reedy categories, with their Reedy structure inherited from $\mathbf{\Omega}$.
\end{remark}

The category $\mathbf{Sets}^{\mathbf{\Omega}^{\mathrm{op}}}$ is called the category of (open) dendroidal sets and here denoted $\mathbf{dSets}$; Cisinski and Moerdijk \cite{cisinskimoerdijk1} established a model structure on this category which in \cite{hhm} is referred to as the \emph{operadic model structure}. Recall the \emph{normal monomorphisms}, which are generated as a saturated class by the boundary inclusions of trees $\partial T \rightarrow T$.  Here $\partial T$ is the union of all faces of $T$. (We will not distinguish in notation between a tree $T$ and the dendroidal set that it represents.) An alternative characterization of these maps is as follows: a monomorphism $f: X \rightarrow Y$ of dendroidal sets is normal if for every $T$, the action of $\mathrm{Aut}(T)$ on $Y(T) - f(X(T))$ is free. Also, if $T$ is a tree and $e$ an inner edge of $T$, we write $\Lambda^e[T]$ for the \emph{inner horn} of $T$ associated to $e$; it is the union of all faces of $T$ except for the inner face corresponding to $e$. A dendroidal set $X$ is called an \emph{$\infty$-operad} if it has the right lifting property with respect to all inner horn inclusions $\Lambda^e[T] \rightarrow T$. The operadic model structure is characterized by the fact that its cofibrations are the normal monomorphisms and its fibrant objects are the $\infty$-operads.

There are evident analogues of the above definitions in the categories $\mathbf{Sets}^{\mathbf{\Omega}_n^{\mathrm{op}}}$ and $\mathbf{Sets}^{\mathbf{\Psi}_n^{\mathrm{op}}}$. It should be noted that in the case of $\mathbf{Sets}^{\mathbf{\Psi}_n^{\mathrm{op}}}$ the boundary of a tree $T$ only consists of those faces of the tree that are themselves contained in $\mathbf{\Psi}_n$; it may therefore only be a subobject of the boundary of $T$ considered as a dendroidal set. A similar comment applies to inner horns. With these definitions, the proofs for dendroidal sets (as for example given in \cite{cisinskimoerdijk1}) carry over verbatim to prove the following:

\begin{theorem}
The categories $\mathbf{Sets}^{\mathbf{\Omega}_n^{\mathrm{op}}}$ and $\mathbf{Sets}^{\mathbf{\Psi}_n^{\mathrm{op}}}$ admit model structures in which the cofibrations are the normal monomorphisms and the fibrant objects are those objects having the right lifting property with respect to inner horn inclusions.
\end{theorem}

We will also refer to the model structures of the previous theorem as the operadic model structures. 

\begin{remark}
It is straightforward to verify that the chain of Quillen equivalences connecting $\mathbf{dSets}$ and Lurie's model category of $\infty$-preoperads $\mathbf{POp}$ of \cite{hhm} restricts to a chain of equivalences between $\mathbf{sSets}^{\mathbf{\Omega}_n^{\mathrm{op}}}$ and the model category of marked simplicial sets over $(\NFin^{\leq n})^{\mathrm{\natural}}$ described in Remark \ref{rmk:adhoctruncation}. Therefore, we can use the operadic model structure on $\mathbf{Sets}^{\mathbf{\Omega}_n^{\mathrm{op}}}$ as a model for the homotopy theory of $n$-truncated $\infty$-operads. Recall that in Chapter \ref{sec:coalgebras} we denoted the corresponding $\infty$-category by $\mathbf{Op}_{\leq n}$.
\end{remark}

The functors $u^*$ and $v^*$ enjoy several pleasant properties with respect to the operadic model structures, summarized in the following two results:

\begin{theorem}
\label{thm:Omegan}
The functor $u^*: \mathbf{Sets}^{\mathbf{\Psi}_n^{\mathrm{op}}} \rightarrow \mathbf{Sets}^{\mathbf{\Omega}_n^{\mathrm{op}}}$ is both left and right Quillen. Furthermore, both the resulting adjunctions are Quillen equivalences.
\end{theorem}

\begin{theorem}
\label{thm:Psin}
The functor $v^*: \mathbf{dSets} \rightarrow \mathbf{Sets}^{\mathbf{\Psi}_n^{\mathrm{op}}}$ is both left and right Quillen.
\end{theorem}

In our proofs we will frequently use the notion of \emph{Segal core}, which is the dendroidal analogue of the spine of a simplex. Recall that a \emph{corolla} is a tree with precisely one vertex. We write $C_k$ for the corolla with $k$ leaves. If $T$ is a tree, then its Segal core $\mathrm{Sc}(T)$ is the union of its corollas; to be more precise, as a subpresheaf of the representable presheaf $T$ it is described by
\begin{equation*}
\mathrm{Sc}(T) = \bigcup_{v \in V(T)} C_{n(v)}
\end{equation*}
where $V(T)$ is the set of vertices of $T$ and $n(v)$ is the number of inputs of a vertex $v$. This definition makes sense in all three of the presheaf categories we consider; moreover, for $T \in \mathbf{\Omega}_n$ and $S \in \mathbf{\Psi}_n$ we have the compatibilities
\begin{equation*}
u_!(\mathrm{Sc}(T)) = \mathrm{Sc}(T) \quad\quad \text{and} \quad\quad v_!(\mathrm{Sc}(S)) = \mathrm{Sc}(S).
\end{equation*}

The following lemma will be an important tool. (Note that its analogue for simplicial sets is a well-known fact about the Joyal model structure.) It is a standard result for dendroidal sets, but the same argument applies to our other two presheaf categories:

\begin{lemma}
\label{lem:FleftSegcore}
Let $F$ be a left adjoint functor from either of the three categories $\mathbf{Sets}^{\mathbf{\Omega}_n^{\mathrm{op}}}$, $\mathbf{Sets}^{\mathbf{\Psi}_n^{\mathrm{op}}}$ or $\mathbf{dSets}$ to a model category $\mathcal{E}$ and suppose $F$ preserves cofibrations. If $F$ sends the maps below to weak equivalences, then $F$ is left Quillen:
\begin{itemize}
\item[(a)] For all $T \in \mathbf{\Omega}_n$ (resp. $\mathbf{\Psi}_n$, $\mathbf{\Omega}$), the Segal core inclusion $\mathrm{Sc}(T) \rightarrow T$.
\item[(b)] The inclusion $\{0\} \rightarrow J$, where $J$ denotes the nerve of the usual groupoid interval, i.e. the category with two objects $0$ and $1$ and an isomorphism between them.
\end{itemize}
\end{lemma}
\begin{proof}
Under the assumption that $F$ preserves cofibrations, a standard argument shows that it preserves trivial cofibrations if and only if its right adjoint preserves fibrations between fibrant objects. In either of the three model categories mentioned, the fibrations between fibrant objects are precisely the $J$-fibrations, i.e. the inner fibrations which also have the right lifting property with respect to the inclusion $\{0\} \rightarrow J$. Therefore it suffices to show that $F$ sends inner horn inclusions and the map $\{0\} \rightarrow J$ to trivial cofibrations. Reducing from inner horn inclusions to Segal core inclusions is for example done as in Proposition 3.6.8 of \cite{hhm}. To be precise, if $\mathcal{A}$ is a saturated class of cofibrations closed under two-out-of-three (among cofibrations) and contains all Segal core inclusions, then it contains all inner horn inclusions.
\end{proof}

Another standard result is the characterization of weak equivalences between fibrant dendroidal sets (see Theorem 3.5 of \cite{cisinskimoerdijk3}), which carries through without change to the two subcategories of $\mathbf{dSets}$ we consider. Recall that by using the tensor product of dendroidal sets, one can define simplicial mapping objects $\mathrm{Map}(X,Y)$ for dendroidal sets $X$ and $Y$. 

\begin{proposition}
\label{prop:wefibrants}
Let $f: X \rightarrow Y$ be a map of fibrant dendroidal sets. Then $f$ is a weak equivalence if and only if the following maps are homotopy equivalences of simplicial sets:
\begin{itemize}
\item[(a)] For any corolla $C_k$, the map $\mathrm{Map}(C_k, X) \rightarrow \mathrm{Map}(C_k, Y)$.
\item[(b)] The map $\mathrm{Map}(\eta, X) \rightarrow \mathrm{Map}(\eta, Y)$, where $\eta$ denotes the `trivial' tree with one edge and no vertices (or equivalently, the image of the 0-simplex $\Delta^0$ in $\mathbf{dSets}$).
\end{itemize}
The analogous statement is true in the model categories $\mathbf{Sets}^{\mathbf{\Omega}_n^{\mathrm{op}}}$ and $\mathbf{Sets}^{\mathbf{\Psi}_n^{\mathrm{op}}}$, where in (a) one only considers corollas $C_k$ with $k \leq n$.
\end{proposition}

\begin{remark}
Another way to state the previous proposition is to say that a map of $\infty$-operads is an equivalence if and only if it is fully faithful, which is part (a), and a weak equivalence on the level of underlying simplicial sets, which is part (b). Since fully faithfulness is already guaranteed by (a), the only additional information provided by (b) is that $f$ is essentially surjective.
\end{remark}

Before we prove Theorem \ref{thm:Omegan} we need a convenient description of the functor $u^*$. For a tree $T$, a \emph{subtree} of $T$ is an inclusion of trees $S \subseteq T$ which can be written as a composition of external face maps. In others words, $S$ is obtained from $T$ by iteratively chopping off leaf corollas and root corollas. Write $\mathrm{Sub}(T)$ for the diagram of subtrees of $T$ (i.e. the full subcategory of $\mathbf{\Omega}/T$ spanned by the subtrees). Also, write $\mathrm{Sub}_n(T)$ for the diagram of subtrees with at most $n$ leaves.

\begin{lemma}
\label{lem:u*}
For a tree $T \in \mathbf{\Psi}_n$, the natural map
\begin{equation*}
\varinjlim_{S \in \mathrm{Sub}_n(T)} S \longrightarrow u^*(T)
\end{equation*}
is an isomorphism.
\end{lemma}
\begin{proof}
Write $(\mathbf{\Omega}_n/T)^{\mathrm{nd}}$ for the category whose objects are non-degenerate maps $S \rightarrow T$, where $S$ is a tree with at most $n$ leaves. Arrows in this category are maps in $\mathbf{\Omega}_n$ compatible with the maps to $T$. Then
\begin{equation*}
u^*(T) \simeq \varinjlim_{S \in (\mathbf{\Omega}_n/T)^{\mathrm{nd}}} S.
\end{equation*}
To prove the lemma it suffices to show that the inclusion $\mathrm{Sub}_n(T) \rightarrow (\mathbf{\Omega}_n/T)^{\mathrm{nd}}$ is cofinal. By standard arguments, this follows if for every $S \in \mathbf{\Omega}_n/T$ the slice category 
\begin{equation*}
\mathrm{Sub}_n(T) \times_{(\mathbf{\Omega}_n/T)^{\mathrm{nd}}} S/(\mathbf{\Omega}_n/T)^{\mathrm{nd}}
\end{equation*}
is connected. In fact, this slice category is contractible: it has an initial object $\widetilde{S}$, which is the subtree of $T$ whose leaves are the (images of the) leaves of $S$ and whose root is the (image of the) root of $S$, so that in particular the map $S \rightarrow \widetilde{S}$ is a composition of inner face maps. 
\end{proof}

\begin{proof}[Proof of Theorem \ref{thm:Omegan}]
Proving that $u^*$ is right Quillen is equivalent to proving that $u_!$ is left Quillen, which we will do by verifying the assumptions of Lemma \ref{lem:FleftSegcore}. Since $u_!$ is the inclusion of a full subcategory, it is clear that it preserves normal monomorphisms. Furthermore, we already observed that for $T \in \mathbf{\Omega}_n$ we have the formula $u_!(\mathrm{Sc}(T)) = \mathrm{Sc}(T)$; also, $u_!$ sends the inclusion $\{0\} \rightarrow J$ to precisely the same map in $\mathbf{Sets}^{\mathbf{\Psi}_n^{\mathrm{op}}}$. 

To show that $u^*$ is also left Quillen, first observe that it preserves monomorphisms (as does any right adjoint). To check that $u^*$ preserves normal monomorphisms, we should check that a map of the form $u^*(\partial T \rightarrow T)$ is normal. Any monomorphism whose codomain is normal is a normal monomorphism; therefore, it suffices to show that $u^*(T)$ is normal for any $T \in \mathbf{\Psi}_n$, i.e. for any $S \in \mathbf{\Omega}_n$ the action of $\mathrm{Aut}(S)$ on $u^*(T)(S)$ should be free. But $u^*(T)(S) = \mathbf{Sets}^{\mathbf{\Psi}_n^{\mathrm{op}}}(u_!S, T)$, so that the freeness of this action follows from the fact that $T$ is normal in $\mathbf{Sets}^{\mathbf{\Psi}_n^{\mathrm{op}}}$. We should now verify that $u^*$ sends the maps of Lemma \ref{lem:FleftSegcore} to weak equivalences. This is immediate for the map in (b), since it is just sent to the corresponding map in $\mathbf{Sets}^{\mathbf{\Omega}_n^{\mathrm{op}}}$. The maps of (a) require a more elaborate argument.

Consider a map of the form $u^*(\mathrm{Sc}(T) \rightarrow T)$, where $T \in \mathbf{\Psi}_n$. We will factor this map into a sequence of maps, each of which we will show to be inner anodyne. Recall from Lemma \ref{lem:u*} the description of $u^*(T)$ as colimit over $\mathrm{Sub}_n(T)$. Write $\mathrm{Sub}^i_n(T)$ for the subdiagram of $\mathrm{Sub}_n(T)$ spanned by the subtrees with at most $i$ vertices and write $A_i$ for the corresponding colimit over this diagram. Then we obtain a sequence of inclusions
\begin{equation*}
u^*(\mathrm{Sc}(T)) = A_1 \subseteq A_2 \subseteq \cdots \subseteq A_N = u^*(T)
\end{equation*}
for sufficiently large $N$. Consider one of the maps $A_i \subseteq A_{i+1}$ and factor it further as
\begin{equation*}
A_i^1 \subseteq A_i^2 \subseteq \cdots \subseteq A_i^M = A_{i+1}
\end{equation*}
by adjoining the subtrees with $i + 1$ vertices one by one (in arbitrary order). Consider one of the inclusions $A_i^j \subseteq A_i^{j+1}$ in this sequence, given by adjoining a such a subtree $S$. It fits into a pushout square
\[
\xymatrix{
\partial^{\mathrm{ext}}S \ar[d]\ar[r] & A_i^j \ar[d] \\
S \ar[r] & A_i^{j+1},
}
\]
where $\partial^{\mathrm{ext}}S$ denotes the union of all external faces of $S$. Indeed, each external face is already contained in $A_i$ because it has one fewer vertex than $S$ itself, whereas $A_i$ cannot contain any of the inner faces of $S$ by the way we have set up our induction. It is a standard fact that the inclusion $\partial^{\mathrm{ext}}S \rightarrow S$ is an inner anodyne map (see for example \cite{cisinskimoerdijk3} or \cite{hhm}).

Finally, we will show that the Quillen adjunction $(u_!, u^*)$ is a Quillen equivalence. It suffices to show that the right derived functor $\mathbf{R}u^*$ detects weak equivalences and the derived unit $\mathrm{id} \rightarrow \mathbf{R}u^*\mathbf{L}u_!$ is a weak equivalence. Indeed, it then follows from the triangle identities for the adjunction that the derived counit is a weak equivalence as well. So, consider a map $f: X \rightarrow Y$ between fibrant objects of $\mathbf{Sets}^{\mathbf{\Psi}_n^{\mathrm{op}}}$ and assume that $u^*f$ is a weak equivalence. Using Proposition \ref{prop:wefibrants}, we need to check that $\mathrm{Map}(C_k, X) \rightarrow \mathrm{Map}(C_k, Y)$ and $\mathrm{Map}(\eta, X) \rightarrow \mathrm{Map}(\eta, Y)$ are homotopy equivalences for $k \leq n$. But we simply have isomorphisms $\mathrm{Map}(C_k, u^*X) \simeq \mathrm{Map}(C_k, X)$ and $\mathrm{Map}(\eta, u^*X) \simeq \mathrm{Map}(\eta, u^*Y)$ (and of course similarly for $Y$), so that this follows directly from the assumption that $u^*f$ is a weak equivalence. 

Let us now prove that the derived unit is a weak equivalence. Consider a cofibrant object $X \in \mathbf{Sets}^{\mathbf{\Omega}_n^{\mathrm{op}}}$ and pick a trivial cofibration $u_!X \rightarrow (u_!X)_f$ so that $(u_!X)_f$ is fibrant. We should verify that the composition
\begin{equation*}
X \rightarrow u^*u_!X \rightarrow u^*(u_!X)_f
\end{equation*}
is a weak equivalence. Since $u^*$ is also left Quillen, the second map is a trivial cofibration. Since $u_!$ is fully faithful, the first map is an isomorphism and the desired conclusion follows.
\end{proof}

To prove Theorem \ref{thm:Psin} it is convenient to have a description of $v^*(T)$ for a tree $T$. The presheaf $v^*(T)$ is a disjoint union of representable presheaves: it is obtained from $T$ by deleting all vertices of $T$ with more than $n$ inputs. In other words, $v^*(T)$ is the coproduct of the maximal subtrees of $T$ whose vertices have no more than $n$ inputs.

\begin{proof}[Proof of Theorem \ref{thm:Psin}]
Showing that $v^*$ is right Quillen is equivalent to showing that $v_!$ is left Quillen: as with $u_!$ (see the beginning of the previous proof) this is obvious given Lemma \ref{lem:FleftSegcore}. To show that $v^*$ is left Quillen, we first check that it preserves normal monomorphisms. Again this is much the same as in the previous proof: $v^*$ preserves monomorphisms since it is a right adjoint and sends trees $T$ to normal objects, which is clear from the description of $v^*(T)$ given above. To verify that $v^*$ sends the maps of Lemma \ref{lem:FleftSegcore} to weak equivalences, note that it sends the map of (b) to precisely the same map in $\mathbf{Sets}^{\mathbf{\Omega}_n^{\mathrm{op}}}$ and a Segal core inclusion $\mathrm{Sc}(T) \rightarrow T$ for $T \in \mathbf{\Omega}$ to the disjoint union of the Segal core inclusions of the trees making up $v^*(T)$. A disjoint union of trivial cofibrations is again a trivial cofibration, concluding the proof.
\end{proof}

The main reason to introduce the category of presheaves over $\mathbf{\Psi}_n$, rather than just over $\mathbf{\Omega}_n$, is that the pushforward functor $v_!: \mathbf{Sets}^{\mathbf{\Psi}_n^{\mathrm{op}}} \rightarrow \mathbf{dSets}$ behaves very well with respect to fibrant objects. This is perhaps surprising, since $v_!$ is a left Quillen functor, but definitely not a right Quillen functor (it is not even a right adjoint). To explain our results it is most convenient to consider dendroidal complete Segal spaces, rather than dendroidal sets. 

Consider the category $\mathbf{sSets}^{\mathbf{\Omega}^{\mathrm{op}}}$ of simplicial presheaves on $\mathbf{\Omega}$, which we will denote by $\mathbf{sdSets}$. Since $\mathbf{\Omega}^{\mathrm{op}}$ is a generalized Reedy category, $\mathbf{sdSets}$ admits a Reedy model structure. We can regard any dendroidal set as an object of $\mathbf{sdSets}$ through the embedding $\mathbf{dSets} \rightarrow \mathbf{sdSets}$ sending a presheaf to the corresponding constant simplicial presheaf.
The model category of complete dendroidal Segal spaces is then obtained by taking the Bousfield localization of the Reedy model structure with respect to the following maps:
\begin{itemize}
\item[(a)] For any tree $T$, the inclusion $\mathrm{Sc}(T) \rightarrow T$.
\item[(b)] The inclusion $\{0\} \rightarrow J$.
\end{itemize}
By analogy we can put a corresponding model structure on the category $\mathbf{sSets}^{\mathbf{\Psi}_n^{\mathrm{op}}}$, which we will refer as the model category of complete $\mathbf{\Psi}_n$-Segal spaces. Pushforward and pullback along $v$ define an adjunction between this category and $\mathbf{sdSets}$; we will again denote the resulting functors by $v_!$ and $v^*$. Also, we denote both the constant embeddings $\mathbf{dSets} \rightarrow \mathbf{sdSets}$ and $\mathbf{Sets}^{\mathbf{\Psi}_n^{\mathrm{op}}} \rightarrow \mathbf{sSets}^{\mathbf{\Psi}_n^{\mathrm{op}}}$ by $\mathrm{con}$.

\begin{proposition}
In the commutative square
\[
\xymatrix{
\mathbf{Sets}^{\mathbf{\Psi}_n^{\mathrm{op}}} \ar[r]^{v_!}\ar[d]_{\mathrm{con}} & \mathbf{dSets} \ar[d]^{\mathrm{con}} \\
\mathbf{sSets}^{\mathbf{\Psi}_n^{\mathrm{op}}} \ar[r]_{v_!} & \mathbf{sdSets}
}
\]
all arrows are left Quillen functors. Both vertical arrows are part of Quillen equivalences.
\end{proposition}
\begin{proof}
The fact that $\mathrm{con}$ is a left Quillen equivalence is proved for dendroidal complete Segal spaces (i.e. the vertical functor on the right) in \cite{cisinskimoerdijk3}; the proof for complete $\mathbf{\Psi}_n$-Segal spaces is identical. It is straightforward to see that the bottom horizontal arrow induces a left Quillen functor for the Reedy model structures on both categories. Furthermore, it sends the localizing morphisms in the $\mathbf{sSets}^{\mathbf{\Psi}_n^{\mathrm{op}}}$ to localizing morphisms in $\mathbf{dSets}$, so that it is also left Quillen with respect to the model structures for complete Segal spaces.
\end{proof}

The reason for considering these complete Segal spaces is that it allows us to formulate a useful technical property of $v_!$. It states that for a fibrant-cofibrant object $X \in  \mathbf{sSets}^{\mathbf{\Psi}_n^{\mathrm{op}}}$, the pushforward $v_!X$ is only a Reedy fibrant replacement away from being fibrant in $\mathbf{sdSets}$:

\begin{lemma}
\label{lem:v!fibrant}
Suppose $X \in \mathbf{sSets}^{\mathbf{\Psi}_n^{\mathrm{op}}}$ is a complete $\Psi_n$-Segal space (i.e. a Reedy fibrant simplicial presheaf local with respect to the maps (a) and (b) described above), which is also Reedy cofibrant. Write $(v_!X)_f$ for a Reedy fibrant replacement of $v_!X$. Then $(v_!X)_f$ is a complete dendroidal Segal space, i.e. a fibrant object of $\mathbf{sdSets}$.
\end{lemma}
\begin{proof}
Throughout this proof we will repeatedly use the fact that for every $T \in \mathbf{\Omega}$, the map $v_!X(T) \rightarrow (v_!X(T))_f$ is a weak equivalence of simplicial sets. We need to check that $(v_!X)_f$ is local with respect to $\{0\} \rightarrow J$ and Segal core inclusions $\mathrm{Sc}(T) \rightarrow T$. The first is clear from the fact $v_!$ induces an isomorphism $\mathrm{Map}(J, X) \simeq \mathrm{Map}(J, v_!X)$. For the second, consider a tree $T \in \mathbf{\Omega}$. Write $\mathrm{Dec}_n(T)$ for the category whose objects are maps of trees $T \rightarrow S$ which are compositions of inner face maps and where $S \in \mathbf{\Psi}_n$. Its morphisms are maps $S \rightarrow S'$ in $\mathbf{\Psi}_n$ compatible with the structure maps from $T$. Another way to phrase the condition that $T \rightarrow S$ is a composition of inner face maps is to say that this map is injective, sends the root of $T$ to the root of $S$ and gives a bijection between the leaves of $T$ and the leaves of $S$. We will refer to an object of $\mathrm{Dec}_n(T)$ as an $n$-decomposition of $T$. By Lemma \ref{lem:v!dec} below we have a weak equivalence of simplicial sets
\begin{equation*}
\mathrm{hocolim}_{S \in \mathrm{Dec}_n(T)^{\mathrm{op}}} X(S) \longrightarrow v_!X(T).
\end{equation*}
As before, write $V(T)$ for the set of vertices of $T$. For a vertex $x$, write $n(x)$ for the number of inputs of $x$. Observe that there is an equivalence of categories
\begin{equation*}
\gamma: \mathrm{Dec}_n(T) \longrightarrow \prod_{x \in V(T)} \mathrm{Dec}_n(C_{n(x)})
\end{equation*}
given by restricting an $n$-decomposition $f: T \rightarrow S$ to every corolla $C_{n(x)}$ of $T$ to obtain a corresponding decomposition $C_{n(x)} \rightarrow \gamma_x(S)$. Here $\gamma_x(S)$ is the subtree of $S$ with as its root the image under $f$ of the root of $C_{n(x)}$, and similarly for the leaves. Write $\mathrm{Sc}_{\gamma}(S)$ for the union of the subtrees $\gamma_x(S)$ in $S$, where $x$ ranges through $V(T)$. Using $\gamma$ we may form the commutative square
\[
\xymatrix{
\mathrm{hocolim}_{S \in \mathrm{Dec}_n(T)^{\mathrm{op}}} X(S) \ar[r]\ar[d] & v_!X(T) \ar[d]^{\psi} \\
\mathrm{hocolim}_{S \in \mathrm{Dec}_n(T)^{\mathrm{op}}} \mathrm{Map}(\mathrm{Sc}_{\gamma}(S), X) \ar[r]_-{\varphi} & \mathrm{Map}(\mathrm{Sc}(T), v_! X). 
}
\]
The bottom horizontal map $\varphi$ in this square can be built by iterated homotopy pullbacks from the maps
\begin{equation*}
 \mathrm{hocolim}_{S \in \mathrm{Dec}_n(C_{n(x)})^{\mathrm{op}}} \mathrm{Map}(\gamma_x(S), X) \longrightarrow \mathrm{Map}(C_{n(x)}, v_! X)
\end{equation*}
and
\begin{equation*}
\mathrm{Map}(\eta, X) \longrightarrow \mathrm{Map}(\eta, v_! X),
\end{equation*}
where we used the fact that $\eta$ does not admit nontrivial decompositions. Both these maps are weak equivalences (the second even an isomorphism), so that $\varphi$ is a weak equivalence as well. Since the left vertical map in the square is a weak equivalence by the assumption that $X$ is a $\mathbf{\Psi}_n$-Segal space, it follows that $\psi$ is a weak equivalence, so that $(v_!X)_f$ is local with respect to $\mathrm{Sc}(T) \rightarrow T$.
\end{proof}

In the previous proof we needed the following:

\begin{lemma}
\label{lem:v!dec}
For $X \in \mathbf{sSets}^{\mathbf{\Psi}_n^{\mathrm{op}}}$ and $T \in \mathbf{\Omega}$, the natural map
\begin{equation*}
\varinjlim_{S \in \mathrm{Dec}_n(T)^{\mathrm{op}}} X(S) \longrightarrow v_!X(T)
\end{equation*}
is an isomorphism. If $X$ is Reedy cofibrant and we replace the colimit above by a homotopy colimit, the resulting map is a weak equivalence of simplicial sets.
\end{lemma}
\begin{proof}
Write $(T / \mathbf{\Psi}_n)^{\mathrm{nd}}$ for the category of non-degenerate maps $T \rightarrow S$ in $\mathbf{\Omega}$ such that $S$ is in $\mathbf{\Psi}_n$. Then by definition we have the formula
\begin{equation*}
\varinjlim_{S \in ((T / \mathbf{\Psi}_n)^{\mathrm{nd}})^{\mathrm{op}}} X(S) \simeq v_!X(T).
\end{equation*}
The prove the lemma we should show that the inclusion
\begin{equation*}
\mathrm{Dec}_n(T)^{\mathrm{op}} \longrightarrow \bigl((T / \mathbf{\Psi}_n)^{\mathrm{nd}}\bigr)^{\mathrm{op}}
\end{equation*}
is cofinal (or homotopy cofinal, for the second part). Both these facts follow if we can show that for any object $f: T \rightarrow S$ of $(T / \mathbf{\Psi}_n)^{\mathrm{nd}}$, the slice category
\begin{equation*}
\mathrm{Dec}_n(T) \times_{(T / \mathbf{\Psi}_n)^{\mathrm{nd}}} (T / \mathbf{\Psi}_n)^{\mathrm{nd}}/f
\end{equation*}
is weakly contractible. In fact this category has a terminal object. Indeed, there is a unique factorization of $f$ as a composition $h \circ g$, where $h$ is a composition of inner face maps and $g$ a composition of external face maps. Then $g \in \mathrm{Dec}_n(T)$ is the desired terminal object. Said differently, if $R$ is the subtree of $S$ with as its root the image under $f$ of the root of $T$ and similarly for its leaves, then $g$ is the evident map $T \rightarrow R$.
\end{proof}

For ease of reference, let us record the following consequence of the proof of Lemma \ref{lem:v!fibrant}, which will be the essential step in proving Proposition \ref{prop:truncatedtensor}:

\begin{corollary}
\label{cor:mappingspacev!}
Let $X \in \mathbf{sSets}^{\mathbf{\Psi}_n^{\mathrm{op}}}$ be fibrant and cofibrant and let $Y$ be a fibrant replacement of $v_!X$. Then there is a natural weak equivalence
\begin{equation*}
\mathrm{hocolim}_{S \in \mathrm{Dec}_n(T)^{\mathrm{op}}} X(S) \longrightarrow Y(T).
\end{equation*}
\end{corollary}

The following is also a consequence of Lemma \ref{lem:v!fibrant}.

\begin{corollary}
\label{cor:unitv}
The derived unit of the Quillen adjunction
\[
\xymatrix{
\mathbf{sSets}^{\mathbf{\Psi}_n^{\mathrm{op}}} \ar@<.5ex>[r]^{v_!} & \mathbf{sdSets} \ar@<.5ex>[l]^{v^*}
}
\]
is a weak equivalence.
\end{corollary}
\begin{proof}
Consider $X \in \mathbf{sSets}^{\mathbf{\Psi}_n^{\mathrm{op}}}$ fibrant and cofibrant and consider a fibrant replacement $Y$ of $v_!X$. We should verify that the map $X \rightarrow v^*Y$ is a weak equivalence. Note that for every $T \in \mathbf{\Psi}_n$, the category $\mathrm{Dec}_n(T)$ has an initial object (namely the identity map of $T$), so that the map
\begin{equation*}
\mathrm{hocolim}_{S \in \mathrm{Dec}_n(T)^{\mathrm{op}}} X(S) \longrightarrow X(T)
\end{equation*}
is a weak equivalence. It follows that $X(T) \rightarrow v^*Y(T)$ is a weak equivalence for every such $T$.
\end{proof}

\begin{remark}
In terms of the corresponding $\infty$-categories $\mathbf{Op}_{\leq n}$ and $\mathbf{Op}$, the previous corollary states that the inclusion $\mathbf{Op}_{\leq n} \rightarrow \mathbf{Op}$ exhibits the former as a colocalization of the latter.
\end{remark}

We conclude this section with a result relating the $n$-truncation and the $(n-1)$-truncation of an $\infty$-operad. To state it, write $w$ for the inclusion $\mathbf{\Psi}_{n-1} \rightarrow \mathbf{\Psi}_n$, which induces an adjunction
\[
\xymatrix{
w_!: \mathbf{sSets}^{\mathbf{\Psi}_{n-1}^{\mathrm{op}}} \ar@<.5ex>[r] & \mathbf{sSets}^{\mathbf{\Psi}_n^{\mathrm{op}}}: w^*. \ar@<.5ex>[l]
}
\]
One sees this is a Quillen adjunction in the same way as for the pair $(v_!, v^*)$. Write $t_{n-1}$ for the functor $\mathbf{L}w_!\mathbf{R}w^*$. For $X \in \mathbf{sSets}^{\mathbf{\Psi}_n^{\mathrm{op}}}$ we wish to express the difference between $t_{n-1} X$ and $X$ in terms of \emph{$n$-homogeneous $\infty$-operads}. Informally speaking, an $n$-homogeneous operad is one that only has nontrivial operations of arities $1$ and $n$. To make this precise in the setting of $\infty$-operads, consider the full subcategory $g_n: \mathbf{\Gamma}_n \rightarrow \mathbf{\Psi}_n$ spanned by the trees $T$ which satisfy one of the following two conditions:
\begin{itemize}
\item[(1)] All the vertices of $T$ are unary, i.e. $T$ is just a simplex.
\item[(2)] All the vertices of $T$ except one are unary, where the non-unary vertex has valence $n$.
\end{itemize}
In particular, any $T \in \mathbf{\Gamma}_n$ has either one leaf or $n$ leaves. The category $\mathbf{\Gamma}_n$ inherits a generalized Reedy structure from $\mathbf{\Psi}_n$; again we may localize the category of simplicial presheaves over $\mathbf{\Gamma}_n$ with respect to Segal cores and $\{0\} \rightarrow J$ to obtain the homotopy theory of \emph{complete $\mathbf{\Gamma}_n$-Segal spaces}. The pushforward
\begin{equation*}
(g_n)_!: \mathbf{sSets}^{\mathbf{\Gamma}_n^{\mathrm{op}}} \longrightarrow \mathbf{sSets}^{\mathbf{\Psi}_n^{\mathrm{op}}}
\end{equation*}
is then a left Quillen functor. Let us write $h_n$ for the composite $\mathbf{L}(g_n)_!\mathbf{R}g_n^*$. We will sometimes refer to $h_n X$ as the \emph{$n$-homogeneous part of $X$}. The following result will be the key ingredient in proving Proposition \ref{prop:mapstruncations}:

\begin{proposition}
\label{prop:wnX}
For a fibrant object $X \in \mathbf{sSets}^{\mathbf{\Psi}_n^{\mathrm{op}}}$ the square
\[
\xymatrix{
h_n t_{n-1} X \ar[r]\ar[d] & t_{n-1} X \ar[d] \\
h_n X \ar[r] & X.
}
\]
is a homotopy pushout.
\end{proposition}
\begin{proof}
By Theorem \ref{thm:Omegan} it suffices to check that this square is a homotopy pushout after applying $\mathbf{L}u^*$, i.e. after restricting to $\mathbf{\Omega}_n$ (and Reedy cofibrant replacements, if necessary). We claim that the evaluation of the vertical map on the right at any tree $T \in \mathbf{\Omega}_n \cap \mathbf{\Psi}_{n-1}$ is a weak equivalence, which we will justify in the second part of this proof. To conclude the statement of the proposition, observe that the map $h_n X \rightarrow X$ is a homotopy equivalence for any tree $T \in \mathbf{\Gamma}_n$,  whereas $h_n X(T) = \varnothing$ for $T \notin \mathbf{\Gamma}_n$ (and similarly for the map $h_n t_{n-1} X \rightarrow t_{n-1} X$). Indeed, the proposition then follows from the observation that any $T \in \mathbf{\Omega}_n$ is contained in either $\mathbf{\Psi}_{n-1}$ or $\mathbf{\Gamma}_n$.

To verify the claim above it suffices to check that the map $\mathbf{R}w^* t_{n-1} X \rightarrow \mathbf{R}w^*X$ is a weak equivalence. The object $t_{n-1}X$ may be computed as $w_! (w^*X)_c$, where $(w^*X)_c$ is a Reedy cofibrant replacement of $w^*X$ (and is therefore in particular pointwise weakly equivalent to $w^*X$). Write $(w_! (w^*X)_c)_f$ for a Reedy fibrant replacement of this object, which is then a complete $\mathbf{\Psi}_n$-Segal space by Lemma \ref{lem:v!fibrant}. It follows that we may compute $\mathbf{R}w^* t_{n-1} X$ as $w^*(w_! (w^*X)_c)_f$. Our claim then follows from Corollary \ref{cor:unitv}.
\end{proof}

Observe that if $S \rightarrow T$ is a map of trees in $\mathbf{\Psi}_n$ such that $T \in \mathbf{\Gamma}_n$, then $S$ must be in $\mathbf{\Gamma}_n$ as well. It follows that for any such $T$, the inclusion
\begin{equation*}
\mathbf{sSets}^{\mathbf{\Gamma}_n^{\mathrm{op}}}/T \longrightarrow \mathbf{sSets}^{\mathbf{\Psi}_n^{\mathrm{op}}}/T
\end{equation*}
is an equivalence of categories. It is immediate from this observation that the pushforward functor $g_!$ preserves fibrant objects. This allows us to prove the following:

\begin{proposition}
\label{prop:wnXpullback}
The square of Proposition \ref{prop:wnX} is also a homotopy pullback.
\end{proposition}
\begin{proof}
As before it suffices to check this after restricting to $\mathbf{\Omega}_n$. It is clear from the descriptions of $h_n X$ and $h_n t_{n-1}X$ offered in the previous proof that the square is a homotopy pullback in the Reedy model structure on $\mathbf{sSets}^{\mathbf{\Omega}_n^{\mathrm{op}}}$. Without loss of generality we may assume the two objects on the right of the square are Reedy fibrant (and therefore also fibrant in the model structure for complete $\mathbf{\Omega}_n$-Segal spaces by virtue of Lemma \ref{lem:v!fibrant}). We observed above that the two objects on the left are fibrant as well. It follows that the square is also a homotopy pullback in the model structure for complete $\mathbf{\Omega}_n$-Segal spaces.
\end{proof}

\section{Truncations of stable $\infty$-operads}
\label{subsec:truncatedstable}

In this section we will adapt the theory of truncations to the setting of stable $\infty$-operads. Recall that we write $\mathbf{Op}^{\mathrm{St}}$ (resp. $\mathbf{Op}^{\mathrm{St}}_{\leq n}$) for the $\infty$-category of stable $\infty$-operads (resp. $n$-truncated stable $\infty$-operads). Theorem \ref{thm:Psin} shows that the restriction functor $\mathbf{Op} \rightarrow \mathbf{Op}_{\leq n}$ has both a left and right adjoint, for which we write $\mathbf{L}v_!$ and $\mathbf{R}v_*$ respectively. In this section we will need to vary $n$, so that we sometimes write $v_n$ for $v$ to make the dependence on $n$ explicit. The main result of this section is the following, which in particular proves Theorem \ref{thm:ntruncation}:

\begin{theorem}
\label{thm:stableleftadjoint}
The restriction functor $(-)_{\leq n}: \mathbf{Op}^{\mathrm{St}} \rightarrow \mathbf{Op}^{\mathrm{St}}_{\leq n}$ has both a fully faithful left adjoint (which we denote $i_n$) and a fully faithful right adjoint. Furthermore, for $\mathcal{O}^\otimes \in \mathbf{Op}^{\mathrm{St}}_{\leq n}$ and $\mathcal{N}^\otimes \in \mathbf{Op}^{\mathrm{St}}$, there is a natural weak equivalence
\begin{equation*}
\mathrm{Map}(i_n\mathcal{O}^\otimes, \mathcal{N}^\otimes) \simeq \mathrm{Map}(\mathbf{L}(v_n)_!\mathcal{O}^\otimes, \mathcal{N}^\otimes).
\end{equation*} 
\end{theorem}
\begin{remark}
The $\infty$-operad $\mathbf{L}(v_n)_!\mathcal{O}^\otimes$ need not be stable itself; however, the previous result shows that when mapping into a stable $\infty$-operad, the difference between $i_n$ and $\mathbf{L}(v_n)_!$ is irrelevant.
\end{remark}

\begin{proof}[Proof of Theorem \ref{thm:stableleftadjoint}] 
Consider the square of functors
\[
\xymatrix{
\mathbf{Op} \ar[r] & \mathbf{Op}_{\leq n} \\
\mathbf{Op}^{\mathrm{St}} \ar[r]_-{(-)_{\leq n}}\ar[u] & \mathbf{Op}^{\mathrm{St}}_{\leq n} \ar[u]
}
\]
in which the vertical arrows are fully faithful. We will first show the existence of a right adjoint to the bottom horizontal arrow. For this it suffices to show that for $\mathcal{O}^\otimes \in \mathbf{Op}^{\mathrm{St}}_{\leq n}$, the $\infty$-operad $\mathbf{R}v_*\mathcal{O}^\otimes$ is stable as well, so that the restriction of $\mathbf{R}v_*$ to stable $\infty$-operads gives the desired adjoint. The underlying $\infty$-categories of $\mathcal{O}$ and $\mathbf{R}v_*\mathcal{O}^\otimes$ agree. Furthermore, it is clear directly from the definition of $v_*$ that for any collection of objects $x_1, \ldots, x_k, y$ of $\mathcal{O}$ with $k \leq n$, the space of operations $\mathbf{R}v_*\mathcal{O}^\otimes(x_1, \ldots, x_k; y)$ is naturally equivalent to $\mathcal{O}^\otimes(x_1, \ldots, x_k; y)$, whereas $\mathbf{R}v_*\mathcal{O}^\otimes(x_1, \ldots, x_k; y)$ is contractible for $k > n$. It follows that $\mathbf{R}v_*\mathcal{O}^\otimes$ is indeed stable. Fully faithfulness of $\mathbf{R}v_*$ is also clear from this description.

For the existence of the left adjoint we appeal to the adjoint functor theorem (Corollary 5.5.2.9 of \cite{htt}). A formal argument shows that the $\infty$-categories $\mathbf{Op}^{\mathrm{St}}$ and $\mathbf{Op}^{\mathrm{St}}_{\leq n}$ are presentable, so that it suffices to show that the functor $\mathbf{Op}^{\mathrm{St}} \rightarrow \mathbf{Op}^{\mathrm{St}}_{\leq n}$ is accessible and preserves small limits. Accessibility is immediate from the fact that this functor admits a right adjoint and thus preserves small colimits. To show preservation of limits, it suffices to show that the other three functors in the square above preserve small limits (and use that the inclusion $\mathbf{Op}^{\mathrm{St}} \rightarrow \mathbf{Op}$ is fully faithful). We already know this for the functor $\mathbf{Op} \rightarrow \mathbf{Op}_{\leq n}$. For the vertical functors, we need to argue that the class of ($n$-truncated) stable $\infty$-operads is closed under taking small limits in $\mathbf{Op}$ (or $\mathbf{Op}_{\leq n}$). We do this for $\mathbf{Op}$, the other case being entirely analogous. Let $I \rightarrow \mathbf{Op}^{\mathrm{St}}: i \mapsto \mathcal{O}_i^\otimes$ be a diagram and let $\mathcal{O}^\otimes$ be a limit of the induced diagram $I \rightarrow \mathbf{Op}^{\mathrm{St}} \rightarrow \mathbf{Op}$. Taking underlying $\infty$-categories gives a limit-preserving functor $\mathbf{Op} \rightarrow \widehat{\mathbf{Cat}}$. We may conclude that the underlying $\infty$-category of $\mathcal{O}^\otimes$ is compacty generated by Proposition 5.5.7.6 of \cite{htt}; also, it is manifestly stable. Write $G_i: \mathcal{O} \rightarrow \mathcal{O}_i$ for the induced functor and $F_i$ for its left adjoint. For a collection of objects $x_1, \ldots, x_k, y$ of $\mathcal{O}$ we have natural equivalences
\begin{eqnarray*}
\mathcal{O}^\otimes(x_1, \ldots, x_k; y) & \simeq & \varprojlim_I \mathcal{O}_i^\otimes(G_i(x_1), \ldots, G_i(x_k); G_i(y)) \\
& \simeq & \varprojlim_I \mathcal{O}_i(G_i(x_1) \otimes_{\mathcal{O}_i} \cdots \otimes_{\mathcal{O}_i} G_i(x_k), G_i(y)) \\
& \simeq & \varprojlim_I \mathcal{O}(F_i(G_i(x_1) \otimes_{\mathcal{O}_i} \cdots \otimes_{\mathcal{O}_i} G_i(x_k)), y),
\end{eqnarray*}
showing that the functor $\mathcal{O}^\otimes(x_1, \ldots, x_k; -)$ is corepresented by the object 
\begin{equation*}
\varinjlim_I F_i(G_i(x_1) \otimes_{\mathcal{O}_i} \cdots \otimes_{\mathcal{O}_i} G_i(x_k)).
\end{equation*}
We need to show that the tensor products induced by $\mathcal{O}^\otimes$ preserve colimits in each variable separately. This follows from the above formula and the fact that the functors $G_i$ also preserve colimits. Indeed, they preserve filtered colimits by assumption and finite colimits since they are exact (being limit-preserving functors between stable $\infty$-categories). We conclude that $\mathcal{O}^\otimes$ is indeed a stable $\infty$-operad. 

The natural equivalence of the theorem arises from the natural equivalences
\begin{equation*}
\mathrm{Map}(i_n\mathcal{O}^\otimes, \mathcal{N}^\otimes) \simeq \mathrm{Map}(\mathcal{O}^\otimes, \mathcal{N}^\otimes_{\leq n}) \simeq \mathrm{Map}(\mathbf{L}v_!\mathcal{O}^\otimes, \mathcal{N}^\otimes).
\end{equation*}
Finally, we will show that $i_n$ is fully faithful by demonstrating that for each $\mathcal{O}^\otimes \in \mathbf{Op}^{\mathrm{St}}_{\leq n}$, the unit map $\mathcal{O}^\otimes \rightarrow (i_n\mathcal{O}^\otimes)_{\leq n}$ is an equivalence. This follows by considering, for any $\mathcal{N}^\otimes \in \mathbf{Op}^{\mathrm{St}}_{\leq n}$, the sequence of natural equivalences
\begin{eqnarray*}
\mathrm{Map}((i_n\mathcal{O}^\otimes)_{\leq n}, \mathcal{N}^\otimes) & \simeq & \mathrm{Map}(i_n\mathcal{O}^\otimes, \mathbf{R}v_*\mathcal{N}^\otimes) \\
& \simeq & \mathrm{Map}(\mathbf{L}v_!\mathcal{O}^\otimes, \mathbf{R}v_*\mathcal{N}^\otimes) \\
& \simeq & \mathrm{Map}(\mathcal{O}^\otimes, (\mathbf{R}v_*\mathcal{N}^\otimes)_{\leq n}) \\
& \simeq & \mathrm{Map}(\mathcal{O}^\otimes, \mathcal{N}^\otimes)
\end{eqnarray*}
where in the last step we used the fully faithfulness of $\mathbf{R}v_*$.
\end{proof}

Consider a stable $\infty$-operad $\mathcal{O}^\otimes$ with associated tensor products $\otimes^k$. Before stating our next result we introduce some notation. Let $T$ be an object in $\mathbf{\Psi}_n$. Then we can inductively define a tensor product $\otimes^T$ as follows: if $T$ is a corolla $C_k$, then $\otimes^T = \otimes^k$, and if $T$ is obtained by grafting a corolla $C_k$ onto a leaf $l$ of a smaller tree $T'$, then 
\begin{equation*}
\otimes^{T} = \otimes^{T'} \circ (\mathrm{id}, \ldots, \otimes^k, \ldots, \mathrm{id}), 
\end{equation*}
where on the right-hand side $\otimes^k$ occurs in the slot corresponding to $l$ and the identity functor occurs in all others. The tensor products $\otimes^T$ are covariantly functorial in $T$. Note that the definition of $\otimes^T$ only involves the tensor products $\otimes^k$ for $k \leq n$. Also, for a tree $S \in \mathbf{\Omega}$, recall the category $\mathrm{Dec}_n(S)$ of $n$-decompositions of $S$. Its objects are the maps of trees $f: S \rightarrow T$ such that $T \in \mathbf{\Psi}_n$ and $f$ is a composition of inner face maps. 

\begin{lemma}
\label{lem:odotk}
Consider $\mathcal{O}^\otimes$ as above and write $\odot^k$ for the $k$-fold tensor product determined by the stable $\infty$-operad $\tau_n\mathcal{O}^\otimes = i_n(\mathcal{O}^\otimes_{\leq n})$. Then there is a natural equivalence
\begin{equation*}
\odot^k \longrightarrow \varprojlim_{T \in \mathrm{Dec}_n(C_k)} \otimes^T.
\end{equation*}
\end{lemma}

\begin{remark}
It should be noted that the nerve of the category $\mathrm{Dec}_n(C_k)$ has the homotopy type of a finite simplicial set: indeed, since any automorphism of a tree $T$ is completely determined by its action on the leaves of $T$, the objects of $\mathrm{Dec}_n(C_k)$ have no nontrivial automorphisms. Furthermore, if $k \geq 3$ the length of a chain of decompositions $C_k \rightarrow T_1 \rightarrow \cdots \rightarrow T_m$ containing no isomorphisms is bounded above by $m = k-2$ (corresponding to a `maximal binary expansion' of $C_k$). In case $k=2$ the category $\mathrm{Dec}_n(C_2)$ is equivalent to the trivial category, since there are no nontrivial decompositions of a binary vertex.
\end{remark}

Before proving the lemma, let us fix another piece of convenient notation. In the previous section we introduced $n$-homogeneous $\infty$-operads by means of the homotopy theory of complete $\mathbf{\Gamma}_n$-Segal spaces. Write $\mathbf{Op}_{=n}$ for the $\infty$-category associated to this model category and write $\mathbf{Op}_{=n}^{\mathrm{St}}$ for its full subcategory spanned by the stable $n$-homogenous $\infty$-operads, i.e. those complete $\mathbf{\Gamma}_n$-Segal spaces satisfying the evident versions of the axioms imposed on stable $\infty$-operads. Pullback along the inclusion $\mathbf{\Gamma}_n \rightarrow \mathbf{\Omega}_n$ then defines a functor
\begin{equation*}
\mathbf{Op}_{\leq n} \longrightarrow \mathbf{Op}_{=n}: \mathcal{O}^\otimes_{\leq n} \longmapsto \mathcal{O}^\otimes_{=n}
\end{equation*}
which restricts to a functor $\mathbf{Op}^{\mathrm{St}}_{\leq n} \longrightarrow \mathbf{Op}^{\mathrm{St}}_{=n}$.

\begin{proof}[Proof of Lemma \ref{lem:odotk}]
We will prove the lemma by induction on $k$. To establish the base of the induction, note that the lemma is true for $k \leq n$. Indeed, since $i_n$ is fully faithful, the unit map $\mathcal{O}^\otimes_{\leq n} \rightarrow (\tau_n\mathcal{O}^\otimes)_{\leq n}$ is an equivalence, implying that for $k \leq n$ the natural transformations $\otimes^k \rightarrow \odot^k$ are equivalences. Note that in this case the category $\mathrm{Dec}_n(C_k)$ has an initial object (namely the identity map of $C_k$), so that the formula of the lemma holds true.

We first prove the lemma for $k = n+1$. Write $X = (\tau_n\mathcal{O}^\otimes)_{\leq n+1}$. Then Proposition \ref{prop:wnX} gives a pushout square
\[
\xymatrix{
h_{n+1}t_n X \ar[d]\ar[r] & t_n X \ar[d] \\
h_{n+1} X \ar[r] & X.
}
\]
For an arbitrary stable $\infty$-operad $\mathcal{N}^\otimes$, we then find a pullback square 
\[
\xymatrix{
\mathrm{Map}_{\mathcal{O}}(X, \mathcal{N}^\otimes_{\leq n+1}) \ar[r]\ar[d] & \mathrm{Map}_{\mathcal{O}}(X_{=n+1}, \mathcal{N}^\otimes_{=n+1}) \ar[d] \\
\mathrm{Map}_{\mathcal{O}}(t_n X, \mathcal{N}^\otimes_{\leq n+1}) \ar[r] & \mathrm{Map}_{\mathcal{O}}((t_n X)_{=n+1}, \mathcal{N}^\otimes_{=n+1})
}
\]
where $\mathrm{Map}_{\mathcal{O}}$ denotes the space of maps of $\infty$-operads which restrict to the identity $\mathcal{O} \rightarrow \mathcal{N}$ on underlying $\infty$-categories. Note that $\mathrm{Map}_{\mathcal{O}}(t_n X, \mathcal{N}^\otimes_{\leq n+1}) \simeq \mathrm{Map}_{\mathcal{O}}(\mathbf{L}(v_n)_! \mathcal{O}^\otimes_{\leq n}, \mathcal{N}^\otimes)$ and applying the equivalence of Theorem \ref{thm:stableleftadjoint} this is in turn naturally equivalent to $\mathrm{Map}_{\mathcal{O}}(\tau_n\mathcal{O}^\otimes, \mathcal{N}^\otimes)$. Therefore, the left vertical map in the square is an equivalence, so that the homotopy fiber of the vertical map on the right is contractible. The spaces on the right-hand side can be understood using the complete $\mathbf{\Gamma}_{n+1}$-Segal spaces of the previous section. For example, the space $\mathrm{Map}_{\mathcal{O}}(X_{=n+1}, \mathcal{N}^\otimes_{=n+1})$ is the space of $\Sigma_{n+1}$-equivariant natural transformations
\begin{equation*}
\tau_n\mathcal{O}^\otimes(x_1, \ldots, x_{n+1}; y) \longrightarrow \mathcal{N}^\otimes(x_1, \ldots, x_{n+1}; y),
\end{equation*}
where both sides are considered as functors $(\mathcal{O}^\mathrm{op})^{n+1} \times \mathcal{O} \rightarrow \mathcal{S}$. By the corepresentability of the $\infty$-operads involved, this space is equivalent to the space of $\Sigma_{n+1}$-equivariant natural transformations
\begin{equation*}
\otimes_{\mathcal{N}}^{n+1} \longrightarrow \odot^{n+1}.
\end{equation*}
Similarly, the space $\mathrm{Map}_{\mathcal{O}}((t_n X)_{=n+1}, \mathcal{N}^\otimes_{=n+1})$ is the space of $\Sigma_{n+1}$-equivariant natural transformations
\begin{equation*}
t_n\mathcal{O}^\otimes(x_1, \ldots, x_{n+1}; y) \longrightarrow \mathcal{N}^\otimes(x_1, \ldots, x_{n+1}; y).
\end{equation*}
Applying Corollary \ref{cor:mappingspacev!} we see that this space is naturally equivalent to the space of $\Sigma_{n+1}$-equivariant natural transformations
\begin{equation*}
\otimes_{\mathcal{N}}^{n+1} \longrightarrow \varprojlim_{T \in \mathrm{Dec}_n(C_{n+1})} \otimes_{\mathcal{O}}^T.
\end{equation*}
We conclude that the natural transformation
\begin{equation*}
\odot^{n+1} \longrightarrow \varprojlim_{T \in \mathrm{Dec}_n(C_{n+1})} \otimes_{\mathcal{O}}^T
\end{equation*}
is an equivalence, which proves the case $k = n+1$ of the lemma. More generally, the same argument as above shows that for $k > n+1$ there is an equivalence
\begin{equation*}
\odot^k \longrightarrow \varprojlim_{T \in \mathrm{Dec}_{k-1}(C_k)} \odot^T.
\end{equation*}
To carry out our induction, suppose we have proved the formula of the lemma for $k-1$. Then for any $T \in \mathbf{\Psi}_{k-1}$ the natural map
\begin{equation*}
\odot^T \longrightarrow \varprojlim_{S \in \mathrm{Dec}_n(T)} \otimes^S
\end{equation*}
is an equivalence; indeed, this follows from the decomposition $\gamma$ of the category $\mathrm{Dec}_n(T)$ discussed in the proof of Lemma \ref{lem:v!fibrant} and the fact that the tensor product functors $\otimes^S$ preserve finite limits in each variable separately, since $\mathcal{O}^\otimes$ is stable. We therefore have an equivalence
\begin{equation*}
\odot^k \longrightarrow \varprojlim_{T \in \mathrm{Dec}_{k-1}(C_k)}\Bigl(\varprojlim_{S \in \mathrm{Dec}_n(T)} \otimes^S\Bigr).
\end{equation*}
Write $\mathrm{Dec}_n^+(C_k)$ for the category indexing the limit on the right, i.e. the category whose objects are $C_k \rightarrow T \rightarrow S$, with $T$ a $(k-1)$-decomposition of $C_k$ and $S$ an $n$-decomposition of $T$, with the evident maps between them. To prove the lemma we should show that the functor
\begin{equation*}
\mathrm{Dec}_n^+(C_k) \longrightarrow \mathrm{Dec}_n(C_k): (C_k \rightarrow T \rightarrow S) \longmapsto (C_k \rightarrow S)
\end{equation*}
is final. This is clear, since for any $n$-decomposition $f: C_k \rightarrow S$, the slice
\begin{equation*}
\mathrm{Dec}_n^+(C_k) \times_{\mathrm{Dec}_n(C_k)} \mathrm{Dec}_n(C_k)/f
\end{equation*}
is contractible; indeed, it has a final object $C_k \rightarrow S = S$.
\end{proof}

The technique we just used to prove Lemma \ref{lem:odotk} in particular proves the following:

\begin{lemma}
\label{lem:homogstable}
If $\mathcal{O}^\otimes_{\leq n}$ and $\mathcal{N}_{\leq n}^\otimes$ are objects of $\mathbf{Op}_{\leq n}^{\mathrm{St}}$, then there is a natural equivalence
\begin{equation*}
\mathrm{Map}((\tau_{n-1}\mathcal{O}^\otimes)_{=n}, \mathcal{N}^\otimes_{=n}) \simeq \mathrm{Map}((t_{n-1}\mathcal{O}^\otimes)_{=n}, \mathcal{N}_{=n}^\otimes).
\end{equation*}
\end{lemma}
\begin{proof}
Assume for simplicity that $\mathcal{O}^\otimes_{\leq n}$ and $\mathcal{N}_{\leq n}^\otimes$ have the same underlying $\infty$-category $\mathcal{O}$ (the general case is only notationally more difficult). Then $\mathrm{Map}_{\mathcal{O}}((\tau_{n-1}\mathcal{O}^\otimes)_{=n}, \mathcal{N}^\otimes_{=n})$ can be identified with the space of $\Sigma_n$-equivariant natural transformations $\otimes_\mathcal{N}^n \rightarrow \odot_{\mathcal{O}}^n$, where the codomain denotes the $n$-fold tensor product in $\tau_{n-1}\mathcal{O}^\otimes$. That this space is naturally equivalent to the space on the right follows from Corollary \ref{cor:mappingspacev!} (applied to $\mathbf{L}(v_{n-1})_!\mathcal{O}^\otimes_{\leq n-1}$ and $T = C_n$) and Lemma \ref{lem:odotk}.
\end{proof}

\begin{corollary}
\label{cor:wnXstable}
If $\mathcal{O}^\otimes$ is a stable $\infty$-operad, then the square
\[
\xymatrix{
h_n(\tau_{n-1}\mathcal{O}^\otimes_{\leq n}) \ar[r]\ar[d] & \tau_{n-1}\mathcal{O}^\otimes_{\leq n} \ar[d] \\
h_n(\mathcal{O}^\otimes_{\leq n}) \ar[r] & \mathcal{O}^\otimes_{\leq n}
}
\]
is a pushout in the $\infty$-category $\mathbf{Op}_{\leq n}$.
\end{corollary}
\begin{proof}
It is straightforward to verify that the pushout in the square is itself a stable $\infty$-operad. To see that it is equivalent to $\mathcal{O}_{\leq n}^\otimes$, combine Proposition \ref{prop:wnX} with the equivalences provided by Theorem \ref{thm:stableleftadjoint} and Lemma \ref{lem:homogstable}.
\end{proof}

We end this section by proving the following results claimed in Chapter \ref{sec:coalgebras}:
\begin{proof}[Proof of Proposition \ref{prop:mapstruncations}]
Consider stable $\infty$-operads $\mathcal{O}^\otimes$ and $\mathcal{N}^\otimes$ and assume they have the same underlying $\infty$-category $\mathcal{O}$. The previous corollary yields a pullback square
\[
\xymatrix{
\mathrm{Map}_{\mathcal{O}}(\tau_n\mathcal{O}^\otimes, \mathcal{N}^\otimes) \ar[d]\ar[r] & \mathrm{Map}_{\mathcal{O}}(\mathcal{O}^\otimes_{=n}, \mathcal{N}_{=n}^\otimes) \ar[d] \\
\mathrm{Map}_{\mathcal{O}}(\tau_{n-1}\mathcal{O}^\otimes, \mathcal{N}^\otimes) \ar[r] & \mathrm{Map}_{\mathcal{O}}((\tau_{n-1}\mathcal{O}^\otimes)_{= n}, \mathcal{N}_{= n}^\otimes).
}
\]
The space $\mathrm{Map}_{\mathcal{O}}(\mathcal{O}^\otimes_{=n}, \mathcal{N}_{=n}^\otimes)$ is equivalent to the space of $\Sigma_n$-equivariant natural transformations $\otimes_{\mathcal{N}}^n \rightarrow \otimes_{\mathcal{O}}^n$. Also, the space $\mathrm{Map}_{\mathcal{O}}((\tau_{n-1}\mathcal{O}^\otimes)_{= n}, \mathcal{N}_{= n}^\otimes)$ is equivalent to the space of $\Sigma_n$-equivariant natural transformations $\otimes_{\mathcal{N}}^n \rightarrow \odot_{\mathcal{O}}^n$, where the latter denotes the $n$-fold tensor product determined by $\tau_{n-1}\mathcal{O}^\otimes$.
\end{proof}

\begin{proof}[Proof of Proposition \ref{prop:truncatedtensor}]
We will deduce the proposition from Lemma \ref{lem:odotk} and a variation on some well-known observations relating partition complexes to spaces of trees. First of all there is a functor
\begin{equation*}
\omega: \mathbf{Part}_n(k) \rightarrow \mathrm{Dec}_n(k)
\end{equation*}
which may be described as follows. For a vertex of $\mathbf{Part}_n(k)$ corresponding to a simplex $\zeta: \Delta^m \rightarrow \N\mathbf{Equiv}(k)$, we define a tree $T_{\zeta}$ whose set of edges is the disjoint union
\begin{equation*}
\coprod_{i=0}^m \{1, \ldots, k\}/\zeta(i). 
\end{equation*}
Furthermore, for any $1 \leq i \leq m$ and $e \in \{1, \ldots, k\}/\zeta(i)$ the tree $T_{\zeta}$ has a vertex with outgoing edge $e$ and ingoing edges all those $l \in \{1, \ldots, k\}/\zeta(i-1)$ which are sent to $e$ under the quotient map
\begin{equation*}
\{1, \ldots, k\}/\zeta(i-1) \longrightarrow \{1, \ldots, k\}/\zeta(i).
\end{equation*}
The map $\omega(\zeta)$ is then the unique $n$-decomposition $C_k \rightarrow T_{\zeta}$ which sends the $j$th leaf of $C_k$ to the leaf $j \in \{1, \ldots, k\}$ of $T_{\zeta}$. Note that the diagram $\psi_n^k: \mathbf{Part}_n(k) \rightarrow \mathrm{Fun}(\mathcal{O}^\otimes_{\langle k \rangle}, \mathcal{O}^\otimes_{\langle 1 \rangle})$ may be identified with the composition of (the nerve of) $\omega$ with the assignment $(C_k \rightarrow T) \mapsto \otimes^T$. Thus to prove the proposition it suffices to show that $\omega$ is final, i.e. that for any $n$-decomposition $f: C_k \longrightarrow T$ the slice category
\begin{equation*}
\mathbf{Part}_n(k) \times_{\mathrm{Dec}_n(k)} \mathrm{Dec}_n(k)/f
\end{equation*}
has contractible nerve. This will require a fair amount of combinatorics. We write $L(f)$ for the category above.

First observe that the functor $L(f) \rightarrow \mathrm{Dec}_n(k)/f$ is faithful. The objects in its image are those $n$-decompositions $g: C_k \rightarrow T_{\zeta}$ (lying over $f$) such that each leaf of $T_{\zeta}$ has the same \emph{height}; here the height of an edge $e$ of $T_{\zeta}$ is defined as the number of vertices on the directed path from $e$ to the root of $T_{\zeta}$. We will call such a tree \emph{layered}. The morphisms in the image are those maps of trees $\varphi$ satisfying the condition that if two edges have the same height, their images under $\varphi$ have the same height as well. Write $L(f)'$ for the full subcategory of $L(f)$ spanned by those objects corresponding to nondegenerate simplices of $\N\mathbf{Equiv}(k)$. It is standard (and easy to show) that the inclusion $\N L(f)' \subseteq \N L(f)$ is a weak homotopy equivalence, so that it suffices to show that the former is contractible. Also observe that the category $L(f)'$ is a partially ordered set.

Let us say an object $X$ of $L(f)'$ is \emph{elementary} if it is maximal in the sense that there are no non-identity maps out of $X$. A more explicit description is as follows. The object $X$ corresponds to a simplex $\zeta: \Delta^m \rightarrow \N\mathbf{Equiv}(k)$ and a triangle
\[
\xymatrix{
& C_k \ar[dl]_g\ar[dr]^f & \\
T_{\zeta} \ar[rr]_t && T, 
}
\]
with $T_{\zeta}$ the corresponding layered tree. Then for each $1 \leq i \leq m$ the map
\begin{equation*}
\{1, \ldots, k\}/\zeta(i-1) \rightarrow \{1, \ldots, k\}/\zeta(i)
\end{equation*}
corresponds to a collection of corollas of $T_{\zeta}$, one for each element $e$ of $\{1, \ldots, k\}/\zeta(i)$. Writing $C_e$ for the corolla corresponding to $e$, the object $X$ is elementary if for all but precisely one such $e$, the corolla $t(C_e)$ is a degenerate corolla of $T$ and moreover $t$ is surjective (i.e. a composition of degeneracies). More informally, $X$ is elementary if each layer of $T_{\zeta}$ contains precisely one corolla whose image under $t$ is nondegenerate in $T$ and $t$ is surjective. If $Y \rightarrow X$ is a map in $L(f)'$, let us say that $Y$ is a \emph{face} of $X$. The reason for this terminology is as follows: if $X$ is as above then the simplex of $\N\mathbf{Equiv}(k)$ corresponding to $Y$ is a face of $\zeta$. Observe that every object of $L(f)'$ is a face of an elementary object, although not in a unique way (i.e. an object may be a face of multiple elementary objects). If $X$ is elementary, write $\mathrm{face}(X)$ for the full subcategory of $L(f)'$ spanned by the faces of $X$.

We will prove that $\N L(f)'$ is contractible by an induction on elementary objects. To do this we need to order them. The set $V(T)$ of vertices of $T$ is partially ordered in an evident way, where $v<w$ if $v$ is on the directed path from $w$ to the root of $T$ (i.e. if $v$ is `lower' than $w$). Observe that elementary objects of $L(f)'$ are in one-to-one correspondence with linear orderings on $V(T)$ which extend this partial ordering. Indeed, if $X$ is elementary then the height function of $T_{\zeta}$ induces such a linear order, where $v < w$ if the preimage of $w$ in $T_{\zeta}$ (which is unique) is higher than that of $v$ (where the height of a vertex is by definition the height of its outgoing edge). Conversely, it is clear that every such linear order on $V(T)$ arises from the height function of a layered tree. 

Arbitrarily pick an elementary object $X_0 \in L(f)'$, determining a linear order $<_{X_0}$ on $V(T)$. We may construct the orderings corresponding to other elementary objects from $<_{X_0}$ by `shuffling': to be precise, we say that an elementary object $Z$ is a \emph{swap} of another elementary object $Y$ if the linear order $<_Z$ is obtained from $<_Y$ by swapping two consecutive vertices $v <_Y w$ (note that these vertices must be incomparable in the tree $T$ for this to be possible). Moreover, let us say $Y < Z$ if we have the inequality $v <_{X_0} w$. In other words $Z$ is obtained from $Y$ by `shuffling up' the vertex $v$, where the meaning of `up' is determined by $X_0$. In this way we generate a partial ordering on the set of elementary objects of $L(f)'$ in which $X_0$ is minimal.

Now set $F_0 = \mathrm{face}(X_0)$. Choose a linear ordering on the set of elementary objects of $L(f)'$ extending the partial order just described and define $F_i = F_{i-1} \cup \mathrm{face}(X_i)$, where $X_i$ is the $i$th elementary object. In this way we obtain a filtration 
\begin{equation*}
F_0 \subseteq F_1 \subseteq \cdots \subseteq \bigcup_{i} F_i = L(f)'.
\end{equation*}
Note that $F_0$ is weakly contractible since it has a final object $X_0$. Consider the stage $F_i$ of this filtration, which is obtained from $F_{i-1}$ by adjoining an elementary object $X_i$ and all its faces. We claim (see below) that the intersection $\mathrm{face}(X_i) \cap F_{i-1}$ is weakly contractible. By induction we may assume $F_{i-1}$ is weakly contractible. Since $\mathrm{face}(X_i)$ is weakly contractible, we conclude that $F_i$ is weakly contractible as well. It follows that $L(f)'$ is weakly contractible.

To verify our claim, note that $X_i$ is obtained from an elementary object $X_j$, for some $j < i$, by swapping two consecutive vertices $v <_{X_j} w$. Write $\zeta_i$ and $\zeta_j$ for the simplices of $\N\mathbf{Equiv}(k)$ corresponding to $X_i$ and $X_j$ respectively. Then the outgoing edge of the vertex $w$ corresponds to an element of 
\begin{equation*}
\{1, \ldots, k\}/\zeta_j(h)
\end{equation*}
for some positive integer $h$. Observe that $\partial_h \zeta_j = \partial_h \zeta_i$, which corresponds to a simplex where $v$ and $w$ are at the same height. Write $\partial_h X_j$ (or equivalently $\partial_h X_i$) for the corresponding object of $L(f)'$. Then the intersection of $\mathrm{face}(X_i) \cap F_{i-1}$ is the full subcategory of $L(f)'$ spanned by the faces of $\partial_h X_j$. This category has $\partial_h X_j$ as a terminal object and is therefore weakly contractible.
\end{proof}

\section{A cobar construction for stable $\infty$-operads}
\label{subsec:cobar}



In this section we prove Proposition \ref{prop:SpPnCntruncated}, which states that for a pointed compactly generated $\infty$-category $\mathcal{C}$ the maps $\tau_n\mathrm{Sp}(\mathcal{P}_n\mathcal{C})^\otimes \rightarrow \mathrm{Sp}(\mathcal{P}_n\mathcal{C})^\otimes$ and $\tau_n\mathrm{Sp}(\mathcal{P}_n\mathcal{C})^\otimes \rightarrow \tau_n\mathrm{Sp}(\mathcal{C})^\otimes$ are equivalences. To do this we need to exploit the relationship between the tensor products defined by the stable $\infty$-operad $\mathrm{Sp}(\mathcal{C})^\otimes$ and the derivatives of the identity functor of $\mathcal{C}$. Recall that the tensor product $\otimes^k$ can be identified with the derivative $\partial_k(\Sigma_{\mathcal{C}}^\infty\Omega_\mathcal{C}^\infty)$. Corollary \ref{cor:cobarid} allows us to compute the derivatives $\partial_k\mathrm{id}_{\mathcal{C}}$ from these tensor products by a cobar construction, which we will explain below. We will deduce Proposition \ref{prop:SpPnCntruncated} from the following: 

\begin{proposition}
\label{prop:derivativesid}
Fix $k \geq 2$ and write $\odot^k$ for the $k$-fold tensor product in the $\infty$-operad $\tau_{k-1}\mathrm{Sp}(\mathcal{C})^\otimes$. Then there is a natural equivalence 
\begin{equation*}
\partial_k \mathrm{id}_{\mathcal{C}} \longrightarrow \Omega\mathrm{fib}(\otimes^k \rightarrow \odot^k).
\end{equation*}
\end{proposition}
\begin{proof}
We can compute the derivatives $\partial_\ast(\Omega^\infty_{\mathcal{C}}(\Sigma^\infty_{\mathcal{C}}\Omega^\infty_{\mathcal{C}})^{\bullet}\Sigma^\infty_{\mathcal{C}})$ featuring in Corollary \ref{cor:cobarid} using the chain rule (see Theorem 6.3.2.1 of \cite{higheralgebra}). They are given by the (somewhat informal) formula
\begin{equation*}
\partial_\ast\Omega_{\mathcal{C}}^\infty \circ \partial_\ast(\Sigma_{\mathcal{C}}^\infty\Omega_{\mathcal{C}}^\infty)^{\circ \bullet} \circ \partial_\ast\Sigma_{\mathcal{C}}^\infty
\end{equation*}
with $\circ$ denoting the composition product of symmetric sequences; a rigorous justification of this formula can be extracted from Section 6.3.2 of \cite{higheralgebra}, using the notion of thin $\Delta^n$-families of $\infty$-operads. (Note that we do not need to distinguish between the derivatives and coderivatives of $\Sigma_{\mathcal{C}}^\infty\Omega_{\mathcal{C}}^\infty$ by virtue of Remark \ref{rmk:derivcoderiv}.) We denote the totalization of the corresponding cosimplicial object by 
\begin{equation*}
\mathrm{cobar}(\partial_\ast\Omega_{\mathcal{C}}^\infty, \partial_\ast(\Sigma_{\mathcal{C}}^\infty\Omega_{\mathcal{C}}^\infty), \partial_\ast\Sigma_{\mathcal{C}}^\infty).
\end{equation*}
The derivative $\partial_i\Omega_{\mathcal{C}}^\infty$ is the identity for $i=1$ and is 0 otherwise (and similarly for $\partial_i\Sigma_{\mathcal{C}}^\infty$); let us write $\mathbf{1}$ for the corresponding symmetric sequence. The symmetric sequence $\partial_\ast(\Sigma_{\mathcal{C}}^\infty\Omega_{\mathcal{C}}^\infty)$ admits an evident augmentation to $\mathbf{1}$ and we write $\overline{\partial_\ast(\Sigma_{\mathcal{C}}^\infty\Omega_{\mathcal{C}}^\infty)}$ for the fiber of this map. It agrees with $\partial_\ast(\Sigma_{\mathcal{C}}^\infty\Omega_{\mathcal{C}}^\infty)$ in every degree except 1, in which it is trivial. By standard reasoning we may replace the cobar construction above by a \emph{reduced} cobar construction, giving (for $k \geq 2$) the following formula for $\partial_k \mathrm{id}_{\mathcal{C}}$ in terms of the totalization of a semicosimplicial object:

\[
\xymatrix{
\mathrm{Tot}\Bigl( \ast \ar@<.5ex>[r]^-0\ar@<-.5ex>[r]_-0 & \partial_k(\Sigma_{\mathcal{C}}^\infty\Omega_{\mathcal{C}}^\infty) \ar@<1ex>[r]^-0\ar[r]\ar@<-1ex>[r]_-0 & \bigl(\overline{\partial_\ast(\Sigma_{\mathcal{C}}^\infty\Omega_{\mathcal{C}}^\infty)} \circ \overline{\partial_\ast(\Sigma_{\mathcal{C}}^\infty\Omega_{\mathcal{C}}^\infty)}\bigr)_k  \ar@<1.5ex>[r]^-0\ar@<.5ex>[r]\ar@<-.5ex>[r]\ar@<-1.5ex>[r]_-0 & \cdots \Bigr).
}
\]
The bottom and top coface maps are always null (indeed, these are induced by the comultiplication $\mathbf{1} \rightarrow \mathbf{1} \circ \overline{\partial_\ast(\Sigma_{\mathcal{C}}^\infty\Omega_{\mathcal{C}}^\infty)}$, which is null). A standard argument (formally dual to the Kan simplicial suspension) then gives an equivalence
\[
\xymatrix{
\partial_k\mathrm{id}_{\mathcal{C}} \ar[r] & \Omega\mathrm{Tot}\Bigl(\partial_k(\Sigma_{\mathcal{C}}^\infty\Omega_{\mathcal{C}}^\infty) \ar@<.5ex>[r]\ar@<-.5ex>[r]_-0 & \bigl(\overline{\partial_\ast(\Sigma_{\mathcal{C}}^\infty\Omega_{\mathcal{C}}^\infty)} \circ \overline{\partial_\ast(\Sigma_{\mathcal{C}}^\infty\Omega_{\mathcal{C}}^\infty)}\bigr)_k  \ar@<1ex>[r]\ar[r]\ar@<-1ex>[r]_-0 & \cdots \Bigr), 
}
\]
where the cosimplicial object in brackets is the d\'{e}calage of the previous one, i.e. its composition with the functor
\begin{equation*}
\mathbf{\Delta} \longrightarrow \mathbf{\Delta}: [n] \longmapsto [0] \amalg [n] = [n+1].
\end{equation*}
We conclude that there is an equivalence between $\Sigma\partial_k\mathrm{id}_{\mathcal{C}}$ and the fiber of the comultiplication map
\[
\xymatrix{
\partial_k (\Sigma_{\mathcal{C}}^\infty\Omega_{\mathcal{C}}^\infty) \ar[r] & \mathrm{Tot}\Bigl(\bigl(\overline{\partial_\ast(\Sigma_{\mathcal{C}}^\infty\Omega_{\mathcal{C}}^\infty)}^{\circ 2}\bigr)_k  \ar@<.5ex>[r]\ar@<-.5ex>[r] & \bigl(\overline{\partial_\ast(\Sigma_{\mathcal{C}}^\infty\Omega_{\mathcal{C}}^\infty)}^{\circ 3}\bigr)_k  \ar@<1ex>[r]\ar[r]\ar@<-1ex>[r] & \cdots \Bigr)
}
\]
by applying another d\'{e}calage, now composing with $[n] \mapsto [n] \amalg [0]$. Observe that $\partial_k(\Sigma_{\mathcal{C}}^\infty\Omega_{\mathcal{C}}^\infty)$ may be identified with $\otimes^k$ and that the totalization on the right may be identified with the following limit (see Section \ref{subsec:truncations}):
\begin{equation*}
\varprojlim_{\N\mathbf{Part}_{k-1}(k)}\psi_{k-1}^k.
\end{equation*}
The desired result now follows by applying Proposition \ref{prop:truncatedtensor}.
\end{proof}

\begin{proof}[Proof of Proposition \ref{prop:SpPnCntruncated}]
To prove that $\tau_n\mathrm{Sp}(\mathcal{P}_n\mathcal{C})^\otimes \rightarrow \tau_n\mathrm{Sp}(\mathcal{C})^\otimes$ is an equivalence it suffices to show that $\mathrm{Sp}(\mathcal{P}_n\mathcal{C})^\otimes_{\leq n} \rightarrow \mathrm{Sp}(\mathcal{C})^\otimes_{\leq n}$ is an equivalence of $n$-truncated stable $\infty$-operads. The tensor products induced by $\mathrm{Sp}(\mathcal{P}_n\mathcal{C})^\otimes_{\leq n}$ are the first $n$ derivatives of the functor $\Sigma^\infty_{\mathcal{P}_n\mathcal{C}}\Omega^\infty_{\mathcal{P}_n\mathcal{C}}$, the tensor products induced by  $\mathrm{Sp}(\mathcal{C})^\otimes_{\leq n}$ are the first $n$ derivatives of $\Sigma^\infty_{\mathcal{C}}\Omega^\infty_{\mathcal{C}}$. That the natural map between these is an equivalence is immediate from the fact that $\mathcal{P}_n\mathcal{C}$ is a weak $n$-excisive approximation to $\mathcal{C}$.

We now show that the map $\tau_n\mathrm{Sp}(\mathcal{P}_n\mathcal{C})^\otimes \rightarrow \mathrm{Sp}(\mathcal{P}_n\mathcal{C})^\otimes$ is an equivalence of stable $\infty$-operads. We will do this by proving that for every $k \geq n$ the map
\begin{equation*}
f_k: \tau_n \mathrm{Sp}(\mathcal{P}_n\mathcal{C})^\otimes \longrightarrow \tau_k\mathrm{Sp}(\mathcal{P}_n\mathcal{C})^\otimes
\end{equation*}
is an equivalence. This is clear for $k = n$. To establish the inductive step, assume that $k > n$ and that $f_{k-1}$ is an equivalence. The derivative $\partial_k \mathrm{id}_{\mathcal{P}_n\mathcal{C}}$ is contractible since $\mathrm{id}_{\mathcal{P}_n\mathcal{C}}$ is $n$-excisive, so Proposition \ref{prop:derivativesid} (applied to $\mathcal{P}_n\mathcal{C}$) implies that the natural transformation $\otimes^k \rightarrow \odot^k$ is an equivalence. Hence $f_k$ is an equivalence as well.
\end{proof}

\section{Coalgebras}
\label{subsec:appendixcoalgebras}

In this section we prove Lemma \ref{lem:algtaun+1}, which gives an inductive description of coalgebras in a stable $\infty$-operad of the form $\tau_n \mathcal{O}^\otimes$. Consider the $\infty$-category $\mathcal{B}$ defined as the pullback in the following square:
\[
\xymatrix{
\mathcal{B} \ar[r]\ar[d] & \bigl\{X \rightarrow (X \otimes^n \cdots \otimes^n X)^{h\Sigma_n} \bigr\}^c_{\mathcal{O}} \ar[d] \\
\mathrm{coAlg}^c(\tau_{n-1}\mathcal{O}^\otimes) \ar[r] & \bigl\{X \rightarrow (X \odot^n \cdots \odot^n X)^{h\Sigma_n} \bigr\}^c_{\mathcal{O}}.
}
\]
As observed in Section \ref{subsec:truncatedcoalgebras} there is an evident functor
\begin{equation*}
\beta: \mathrm{coAlg}^c(\tau_n\mathcal{O}^\otimes) \longrightarrow \mathcal{B}.
\end{equation*}

The following is a reformulation of Lemma \ref{lem:algtaun+1}:

\begin{lemma}
\label{lem:beta}
The functor $\beta$ is an equivalence of $\infty$-categories.
\end{lemma}

Before we prove this lemma we investigate the $\infty$-category of coalgebras in an $n$-homogeneous stable $\infty$-operad. To this end, consider a map $p^\otimes: \mathcal{X}^\otimes_{=n} \rightarrow \mathcal{O}^\otimes_{=n}$ in the $\infty$-category $\mathbf{Op}^{\mathrm{St}}_{=n}$. We call such a map a coalgebra in $\mathcal{O}^\otimes_{=n}$ if it satisfies the evident analogue of Definition \ref{def:coalgebra} and write $\mathrm{coAlg}(\mathcal{O}^\otimes_{=n})$ for the $\infty$-category of such coalgebras. Informally speaking, a coalgebra in $\mathcal{O}^\otimes_{=n}$ is simply an object $X$ of $\mathcal{O}$ together with a $\Sigma_n$-equivariant map $X \rightarrow X^{\otimes n}$. This is articulated by the following result:

\begin{lemma}
\label{lem:homogcoalgebras}
The evident functor
\begin{equation*}
\mathrm{coAlg}^c(\mathcal{O}^\otimes_{=n}) \longrightarrow \bigl\{X \rightarrow (X \otimes^n \cdots \otimes^n X)^{h\Sigma_n} \bigr\}^c_{\mathcal{O}} 
\end{equation*}
is an equivalence of $\infty$-categories.
\end{lemma}
\begin{proof}
We first show that this functor is fully faitfhul. Consider compact coalgebras $p^\otimes: \mathcal{X}_{=n}^\otimes \rightarrow \mathcal{O}^\otimes_{=n}$ and $q^\otimes: \mathcal{Y}_{=n}^\otimes \rightarrow \mathcal{O}^\otimes_{=n}$ with underlying objects $X$ and $Y$ respectively. Write $\mathrm{Map}(p^\otimes, q^\otimes)$ for the space of maps in $\mathrm{coAlg}^c(\mathcal{O}_{=n}^\otimes)$ between these coalgebras, which is by definition the space of maps $\mathcal{Y}_{=n}^\otimes \rightarrow \mathcal{X}_{=n}^\otimes$ compatible with the maps down to $\mathcal{O}_{=n}^\otimes$. The forgetful functor induces a map
\begin{equation*}
\mathrm{Map}(p^\otimes, q^\otimes) \longrightarrow \mathrm{Map}_{\mathcal{O}}(X, Y).
\end{equation*}
Write $F_{\varphi}$ for the fiber of this map over a morphism $\varphi: X \rightarrow Y$ and consider objects $f_i: Y \rightarrow W_i$ and $g: Y \rightarrow Z$ of $\mathcal{O}_{Y/}$. Informally speaking $F_{\varphi}$ may be described as the space of $\Sigma_n$-equivariant maps
\begin{equation*}
\mathcal{Y}^\otimes_{=n}(f_1, \ldots, f_n; g) \longrightarrow \mathcal{X}^\otimes_{=n}(f_1\varphi, \ldots, f_n\varphi; g\varphi),
\end{equation*}
natural in the $f_i$ and $g$, which are moreover compatible with the maps down to $\mathcal{O}^\otimes(W_1, \ldots, W_n; Z)$. By the corepresentability of the $\infty$-operads involved this is the space of $\Sigma_n$-equivariant triangles
\[
\xymatrix@C=45pt{
X \ar[rr]^{\varphi} \ar[dr]_-{\delta_n^X(f_1\varphi, \ldots, f_n\varphi)} && Y \ar[dl]^-{\delta_n^Y(f_1, \ldots, f_n)} \\
& W_1 \otimes \cdots \otimes W_n &
}
\]
natural in the $f_i$. Here the map $\delta^Y_n(f_1, \ldots, f_n): Y \rightarrow W_1 \otimes \cdots \otimes W_n$ is induced by the tensor product of $\mathcal{Y}^{\otimes}_{=n}$ and similarly for $\delta^X_n$. To be more precise, consider the space of $\Sigma_n$-equivariant natural transformations between the functors $\delta_n^X \circ (\varphi^*)^n$ and $\varphi^*\delta_n^Y$. The $\infty$-category $\mathrm{Fun}(\mathcal{O}_{Y/}^n, \mathcal{O}_{X/})$ admits an equivariant functor to $\mathrm{Fun}(\mathcal{O}_{Y/}^n, \mathcal{O})$ by composition with the projection $\pi_X: \mathcal{O}_{X/} \rightarrow \mathcal{O}$. The composition $\otimes^n \circ \pi_Y^n$ is an object of the latter and $F_{\varphi}$ can be identified with the fiber of $\mathrm{Nat}^{\Sigma_n}(\delta_n^X \circ (\varphi^*)^n, \varphi^*\delta_n^Y)$ over the identity map of that object. Write $\widetilde{F}_{\varphi}$ for the fiber of $\mathrm{Fun}(\mathcal{O}_{Y/}^n, \mathcal{O}_{X/})$ over the object $\otimes^n \circ \pi_Y^n$ of $\mathrm{Fun}(\mathcal{O}_{Y/}^n, \mathcal{O})$. Since the inclusion of the vertex $\{\mathrm{id}_Y\}^n \rightarrow \mathcal{O}_{Y/}^n$ is left anodyne and the projection $\mathcal{O}_{X/} \rightarrow \mathcal{O}$ is a left fibration, evaluation at that vertex induces an equivalence between $\widetilde{F}_{\varphi}$ and the fiber of $\mathcal{O}_{X/} \rightarrow \mathcal{O}$ over the object $Y^{\otimes n}$. Hence there is a corresponding fiber sequence
\begin{equation*}
F_{\varphi} \longrightarrow \mathrm{Map}^{\Sigma_n}_{\mathcal{O}_{X/}}(\delta_n^X(\varphi, \ldots, \varphi), \varphi \circ \delta_n^Y(\mathrm{id}_Y, \ldots, \mathrm{id}_Y)) \longrightarrow \mathrm{Map}^{\Sigma_n}_{\mathcal{O}}(Y^{\otimes n}, Y^{\otimes n}).
\end{equation*}
Note also that the map $\delta_n^X(\varphi, \ldots, \varphi)$ canonically factors as follows
\[
\xymatrix@C=60pt{
X \ar[r]^-{\delta_n^X(\mathrm{id}_X, \ldots, \mathrm{id}_X)} & X \otimes \cdots \otimes X \ar[r]^{\varphi \otimes \cdots \otimes \varphi} & Y \otimes \cdots \otimes Y. 
}
\]
Therefore one may describe $F_{\varphi}$ as the space of homotopies that fill the following square:
\[
\xymatrix{
X \ar[r]^{\varphi}\ar[d]_{} & Y \ar[d] \\
(X \otimes \cdots \otimes X)^{h\Sigma_n} \ar[r]_{(\varphi^{\otimes n})^{\Sigma_n}} & (Y \otimes \cdots \otimes Y)^{h\Sigma_n}.
}
\]
It should be clear that this is precisely the homotopy type of the space of maps between the objects 
\begin{equation*}
\delta^X_n(\mathrm{id}_X, \ldots, \mathrm{id}_X): X \longrightarrow X^{\otimes n} \quad\quad \text{and} \quad\quad \delta^Y_n(\mathrm{id}_Y, \ldots, \mathrm{id}_Y): Y \longrightarrow Y^{\otimes n},
\end{equation*}
compatible with the map $\varphi: X \rightarrow Y$, in the $\infty$-category 
\begin{equation*}
\bigl\{X \rightarrow (X \otimes^n \cdots \otimes^n X)^{\Sigma_n} \bigr\}^c_{\mathcal{O}}.
\end{equation*}
This proves that the functor of the lemma is fully faithful.

It remains to show essential surjectivity. If $\delta_n: X \rightarrow (X^{\otimes n})^{h\Sigma_n}$ is a map in $\mathcal{O}$ we define a corresponding coalgebra $\mathcal{X}^\otimes_{=n} \rightarrow \mathcal{O}^\otimes_{=n}$ as follows. Unraveling the definitions, it suffices to specify a $\Sigma_n$-equivariant functor associating to every $n$-tuple of maps $f_1: X \rightarrow Y_i, \ldots, f_n: X \rightarrow Y_n$ a map $X \rightarrow Y_1 \otimes \cdots \otimes Y_n$. Clearly the formation of the composition
\[
\xymatrix@C=35pt{
X \ar[r] & X^{\otimes n} \ar[r]^-{f_1 \otimes \cdots \otimes f_n} & Y_1 \otimes \cdots \otimes Y_n
}
\end{equation*}
does the job.
\end{proof}

We will also need the following collection of technical facts:

\begin{lemma}
\label{lem:Xtruncated}
Let $p^\otimes: \mathcal{X}^\otimes \rightarrow \tau_n\mathcal{O}^\otimes$ be a coalgebra in $\tau_n\mathcal{O}^\otimes$. Then the map $\tau_n\mathcal{X}^\otimes \rightarrow \mathcal{X}^\otimes$ is an equivalence of $\infty$-operads. Furthermore, the pushforward of $p^\otimes$ to $\mathrm{coAlg}(\tau_{n-1}\mathcal{O}^\otimes)$ is equivalent to $\tau_{n-1}p^\otimes: \tau_{n-1}\mathcal{X}^\otimes \rightarrow \tau_{n-1}\mathcal{O}^\otimes$. Also, the map $\mathcal{X}^\otimes_{=n} \rightarrow (\tau_n\mathcal{O}^\otimes)_{=n}$ defines a coalgebra in $(\tau_n\mathcal{O}^\otimes)_{=n}$.
\end{lemma}
\begin{remark}
We are cutting a corner in the statement of this lemma: the $\infty$-operad $\mathcal{X}^\otimes$ need itself not be stable, so that the expression $\tau_n\mathcal{X}^\otimes$ is as of yet ill-defined. However, the relevant definitions can easily be adapted to apply to this corepresentable $\infty$-operad as well. We will not belabour the details here.
\end{remark}
\begin{proof}
Write $\otimes^k_{\mathcal{X}}$ for the $k$-fold tensor product on $\mathcal{O}_{X/}$ defined by $p^\otimes$. To prove that $\tau_n\mathcal{X}^\otimes \rightarrow \mathcal{X}^\otimes$ is an equivalence, it suffices (by Lemma \ref{lem:odotk}) to show that for every $k > n$ the natural map
\begin{equation*}
\varphi_k: \otimes^k_{\mathcal{X}} \longrightarrow \varprojlim_{T \in \mathrm{Dec}_n(C_k)} \otimes^T_{\mathcal{X}}
\end{equation*}
is an equivalence. Let $f_1: X \rightarrow Y_1, \ldots, f_k: X \rightarrow Y_k$ be objects of $\mathcal{O}_{X/}$. Evaluated at this tuple, the map above corresponds to a triangle in $\mathcal{O}$ as follows:
\[
\xymatrix{
& X \ar[dr]\ar[dl]& \\
\otimes^k_{\mathcal{O}}(Y_1, \ldots, Y_k) \ar[rr] &&  \varprojlim_{T \in \mathrm{Dec}_n(C_k)} \otimes^T_{\mathcal{O}}(Y_1, \ldots, Y_k).
}
\]
Then $\varphi_k$ is an equivalence if and only if the bottom arrow is an equivalence, which is the case because $\tau_n\mathcal{O}^\otimes$ is $n$-truncated. For the second claim, recall that the pushforward of $p^\otimes$ to $\mathrm{coAlg}(\tau_{n-1}\mathcal{O}^\otimes)$ (say $p_{n-1}^\otimes$) is defined by the pullback square
\[
\xymatrix{
\mathcal{X}_{n-1}^\otimes \ar[d]_{p_{n-1}^\otimes} \ar[r] & \mathcal{X}^\otimes \ar[d]^{p^\otimes} \\
\tau_{n-1}\mathcal{O}^\otimes \ar[r] & \tau_n\mathcal{O}^\otimes.
}
\]
Applying the truncation functor $(-)_{\leq n-1}$ to the square and using that it preserves limits, we see that $(\mathcal{X}_{n-1}^\otimes)_{\leq n-1} \rightarrow (\mathcal{X}^\otimes)_{\leq n-1}$ is an equivalence. Furthermore, the map $\tau_{n-1}\mathcal{X}_{n-1}^\otimes \rightarrow \mathcal{X}_{n-1}^\otimes$ is an equivalence by what we proved before, so that we obtain an equivalence $\mathcal{X}_{n-1}^\otimes \rightarrow \tau_{n-1}\mathcal{X}^\otimes$ over $\tau_{n-1}\mathcal{O}^\otimes$. Finally, the claim in the lemma about $\mathcal{X}^\otimes_{=n}$ is immediate from the definitions.
\end{proof}

\begin{remark}
\label{rmk:truncatedcoalg}
A consequence of the previous lemma is that for a coalgebra $\mathcal{X}^\otimes \rightarrow \tau_n\mathcal{O}^\otimes$, there is no essential loss of information in passing to the map $\mathcal{X}^\otimes_{\leq n} \rightarrow \tau_n\mathcal{O}^\otimes_{\leq n}$. More precisely, there is an evident definition of the notion of coalgebra in the $n$-truncated stable $\infty$-operad $\tau_n\mathcal{O}^\otimes_{\leq n}$ and a consequence of the previous lemma is that the pullback functor
\begin{equation*}
\mathrm{coAlg}(\tau_n\mathcal{O}^\otimes) \longrightarrow \mathrm{coAlg}(\tau_n\mathcal{O}^\otimes_{\leq n})
\end{equation*}
is an equivalence of $\infty$-categories.
\end{remark}

\begin{proof}[Proof of Lemma \ref{lem:beta}]
By Lemma \ref{lem:homogcoalgebras} and the previous remark it suffices to show that the square
\[
\xymatrix{
\mathrm{coAlg}^c(\tau_n\mathcal{O}^\otimes_{\leq n}) \ar[r]\ar[d] & \mathrm{coAlg}^c(\mathcal{O}_{=n}^\otimes) \ar[d] \\
\mathrm{coAlg}^c(\tau_{n-1}\mathcal{O}^\otimes_{\leq n}) \ar[r] & \mathrm{coAlg}^c((\tau_{n-1}\mathcal{O}^\otimes)_{=n}) 
}
\]
is a pullback of $\infty$-categories. Recall that we denote the pullback in this square by $\mathcal{B}$. Let us describe a functor $\alpha: \mathcal{B} \rightarrow \mathrm{coAlg}^c(\tau_n\mathcal{O}^\otimes_{\leq n})$ which we will show to be inverse to $\beta$. A vertex $v$ of $\mathcal{B}$ corresponds to a diagram
\[
\xymatrix{
\mathcal{W}^\otimes \ar[d] & \mathcal{U}^\otimes \ar[l]\ar[d]\ar[r] & \mathcal{V}^\otimes \ar[d] \\
\mathcal{O}^\otimes_{=n} & (\tau_{n-1}\mathcal{O}^\otimes)_{=n} \ar[r]\ar[l] & \tau_{n-1}\mathcal{O}^\otimes_{\leq n}
}
\]
in which the vertical arrows are coalgebras. Forming pushouts of the top and bottom rows of we obtain a map $\alpha(v): \mathcal{X}^\otimes \rightarrow \tau_n\mathcal{O}^\otimes_{\leq n}$, where we applied Corollary \ref{cor:wnXstable} to identify the codomain. It is straightforward to verify that $\alpha(v)$ is a coalgebra in $\tau_n\mathcal{O}^\otimes_{\leq n}$. Clearly the construction of $\alpha$ can be made natural to yield the desired map $\mathcal{B} \rightarrow \mathrm{coAlg}^c(\tau_n\mathcal{O}^\otimes_{\leq n})$.

Now let $p^\otimes: \mathcal{X}^\otimes \rightarrow \tau_n\mathcal{O}^\otimes$ be a compact coalgebra. The same argument used to prove Proposition \ref{prop:wnXpullback} proves that the following square is a pullback of $\infty$-operads:
\[
\xymatrix{
(\tau_{n-1}\mathcal{O}^\otimes)_{=n} \ar[d]\ar[r] & \tau_{n-1}\mathcal{O}^\otimes_{\leq n} \ar[d] \\
\mathcal{O}^\otimes_{=n} \ar[r] & \tau_n\mathcal{O}^\otimes_{\leq n}.
}
\]
Pulling back $p^\otimes$ along the maps in this square we produce a cube 
\[
\xymatrix{
(\tau_{n-1}\mathcal{X}^\otimes)_{=n} \ar[rr]\ar[dd]\ar[dr] && \tau_{n-1}\mathcal{X}^\otimes_{\leq n} \ar[dd]\ar[dr] & \\
& \mathcal{X}^\otimes_{=n} \ar[dd] \ar[rr] && \mathcal{X}^\otimes_{\leq n} \ar[dd] \\
(\tau_{n-1}\mathcal{O}^\otimes)_{=n} \ar'[r][rr]\ar[dr]&& \tau_{n-1}\mathcal{O}^\otimes_{\leq n} \ar[dr] &  \\
& \mathcal{O}^\otimes_{=n} \ar[rr] && \tau_n\mathcal{O}^\otimes_{\leq n},
}
\]
where we used Lemma \ref{lem:Xtruncated} to identify the $\infty$-operads making up the top square. Then $\beta(p^\otimes)$ is described by the diagram obtained from this cube by deleting the vertices $\mathcal{X}^\otimes_{\leq n}$ and $\tau_n\mathcal{O}^\otimes_{\leq n}$. All faces of the cube are pullbacks; moreover, the top and bottom faces are pushouts by Corollary \ref{cor:wnXstable}. These observations imply that the composites $\alpha\circ\beta$ and $\beta \circ \alpha$ are naturally equivalent to the identity, proving the lemma.
\end{proof}

We conclude this section with a result relating the $\infty$-category of coalgebras in a stable $\infty$-operad $\mathcal{O}^\otimes$ to the $\infty$-categories of coalgebras in its truncations $\tau_n\mathcal{O}^\otimes$. Recall that the morphism $\tau_n\mathcal{O}^\otimes \rightarrow \mathcal{O}^\otimes$ induces a functor
\begin{equation*}
\mathrm{coAlg}(\mathcal{O}^\otimes) \rightarrow \mathrm{coAlg}(\tau_n\mathcal{O}^\otimes)
\end{equation*}
via Construction \ref{constr:pullbackcoalg}.

\begin{lemma}
\label{lem:coalglimit}
The functor
\begin{equation*}
\mathrm{coAlg}(\mathcal{O}^\otimes) \rightarrow \varprojlim_n \mathrm{coAlg}(\tau_n\mathcal{O}^\otimes)
\end{equation*}
is an equivalence of $\infty$-categories.
\end{lemma}
\begin{proof}
We defined the $\infty$-category $\mathrm{coAlg}(\mathcal{O}^\otimes)$ as a full subcategory of $\bigl((\mathbf{Cat}_\infty)_{/\mathcal{O}^\otimes}\bigr)^{\mathrm{op}}$ and as in Remark \ref{rmk:truncatedcoalg} the $\infty$-category $\mathrm{coAlg}(\tau_n\mathcal{O}^\otimes)$ is equivalent to a full subcategory of $\bigl((\mathbf{Cat}_\infty)_{/\mathcal{O}^{\otimes}_{\leq n}}\bigr)^{\mathrm{op}}$. Identified in this way, the functor described right before the lemma is induced by the functor
\begin{equation*}
(\mathbf{Cat}_\infty)_{/\mathcal{O}^\otimes} \rightarrow (\mathbf{Cat}_\infty)_{/\mathcal{O}_{\leq n}^\otimes}
\end{equation*}
which takes the pullback along the inclusion $\mathcal{O}^\otimes_{\leq n} \rightarrow \mathcal{O}^\otimes$. The lemma now follows from the fact that 
\begin{equation*}
(\mathbf{Cat}_\infty)_{/\mathcal{O}^{\otimes}} \rightarrow \varprojlim_n \, (\mathbf{Cat}_\infty)_{/\mathcal{O}^{\otimes}_{\leq n}}
\end{equation*}
is an equivalence of $\infty$-categories (see Lemma \ref{lem:inftycatlimit} below) and that this equivalence identifies the appropriate full subcategories, as one easily verifies.
\end{proof}

The following is a version of the elementary observation that to give a fibration over a filtered space $X = \cup_n X_n$ is essentially the same as to give a fibration over each stage $X_n$ compatible with the inclusions $X_n \rightarrow X_{n+1}$:

\begin{lemma}
\label{lem:inftycatlimit}
If $\mathcal{C}$ is an $\infty$-category with a filtration by subcategories
\begin{equation*}
\mathcal{C}_1 \subseteq \mathcal{C}_2 \subseteq \mathcal{C}_3 \subseteq \cdots, \quad\quad \bigcup_n \mathcal{C}_n = \mathcal{C},
\end{equation*}
then the functor
\begin{equation*}
(\mathbf{Cat}_\infty)_{/\mathcal{C}} \rightarrow \varprojlim_n \, (\mathbf{Cat}_\infty)_{/\mathcal{C}_n}
\end{equation*}
is an equivalence of $\infty$-categories.
\end{lemma}
\begin{proof}
The $\infty$-category $(\mathbf{Cat}_\infty)_{/\mathcal{C}}$ is equivalent to the homotopy-coherent nerve of a simplicial category $\mathbf{C}$ described as follows: its objects are categorical fibrations $X \xrightarrow{p} \mathcal{C}$ (which are precisely the fibrant objects in the Joyal model structure on the slice category $\mathbf{sSets}/\mathcal{C}$) and for two such fibrations $X \xrightarrow{p} \mathcal{C}$ and $Y \xrightarrow{q} \mathcal{C}$ the simplicial set of maps between them is the maximal Kan complex in the $\infty$-category
\begin{equation*}
\mathrm{Map}_\mathcal{C}(X, Y) = \Delta^0 \times_{\mathcal{C}^X} Y^X,
\end{equation*}
where the map $\Delta^0 \rightarrow \mathcal{C}^X$ used to define the pullback simply picks out the vertex $p$. The fact that $\mathrm{Map}_\mathcal{C}(X, Y)$ is indeed an $\infty$-category follows from the fact that the Joyal model structure is Cartesian. There is for every $n$ a similar simplicial category $\mathbf{C}_n$ whose nerve is equivalent to $(\mathbf{Cat}_\infty)_{/\mathcal{C}_n}$. Pullback along the inclusion $\mathcal{C}_n \rightarrow \mathcal{C}_{n+1}$ defines a simplicial functor
\begin{equation*}
\mathbf{C}_{n+1} \rightarrow \mathbf{C}_n
\end{equation*}
which is easily checked to be a fibration of simplicial categories (i.e. it induces Kan fibrations on mapping spaces). The evident functor
\begin{equation*}
\mathbf{C} \rightarrow \varprojlim_n \mathbf{C}_n
\end{equation*}
is an equivalence of simplicial categories in a strict sense, meaning it induces \emph{isomorphisms} on mapping spaces rather than just weak equivalences. Indeed, an explicit inverse is given by the colimit functor
\begin{equation*}
\{X_n \xrightarrow{p_n} \mathcal{C}_n\}_{n \geq 1} \longmapsto (\varinjlim_n X_n \rightarrow \varinjlim_n \mathcal{C}_n = \mathcal{C})
\end{equation*}
with its obvious simplicial structure (using that pullbacks of simplicial sets commute with colimits). Taking homotopy-coherent nerves produces a diagram (equivalent to)
\begin{equation*}
\cdots \rightarrow (\mathbf{Cat}_\infty)_{/\mathcal{C}_{n+1}} \rightarrow (\mathbf{Cat}_\infty)_{/\mathcal{C}_n} \rightarrow \cdots \rightarrow (\mathbf{Cat}_\infty)_{/\mathcal{C}_1}
\end{equation*}
in which all maps are categorical fibrations, since the homotopy-coherent nerve is a right Quillen functor. Therefore the actual limit is also the homotopy limit and we conclude that the limit
\begin{equation*}
\varprojlim_n \, (\mathbf{Cat}_\infty)_{/\mathcal{C}_n}
\end{equation*}
computed in the $\infty$-category of (large) $\infty$-categories is equivalent to the homotopy-coherent nerve of $\mathbf{C}$, which in turn is equivalent to $(\mathbf{Cat}_\infty)_{/\mathcal{C}}$.
\end{proof}


\backmatter
\bibliographystyle{amsalpha}
\bibliography{biblio}

\providecommand{\bysame}{\leavevmode\hbox to3em{\hrulefill}\thinspace}
\providecommand{\MR}{\relax\ifhmode\unskip\space\fi MR }
\providecommand{\MRhref}[2]{%
  \href{http://www.ams.org/mathscinet-getitem?mr=#1}{#2}
}
\providecommand{\href}[2]{#2}
\begin{thebibliography}{HHM16}

\bibitem[AC11]{aroneching}
G.~Arone and M.~Ching, \emph{Operads and chain rules for the calculus of
  functors}, Soci{\'e}t{\'e} math{\'e}matique de {F}rance, 2011.

\bibitem[AC15]{aronechingtowers}
G.~Arone and M.~Ching, \emph{A classification of {T}aylor towers of functors of
  spaces and spectra}, Advances in Mathematics \textbf{272} (2015), 471--552.

\bibitem[AK98]{aronekankaanrinta}
G.~Arone and M.~Kankaanrinta, \emph{A functorial model for iterated {S}naith
  splitting with applications to calculus of functors, from: `stable and
  unstable homotopy'}, Fields Inst. Commun., vol.~19, A.M.S., 1998.

\bibitem[AK02]{ahearnkuhn}
S.~T. Ahearn and N.~J. Kuhn, \emph{Product and other fine structure in
  polynomial resolutions of mapping spaces}, Algebraic \& Geometric Topology
  \textbf{2} (2002), no.~2, 591--647.

\bibitem[BM05]{basterramandell}
M.~Basterra and M.A. Mandell, \emph{Homology and cohomology of ${E}_{\infty}$
  ring spectra}, Mathematische Zeitschrift \textbf{249} (2005), no.~4,
  903--944.

\bibitem[BR14]{biedermannroendigs}
G.~Biedermann and O.~R\"ondigs, \emph{Calculus of functors and model
  categories, {II}}, Algebr. Geom. Topol. \textbf{14} (2014), no.~5,
  2853--2913.

\bibitem[BS16]{barwickshah}
C.~Barwick and J.~Shah, \emph{Fibrations in $\infty$-category theory},
  Available online at arXiv:1607.04343, 2016.

\bibitem[CHH]{chh}
H.~Chu, R.~Haugseng, and G.S.K.S. Heuts, \emph{Two models for the homotopy
  theory of infinity-operads}, arXiv: 1606.03826, to appear in Journal of
  Topology.

\bibitem[Chi]{chingderivatives}
M.~Ching, \emph{Infinity-operads and {D}ay convolution in {G}oodwillie
  calculus}, Available online at arXiv:1801.03467.

\bibitem[Chi05]{ching}
M.~Ching, \emph{Bar constructions for topological operads and the {G}oodwillie
  derivatives of the identity}, Geom. Topol. \textbf{9} (2005), 833--933.

\bibitem[Chi12]{chingbar}
M.~Ching, \emph{Bar-cobar duality for operads in stable homotopy theory},
  Journal of {T}opology \textbf{5} (2012), no.~1, 39--80.

\bibitem[CM11]{cisinskimoerdijk1}
D.C. Cisinski and I.~Moerdijk, \emph{Dendroidal sets as models for homotopy
  operads}, Journal of {T}opology \textbf{4} (2011), no.~2, 257--299.

\bibitem[CM13]{cisinskimoerdijk3}
\bysame, \emph{Dendroidal sets and simplicial operads}, Journal of {T}opology
  \textbf{6} (2013), no.~3, 705--756.

\bibitem[Eld16]{eldred}
R.~Eldred, \emph{Goodwillie calculus via adjunction and {L}{S} cocategory},
  Homology, {H}omotopy and {A}pplications \textbf{18} (2016), no.~2, 31--58.

\bibitem[FG12]{francisgaitsgory}
J.~Francis and D.~Gaitsgory, \emph{Chiral {K}oszul duality}, Selecta
  Mathematica \textbf{18} (2012), no.~1, 27--87.

\bibitem[GJ]{getzlerjones}
E.~Getzler and J.D.S. Jones, \emph{Operads, homotopy algebra and iterated
  integrals for double loop spaces}, Available online at arXiv hep-th/9305013.

\bibitem[GK94]{ginzburgkapranov}
V.~Ginzburg and M.~Kapranov, \emph{Koszul duality for operads}, Duke
  {M}athematical {J}ournal \textbf{76} (1994), no.~1, 203--272.

\bibitem[GM95]{greenleesmay}
J.P.C. Greenlees and L.P. May, \emph{Generalized {T}ate cohomology}, Mem. Amer.
  Math. Soc. \textbf{113} (1995), no.~543.

\bibitem[Goo03]{goodwillie3}
T.~G. Goodwillie, \emph{Calculus {III}: {T}aylor {S}eries}, Geometry and
  Topology \textbf{7} (2003), 645--711.

\bibitem[GS96]{greenleessadofsky}
J.P.C. Greenlees and H.~Sadofsky, \emph{The {T}ate spectrum of $v_n$-periodic
  complex oriented theories}, Mathematische Zeitschrift \textbf{222} (1996),
  no.~3, 391--405.

\bibitem[Heu]{unstableperiodicity}
G.S.K.S. Heuts, \emph{Lie algebras and $v_n$-periodic spaces}, Available online
  at arXiv:1803.06325.

\bibitem[HHM16]{hhm}
G.S.K.S. Heuts, V.~Hinich, and I.~Moerdijk, \emph{On the equivalence between
  {L}urie's model and the dendroidal model for infinity-operads}, Advances in
  Mathematics \textbf{302} (2016), 869--1043.

\bibitem[HL13]{ambidexterity}
M.~J. Hopkins and J.~Lurie, \emph{Ambidexterity in {K}(n)-local stable homotopy
  theory}, Available online at math.harvard.edu/~lurie/, 2013.

\bibitem[Joy02]{joyalpaper}
A.~Joyal, \emph{Quasi-categories and {K}an complexes}, Journal of Pure and
  Applied Algebra \textbf{175} (2002), no.~1, 207--222.

\bibitem[Joy08]{joyal}
\bysame, \emph{The theory of quasi-categories {I}}, preprint (2008).

\bibitem[Kle02]{klein}
J.~R. Klein, \emph{Axioms for generalized {F}arrell-{T}ate cohomology}, Journal
  of Pure and Applied Algebra \textbf{172} (2002), no.~2, 225--238.

\bibitem[Kle05]{kleinmoduli}
J.R. Klein, \emph{Moduli of suspension spectra}, Transactions of the American
  Mathematical Society \textbf{357} (2005), no.~2, 489--507.

\bibitem[KP14]{kleinpeter}
J.R. Klein and J.~Peter, \emph{Fake wedges}, Transactions of the American
  Mathematical Society \textbf{366} (2014), no.~7, 3771--3786.

\bibitem[Kuh04]{kuhntate}
N.~J. Kuhn, \emph{Tate cohomology and periodic localization of polynomial
  functors}, Inventiones {M}athematicae \textbf{157} (2004), no.~2, 345--370.

\bibitem[Kuh07]{kuhngoodwillie}
\bysame, \emph{Goodwillie towers and chromatic homotopy: an overview}, Geometry
  and Topology Monographs \textbf{10} (2007), 245--279.

\bibitem[Kuh08]{kuhntelescopic}
\bysame, \emph{A guide to telescopic functors}, Homology, {H}omotopy and
  {A}pplications \textbf{10} (2008), no.~3, 291--319.

\bibitem[LNR12]{LNR}
S.~Lun{\o}e-Nielsen and J.~Rognes, \emph{The topological {S}inger
  construction}, Documenta Mathematica \textbf{17} (2012), 861--909.

\bibitem[Lur09]{htt}
J.~Lurie, \emph{Higher topos theory}, vol. 170, Princeton University Press,
  2009.

\bibitem[Lur14]{higheralgebra}
\bysame, \emph{Higher algebra}, Available online at math.harvard.edu/~lurie/,
  2014.

\bibitem[McC99]{mccarthy}
R.~McCarthy, \emph{Dual calculus for functors to spectra}, Contemp. Math.
  Series, vol. 271, pp.~183--215, A.M.S., 1999.

\bibitem[MW09]{moerdijkweiss}
I.~Moerdijk and I.~Weiss, \emph{On inner {K}an complexes in the category of
  dendroidal sets}, Advances in Mathematics \textbf{221} (2009), no.~2,
  343--389.

\bibitem[Per13a]{pereira}
L.~Pereira, \emph{A general context for {G}oodwillie calculus}, Available
  online at arXiv:1301.2832, 2013.

\bibitem[Per13b]{pereirathesis}
L.~Pereira, \emph{Goodwillie calculus and algebras over a spectral operad},
  Ph.D. thesis, MIT, 2013.

\bibitem[Qui69]{rationalhomotopy}
D.~G. Quillen, \emph{Rational homotopy theory}, Annals of Mathematics (1969),
  205--295.

\end{thebibliography}
\printindex

\end{document}